\documentclass[10pt,a4paper,dvipsnames]{article}

\usepackage{fullpage,amssymb,amsmath,amsfonts,amsthm,paralist,graphicx,color,float, mathtools,tikz}
\usepackage[utf8]{inputenc}
\usepackage{here}
\usepackage[normalem]{ulem}
\usepackage{centernot}
\usepackage{ifthen,xkeyval, tikz, calc, graphicx}
\usepackage{xcolor}
\usepackage{framed}
\usepackage[textsize=scriptsize, backgroundcolor=blue!20!white,bordercolor=red]{todonotes}
\usepackage{dsfont}
\usepackage{mathrsfs}
\usepackage{comment}
\usepackage{stmaryrd}
\usepackage[nice]{nicefrac}
\usepackage[colorlinks=true,linkcolor=RedViolet,citecolor=blue]{hyperref}

\newtheorem{theorem}{Theorem}[section]
\newtheorem{lemma}[theorem]{Lemma}
\newtheorem{prop}[theorem]{Proposition}
\newtheorem{corollary}[theorem]{Corollary}
\newtheorem{observation}[theorem]{Observation}

\newtheorem{remark}[theorem]{Remark}
\theoremstyle{definition}
\newtheorem{definition}[theorem]{Definition}

\definecolor{shadecolor}{named}{GreenYellow}

\newcommand{\al}[1]{\begin{align*}#1\end{align*}}

\newcommand{\algn}[1]{\begin{align}#1\end{align}}
\newcommand{\eqq}[1]{\begin{equation}#1\end{equation}}

\newcommand{\p}{\mathbb P}
\newcommand{\pla}{\mathbb P_\lambda}
\newcommand{\E}{\mathbb E}
\newcommand{\R}{\mathbb R}
\newcommand{\Rd}{\mathbb R^d}

\newcommand{\N}{\mathbb N}
\newcommand{\B}{\mathcal B}
\newcommand{\dd}{\, \mathrm{d}}
\newcommand{\C}{\mathscr {C}}
\newcommand{\thinn}[1]{\langle #1 \rangle}
\newcommand{\piv}[1]{\textsf {Piv}(#1)}
\newcommand{\conn}[3]{#1 \longleftrightarrow #2\textrm { in } #3}

\newcommand{\nconn}[3]{#1 \centernot\longleftrightarrow #2\textrm { in } #3}
\newcommand{\dconn}[3]{#1 \Longleftrightarrow #2\textrm { in } #3}
\newcommand{\ndconn}[3]{#1 \centernot\Longleftrightarrow #2\textrm { in } #3}
\newcommand{\xconn}[4]{#1 \xleftrightarrow{\,\,#4\,\,} #2\textrm { in } #3}
\newcommand{\offconn}[4]{#1 \longleftrightarrow #2\textrm { in } #3 \textrm{ off } #4}
\newcommand{\viaconn}[4]{#1 \longleftrightarrow #2\textrm { in } #3 \textrm{ through } #4}
\newcommand{\sqconn}[3]{#1 \leftrightsquigarrow #2\textrm { in } #3 }
\newcommand{\sqarrow}{\leftrightsquigarrow}

\newcommand{\thinning}[2]{#1_{\langle #2 \rangle}}
\newcommand{\tlam}{\tau_\lambda}
\newcommand{\tlamo}{\tau_\lambda^\circ}
\newcommand{\tlame}{\tlam^{(\varepsilon)}}

\newcommand{\tillam}{\tilde\tau_\lambda}
\newcommand{\tklam}{\tau_{\lambda,k}}
\newcommand{\tilklam}{\tilde\tau_{\lambda,k}}
\newcommand{\ftlam}{\widehat\tau_\lambda}
\newcommand{\ftillam}{\widehat{\tilde\tau}_\lambda}
\newcommand{\ftklam}{\widehat\tau_{\lambda,k}}

\newcommand{\fpilam}{\widehat\Pi_\lambda}

\newcommand{\vardbtilde}[1]{\tilde{\raisebox{0pt}[0.85\height]{$\tilde{#1}$}}}
\newcommand{\dtilklam}{\vardbtilde{\tau}_{\lambda,k}}
\newcommand{\fdtilklam}{\widehat{\vardbtilde{\tau}}_{\lambda,k}}
\newcommand{\trilam}{\triangle_\lambda}
\newcommand{\trilamo}{\triangle^\circ_\lambda}
\newcommand{\trilamoo}{\triangle^{\circ\circ}_\lambda}
\newcommand{\trilame}{\triangle^{(\varepsilon)}_\lambda}
\newcommand{\greensmu}{G_\mu}
\newcommand{\gmu}{G_{\mulam}}
\newcommand{\fgreensmu}{\widehat G_\mu}
\newcommand{\fgmu}{\widehat G_{\mulam}}

\newcommand{\cmu}{C_{\mulam}}
\newcommand{\fcreensmu}{\widehat C_\mu}
\newcommand{\fcmu}{\widehat C_{\mulam}}

\newcommand{\phiint}{q_\connf}
\newcommand{\phiintL}{q_{\connf_L}}
\newcommand{\mulam}{\mu_\lambda}
\newcommand{\orig}{\mathbf{0}}
\newcommand{\e}{\text{e}}
\renewcommand{\i}{\text{i}}
\newcommand{\parl}{\mathrel{\raisebox{-0.05 cm}{\includegraphics{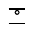}}} \hspace{-0.05cm}}
\newcommand{\connf}{\varphi}

\newcommand{\fconnf}{\widehat\connf}

\newcommand{\unitball}{\mathbb B^d}
\newcommand{\ballep}{\mathsf B^{(\varepsilon)}}
\newcommand{\const}{c_f}

\definecolor{darkorange}{RGB}{255,165,0}
\definecolor{altviolet}{RGB}{139,0,139}
\definecolor{turquoise}{RGB}{64,224,208}
\definecolor{lblue}{RGB}{173,216,230}
\definecolor{violet}{RGB}{238,130,238}
\definecolor{darkgreen}{RGB}{0,100,0}

\newcommand{\twodoin}[1]{\todo[inline, backgroundcolor=green!20!white,bordercolor=red]{#1}}

\newcommand{\col}[1]{#1}
\newcommand{\ch}[1]{#1}
\newcommand{\chr}[1]{#1}
\newcommand{\cov}[1]{#1}
\newcommand{\red}[1]{#1}

\newcommand{\eqn}[1]{\begin{equation}#1\end{equation}}
\newcommand{\eqan}[1]{\begin{align}#1\end{align}}
\newcommand{\nn}{\nonumber}
\newcommand{\sss}{\scriptscriptstyle}

\numberwithin{equation}{section}
\allowdisplaybreaks

\title{Lace Expansion and Mean-Field Behavior\linebreak for the Random Connection Model}
\author{Markus Heydenreich\footnote{Universität Augsburg, Institut für Mathematik, 86135 München, Germany; Email: markus.heydenreich@uni-a.de} \and Remco van der Hofstad\footnote{Eindhoven University of Technology, Department of Mathematics and Computer Science, P.O. Box 513, 5600 MB Eindhoven, The Netherlands; Email: r.w.v.d.hofstad@tue.nl} \and G\"unter Last\footnote{Karlsruher Institut für Technologie (KIT), Institut für Stochastik, Postfach 6980, 76049 Karlsruhe, Germany; Email: guenter.last@kit.edu} \and Kilian Matzke\footnote{Ludwig-Maximilians-Universität München, Mathematisches Institut, Theresienstr.\ 39, 80333 München, Germany; Email: kilianmatzke@web.de}
}
 
\date{}

\begin{document}
\maketitle

\vspace{-1em}

{\centering{ June 7, 2022}\par}

\vskip-3em

\begin{abstract}
We study the random connection model 
driven by a stationary Poisson process.
In the first part of the paper, we derive a 
lace expansion with remainder term in the continuum and bound the coefficients using
a new version of the BK inequality. For our main results, 
we consider three versions of the
  connection function $\connf$: a finite-variance version (including
  the Boolean model), a spread-out version, and a long-range
  version. For sufficiently large dimension (resp., spread-out parameter and $d>6$), we then prove the convergence of the lace expansion,
derive the triangle condition, and
  establish an infra-red bound. From this, mean-field behavior of the
  model can be deduced. As an example, we show that the critical
  exponent $\gamma$ takes its mean-field value $\gamma=1$ and that the
  percolation function is continuous.
\end{abstract}



{\footnotesize
\tableofcontents
}

\section{Introduction}
\subsection{The random connection model}
Consider a stationary Poisson point process (PPP) $\eta$ on $\Rd$ with intensity $\lambda \geq 0$
along with a measurable \emph{connection function}
$\connf\colon \Rd \to [0,1]$. 
We assume that
\eqq{\connf(x) = \connf(-x)\qquad (x\in\Rd),
 \label{eq:phisymmetry}} 
as well as 
\eqq{\col{0 < \phiint := \int\connf(x) \dd x < \infty}. \label{eq:phisummability}} 
\ch{A careful reader might observe that  without loss of generality  we could set $\phiint=1$ by rescaling space, in fact we make use of this in Section \ref{sec:rescaling}. However, since such a normalization is not common in stochastic geometry, we state our results for general $\phiint$.} 
Suppose any two distinct points $x,y \in\eta$ form an edge with probability
$\connf(y-x)$ independently of all other pairs and independently of
$\eta$. This yields the random connection model (RCM), an undirected
random graph denoted by $\xi$, whose vertex set $V(\xi)$ is $\eta$ and
whose edge set we denote by $E(\xi)$. We point to
Section~\ref{sec:literature} for a brief literature overview of the
RCM.

We stress the difference between $\eta$ and $\xi$, as the former is
used to denote a random set of \emph{points}, whereas the latter is a
\emph{random graph}, which contains complete information about $\eta$
as well as the additional information about edges between the points
of $\eta$. It is convenient to define $\xi$ on a probability space
$(\Omega, \mathcal F, \pla)$ \cov{with associated expectation operator
$\E_\lambda$} and to treat $\lambda$ as a parameter.

For $x,y \in \Rd$, we use $\eta^x$ and $\xi^x$ (respectively,
$\eta^{x,y}$ and $\xi^{x,y}$) to denote a PPP and an RCM
augmented by $x$ (respectively, by $x$ and $y$). For $\xi$, this
includes the random set of edges incident to $x$ or $\{x,y\}$, with
the edge probabilities governed by $\connf$. We refer to
Section~\ref{sec:constructionofRCM} for a formal construction of the
model.

For $x,y \in \Rd$, we write $x \sim y$ if $\{x,y\} \in E(\xi)$ and say that $x$ and $y$
are \emph{neighbors} (or \emph{adjacent}). We say 
that $x$ and $y$ are \emph{connected} and write $\conn{x}{y}{\xi}$ if
either $x=y$ or if there is a \emph{path} in $\xi$ connecting $x$ and
$y$---that is, there are distinct
$x=v_0,v_1, \ldots, v_k, v_{k+1}=y \in\eta$ (with $k\in\N_0=\N\cup\{0\}$) such
that $v_i \sim v_{i+1}$ for $0 \leq i \leq k$. 
Given $x,y\in\R^d$, we shall use this definition also with $(\eta^x,\xi^x)$ (resp.\ $\eta^{x,y},\xi^{x,y}$) in place of $(\eta,\xi)$.
In particular we can then talk about the event $\{\conn{x}{y}{\xi^{x,y}}\}$.
For $x\in\Rd$, we define $\C(x) = \C(x, \xi^x)= \{y \in \eta^x: \conn{x}{y}{\xi^x} \}$ to be
the \emph{cluster} of $x$.

We define the pair-connectedness (or two-point) function $\tlam\colon \Rd \to [0,1]$ to be
	\eqq{\tlam(x) = \pla (\conn{\orig}{x}{\xi^{\orig,x}}), \label{eq:two-point-fct}}
where $\orig$ denotes the origin in $\Rd$. The two-point function can also be defined as $\tlam\colon\Rd\times\Rd\to[0,1]$ via $\tlam(x,y) = \pla (\conn{x}{y}{\xi^{x,y}})$. The two definitions relate by translation invariance, as $\tlam(x,y) = \tlam(x-y)$, and we stick to~\eqref{eq:two-point-fct} throughout this paper. Observe that $\lambda \mapsto \tlam(x)$ is \chr{non-decreasing.} We next define the \emph{percolation function} $\lambda \mapsto \theta(\lambda)$ as
	\[\theta(\lambda) = \pla(|\mathscr C(\orig)| = \infty ), \]
where $| \mathscr C(x)|$ denotes the number of points in $\mathscr C(x)$. Note that $|\C(\orig)|$ has the same distribution as $|\C(x)|$ for any $x \in\Rd$ due to translation invariance. We next define the \emph{critical value} for the RCM as
	\[\lambda_c = \inf \{ \lambda \geq 0: \theta(\lambda)>0\}.\]
To state our main theorem, for an (absolutely) integrable function $f\colon\Rd\to \R$, we define the Fourier transform of $f$ to be
	\[\widehat f(k) = \int \e^{\i k\cdot x} f(x) \dd x \qquad (k\in\Rd),\]
where $k\cdot x = \sum_{j=1}^{d} k_j x_j$ denotes the scalar product \chr{and the integral is over $\R^d$}. We next define the \emph{expected cluster size} as
	\eqq{ \chi(\lambda) := \E_\lambda[|\C(\orig)|]= 1+ \lambda \int \tlam(x) \dd x = 1+ \lambda \ftlam(\orig), \label{eq:ftlam_chi_relation} } 
where $\orig$ also denotes the origin in the Fourier dual (which is also $\Rd$). The second, elementary but helpful, identity is proved in~\eqref{eq:prelim:ftlam_chi_relation_proof}. This allows us to define
	\[ \lambda_T := \sup\{ \lambda \geq 0: \chi(\lambda) < \infty\},\]
and to point out that $\lambda_T=\lambda_c$ (proven by Meester~\cite{Mee95}).

Let us now specify $\connf$. \col{In the finite-variance model defined below,} our goal is to obtain a result valid for all dimensions $d \geq d_0$ for some $d_0$, \col{and so} we are interested in a sequence $(\connf_d)_{d\in\N}$, where $\connf_d\colon\Rd \to [0,1]$. Meester et al.~\cite{MeePenSar97} demonstrate a simple way to do this: They take a function $\tilde\connf\colon\R_{\geq 0} \to [0,1]$ and model the RCM such that two points $x,y\in\Rd$ are connected with probability $\tilde\connf(|x-y|)$, where $|\cdot |$ denotes Euclidean distance. The integrability condition for $\connf$ then becomes $\int t^{d-1} \tilde\connf(t) \dd t < \infty$. 

\subsection{Conditions on the connection function}
In the first four sections we will not impose
further restrictions on the connection function.
But for our main results (Theorems \ref{thm:maintheorem} and \ref{thm:mainthmcorollaries} below) we consider, similarly to Heydenreich et al.~\cite{HeyHofSak08}, three
classes (or versions) for $\connf\colon\Rd \to [0,1]$. Those are
the ``finite-variance'' case, the ``finite-variance spread-out'' (or
simply ``spread-out'') case, and the ``long-range spread-out'' (or
simply ``long-range'') case. The second version comes with an extra
spread-out parameter $L$, the third with $L$ as well as another
parameter $\alpha$, controlling the power-law decay of $\connf$. For
each of these three cases, we make several assumptions. 
We give at least one example for each of the three versions.

\paragraph{(H1): Finite-variance model.} We require $\connf$ to satisfy the following three properties: 
\begin{itemize}
\item[(H1.1)] The density $\connf\equiv \connf_d$ has a second moment, i.e. $\int |x|^2 \connf(x) \dd x<\infty$ for all $d$. 
\item[(H1.2)] There is a function $g\colon\mathbb N \to \R_{\geq 0}$ 
\chr{converging to zero as $d\to\infty$ slowly enough so that there exists $\zeta\in(0,1)$ with $g(d)\ge\zeta^d$ and such that the following is true. 
}
For $m \geq 3$, the $m$-fold convolution $\connf^{\star m}$ of $\connf$ satisfies
\[ \red{ \phiint^{-(m-1)} \sup_{x\in\Rd}\connf^{\star m}(x) \leq g(d)} \] 
(we point to the notational remarks in Section~\ref{sec:notation} for a definition of the $m$-fold convolution). For $m=2$, we make the (weaker) assumption that there exists some $\varepsilon$ with $0 \leq \varepsilon < r_d := \pi^{-1/2} \Gamma(\tfrac d2+1)^{1/d}$ (i.e.,~$r_d$ is the radius of the ball of volume $1$) such that
\[  \red{  \phiint^{-1} \sup_{x: |x| \geq \varepsilon\phiint^{1/d}} (\connf\star\connf) (x) \leq g(d)}. \]
\item[(H1.3)] \ch{There are constants $b,c_1,c_2>0$ (independent of $d$) such that} the Fourier transform $\fconnf$ of $\connf$ satisfies 
	\eqq{
	\ch{\inf_{|k|\leq b/\red{\phiint^{1/d}}}  |k|^{-2}\,(1-\phiint^{-1}\fconnf(k)) >c_1 \label{eq:Boolean_D_infraredbound},}}
and
	\eqq{
	\ch{
	\inf_{|k|> b/\red{\phiint^{1/d}}} (1-\phiint^{-1}\fconnf(k)) > c_2 \label{eq:Boolean_D_infraredbound2}.}
	}
\end{itemize}
The assumption (H1.2) will make sure that our results below apply in all \chr{dimensions} $d$ that are larger than some unspecified, but large, value. 

We now present two examples of functions that satisfy (H1). The first
is $\connf(x) = \mathds 1_{\{|x| \leq r\}}$ for some $r>0$, which corresponds to the Poisson
blob model or (spherical) Boolean model. It is the most prominent
example of a continuum-percolation model and in some sense the easiest
of the ones commonly investigated.

Another very natural connection function is the density of a $d$-dimensional (independent) Gaussian, given in its standardized form by $\connf(x)=(2\pi)^{-d/2} \exp(-|x|^2/2)$. Even though the Gaussian case, very much unlike the Poisson blob model, is supported on the whole space $\Rd$, its light tail allows us to treat it in the same way.

\paragraph{(H2): Spread-out model.}
In this version, we introduce a new parameter $L \geq 1$, upon which $\connf=\connf_L$ depends. It describes the range of the model and will be taken to be large. 
Later the dimension $d$ will be taken to be larger than some specified value, depending on the precise model setting. 
We make the following assumptions:

\begin{itemize}
\item[(H2.1)] For every $L \geq 1$, the second moment exists, i.e.~$\int |x|^2 \connf_L(x) \dd x < \infty$.
\item[(H2.2)] There exists a constant $C$ such that, for all $L \geq 1$, $\| \connf_L \|_\infty \leq C L^{-d}$, where $\| \connf_L \|_\infty = \sup_{x\in\Rd} \connf_L(x)$.
\item[(H2.3)] There are constants $b,c_1, c_2>0$ (independent of $L$) such that
	\algn{1- \col{\phiintL^{-1}} \widehat \connf_L(k) \geq \begin{cases} c_1 L^2 |k|^2 &\mbox{for } |k| \leq b L^{-1}, \notag \\
	 	c_2 &\mbox{for } |k| > b L^{-1}. \end{cases}\label{eq:stepdistFTbounds}}
\end{itemize}

In this setup, we will prove our results for all dimensions $d>6$. This indicates that $d>6$ \emph{is} the right requirement for (H1) as well.

Let us now give an example for (H2). Let \col{$h\colon \Rd \to [0,1]$ } satisfy $0 < \int h(x) \dd x<\infty$ and \col{$h(x) = h(-x)$}. Furthermore, \col{assume $\int |x|^2 h(x) \dd x <\infty$} and that the matrix $\big(\int x_jx_lh(x)\dd x\big)_{j,l\in\{1,\ldots,d\}}$ is regular (i.e., invertible). Then 
	\eqq{\connf_L(x) = \col{L^{-d} h(x/L)}\qquad \text{for } x \in\Rd \label{eq:Dspreadoutdefinition}}
satisfies (H2.1)-(H2.3), see Proposition \ref{thm:prelim:conn_fct_examples}. An explicit example in this class is $h(x) = \mathds 1_{\{|x| \leq r\}}$. Another explicit example is $h(x) = \mathds 1_{\{|x_i| \leq b_i \forall i \in [d]\}}$ for a collection $(b_i)_{i=1}^d$ of positive values.

\paragraph{(H3): Long-range spread-out model.}
We introduce an additional parameter $\alpha>0$ (describing the long-range behavior of $\connf$) so that $\connf=\connf_L=\connf_{L,\alpha}$ depends on both $L$ and $\alpha$. As for (H2), we assume later that $d$ is larger than some specified value. 
The assumptions are now as follows:
\begin{itemize}
\item[(H3.1)] For all $0<\varepsilon \leq\alpha$, we have $\int |x|^{\alpha-\varepsilon} \connf_L(x) \dd x < \infty$.
\item[(H3.2)] =(H2.2).
\item[(H3.3)] There are constants $b, c_1, c_2>0$ (independent of $L$) such that
	\al{1- \col{\phiintL^{-1} \widehat \connf_L(k) \geq \begin{cases} c_1 (L|k|)^{\alpha \wedge 2} &\mbox{for } |k| \leq b L^{-1} \text{ and } \alpha \neq 2, \\
	 	c_1 (L|k|)^2 \log \frac{1}{L |k|} &\mbox{for } |k| \leq b L^{-1} \text{ and } \alpha = 2 \end{cases} }}
and
	\[ 1- \col{\phiintL^{-1}} \widehat \connf_L(k) \geq c_2 \text{ for } |k| > b L^{-1}. \]
\end{itemize}

We introduce (H3) in order to model long-range interactions, where $\connf(x)$ decays as some inverse power of $|x|$ as $|x|\to\infty$. 
Our results apply for all dimensions $d>3(\alpha  \wedge 2)$ (where $\alpha  \wedge 2$ denotes the minimum of $\alpha$ and 2) \ch{and also for $d=6$ if $\alpha=2$}. 
Even though the long-range model is defined for
$\alpha>0$, we remark that the interesting new regime arises from
$\alpha \leq 2$, as $\alpha>2$ does not differ from the spread-out
model (H2). As an example, set \eqq{ \col{h(x) = \frac{1}{( |x| \vee
      1)^{d+\alpha}} \quad\qquad \text{for } x
    \in\Rd,} \label{eq:Dlongrangedefinition}} 
  and define $\connf_L$ as
in \eqref{eq:Dspreadoutdefinition}.

The following proposition verifies the required conditions for the
respective examples. It is proved in Appendix~\ref{sec:appendix}.
\begin{prop}[Verification of conditions for connection-function
  examples] \label{thm:prelim:conn_fct_examples} \
\begin{compactitem}
\item[(a)] There exists $\rho \in (0,1)$ \chr{independent of $d$} such that the Boolean as well
  as the Gaussian \ch{densities} satisfy (H1.1)-(H1.3) with $g(d) = \rho^d$.
\item[(b)] 
The density defined in (and above) 
  \eqref{eq:Dspreadoutdefinition} satisfies (H2.1)-(H2.3).
\item[(c)] The density defined in~\eqref{eq:Dspreadoutdefinition} with
  $h$ as in~\eqref{eq:Dlongrangedefinition} satisfies (H3.1)-(H3.3).
\end{compactitem}
\end{prop}

We set $\alpha=\infty$ for (H1) and (H2). This convention allows us to
refrain from tedious case distinctions in later statements. In order
to state the main theorem, we introduce a parameter $\beta$, dependent
on the version of $\connf$, as \eqq{ \beta := \begin{cases}
    g(d)^{\frac 14} &\mbox{for (H1),} \\ L^{-d} & \mbox{for (H2),
      (H3)}. \end{cases} \label{eq:def:beta}} The function $g$ in the
definition of $\beta$ is the same as in (H1.2). Our methods crucially
depend on the fact that $\beta$ can be made arbitrarily small. Not
only is $\beta$ important for the statement of the main theorem, it
will also appear prominently throughout the paper. Whenever we speak
of small $\beta$ in this paper, we refer to large $d$ for (H1) and to
large $L$ (with \emph{fixed} $d$) for (H2), (H3). In particular,
whenever the Landau notation $\mathcal O(\beta)$ appears, the
asymptotics are $d\to\infty$ for (H1) and $L\to\infty$ for (H2) and
(H3).

\subsection{Main results}

The main contribution of this paper is the establishment of the triangle condition as well as the infra-red bound for the RCM in dimension $d> 3(\alpha \wedge2)$  (\ch{and also for $d=6$ if $\alpha=2$}) and for $\beta$ sufficiently small. This is achieved using the lace expansion. To formulate our main theorem, we define the \emph{triangle} by
	\eqn{
	\label{triangle-defs}
	\triangle_\lambda(x):= \lambda^2\iint \tlam(z)\tlam(y-z)\tlam(x-y) \dd z \dd y, \qquad \text{and } \qquad \triangle_\lambda := \sup_{x\in\Rd} \triangle_\lambda(x).
	}
Then by the \emph{triangle condition}, we mean that $\triangle_{\lambda_c}<\infty$. Our main result is the following theorem. 
\begin{theorem}[Infra-red bound and triangle condition]\label{thm:maintheorem}~
\begin{enumerate}
\item If $\connf$ satisfies (H1), then there is \ch{$d^* >6$} and a
  constant $C$ such that, for all $d\geq d^*$, \eqq{
    \col{\lambda|\ftlam(k)| \leq \frac{|\fconnf(k)| + C
        \beta}{\fconnf(\orig)-\fconnf(k)}} \qquad
    (k\in\Rd), \label{eq:mainthm:IRB}} as well as
  $\trilam \leq C\beta$, and both bounds are uniform in
  $\lambda\in[0,\lambda_c]$ (the right-hand side
  of~\eqref{eq:mainthm:IRB} is understood to be $+\infty$ for
  $k=\orig$).
\item \ch{Let $d>3(\alpha \wedge 2)$ (when $\alpha\neq2$) or $d \geq 6$ (when $\alpha=2$).} If $\connf$ satisfies (H2) or (H3), then there \ch{are} $L^* \geq 1$ and $C$ such that~\eqref{eq:mainthm:IRB} and $\trilam\leq C\beta$ both hold uniformly in $\lambda\in[0,\lambda_c]$ for all $L \geq L^*$.
\end{enumerate}
\end{theorem}
Theorem~\ref{thm:maintheorem} has multiple consequences (some of them are listed in Theorem~\ref{thm:mainthmcorollaries}), such as asymptotics of $\lambda_c$ (as $d \to \infty$ or $L\to\infty$) and continuity of $\lambda \mapsto \theta(\lambda)$. Furthermore, Theorem~\ref{thm:maintheorem} enables us to prove mean-field behavior in the sense that several critical exponents take their mean-field values. For example, the exponent $\gamma$ is the dimension-dependent value that governs the predicted behavior
	\eqq{ \chi(\lambda) \approx (\lambda_c-\lambda)^{-\gamma} \qquad \text{as $\lambda \nearrow \lambda_c$}. \label{eq:def:gamma}}
This definition of $\gamma$ already assumes a certain behavior of $\chi(\lambda)$. It is also predicted that for $d>d_c$, where $d_c$ is the \emph{upper critical dimension} believed to be $d_c=6$ for percolation, $\gamma$ no longer depends on the dimension (it takes its \emph{mean-field value}). We prove that the critical exponent $\gamma$ exists (in a bounded-ratio sense) and takes its mean-field value 1. 
\red{We refer to the recent preprint \cite{CaiceDicks23}, where further critical exponents are derived,} 
and to the book~\cite{HeyHof17}, where other exponents (in bond percolation on the lattice) are discussed.

\begin{theorem}[The critical point and $\gamma=1$]\label{thm:mainthmcorollaries}
Under (H1), there is $d^*>\ch{6}$ such that for all $d \geq d^*$, and under (H2) or (H3) 
\ch{in dimension $d>3(\alpha\wedge2)$ (or $d=6$ if $\alpha=2$)}, there is $L^*\geq 1$ such that for all $L \geq L^*$
	\eqq{ \lambda (\lambda_c-\lambda)^{-1} \leq \chi(\lambda) \leq \lambda(1+C\beta) (\lambda_c-\lambda)^{-1} \qquad \text{for } \lambda < \lambda_c, \label{eq:critexp_gamma_boundedratio}}
that is, the critical exponent $\gamma$ takes its mean-field value $1$. \col{Under (H1), $C=C(d^*)$ and under (H2) or (H3), $C=C(d,L^*)$}. Furthermore, $\theta(\lambda_c) =0$ and $1 \leq \lambda_c \phiint\leq 1+ C\beta$.
\end{theorem}
We note that the proof of Theorem~\ref{thm:mainthmcorollaries} gives the stronger result that the lower bound on $\chi(\lambda)$ in~\eqref{eq:critexp_gamma_boundedratio} is valid in all dimensions and does not require any set of assumptions (H1), (H2), or (H3). This implies $\chi(\lambda_c) = \infty$ for the general RCM provided that \eqref{eq:phisummability} is valid. 

As a consequence of
\col{Corollary~\ref{cor:convergenceoflaceexpansion_uniform}}, we get
an explicit identity for $\lambda_c$. In particular, we define a
function $\Pi_\lambda$ in
Proposition~\ref{thm:convergenceoflaceexpansion} that satisfies \eqq{
  \lambda_c =
  \frac{1}{1+\widehat\Pi_{\lambda_c}(\orig)}. \label{eq:lambda_c_identity_fpilam}}
As $\lambda \mapsto \theta(\lambda)$ is non-decreasing and the
decreasing limit of \ch{the continuous functions that the origin is connected to a point outside a large box},
we have that $\theta$ is
continuous from the right for all $\lambda \geq 0$ (see~\cite[Lemma
8.9]{Gri99}). \ch{The fact that $\theta(\lambda_c)=0$ implies that
$\lambda\mapsto\theta(\lambda)$ is continuous on $[0,\infty)$, since
the left-continuity of $\theta$ for $\lambda>\lambda_c$ follows as in \cite[Lemma~8.10]{Gri99} in the discrete setting or \cite{Sarka97} for the Poisson Boolean model; the extension to the RCM is straightforward.} 
The asymptotics
of $\lambda_c$ \col{for $d \to\infty$} was already shown by Meester, Penrose and Sarkar~\cite{MeePenSar97}. \col{The asymptotics in the spread-out case were shown in a slightly weaker form by Penrose~\cite{Pen93}.}

\subsection{Literature overview and discussion}\label{sec:literature}

We first give some general background on percolation theory, then highlight the important literature on continuum percolation and the random connection model. After this, we put the results of this paper into context.

The foundations of percolation theory are generally attributed to Broadbent and Hammersley in 1957~\cite{BroHam57}. Several textbooks were published, we refer to Grimmett~\cite{Gri99} as a standard reference, and Bollob{\'a}s and Riordan~\cite{BolRio06}, which puts an extra focus on two-dimensional percolation. More recent treatments of two-dimensional percolation are the book by Werner~\cite{Wer09} and the survey by Beffara and Duminil-Copin~\cite{BefDum13}.

A book on percolation in high dimensions was written by the first two authors~\cite{HeyHof17}. It contains a self-contained proof of the lace expansion for bond percolation as well as an extensive summary of recent results on high-dimensional percolation. Another detailed description of the lace expansion is given by Slade~\cite{Sla06}, with a focus on self-avoiding walk. One of the corner stones of high-dimensional percolation is the seminal 1990 paper by Hara and Slade~\cite{HarSla90}, successfully applying the lace expansion to bond percolation on $\mathbb Z^d$ in sufficiently high dimension.

While this paper contains many ideas and techniques from percolation, the above references deal with discrete lattices mostly, whereas we deal with a model of continuum percolation. When highlighting the difference of the former models to continuum percolation, we refer to them as \chr{lattice or discrete percolation.}

Continuum percolation may be regarded as a branch of percolation theory, including some aspects of stochastic geometry, and, in particular, the theory of point processes. A textbook on the Poisson point process was written by the third author and Penrose~\cite{LasPen17}. Continuum percolation was first considered in 1961 by Gilbert~\cite{Gil61} for the Poisson blob model. 
Partially motivated by \cite{Gil61}, the random connection model was first introduced in 1991 by Penrose~\cite{Pen91}. A textbook treatment of continuum percolation was given by Meester and Roy~\cite{MeeRoy96}, also summarizing some properties of the random connection model. Among those properties is the essential result that $\lambda_c = \lambda_T$, which was first obtained in full generality in 1995 by Meester~\cite{Mee95}. As a representative treatment of continuum percolation in the physics literature, we point to the book by Torquato~\cite{Tor02}. More recently, the RCM was considered by the third author and Ziesche~\cite{LasZie17}, and they prove that the subcritical two-point function satisfies the Ornstein-Zernike equation (OZE). We point out that~\eqref{eq:LE_identity_OZE} is precisely the OZE (and~\eqref{eq:fourier-lace-expansion-equation} is the OZE in Fourier space).

Continuum percolation has a finite-volume analogue, by restricting to a bounded domain---see the monograph by Penrose~\cite{Pen03} about random geometric graphs. The finite-volume analogue of the RCM was investigated by Penrose~\cite{Pen16}, \ch{who calls it the \emph{soft random geometric graph}. Recently, the RCM has attracted lots of attention, e.g.\ \cite{BetLas21,GraHeyMonMor19,LasNesSch18}.} 

The RCM is related to some fundamental lattice models and has, in fact, features of both site and bond percolation. The Poisson blob model, for instance, can be considered as a continuum version of nearest-neighbor site percolation. The parameter $p$ in the discrete setting is then analogous to the intensity $\lambda$ of the PPP, as both describe the point density. \ch{In that sense, the general RCM corresponds to a mixed bond-site percolation  model with long-range connections governed by $\connf$.} The RCM can also be interpreted as bond percolation (again with long-range interactions governed by $\connf$) on the complete graph generated by the PPP. In this perspective, the parameter $p$ in the discrete setting can be compared to $\lambda$ in the continuum, as the mean degree in the discrete setting is $2dp$, whereas $\lambda\int \connf(x) \dd x$ is the mean degree of a typical point of the PPP.

The results obtained in this paper mirror several results of lattice percolation. The treatment of nearest-neighbor models and their spread-out version, first performed by Hara and Slade~\cite{HarSla90}, can be compared to our versions (H1) and (H2) for $\connf$. For bond percolation on $\mathbb Z^d$, Fitzner and the second author proved that $d \geq 11$ suffices to obtain an analogue of Theorem~\ref{thm:maintheorem}~\cite{FitHof16, FitHof17}. In our corresponding regime, which is (H1), we give no quantitative bound on the dimension $d$. The ``discrete analogue'' of (H3) is long-range percolation, for which the corresponding results were obtained by the first two authors and Sakai~\cite{HeyHofSak08} for $d>3(\alpha\wedge2)$, and by Chen and Sakai~\cite{CheSak19} for the marginal case $d=6$ and $\alpha=2$.

It is worth noting that Tanemura~\cite{Tan96} already devised a lace expansion for the Poisson blob model. For the special case of the Poisson blob model, the expansion itself is the same as the one devised in this paper. However, we were unable to give a proof of the expansion's convergence based on~\cite{Tan96}.

A possible application of the results of this paper is the deduction of the existence of several critical exponents (and their computation) other than $\gamma$. Analogous results for the lattice were proved by Aizenman and Newman~\cite{AizNew84}, by Aizenman and Barsky~\cite{AizBar91}, by Hara~\cite{Har90,Har08}, by Hara, the second author, and Slade~\cite{HarHofSla03}, and furthermore by Kozma and Nachmias~\cite{KozNac09, KozNac11} (this list is not exhaustive). These results have not yet been shown for the RCM. However, the third author together with Penrose and Zuyev~\cite{LasPenZuy17} proved the mean-field upper bound on the critical exponent $\beta$ in $\theta(\lambda)\approx (\lambda-\lambda_c)^{\beta}$ as $\lambda \searrow \lambda_c$ for the Boolean model with random (deterministically bounded) radii. It may also be possible to investigate an asymptotic expansion of the critical point $\lambda_c$ (at least for specific choices of $\connf$). We point to Torquato~\cite{Tor12} for predictions of such results in the continuum, and to \cite{HarSla95,HeyMat19b,HofSla05,HofSla06} for corresponding rigorous results for percolation on the lattice $\mathbb Z^d$.

\subsection{Overview, discussion of proof, and notation} \label{sec:notation}
\paragraph{Overview of the proofs of Theorems \ref{thm:maintheorem}
and  \ref{thm:mainthmcorollaries}.} 
We interpret $\phiint^{-1} \connf$
as a random walk density and define its Green's function as \eqq{
  \greensmu(x) := \sum_{m \geq 0}  \big(\mu/ \phiint\big)^m
  \connf^{\star m}(x), \label{eq:def:greens_fct}} where $|\mu|<1$ and
$\connf^{\star 0}$ is a generalized (Dirac) function. In Fourier
space, this gives \eqq{ \col{ \phiint^{-1} \fgreensmu (k) =
    \frac{1}{\fconnf(\orig)-\mu
      \fconnf(k)},} \label{eq:def:greens_fct_fourier}} \col{noting
  that $\fconnf(\orig) = \phiint$}. The main aim of our paper can be
summarized as follows: For small $\beta$, we intend to show that
$\lambda \tlam$ is close to $(\connf\star \greensmu)$, where
$\mu$ depends on $\lambda$ in an appropriate way. The latter we
understand much better than we understand $\tlam$, in particular, we
know that $(\connf^{\star m} \star G_1^{\star 3})(x)$ is finite for
$m=1$ and small for $m=2$. This ``closeness'' will allow us to
transfer this result to $\tlam$ and prove the triangle condition.

\paragraph{\ch{Organisation of the proof and paper.}} The lace-expansion technique proceeds in three major steps, which also dictate the structure of this paper. Before the first step, we need to make sure to have the relevant tools that are analogous to those used in discrete percolation theory available to us; also, some methodology from point-process theory is introduced. 
This is done in Section~\ref{sec:preliminaries}.

In the first major step of the proof, which is key to proving Theorem~\ref{thm:maintheorem} and is executed in Section~\ref{sec:expansion}, we show that the lace expansion for the two-point function $\tlam$ takes the form
	\eqq{\tlam = \connf + \Pi_{\lambda, n} + \lambda\big((\connf+\Pi_{\lambda, n})\star \tlam\big) + R_{\lambda, n} \quad (n\in\N_0), \label{eq:laceexpansionequation}}
where the lace-expansion coefficients $\Pi_{\lambda, n}$ and $R_{\lambda, n}$ arise during the expansion and will be defined later.

\chr{In the second step of the proof}, we bound $\Pi_{\lambda, n}$ and $R_{\lambda, n}$ by simpler diagrammatic functions. Those diagrams are large integrals over products of two-point functions which can then be decomposed into factors of $\trilam$ and related quantities. We eventually want to prove $\lambda|\widehat\Pi_{\lambda, n}(k)| = \mathcal O(\beta)$ (recall that this means as $d\to \infty$ for (H1) and as  $L\to\infty$ for (H2),(H3)) uniformly in $k\in\Rd$ and $n\in\N_0$, and the intermediate step of diagrammatic bounds is essential to do so.
\chr{The results of the second step are formulated in Section~\ref{sec:diagrammaticbounds}, whereas the proofs of most diagrammatic bounds on the lace-expansion coefficients are deferred to Section \ref{sec:bounds-lace-expansion}.}

The third step is the so-called ``bootstrap argument'' and it is performed in Section~\ref{sec:bootstrapanalysis}. Since the diagrammatic bounds obtained in the previous step are in terms of $\tlam$ itself, we need this step in order to gain meaningful bounds. More details about this are given at the beginning of Section~\ref{sec:bootstrapanalysis}. As the general strategy of proof is standard, we refer to~\cite{HeyHof17} for a more detailed informal description of the methodology. However, we note that in this step, several functions are introduced, among them the fraction $|\ftlam / \fgmu|$, where the parametrization $\mulam$ satisfies $\mulam \nearrow 1$ as $\lambda \nearrow \lambda_c$. The boundedness of these functions is shown, which justifies the notion of $\ftlam$ and $\fgmu$ being ``close''. 

Section~\ref{sec:bootstrapanalysis} also contains the consequences of the completed argument. The most important of these is that if we let $n\to\infty$ in~\eqref{eq:laceexpansionequation}, then $R_{\lambda, n}\to 0$ and $\Pi_{\lambda, n} \to \Pi_\lambda=\mathcal O(\beta)$ for any $\lambda<\lambda_c$. In Fourier space, \eqref{eq:laceexpansionequation} consequently translates to
	\eqq{\ftlam = \frac{\fconnf + \fpilam}{1-\lambda (\fconnf + \fpilam)}. \label{eq:fourier-lace-expansion-equation}}
Together with the obtained bounds on $\lambda\fpilam(k)$, this implies our two main results and justifies the comparison with the Green's function of the random walk with step distribution $\chr{\connf/ \phiint}$. 

After having completed the lace expansion successfully, we prove our main theorems in Section~\ref{sec:maintheorems}.

\paragraph{Differences to percolation on the lattice.}
We informally describe the novelties that the application of the lace
expansion to the RCM setting brings, in contrast to the lace expansion
for, say, bond percolation on $\mathbb Z^d$. 
By stationarity and isotropy of the 
underlying Poisson point process
we can use re-scaling arguments more easily (see
Section~\ref{sec:rescaling}).  
The Mecke equation (see Section~\ref{sec:bkrusso}) 
allows \chr{us} to deal with expectations of multiple sums over
the random Poisson points, while a suitable Margulis-Russo
type formula (see \eqref{eq:prelim:russo}) is used 
to derive differential equalities and inequalities.
The biggest novelty of the expansion
in Section~\ref{sec:expansion} is the occurrence of \emph{thinnings}
(Definition~\ref{def:LE:thinning_events} and
Lemma~\ref{lem:stoppingset}),
a phenomenon which does not occur in a discrete setting with 
deterministic vertices. 
The thinning events from Section~\ref{sec:expansion} 
require some extra work to decompose in the fashion
intended by Section~\ref{sec:diagrammaticbounds}. This is done in
Definition~\ref{def:DB:extended_thinning_events}, in
Lemma~\ref{lem:DB:squigarrow_tlam_equality} and in
Lemma~\ref{lem:DB:bounds_E_by_F}. We also highlight that the
decomposition crucially relies on the BK inequality, which is new for
the RCM (see Theorem~\ref{lem:prelimbk_inequality}).

While several other differences in the decomposition of
Section~\ref{sec:diagrammaticbounds} can be attributed to the site
percolation nature of the RCM, it is a challenge unique to certain
versions of the RCM (among them, the Poisson blob model) that
$\tlam\star\tlam$ is \chr{not} bounded away from $0$ in a neighborhood of the
origin. 
For discrete percolation on $\mathbb Z^d$, 
$(\tlam\star\tlam)$ takes on a (strictly) positive value at the origin as well,
but the continuum space forces us to deal with this issue in a
different way. We point to the introduction of $\ballep$ in
Definition~\ref{def:bubble} and the discussion \chr{after Proposition \ref{thm:Psi_Diag_bound}} for more details.

Lastly, two issues arise in Section~\ref{sec:bootstrapanalysis}. The first is that $\tlam$ is closer to $(\connf\star \greensmu)$ rather than $\greensmu$, which is again a manifestation of the site-percolation nature of the RCM. The second is the fact that, unlike $\mathbb Z^d$, the space $\Rd$ has a non-compact Fourier domain (namely, $\Rd$ itself), which demands some extra care in the Fourier analysis of the bootstrap functions introduced in Section~\ref{sec:bootstrap_functions}.

\paragraph{Some notation.}
Let us fix some helpful notation, which we will use throughout this paper:
\begin{compactitem}
\item If not specified otherwise, $\xi$ is used to refer to an edge-marking of a PPP (an edge-marking is the random object encoding the RCM, see Section~\ref{sec:constructionofRCM}). This PPP is the ``ground process'' of $\xi$ and is always denoted by $\eta$.

\item For two real numbers $a,b$, we use $a \wedge b := \min \{a,b\}, a \vee b := \max\{a,b\}$.
\item For a vector $v \in\Rd$, $|v|=\| v\|_2$ denotes the Euclidean norm. Secondly, for a set $A$, we use $|A|$ to denote the cardinality of $A$.

\item For a natural number $n$, let $[n]=\{1,\ldots, n\}$. Given $ x_1, x_2, \ldots, y_1, y_2, \ldots \in \Rd$ as well as integers $a \leq b$, we write $\vec y_{[a,b]} = (y_a, \ldots, y_b)$, and similarly $ (\vec x, \vec y)_{[a,b]} = (x_a, \ldots, x_b, y_a, \ldots, y_b)$.

\item An integral over a non-specified set is always to be understood as the integral over the whole space.

\item We use $\delta_{x,y}$ to denote the distributional Dirac delta, i.e.~$\int\delta_{x,y} f(x) \dd x = f(y)$ for measurable $f\colon\Rd\to\R$. The use of $\delta_{\cdot,\cdot}$ in this paper is detailed in a remark above Definition~\ref{def:DB:psi_functions}.

\item We recall that for measurable functions $f,g\colon\Rd\to\R$, we write $(f\star g)(x) = \int f(y) g(x-y)\dd y$ for their convolution. Moreover, the $m$-fold convolution is set to be $f^{\star 0}(x) = \delta_{\orig, x}$ and $f^{\star m}(x) = (f^{\star (m-1)} \star f)(x)$. 

\item We recall some basic notation from graph theory. If $G$ is a graph, then $V(G)$ is its set of vertices (points, sites), and $E(G)$ is its set of edges (bonds). Since we will be concerned with \ch{a (random) graph} $\xi$ whose vertex set is a Poisson point process $\eta$, we usually write $V(\xi) = \eta$. A subgraph $G'$ of $G$ is a graph where $V(G') \subseteq V(G)$ and $E(G') \subseteq E(G)$. For $W\subseteq V(G)$, the subgraph of $G$ \emph{induced} by $W$, denoted by $G[W]$, is the subgraph of $G$ whose vertex set is $W$ and whose edge set is such that two vertices in $W$ are adjacent in $G[W]$ if and only if they are adjacent in $G$.

\end{compactitem}

\section{Preliminaries} \label{sec:preliminaries}

\subsection{Point processes} \label{sec:pointprocesses}

We briefly recall some basic facts about point and Poisson processes,
referring to \cite{LasPen17}
for a comprehensive treatment. Let $\mathbb{X}$ be a metric space with
Borel $\sigma$-field $\B(\mathbb{X})$. 
\cov{Call a set $A\subset\mathbb{X}$ locally finite 
if $|A\cap B|<\infty$  
for each bounded $B\in \B(\mathbb{X})$
and let $\mathbf{N}(\mathbb{X})$ denote the space of all those sets.
Equip $\mathbf N(\mathbb{X})$
with the $\sigma$-field $\mathcal N(\mathbb{X})$ generated by the
sets $\{A:|A \cap B| = k\}$,  $B \in\B(\mathbb{X}),k \in \N_0$.
By \cite[Corollary 6.5]{LasPen17}, there exist
measurable mappings $\pi_i\colon \mathbf N(\mathbb{X})\to\mathbb{X}$, $i\in\N$,
such that $A=\{\pi_i(A):i\le |A|\}$.  
(If $|A|=n\in\N_0$ and $i>n$, then $\pi_i(\zeta)$ can be defined
in any measurable way, for instance by choosing a fixed point in $\mathbb{X}$.)}
A point process on $\mathbb X$ is a measurable mapping
$\zeta\colon \Omega \to \mathbf N(\mathbb X)$ for some underlying probability
space $(\Omega, \mathcal{F}, \mathbb{P})$. The intensity measure of a
point process is the measure on $\mathbb{X}$ given by
$B\mapsto \E[|\zeta \cap B|]$ for $B\in \B(\mathbb{X})$.

Let $\mu$ be a \cov{locally finite} non-atomic measure on $\Rd$. A Poisson
point process (PPP) on $\mathbb{X}$ with intensity measure $\mu$ is a point
process $\zeta$ such that the number of points $|\zeta \cap B|$ is
$\text{Poi}(\mu(B))$-distributed for each $B \in \B(\mathbb{X})$ and the
random variables $|\zeta \cap B_1|, \ldots, |\zeta \cap B_m|$ are independent
whenever $B_1, \ldots, B_m\in\B(\mathbb{X})$ are pairwise disjoint. We point
out that in our setting, the first property implies the second
(independence) property. In the case $\mathbb{X}=\Rd$, we call the PPP
homogeneous (or stationary) with intensity $\lambda$ if
$\mu= \lambda \text{Leb}$ with $\lambda \geq 0$ and $\text{Leb}$ the
Lebesgue measure.
In this paper we consider a homogeneous Poisson process $\eta$ with intensity 
$\lambda \geq 0$, and we write
	\eqq{ \eta=\{X_i:i\in\N\} , \label{eq:orderingofPPpoints} }
where $X_i:=\pi_i(\eta)$, $i\in\N$. 
\cov{The intensity $\lambda$ is an important parameter. As mentioned earlier, we write 
$\pla$ for a probability measure on $(\Omega, \mathcal{F})$ such
that $\eta$ is a Poisson process with this intensity. 
The probability space $(\Omega,\mathcal{F},\pla)$
is assumed to be rich enough to support 
all (independent) randomness we may need, e.g.\ marks or thinning
variables.}

\subsection{Formal construction of the RCM} \label{sec:constructionofRCM}

As some of the later statements sensitively depend on the precise construction of the model, we give a detailed formal construction. It is convenient to construct the RCM as a deterministic functional $\Gamma_\connf(\xi)$ of a suitable point process $\xi$. Following~\cite{LasZie17}, we choose $\xi$ as an \emph{independent edge-marking} of a PPP $\eta$. We then show how to extend the construction to include deterministic points and how to extend it to include thinnings. Other (equivalent) ways to construct the RCM can be found in \cite{vdBru03,MeeRoy96}. There, the RCM is introduced as a marked PPP (hence, the information about the edges of $\xi$ is encoded in the marks of the points). We point to the proof of Theorem~\ref{lem:prelimbk_inequality}, where we also require a construction in terms of a suitable marked PPP of (an approximation of) the RCM.

\paragraph{Construction as a point process in $(\Rd)^{[2]} \times [0,1]$.} 
Recall that $\eta$ denotes an $\Rd$-valued PPP of intensity $\lambda$, which can be written as~\eqref{eq:orderingofPPpoints}. Let $(\Rd)^{[2]}$ denote the space of all sets $e\subset\Rd$ containing exactly two elements. Any $e\in (\Rd)^{[2]}$ is a potential edge of the RCM. When equipped with the Hausdorff metric, this space is a Borel subset of a complete separable metric space. Let $<$ denote the strict lexicographic ordering on $\Rd$. Introduce independent random variables $U_{i,j}$, $i,j\in\N$, uniformly distributed on the unit interval $[0,1]$ such that the double sequence $(U_{i,j})$ is independent of $\eta$. We define 
	\eqq{ \xi:=\{(\{X_i,X_j\},U_{i,j}):X_i<X_j,i,j\in\N\}, \label{eq:prelim:xi_def}}
which is a point process on $(\Rd)^{[2]}\times[0,1]$. We interpret $\xi$ as a random marked graph and say that $\xi$ is an {\em independent edge-marking} of $\eta$. Note that $\eta$ can be recovered from $\xi$. The definition
of $\xi$ depends on the ordering of the points of $\eta$. The distribution of $\xi$, however, does not.
\cov{If $A\subset\R^d$ is locally finite, then we denote 
the RCM restricted to the points in $\eta\setminus A$ by $\xi[\eta\setminus A]$.}
Given an independent edge-marking $\xi$ of $\eta$, we can define the
RCM $\Gamma_\connf(\xi)$ as a deterministic functional of $\xi$, given
by its vertex and edge set
as 
\algn{
		V(\Gamma_\connf(\xi)) = \eta = \{X_i:i\in\N\},\quad
		E(\Gamma_\connf(\xi)) &= \{ \{X_i,X_j\}:X_i<X_j, U_{i,j} \leq \connf(X_i-X_j), i,j\in\N\}. \label{eq:prelim:xi_def_edgeset}}
	
The RCM $\Gamma_\connf(\zeta)$ can be defined for every point process $\zeta$ on $(\Rd)^{[2]}\times[0,1]$ with the property that $(\{x,y\},u)\in\zeta$ and $(\{x,y\},u')\in\zeta$ for some $x\ne y$ and $u,u'\in[0,1]$ implies that $u=u'$. Throughout the paper, when speaking of a graph event taking place in $\zeta$, we refer to the graph event taking place in $\Gamma_\connf(\zeta)$.

\paragraph{Adding extra points.}

For $x_1,x_2, x_3\in\Rd$, consider the random connection models driven by the point processes
\[ \eta^{x_1}:=\eta\cup\{x_1\},\quad
  \eta^{x_1,x_2}:=\eta\cup\{x_1,x_2\}, \quad
  \ch{\eta^{x_1,x_2,x_3}:=\eta\cup\{x_1,x_2,x_3\}}. \] To couple these
models in a natural way, we extend the (double) sequence
$(U_{m,n})_{m,n \geq 1}$ to a sequence $(U_{m,n})_{m,n \geq -2}$ of
independent random variables uniformly distributed on $[0,1]$,
independent of the Poisson process $\eta$. We then define the point
process \ch{$\xi^{x_1,x_2, x_3}$ on $(\Rd)^{[2]}\times[0,1]$ as
  \[ \xi^{x_1,x_2,x_3}:=\{(\{X_i,X_j\},U_{i,j}):X_i<X_j,i,j\ge
    -2\}, \] where $(X_0,X_{-1}, X_{-2}):=(x_1,x_2,x_3)$.  We now
  define the point process $\xi^{x_1,x_2}$ and $\xi^{x_1,x_3}$, resp.,
  by removing all (marked) edges incident to $x_3$ resp.\ $x_2$ from
  $\xi^{x_1,x_2,x_3}$.  } Similarly, we define the point process
$\xi^{x_1}$ by removing all (marked) edges incident to $x_2$ from
$\xi^{x_1,x_2}$ or by removing all (marked) edges incident to $x_3$
from $\xi^{x_1,x_3}$.  
\cov{In this way, the point processes 
$\xi^{x_1},\xi^{x_2},\xi^{x_3},\xi^{x_1,x_2},\xi^{x_1,x_3},\xi^{x_2,x_3}$
can be written as measurable functions of $\xi^{x_1,x_2,x_3}$,
at least if $x_1,x_2,x_3$ are distinct and do not pertain to $\eta$. 
The latter can be assumed without restricting generality. To avoid heavy
notation, we do not make this functional dependence explicit.
When talking, for instance, about the event 
$\{\conn{x_1}{x_2}{\xi^{x_1,x_2}},\conn{x_2}{x_3}{\xi^{x_2,x_3}}\}$,
this specific coupling of $\xi^{x_1,x_2}$ and $\xi^{x_2,x_3}$ will always be assumed, so that the addition of points occurs in a consistent manner.
Specifically, this means that $\xi^{x_1,x_3}$, obtained by removing $x_2$ from $\xi^{x_1,x_2,x_3}$ is a.s.\ the same as $\xi^{x_1,x_3}$ obtained by removing $x_2$ from $\xi^{x_1,x_3,x_2}$.}

\cov{Similarly we proceed by adding $m$ points $x_1,\ldots,x_m$
to obtain $\xi^{x_1,\ldots,x_m}$ and the corresponding subgraphs, where we again assume that the addition and removal of points occurs in a consistent manner. This can be achieved by letting $U_{i,j}$ be replaced by $U_{x_i,x_j}$ for $i,j\leq 0$, and $U_{i,j}$ be replaced by $U_{x_i,j}$ for $i\leq 0$ and $j\geq 1$. Note that the mapping $(\omega,x_1,\ldots,x_m)\mapsto \xi^{x_1,\ldots,x_m}(\omega)$
is (jointly) measurable. This transfers to connectivity events
as $\{\conn{x_1}{x_2}{\xi^{x_1,x_2}}\}$, for instance. (The indicator function
is jointly measurable on $\Omega\times \R^d\times\R^d$.) We will not comment
on this any further.}

\paragraph{Including thinnings.}
Let $\mathbb M=[0,1]^\N$. It will be important to work with subgraphs
of $\xi$ that are obtained via a \emph{thinning} of $\eta$ with
respect to some point process $\zeta$. We specify this in
Definition~\ref{def:LE:thinning_events}. For now, it is important that
this thinning requires extra randomness, which, given $\eta$, is
independent of the edge set $E(\xi)$. We model this by adding to every
Poisson point in $\eta$ a mark from $\mathbb M$.

To this end, let $(\Rd\times \mathbb M)^{[2]}$ denote the space of
subsets of $\Rd\times\mathbb M$ containing exactly two elements. Let
$\mathbb{U}$ denote the uniform distribution on $[0,1]$ and let
$Y_i=(Y_{i,k})_{k\in\N}$, $i\in\N$, be independent random elements of
$\mathbb M$ with distribution $\mathbb{U}^\N$, independent of
$\eta$. Assume that $\eta$, $(U_{i,j})_{i,j\in \N}$ and
$(Y_i)_{i\in\N}$ are independent. Proceeding analogously to the
definition in~\eqref{eq:prelim:xi_def} we define \eqq{ \xi:=\Big\{
  \big(\{(X_i,Y_i),(X_j,Y_j)\}, U_{i,j} \big): X_i<X_j, i,j\in\N
  \Big\}, \label{eq:prelim:xi_def_thinnings}} which is a point process
in \col{$(\Rd\times \mathbb M)^{[2]}\times [0,1]$}. The RCM is
constructed as before, that is, the vertex set is $\eta$ and the edge
set is as in~\eqref{eq:prelim:xi_def_edgeset}, ignoring the marks
$(Y_i)$. From now on, when speaking of $\xi$, we assume that $\xi$ is
given by~\eqref{eq:prelim:xi_def_thinnings}. Still, the reader might
prefer to work with the simpler version \eqref{eq:prelim:xi_def}. We
shall clearly point out when the thinning variables are required.

We can add points $x_1,x_2\in\Rd$ as follows. We let
$(U_{m,n})_{m,n \geq -1}$ as above. In addition, we take independent
random elements $Y_0,Y_{-1}$ of $\mathbb M=[0,1]^\N$ and assume that
$\eta$, $(Y_i)_{i\ge -1}$ and $(U_{m,n})_{m,n \geq -1}$ are
independent. Define
\[ 
\xi^{x_1,x_2}:=\{(\{(X_i,Y_i),(X_j,Y_j)\},U_{i,j}):X_i<X_j,i,j\ge
  -1\}, 
\] 
where $(X_0,X_{-1}):=(x_1,x_2)$. The point processes
$\xi^{x_1}$, $\xi^{x_2}$ etc.~are defined as before. 
\cov{Similarly we define $\xi^{x_1,\ldots,x_m}$ and the corresponding
subgraphs.}

\subsection{Mecke, Margulis-Russo, BK, FKG, and a differential inequality} \label{sec:bkrusso}
In this section, we state some useful equalities and inequalities that are standard either in point process theory or in percolation theory.

\paragraph{The Mecke equation.}
Our first crucial tool is a version of the  Mecke equation (see \cite[Chapter 4]{LasPen17}) for the independent edge-marking $\xi$. This fundamental equation allows us to deal with sums over points of $\eta$, which we frequently make use of. We closely follow \cite{LasZie17}. Given $m\in\N$ and a measurable function $f\colon\mathbf{N}\big((\Rd \times\mathbb M)^{[2]}\times[0,1]\big) \times \big(\Rd\big)^m \to \R_{\geq 0}$, the Mecke equation for $\xi$ states that 
	\eqq{ \E_\lambda \bigg[ \sum_{\vec x \in \eta^{(m)}} f(\xi, \vec{x})\bigg] = \lambda^m \int
				\E_\lambda\Big[ f\left(\xi^{x_1, \ldots, x_m}, \vec x\right)\Big] \dd \vec{x},  \label{eq:prelim:mecke_n} }
where $\vec x=(x_1,\ldots,x_m)$ and $\eta^{(m)}=\{(x_1,\ldots,x_m): x_i \in \eta, x_i \neq x_j \text{ for } i \neq j\}$. We only need the statement for $m\leq 3$, and in particular, we mostly use \eqref{eq:prelim:mecke_n} for $m=1$, yielding the univariate Mecke equation
	\eqq{\E_\lambda \bigg[ \sum_{x \in \eta} f(\xi, x)\bigg] = \lambda \int \E_\lambda\left[ f(\xi^x, x)\right] \dd x. \label{eq:prelim:mecke_1}}

\paragraph{Margulis-Russo formula.}
The Margulis-Russo formula is a well-known tool in (discrete) percolation theory and turns out to be necessary for us as well. Our version follows from a more general result (see \cite[Theorem 3.2]{LasZie17}). We write $\mathbf{N}:=\mathbf{N}\big((\Rd \times\mathbb M)^{[2]}\times[0,1]\big)$. Let $\Lambda\in \B(\Rd), \zeta\in\mathbf{N}$, and define
	\eqq{\zeta_\Lambda := \{(\{(x,v),(y,w)\},u)\in \zeta :\{x,y\}\subseteq\Lambda \}.  \label{eq:prelim:def:PP_restriction}}
We call $\zeta_\Lambda$ the restriction of $\zeta$ to $\Lambda$. We say that $f\colon \mathbf N\to\R$ \emph{lives} on $\Lambda$ if $f(\zeta) = f(\zeta_\Lambda)$ for every $\zeta \in\mathbf N$. Assume that there exists a bounded set $\Lambda \in \B(\Rd)$ such that $f$ lives on $\Lambda$. Moreover, assume that there exists  $\lambda_0>0$ such that $\E_{\lambda_0} [|f(\xi)|] < \infty$. Then the Margulis-Russo formula states that, for all $\lambda\leq \lambda_0$,
	\eqq{ \frac{\partial}{\partial \lambda} \E_\lambda[f(\xi)] = \int_\Lambda \E_\lambda[f(\xi^x) - f(\xi)] \dd x. \label{eq:prelim:russo}}
\ch{A corresponding formula holds when we add extra points, i.e., $\xi^{x_1,x_2}$ rather than $\xi$. In that case, $\xi^{x}$ is replaced by $\xi^{x_1,x_2,x}$}.

\paragraph{The BK inequality.}
The BK inequality is a standard tool in discrete percolation and provides a counterpart to the FKG inequality, which we discuss below and which (given existing results) turns out to be much easier to prove.  

Let us first informally describe what type of inequality we are aiming
for, described for the type of events we need it for. Assume that we
are interested in the event that there are two paths, the first
between $x_1,x_2\in\Rd$ and the second between $x_3,x_4 \in\Rd$, not
sharing any vertices. This \emph{(vertex) disjoint occurrence} is
something we think of as being less likely than the probability that
on two independent RCM graphs, one has an $x_1$-$x_2$-path and the
second has an $x_3$-$x_4$-path. Thus, if $E \circ F$ denotes the
former event, we want an inequality of the form
$\pla(E\circ F) \leq \pla(E)\pla(F)$. We now work towards making these
notions rigorous and towards proving such an inequality.

It is convenient to write
$\mathbf{N}:=\mathbf{N}((\Rd\times\mathbb{M})^{[2]}\times[0,1])$ and
to denote the $\sigma$-field on $\mathbf{N}$ (as defined in
Section~\ref{sec:pointprocesses}) by $\mathcal{N}$. For
$\Lambda\in \B(\Rd)$ and $A\in \mathbf{N}$, we define $A_\Lambda$,
the restriction of $A$ to all edges completely contained in
$\Lambda\times\mathbb{M}$, analogously
to~\eqref{eq:prelim:def:PP_restriction}. Define the \emph{cylinder
  event} $\llbracket A \rrbracket_\Lambda$ as
$\llbracket A \rrbracket_\Lambda := \{ B \in\mathbf N: B_\Lambda
= A_\Lambda \}$. We say that $E\in\mathcal{N}$ {\em lives on}
$\Lambda$ if $\mathds 1_{E}$ lives on $\Lambda$. We call a set
$E\subset \mathbf{N}$ \emph{increasing} if $B \in E$ implies
$A\in E$ for each $A\in\mathbf{N}$ with $B\subseteq A$
\ch{(recall that $B$ and $A$ are at most countably infinite
  sets)}. Let $\mathcal{R}$ denote the ring of all finite unions of
half-open rectangles with rational coordinates. For events
$E,F\in\mathcal{N}$ we define \eqq{ E \circ F := \{ A\in\mathbf{N}:
  \exists K,L \in \mathcal{R} \text{ s.t.~} K \cap L = \varnothing
  \text{ and } \llbracket A\rrbracket_K \subseteq E,
  \llbracket A\rrbracket_L \subseteq F \}.
  \label{eq:prelim:def:circ}} If $E,F$ are increasing events, then
$ E\circ F = \{A\in\mathbf{N}: \exists K,L \in \mathcal{R} \text{
  s.t.~} K \cap L = \varnothing, A_K\in E, A_L \in F \}$.
 We will show that \eqq{ \pla(\xi\in E \circ F) \leq \pla(\xi\in
  E)\,\pla(\xi\in F) \label{eq:prelim:BK} } whenever
$E,F\in\mathcal{N}$ live on a bounded set $\Lambda\in\B(\Rd)$ and are
increasing.

We need a slightly more general version of this inequality involving
independent random variables. Let $(\mathbb{X}_1,\mathcal{X}_1)$,
$(\mathbb{X}_2,\mathcal{X}_2)$ be two measurable spaces. We say that a
set $E_i\subset \mathbf{N}\times \mathbb{X}_i$ is {\em increasing} if
$\E^z_i:=\{A\in\mathbf{N}:(A,z)\in E_i\}$ is increasing for each
$z\in \mathbb{X}_i$. For increasing
$E_i\in\mathcal{N}\otimes\mathcal{X}_i$, we define \eqq{ E_1 \circ E_2
  := \{(A,z_1,z_2)\in\mathbf{N}\times\mathbb{X}_1\times\mathbb{X}_2:
  \exists K_1,K_2 \in \mathcal{R} \text{ s.t.~} K_1\cap K_2 =
  \varnothing, (A_{K_1},z_1)\in E_1,(A_{K_2},z_2)\in E_2
  \}. \label{eq:prelim:circ2} } A set
$E_i\in\mathcal{N}\otimes \mathcal{X}_i$ {\em lives on} $\Lambda$ if
$\mathds 1_{E_i}(A,z)=\mathds 1_{E_i}(A_\Lambda,z)$ for each
$(A,z)\in\mathbf{N}\times\mathbb{X}_i$.

\begin{theorem}[BK inequality]\label{lem:prelimbk_inequality}
\ch{Consider random elements $W_1,W_2$ of $\mathbb{X}_1$ and $\mathbb{X}_2$, respectively, and assume that $\xi,W_1,W_2$ are independent.}
Let $E_i\in\mathcal{N}\otimes\mathcal{X}_i$, $i\in\{1,2\}$, be increasing events that live on some bounded set $\Lambda\in\B(\Rd)$. Then $\pla((\xi,W_1,W_2)\in E_1 \circ E_2) \leq \pla((\xi,W_1)\in E_1)\pla((\xi,W_2)\in E_2)$.
\end{theorem}

Let us make some remarks. First, Theorem~\ref{lem:prelimbk_inequality}
is formulated for $\xi$ as in~\eqref{eq:prelim:xi_def_thinnings}, that
is, for a 
\cov{point process on $(\Rd\times\mathbb{M})^{[2]}\times[0,1]$
given as an edge-marking of a marked PPP.}
The proof is independent of the precise structure of $\mathbb M$ and the
probability measure on it; and so
Theorem~\ref{lem:prelimbk_inequality} still holds true if $\mathbb M$
is replaced by any complete separable metric space with a probability
measure on it.

Second, our proof relies on the BK inequality for marked PPPs proved
by Gupta and Rao in~\cite{GupRao99}. The result in \cite{GupRao99} is
phrased for marks in $[0,\infty)$, but an inspection of the proof
shows that the result is valid for rather general mark spaces, and we
use this generalization in our proof.  A BK-type inequality for
increasing events was proved by van den Berg~\cite{vdB96}.

Finally, we point out that the operation $\circ$ is commutative and associative, which allows repeated application of the inequality.

\begin{proof} We first show that it suffices to prove
  \eqref{eq:prelim:BK}. Indeed, if \eqref{eq:prelim:BK} holds, then
  \al{ \pla((\xi,W_1,W_2)\in E_1 \circ E_2) &= \int \pla(\xi\in E^{w_1}_1 \circ E^{w_2}_2)\pla((W_1,W_2)\in\dd(w_1,w_2)) \\
    & \le \iint\pla(\xi\in E^{w_1}_1)\pla(\xi\in E^{w_2}_2)\pla(W_1\in\dd w_1)\pla(W_2\in\dd w_2)\\
    & =\pla((\xi,W_1)\in E_1)\pla((\xi,W_2)\in E_2). } 

To
  prove~\eqref{eq:prelim:BK}, we use the BKR inequality proven
  in~\cite{GupRao99}. To do so, we approximate $\xi$ by functions of
  suitable independent markings of the 
\cov{Poisson process $\eta=\{X_i: i \in \N\}$. Let}
$\eta':=\{(X_i,Y_i):i\in\N\}$ be an independent marking of $\eta$
(see~\eqref{eq:prelim:xi_def_thinnings}). By the marking theorem (see,
e.g.,~\cite[Theorem 5.6]{LasPen17}), $\eta'$ is a Poisson process with
intensity measure $\lambda\text{Leb}\otimes\mathbb{U}^\N$. Let $\xi$
be given as
in~\eqref{eq:prelim:xi_def_thinnings}.

Let $\varepsilon>0$ and set $Q_z^\varepsilon :=z + [0,\varepsilon)^d$ for $z\in \varepsilon \mathbb Z^d$. Define $\mathbb{M}^\varepsilon:= [0,1]^{\varepsilon\mathbb{Z}^d}$ and let $\mathbb{U}^\varepsilon:=\mathbb{U}^{\varepsilon\mathbb{Z}^d}$, where $\mathbb{U}$ is the uniform distribution on $[0,1]$. Let $\eta^\varepsilon$ be an independent $\mathbb{U}^\varepsilon$-marking of $\eta'$. By the marking theorem~\cite[Theorem 5.6]{LasPen17}, $\eta^\varepsilon$ is a Poisson process on $\Rd\times\mathbb M\times\mathbb{M}^\varepsilon$ with intensity measure $\lambda\text{Leb}\otimes\mathbb{U}^\N\otimes \mathbb{U}^\varepsilon$. We write $\eta^\varepsilon$ in the form
	\[ \eta^\varepsilon=\{(x,r,U(x,r)):(x,r)\in\eta'\}, \]
where the $U(x,r)$ are conditionally independent given $\eta'$. Moreover, given $\eta'$ and $(x,r)\in\eta'$, $U(x,r) = (U_z(x,r))_{z \in\varepsilon\mathbb{Z}^d}$ is a sequence of independent uniform random variables on $[0,1]$ (for simplicity of notation, $U(x,r)$ does not reflect the dependence on $\varepsilon$).

We use $\eta^\varepsilon$ to approximate $\xi$ by a point process
$\xi^\varepsilon$ on
\cov{$(\Rd\times\mathbb{M})^{[2]}\times[0,1]$}
as follows. For $x \in \Rd$, let $q(x;\varepsilon)$ be the point
$z \in \varepsilon\mathbb{Z}^d$ such that $x \in Q_z^\varepsilon$. If
$(x,r),(y,s)\in \eta'$ satisfy $x<y$ (we recall that $<$ denotes
strict lexicographical order) and
$|\eta \cap Q^\varepsilon_{q(x;\varepsilon)}|=|\eta \cap
Q^\varepsilon_{q(y;\varepsilon)}| = 1$, we let
$(\{(x,r),(y,s)\},U_{q(y;\varepsilon)}(x,r))$ be a point of
$\xi^\varepsilon$.  Let
\[ R_\varepsilon :=\bigcup_{z\in\Lambda^\varepsilon}\{|\eta \cap
  Q^\varepsilon_z |\ge 2\},\] where
$\Lambda^\varepsilon :=\{z\in\varepsilon \mathbb Z^d:\Lambda\cap
Q_z^\varepsilon\ne\varnothing\}$. A simple but crucial observation is
that \eqq{ \pla(\{\xi_\Lambda\in\cdot\}\cap R^c_\varepsilon)
  =\pla(\{\xi^{\varepsilon}_\Lambda\in\cdot\}\cap
  R^c_\varepsilon). \label{eq:prelim:BK_1}} Next we note that, as
$\varepsilon \searrow 0$, \eqq{ \pla(R_\varepsilon) \le
  \text{diam}(\Lambda)^d\varepsilon^{-d} \left(1-
    \e^{-\lambda\varepsilon^d} - \lambda \varepsilon^d
    \e^{-\lambda\varepsilon^d}\right) = \mathcal
  O(\varepsilon^d), \label{eq:prelim:BK_2}} where $\text{diam}$
denotes the diameter. To exploit~\eqref{eq:prelim:BK_1}
and~\eqref{eq:prelim:BK_2}, we need to recall the BKR inequality for
$\eta^\varepsilon$. Set
$\mathbf{N}^\varepsilon:=\mathbf{N}(\Rd \times\mathbb M \times
\ch{\mathbb{M}^\varepsilon})$. For $B\in\mathbf{N}^\varepsilon$ and
$K\in \B(\Rd)$, we set
$B_K:=B\cap (K\times\mathbb{M}\times\ch{\mathbb{M}^\varepsilon})$
and
$\llbracket B\rrbracket_K:=\{A\in\mathbf{N}^\varepsilon:A_K=B_K\}$. Given
$E',F'\in \mathcal{N}(\Rd\times\mathbb{M}\times
\ch{\mathbb{M}^\varepsilon}))$, we define \eqq{ E' \Box F' := \{
  A\in\mathbf{N}^\varepsilon: \exists K,L \in \mathcal{R} \text{
    s.t.~} K \cap L = \varnothing \text{ and }
  \llbracket A \rrbracket_K \subseteq E', \llbracket A\rrbracket_L
  \subseteq F' \}. \label{eq:prelim:box_event}} If $E',F'$ live on
$\Lambda$, then we can apply the BKR-inequality of Gupta and
Rao~\cite{GupRao99} to get \eqq{ \pla(\eta^\varepsilon\in E' \Box F')
  \le \pla(\eta^\varepsilon\in E')\pla(\eta^\varepsilon\in
  F'). \label{eq:prelim:BK_3}} Let
$E,F\in\mathcal N((\Rd\times\mathbb{M})^{[2]}\times[0,1])$ be
increasing.  By \eqref{eq:prelim:BK_1} and \eqref{eq:prelim:BK_2}, we
have \eqq{ \pla(\xi\in E \circ F) \leq \pla(\{\xi\in E \circ F\}\cap
  R_\varepsilon^c) + \pla(R_\varepsilon) = \pla(\{\xi^\varepsilon \in
  E \circ F\}\cap R_\varepsilon^c) + \mathcal
  O(\varepsilon^d), \label{eq:prelim:BK_5} } We now use that
$\xi^\varepsilon=T(\eta^\varepsilon)$ for a well-defined measurable
mapping
$T\colon
\mathbf{N}(\Rd\times\mathbb{M}\times\ch{\mathbb{M}^\varepsilon})) \to
\mathbf{N}((\Rd\times\mathbb{M})^{[2]}\times[0,1])$. (Again this
notation does not reflect the dependence on $\varepsilon$.) Assume
that $\varepsilon$ is rational. We assert that \eqq{
  \{\xi^\varepsilon\in E \circ F\}\cap R_\varepsilon^c\subseteq
  \{\eta^\varepsilon\in (T^{-1}E) \Box (T^{-1}F)\}\cap
  R_\varepsilon^c.\label{eq:prelim:BK_7}} To prove this, we assume
that $R_\varepsilon^c$ holds. Assume also that there exist disjoint
$K,L\in\mathcal{R}$ (depending on $\eta^\varepsilon$) such that
$T(\eta^\varepsilon)_K \in E$ and $T(\eta^\varepsilon)_L\in F$. Let
	\[ K':=\bigcup_{z\in\Lambda^\varepsilon:|\eta \cap Q^\varepsilon_z\cap K|=1}Q^\varepsilon_z, \]
and define $L'$ similarly. Since $R_\varepsilon^c$ holds, we have that $K'\cap L'=\varnothing$ and, moreover, $T(\eta^\varepsilon)_{K'}=T(\eta^\varepsilon)_{K}$ as well as $T(\eta^\varepsilon)_{L'}=T(\eta^\varepsilon)_{L}$. By definition of $T$, for each $A\in \mathbf{N}(\Rd\times\mathbb M\times[0,1]^{\varepsilon\mathbb Z^d})$, we have that $T(A)_{K'}=T(A_{K'})$. Let $A\in \mathbf{N}(\Rd\times\mathbb M\times[0,1]^{\varepsilon\mathbb Z^d})$ be such that $A_{K'}=\eta^\varepsilon_{K'}$. Then $T(A)_{K'}=T(\eta^\varepsilon)_{K'}$. Since $T(\eta^\varepsilon)_{K'}=T(\eta^\varepsilon)_{K}$ and $E$ is increasing, we obtain that $T(A)\in E$, that is $A\in T^{-1}(E)$. It follows that $\llbracket\eta^\varepsilon\rrbracket_{K'}\subset T^{-1}(E)$. In the same way, we get $\llbracket\eta^\varepsilon\rrbracket_{L'}\subset T^{-1}(F)$. This shows that \eqref{eq:prelim:BK_7} holds.

From \eqref{eq:prelim:BK_5},~\eqref{eq:prelim:BK_7}, and the BKR inequality~\eqref{eq:prelim:BK_3}, we obtain that
	\al{ \pla(\xi\in E \circ F) - \mathcal O(\varepsilon^d) = \pla(\{\xi^\varepsilon \in E \circ F\} \cap R_\varepsilon^c)
		&\le \pla(\{\eta^\varepsilon \in (T^{-1}E) \Box (T^{-1}F)\}\cap R_\varepsilon^c)\\
		&\leq \pla(\eta^\varepsilon \in (T^{-1}E) \Box (T^{-1}F))\\
		&\le \pla(\eta^\varepsilon \in T^{-1}E)\pla(\eta^\varepsilon\in T^{-1}F) \\
		&\le \pla(\{\xi^\varepsilon \in E\}\cap R_\varepsilon^c)\pla(\{\xi^\varepsilon\in F\}\cap R_\varepsilon^c) + \mathcal O(\varepsilon^d)  \\
		&=\pla(\xi \in E)\pla(\xi \in F) + \mathcal O(\varepsilon^d).}
In the second-to-last step, we have used that we can intersect events with $R_\varepsilon^c$ at the cost of adding $\mathcal O(\varepsilon^d)$. In the last step, we have used~\eqref{eq:prelim:BK_1}. Letting $\varepsilon\searrow 0$ concludes the proof.
\end{proof}
In the above proof of Theorem~\ref{lem:prelimbk_inequality}, mind that $T^{-1}(E)$ and $T^{-1}(F)$ are not increasing any more, and so we have made crucial use of the general BKR inequality in \eqref{eq:prelim:BK_3}. 

\paragraph{An application of the BK inequality.}
We now give an example of the use of Theorem~\ref{lem:prelimbk_inequality}. Let $\xi$ be given as in~\eqref{eq:prelim:xi_def_thinnings} and let $x_1,x_2,x_3\in\Rd$. Define $E$ as the event that there is a path between $x_1$ and $x_2$ as well as a second path between $x_2$ and $x_3$ that only shares $x_2$ as common vertex with the first path. More formally, let $E:= \{ \conn{x_1}{x_2}{\xi^{x_1,x_2}}\} \circ \{ \conn{x_2}{x_3}{\xi^{x_2,x_3}} \} $. We assert that
	\eqq{ \pla (E)  \leq \tlam(x_2-x_1) \tlam(x_3-x_2). \label{eq:prelim:BK_application}}
\col{We want to apply Theorem~\ref{lem:prelimbk_inequality} for an RCM $\xi'$ with a slightly modified mark space; moreover, we need to identify $W_1$ and $W_2$.}
For each $i\in\N$, we let $M_i:=(U_{i,0},U_{0,i},U_{i,-1}, U_{-1,i},U_{i,-2},U_{-2,i})$ and define $\xi'$ by~\eqref{eq:prelim:xi_def_thinnings} with $Y_i$ replaced by $(Y_i,M_i)$. We also define $W_1:=(U_{0,-1},U_{-1,0})$ and $W_2:=(U_{-1,-2},U_{-2,-1})$. Then the RCM $\Gamma_\varphi(\xi^{x_1,x_2,x_3})$ is a (measurable) function of $(\xi',W_1,W_2)$ and $(U_{0,-2},U_{-2,0})$. Note that $(U_{0,-2},U_{-2,0})$ is not required for determining the event $E$. Let $(\Lambda_n)_{n\in\N}$ with $\Lambda_n:= [-n,n)^d$ and define $\tlam^n(v,u) := \pla(\conn{v}{u}\xi^{v,u}_{\Lambda_n})$. Then
	\[\pla\big(  \{ \conn{x_1}{x_2}{\xi^{x_1,x_2}_{\Lambda_n}}\} \circ \{ \conn{x_2}{x_3}{\xi^{x_2,x_3}_{\Lambda_n}} \} \big) \leq \tlam^n(x_1,x_2) \tlam^n(x_2,x_3) \]
for every $n\in\N$. Monotone convergence implies \eqref{eq:prelim:BK_application}.

\paragraph{The FKG inequality.}
Adapting the FKG inequality turns out to be rather straightforward. Given two increasing events $E,F$, we want that $\pla(\xi \in E \cap F) \geq \pla(\xi \in E)\pla(\xi \in F)$. Indeed, given two increasing (integrable) functions $f,g$, we have the more general statement
	\eqq{ \E_\lambda[f(\xi) g(\xi)] = \E_\lambda \big[ \E_\lambda[f(\xi)g(\xi) \mid \eta] \big] \geq \E_\lambda \big[ \E_\lambda[f(\xi)\mid \eta] \; \E_\lambda[g(\xi)\mid \eta] \big]
				\geq \E_\lambda[f(\xi)] \; \E_\lambda[g(\xi)]. \label{eq:prelim:FKG}}
The first inequality was obtained by applying FKG to the random graph conditioned to have $\eta$ as its vertex set, the second inequality by applying FKG for point processes (see, e.g.,~\cite[Thm.\ 20.4]{LasPen17}).

\paragraph{Truncation arguments and a differential inequality.}
Next, we prove elementary differentiability results as well as a differential inequality, illustrating how to put the above tools into action. Since Russo-Margulis and BK work only for events on bounded domains and we intend to use them for events of the form $\{\conn{\orig}{x}{\xi^{\orig,x}}\}$, this careful treatment is necessary. We point out that this is the only instance in this paper where we carry out these finite-size approximations in such detail. We start with the differentiability of $\tlam$. Recall that $\Lambda_n = [-n,n)^d$ and $\tau_\lambda^n(v,x) = \pla(\conn{v}{x}{\xi^{v,x}_{\Lambda_n}})$. We write $\tlam^n(x) := \tlam^n(\orig,x)$.

\col{Moreover, we want to give meaning to the event $\{\conn{x}{\Lambda_n^c}{\xi^x_{\Lambda_n}}\}$. To this end, we add a ``ghost vertex'' $g$ in the same way we added deterministic vertices, and we add an edge between $v \in \xi_{\Lambda_n}$ and $g$ with probability $1-\exp(-\int_{\Lambda_n^c} \connf(y-v) \dd y)$. We now identify $\Lambda_n^c$ with $g$.}

\begin{lemma}[Differentiability of $\tlam$]\label{lem:tlam_differentiability}
Let $x\in\Rd$ and $\varepsilon>0$ be arbitrary. Then,
\ch{\begin{enumerate}
\item[(a)] the function $\lambda\mapsto\tau_\lambda^n(x)$ is differentiable on $[0, \lambda_c-\varepsilon]$ for any $n\in\N$; 
\item[(b)] $ \tau_\lambda^n(x)$ converges to $\tlam(x)$ uniformly in $\lambda$, and $\tfrac{\dd}{\dd\lambda} \tau_\lambda^n(x)$ converges to a limit uniformly in $\lambda$ on $[0,\lambda_c-\varepsilon]$; 
\item[(c)] $\tlam(x)$ is differentiable w.r.t.~$\lambda$ on $[0,\lambda_c)$ and
	\eqq{ \lim_{n\to\infty} \frac{\dd}{\dd\lambda} \tau_\lambda^n(x) = \frac{\dd}{\dd\lambda}\tlam(x) =
							\int \pla(\conn{\orig}{x}{\xi^{\orig,y,x,}}, \nconn{\orig}{x}{\xi^{\orig,x}}) \dd y. \label{eq:prelim:tlam_derivative_limit_exchange}}
\end{enumerate}
}
\end{lemma}
\begin{proof}
For $S \subseteq \Rd$, let $\{\viaconn{a}{b}{\xi^{a,b}}{S} \}$ be the event that $a$ is connected to $b$ in $\xi^{a,b}$, and every path uses a vertex in $S$. Also, let $\{\conn{x}{S}{\xi^x}\}$ be the event that there is $y \in \eta^x\cap S$ that is connected to $x$ \ch{in $\xi$}. The convergence $\tau_\lambda^n(x) \to \tlam(x)$ is uniform \col{in $x\in\Rd$ and $\lambda \in [0,\lambda_c-\varepsilon]$}, as
	\al{ |\tlam(x) - \tau_\lambda^n(x)| &= \pla(\viaconn{\orig}{x}{\xi^{\orig,x}}{\Lambda_n^c}) \\
		& \leq \p_{\lambda_c-\varepsilon}(\conn{\orig}{\Lambda_n^c}{\xi^{\orig}})\xrightarrow{n\to\infty} 0 }
uniformly in $\lambda \leq \lambda_c-\varepsilon$. We further claim that
	\eqq{\frac{\dd}{\dd\lambda} \tau_\lambda^n(x) \xrightarrow{n\to\infty} f_\lambda(x) := \int \pla(\conn{\orig}{x}{\xi^{\orig,y,x,}}, \nconn{\orig}{x}{\xi^{\orig,x}}) \dd y \label{eq:prelim:tlam_derivative}}
uniformly in $\lambda$. A helpful identity in the proof of~\eqref{eq:prelim:tlam_derivative} is~\eqref{eq:ftlam_chi_relation}. It follows from the Mecke equation, as
	\eqq{ \chi(\lambda) = 1 + \E_\lambda\Big[ \sum_{x \in \eta} \mathds 1_{\{\conn{\orig}{x}{\xi^{\orig}} \}} \Big] = 1 + \lambda \int \tlam(x) \dd x, \label{eq:prelim:ftlam_chi_relation_proof}}
and it implies that $\int \tlam(x) \dd x < \infty$ for $\lambda<\lambda_c=\lambda_T$. To prove~\eqref{eq:prelim:tlam_derivative}, note that $\{ \conn{\orig}{x}{\xi^{\orig,x}_{\Lambda_n}}\}$ lives on the bounded set $\Lambda_n$, and so we can apply the Margulis-Russo formula~\eqref{eq:prelim:russo}, which gives differentiability and an explicit expression for the derivative as
	\eqq{ \frac{\dd}{\dd\lambda} \tau_\lambda^n(x) = \int_{\Lambda_n} \pla \left( \conn{\orig}{x}{\xi^{\orig,y,x}_{\Lambda_n}}, \nconn{\orig}{x}{\xi^{\orig,x}_{\Lambda_n}} \right) \dd y. 
							\label{eq:prelim:tlam_finite_derivative}}
As a consequence of~\eqref{eq:prelim:tlam_finite_derivative}, we can write
	\al{ \left\vert\frac{\dd}{\dd\lambda} \tlam^n(x) - f_\lambda(x)\right\vert &= \bigg\vert \int_{\Lambda_n} \Big( \pla(\conn{\orig}{x}{\xi^{\orig,y,x}}) -\pla(\conn{\orig}{x}{\xi^{\orig,x}}) \\
		& \qquad\qquad +\pla(\conn{\orig}{x}{\xi^{\orig,x}_{\Lambda_n}}) - \pla(\conn{\orig}{x}{\xi^{\orig,y,x}_{\Lambda_n}}) \Big) \dd y \\
		& \quad + \int_{\Lambda_n^c} \pla(\conn{\orig}{x}{\xi^{\orig,y,x,}}, \nconn{\orig}{x}{\xi^{\orig,x}}) \dd y \bigg\vert \\
		& \leq \bigg\vert \int_{\Lambda_n} \pla(\viaconn{\orig}{x}{\xi^{\orig,y,x}}{\Lambda_n^c}) - \pla(\viaconn{\orig}{x}{\xi^{\orig,x}}{\Lambda_n^c}) \dd y \bigg\vert\\
		& \quad + \int_{\Lambda_n^c} \pla(\conn{\orig}{y}{\xi^{\orig,y}} ) \dd y \\
		& \leq \int_{\Lambda_n} \pla(\viaconn{\orig}{x}{\xi^{\orig,y,x}}{\Lambda_n^c} \text{ and through } y) \dd y 
						+ \int_{\Lambda_n^c} \tau_{\lambda_c-\varepsilon}(y) \dd y .}
Now, observe that the event $\{\viaconn{\orig}{x}{\xi^{\orig,y,x}}{\Lambda_n^c} \text{ and through } y\}$ is contained in
	\al{ & \Big( \{\conn{\orig}{y}{\xi^{\orig,y}}\} \circ \{\conn{y}{\Lambda_n^c}{\xi^{y}}\} \circ \{\conn{\Lambda_n^c}{x}{\xi^{x}}\} \Big) \\
	 \cup & \Big( \{\conn{\orig}{\Lambda_n^c}{\xi^{\orig}}\} \circ \{\conn{\Lambda_n^c}{y}{\xi^{y}}\} \circ \{\conn{y}{x}{\xi^{y,x}}\} \Big). }
Applying the BK inequality together with $\int_{\Lambda_n^c} \tau_{\lambda_c-\varepsilon}(y) \dd y = o(1)$ as $n\to\infty$ gives
	\al{ \left\vert\frac{\dd}{\dd\lambda} \tlam^n(x) - f_\lambda(x)\right\vert & \leq \int_{\Lambda_n} \pla(\conn{y}{\Lambda_n^c}{\xi^y}) \\
		& \qquad \times \Big[\tlam(y)\pla(\conn{x}{\Lambda_n^c}{\xi^x}) + \tlam(x-y)\pla(\conn{\orig}{\Lambda_n^c}{\xi^{\orig}}) \Big] \dd y +o(1) \\
		& \leq 2 \max_{z\in\{\orig,x\}} \p_{\lambda_c-\varepsilon}\big(\conn{z}{\Lambda_n^c}{\xi^z}\big) \int \tau_{\lambda_c-\varepsilon}(y) \dd y + o(1) = o(1), }
as the remaining integral is again bounded and the remaining probability tends to zero uniformly in $\lambda$. 
\ch{This implies assertions (a) and (b) of the lemma. For (c) we observe that the}
 uniform convergence justifies the exchange of limit and derivative in~\eqref{eq:prelim:tlam_derivative_limit_exchange} (see, e.g.,~\cite[Thm. 7.17]{Rud76}.
\end{proof}
We close this section by deriving a useful differential inequality:
\begin{lemma}[A differential inequality for $\chi(\lambda)$] \label{lem:differentialinequalityclustersize} 
Let $\lambda<\lambda_c$. Then
	\[ \frac{\dd}{\dd \lambda} \ftlam(\orig) \leq \ftlam(\orig)^2.\]
\end{lemma}
\begin{proof}
First note that with~\eqref{eq:prelim:tlam_finite_derivative} (and the BK inequality), we can bound
	\eqq{ \frac{\dd}{\dd\lambda} \tau_\lambda^n(x) \leq \int \pla \left( \{\conn{\orig}{y}{\xi^{\orig,y}_{\Lambda_n}} \}\circ\{ \conn{y}{x}{\xi^{y,x}_{\Lambda_n}} \} \right) \dd y \leq \int \tlam^n(y)\tlam^n(y,x) \dd y.
					 \label{eq:prelim:tlam_finite_derivative_bound}}
We can now use Leibniz' integral rule in its measure-theoretic form to write $\tfrac{\dd}{\dd\lambda} \int \tlam(x) \dd x = \int \tfrac{\dd}{\dd\lambda} \tlam(x) \dd x$. This is justified as the integrand $\tlam$ is uniformly bounded by the integrable function $\tau_{\lambda_c-\varepsilon}$ for some small $\varepsilon\in(0,\lambda_c-\lambda)$. Applying Lemma~\ref{lem:tlam_differentiability} as well as~\eqref{eq:prelim:tlam_finite_derivative_bound}, we derive
	\[ \frac{\dd}{\dd \lambda} \ftlam(\orig) = \int \lim_{n\to\infty} \frac{\dd}{\dd\lambda} \tau_\lambda^n(x) \dd x 
				 \leq \int \tlam(y) \left( \int \tlam(x-y) \dd x \right) \dd y = \ftlam(\orig)^2. \qedhere\] \end{proof}
\section{The expansion}\label{sec:expansion}
\subsection{Preparatory definitions}
The aim of this section is to prove the expansion for $\tlam$ stated
in~\eqref{eq:laceexpansionequation}. It is one of the goals of the
subsequent sections to show that $R_{\lambda, n} \to 0$ as
$n\to\infty$ when $\lambda<\lambda_c$. The intuitive idea behind the
expansion is quite simple. Loosely speaking,
$\{\conn{\orig}{x}{\xi^{\orig,x}}\}$ is partitioned over the first
pivotal point $u \in \eta$ for this connection (if such a point
exists). That is, there is a double connection between $\orig$ and $u$
and we recover $\tlam(x-u)$, due to the event that
$\{\conn{u}{x}{\xi^{u,x}}\}$.  However, this is not quite true, and
the overcounting error made by pretending as if the double connection
event between $\orig$ and $u$ was \emph{independent} of the connection
event between $u$ and $x$ has to be subtracted. This constitutes the
first step of the expansion. The second step is to further examine
this error term, in which we recognize a similar structure, allowing
for a similar partitioning strategy again.

Making this informal strategy of proof precise requires definitions,
starting with some extended connection events:

\begin{definition}[Connectivity terminology] \label{def:LE:connection_events} Let $u,v,x \in \Rd$. 
\begin{compactitem}
\item[(1.)] We say $u$ and $x$ are \emph{$2$-connected} in $\xi$ if
  $u=x\in\eta$, if $u,x\in\eta$ and $u\sim x$, or if $u,x\in\eta$ and
  there are (at least) two paths between $u$ and $x$ that are disjoint
  in all their interior vertices; that is, there are two paths that
  only share $u$ and $x$ as common vertices. 
\cov{In this case we write  that $\dconn{u}{x}$ $\xi$. Note that
  $\{\dconn{u}{x}{\xi^{u,x}}\}$ is a well-defined event for any $u,x\in\R^d$.}

\item[(2.)] For \cov{a locally finite set $A\subset\R^d$}, we say that $u$ and $x$ are
  connected in $\xi$ \emph{off} $A$ and write
\cov{that $\offconn{u}{x}{\xi}{ A}$ if
  $\conn{u}{x}{\xi[\eta \setminus A]}$ (where we recall from Section \ref{sec:constructionofRCM} that
    $\xi[\eta\setminus A]$ is the subgraph of $\xi$ induced by $\eta\setminus A$). This
means that $u,x\in\eta$ and there exists a path between
  $u$ and $x$ in $\xi$ not using any vertices of $A$. In particular,
  this fails if $A$ contains $u$ or $x$.
Again $\{\offconn{u}{x}{\xi^{u,x}}{A}\}$ is a well-defined event for any $u,x\in\R^d$.
}
\end{compactitem}
\end{definition}

\cov{
Given $u,x\in\R^d$ we have avoided to refer to  
$\{\dconn{0}{x}{\xi^0}\}$,
$\{\conn{u}{x}{\xi}\}$
or $\{\dconn{u}{x}{\xi}\}$ as events. The reason is that
$\pla(x\in\eta)=0$ for all $x$.
Still we can apply the Mecke equation \eqref{eq:prelim:mecke_1}
to the indicator function
$\mathds 1_{\{x\in\eta,\dconn{0}{x}{\xi^0}\}}$ (in fact we have already
done so at \eqref{eq:prelim:ftlam_chi_relation_proof})
or the bivariate Mecke formula (case $n=2$ of \eqref{eq:prelim:mecke_n})
to the indicator function
$\mathds 1_{\{u,x\in\eta,\dconn{u}{x}{\xi}\}}$. After the application
of the Mecke formula we arrive at events of the type
$\{\dconn{0}{x}{\xi^{0,x}}\}$ or $\{\conn{u}{x}{\xi^{u,x}}\}$,
which make perfect sense for all $u,x\in\R^d$.
}

The next definitions introduce thinnings, a standard concept in
point process literature. Recall from
Section~\ref{sec:constructionofRCM} that every Poisson point
$X_i\in\eta$ comes with a sequence of ``thinning marks''
$(Y_{i,j})_{j\in\N}$.

\begin{definition}[Thinning events] \label{def:LE:thinning_events} Let
  $u,x \in \Rd$, and let $A \subset \Rd$ be locally finite and of
  cardinality $|A|$.
\begin{compactitem}
\item[(1.)] Set \eqq{ \bar\connf(A,x) :=\prod_{y\in
      A}(1-\connf(y-x)) \label{eq:LE:def:thinning_probability}} and
  define $\thinning{\eta}{A}$ as a
  \emph{$\bar\connf(A, \cdot)$-thinning of $\eta$} (or simply
  \emph{$A$-thinning of $\eta$}) as follows. We keep a point
  $w \in \eta$ as a point of $\thinning{\eta}{A}$ with probability
  $\bar\connf(A,w)$ independently of all other points of $\eta$. To
  make this more explicit, we use the mappings $\pi_i$, $i\in\N$,
  introduced in Section~\ref{sec:pointprocesses}. In particular,
  $(\pi_j(A))_{j \leq |A|}$ is an ordering of the points in $A$ and
  $(\pi_i(\eta))_{i\in\N}$ is an ordering of the points in $\eta$.

We keep $\pi_i(\eta) \in \eta$ as a point of $\thinning{\eta}{A}$ if
$Y_{i,j} > \connf(\pi_j(A)-\pi_i(\eta))$ for all $j\leq |A|$ (we say
that $\pi_i(\eta)$ \emph{survives} the $A$-thinning). We further
define $\thinning{\eta}{A}^x$ as a $\bar\connf(A,\cdot)$-thinning of
$\eta^x$ using the marks in $\xi^x$.
\item[(2.)] \cov{We write $\xconn{u}{x}{\xi}{A}$ if $\conn{u}{x}{\xi}$
    and $\nconn{u}{x}{\xi[\thinning{\eta}{A} \cup\{u\}]}$.  This
  means that $u,x\in\eta$ and $u$ is connected to $x$ in $\xi$, but
  this connection does not survive an $A$-thinning of
  $\eta\setminus\{u\}$.} In particular, the connection does not survive
  if $x$ is thinned out.
\item[(3.)] We define
	\eqq{ \tlam^{A}(u,x) = \pla \left(\conn{u}{x}{\xi^{u,x}[\thinning{\eta^x}{A}\cup\{u\}]} \right). \label{eq:def:LE:offconn}}
In words, $\tlam^A(u,x)$ is the probability of the event that there exists an open path between $u$ and $x$ in an RCM driven by an $A$-thinning of $\eta^x$, where the point $u$ is fixed to be present (but $x$ is not).
\end{compactitem}
\end{definition}

Let us make some remarks about
Definition~\ref{def:LE:thinning_events}. First, the definition of
$\thinning{\eta}{A}$ is the first occasion that we need the enriched
version of $\xi$ from Section~\ref{sec:constructionofRCM}. It is due
to the independence of the sequence $(Y_{i,j})_{i,j\in\N}$ that the
(conditional) probability of some point $y\in\eta$ being contained in
$\thinning{\eta}{A}$ is indeed $\bar\connf(A,y)$.

Secondly, it is due to the thinning properties of a Poisson point process that $\thinning{\eta}{A}$ has the distribution of an inhomogeneous PPP with intensity $\lambda \bar\connf(A, \cdot)$ (see, e.g.,~\cite{LasPen17}). Thirdly, $\{ \xconn{u}{x}{\xi^{u,x}}{A} \}$ is \emph{not} symmetric w.r.t.~$u$ and $x$, since $x$ can be thinned out, but $u$ can not. Lastly, note that the events considered in (2.) and (3.) of Definition~\ref{def:LE:thinning_events} are complementary in the sense that 
	\eqq{ \{ \conn{u}{x}{\xi^{u,x}}\} = \{\conn{u}{x}{\xi^{u,x}[\thinning{\eta^x}{A}\cup\{u\}]}\} \cup \{\xconn{u}{x}{\xi^{u,x}}{A}\},  \label{eq:LE:tau_thinning_split}}
and the above union is disjoint. This observation will be a crucial ingredient in the lace expansion.

\subsection{Stopping sets and cutting points}
Before deriving the expansion, we state and prove the Cutting-point Lemma (see Lemma~\ref{lem:LE:cutting_point}). This lemma is crucial in deriving an expansion and quite standard in the literature; we view it as an analogue of the Cutting-bond lemma (see~\cite[Lemma 6.4]{HeyHof17}). 

One central idea in the proof of the Cutting-point Lemma~\ref{lem:LE:cutting_point} is to use the \emph{stopping set} properties of $\mathscr C(v, \xi^v)$. We therefore start with Lemma~\ref{lem:stoppingset}, which rigorously formulates the properties we need.

To stress the dependence of $\xi$ on $\eta$, we write $\xi(\eta):=\xi$ in the statement of Lemma~\ref{lem:stoppingset} and parts of its proof, even though this notation is a bit ambiguous. First, it does not reflect the dependence of $\xi$ on the marks $U_{i,j}$. Secondly, the definition of $\xi$ depends on the ordering of the points of $\eta$. The distribution of $\xi$, however, does not depend on this ordering.

\begin{lemma}[Stopping-set lemma]\label{lem:stoppingset} Let $v\in\Rd$. Then
	\eqq{ \pla\big(\xi^v[\eta^v\setminus \C(v,\xi^v)]\in\cdot \mid \C(v,\xi^v) =A\big)	= \pla (\xi(\thinning{\eta}{A})\in\cdot)\quad 
					\text{for }\pla(\C(v,\xi^v)\in\cdot)\text{-a.e.\ $A$}. \label{eq:stopping_set_lemma}}
\end{lemma}

Before giving a proof we explain the distributional identity \eqref{eq:stopping_set_lemma}.
On the left-hand side, we have the conditional distribution of the restriction of $\xi^v$ to the complement of the cluster $\C(v,\xi^v)$ given that $\C(v,\xi^v)=A$.
On the right-hand side, we have an independent edge-marking based on the inhomogeneous Poisson process $\thinning{\eta}{A}$.
Even though the latter is defined as an independent thinning of $\eta$ (its intensity is bounded by $\lambda$), the point process $\eta\setminus\C(v,\xi^v)$ cannot be constructed this way.
In fact, neither $\C(v,\xi^v)$ nor $\eta\setminus \C(v,\xi^v)$ is a Poisson process.

The proof is based on a recursive construction of the cluster, in increasing graph distance from the root; this is also the reason why subcriticality is not required.

Moreover, we want to point out that the following proof is for the RCM as defined in~\eqref{eq:prelim:xi_def}. The proof for the RCM with additional marks is essentially the same, just heavier on notation.

\begin{proof}[Proof of Lemma \ref{lem:stoppingset}] 

  The proof is similar to \chr{Meester et al.\ \cite[Proposition
  2]{MeePenSar97}.} Since the lemma is crucial for our paper, we give
  more details here.  We interpret $\xi^v[\C(v,\xi^v)]$ as a rooted
  graph with root $v$.  Let $\eta_0:=\{v\}$. For $n\in\N$ let $\eta_n$
  be the vertices of $\C(v,\xi^v)$ whose graph distance from the root
  is at most $n$. \ch{Writing $A_0=\eta_0=\{v\}$,} we assert that, for every $n\in\N$,
  \eqq{ \label{e3.4} \E_\lambda [f(\xi^v[\eta\setminus
    \eta_n],\ch{\eta_0,}\ldots,\eta_n)]
    =\int\mathbb{E}_\lambda[f(\xi(\thinning{\eta}{A_{n-1}}),\ch{A_0},\ldots,A_n)]
    \,\pla((\ch{\eta_0,}\ldots,\eta_n)\in\dd (\ch{A_0,}\ldots,A_n)),} 
  for all measurable non-negative functions $f$
  with suitable domain, and we recall that
  the arguments of $f$ are point processes.

\begin{figure}\centering
	\includegraphics[width=.5\textwidth]{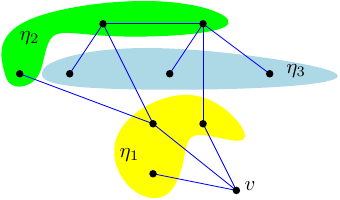}
	\caption{\ch{An example of the breadth-first-exploration as used in the proof of Lemma \ref{lem:stoppingset}.}}
\end{figure}

We prove a slightly more general version of \eqref{e3.4}, which is amenable to induction.
To do so, we need to introduce some further notation.
Given two disjoint point processes $\eta',\eta''\subset \eta^v$, we define $\xi^v[\eta',\eta'']$ as the random marked graph arising from $\xi^v$ by taking all marked edges with at least one vertex in $\eta'$ and no vertex outside of $\eta' \cup \eta''$.
Given a point process $\mu$ on $\Rd$ and a locally finite set $A\subset\Rd$, we define a random marked graph $T(\mu,A)$ as follows.
The edge set is given by $\{\{x,y\}: x \in \mu, y\in \mu \cup A\}$.
The marks are given by independent random variables, uniformly distributed on $[0,1]$ and independent of $\mu$. 

\cov{We assert for each $n\in\N$ that} 
	\begin{align}\label{e3.41}\notag
	\E_\lambda &[f(\xi^v[\eta\setminus \eta_n,\eta_n\setminus\eta_{n-1}],\ch{\eta_0},\ldots,\eta_n)]\\ 
	&=\int\mathbb{E}_\lambda[f(T(\thinning{\eta}{A_{n-1}},A_n \setminus A_{n-1}),\ch{A_0},\ldots,A_n))]
	\pla((\eta_1,\ldots,\eta_n)\in\dd (\ch{A_0},\ldots,A_n))
	\end{align}
for all measurable non-negative functions $f$ with suitable domain, which
is clearly more general than \eqref{e3.4}. This can be written as
	\begin{align}\label{e3.42}
	\mathbb{P}_\lambda(\xi^v[\eta\setminus \eta_n,\eta_n \setminus\eta_{n-1}]&\in\cdot
	\mid (\eta_0,\ldots,\eta_n)=(A_0, \ldots,A_n))
	=\pla(T(\thinning{\eta}{A_{n-1}},A_n\setminus A_{n-1})\in\cdot)
	\end{align}
for $\pla((\eta_0,\ldots,\eta_n)\in\cdot)$-a.e.\ $(A_0,\ldots,A_n)$, and \ch{where} $A_0\ch{=\eta_0} = \{v\}$.

We base the proof of \eqref{e3.42} on the following property.
Let $h\colon\Rd\to[0,\infty)$ be measurable and let $\mu$ be a Poisson process with intensity function $h$.
Further, let $A\subset\Rd$ be locally finite and consider the independent edge-marking $\tilde\xi:=\xi(\mu\cup A)$ of $\mu\cup A$.
Let $\mu^A$ be the set of points from $\mu$ that are directly connected to a point from $A$, where the connection is defined as before in terms of $\tilde\xi$ and the connection function $\connf$. 
Then we have the distributional identity
	\begin{align}\label{e3.63}
	 (\xi[\mu\setminus \mu^A,\mu^A],\mu^A)\overset{d}{=}(T(\mu',\mu''),\mu''),
	\end{align}
where $\mu'$ and $\mu''$ are independent Poisson processes with intensity functions $h(\cdot)\bar\connf(A,\cdot)$ and $h(\cdot)(1-\bar\connf(A,\cdot))$, respectively.
This follows from the marking and mapping theorems for Poisson processes (see~\cite[Theorems~5.6 and~5.1]{LasPen17}) applied to a suitably defined Poisson process $\tilde\xi$ such that $\xi(\mu\cup A)$ is (up to the marks of edges with both vertices in $A$) a deterministic function of $\tilde\xi$.
The details of this construction are left to the reader.

Applying \eqref{e3.63} with $A=\{v\}$ and $\mu=\eta$ gives \eqref{e3.42} for $n=1$. 
Suppose \eqref{e3.42} is true for some $n\in\N$ and let $A_1,\ldots,A_n$ be locally finite subsets of $\Rd$.
Applying \eqref{e3.63} with the conditional probability measure $\mathbb{P}(\cdot \mid (\eta_0,\ldots,\eta_n)=(A_0,\ldots,A_n))$ and with $\mu=\thinning{\eta}{A_{n-1}}$ as well as $A=A_n\setminus A_{n-1}$ gives \eqref{e3.42} for $n+1$.

In fact, this argument also yields that, given $(\eta_0,\ldots,\eta_n)$, the point processes $\eta\setminus\eta_{n+1}$ and $\eta_{n+1}\setminus\eta_n$ are conditionally independent Poisson processes with intensity functions $\lambda\bar\connf(\eta_n,\cdot)$ and $\lambda (1-\bar\connf(\eta_n\setminus\eta_{n-1},\cdot))\bar\connf(\eta_{n-1},\cdot)$, respectively. Since
	\begin{align*}
	1-\bar\connf(\eta_n\setminus\eta_{n-1},x)\le \sum_{w\in \eta_n\setminus\eta_{n-1}}\connf(w-x),
	\quad x\in\Rd,
	\end{align*}
it follows by induction and by the integrability of $\connf$ that the point processes $\eta_n$ are all finite almost surely.

Equation \eqref{e3.4} shows in particular that
	\begin{align}\label{e3.53}
	\E_\lambda [f(\xi^v[\eta\setminus \eta_n],\eta_n)]
	=\int\mathbb{E}_\lambda[f(\xi(\thinning{\eta}{V_{n-1}(G)}),V_n(G))]
	\pla(\xi^v[\C(v,\xi^v)]\in\dd G)
	\end{align}
for all measurable non-negative functions $f$ with suitable domain, where, for a rooted graph $G$ and $n\in\N_0$, $V_n(G)$ denotes the set of vertices of $G$ whose graph distance from the root is at most $n$.
Let $\eta_\infty=\cup_n\eta_n$ denote the vertex set $\C(v,\xi^v)$.
Note that for a bounded Borel set, we have that $|\eta_\infty \cap B|=|\eta_n \cap B|$ for all sufficiently large $n$ almost surely.
Note also that $\xi^v[\eta\setminus \eta_n]\downarrow \xi^v[\eta\setminus \eta_\infty]$ as $n\to\infty$.
Therefore, if $f(\xi^v[\eta\setminus \eta_n],\eta_n)$ is a bounded function of $|\xi^v[\eta\setminus \eta_n] \cap B_1|,\ldots,|\xi^v[\eta\setminus \eta_n] \cap B_k|$ and $|\eta_n \cap B_{k+1}|,\ldots,|\eta_n \cap B_{m}|$ for suitable measurable and bounded sets $B_1,\ldots,B_m$, the left-hand side of \eqref{e3.53} tends to $\mathbb{E}_\lambda [f(\xi^v[\eta\setminus \eta_\infty],\eta_\infty)]$ as $n\to\infty$.

For a similar reason, the integrand on the right-hand side converges for each fixed rooted (locally finite) graph $G$ to $\mathbb{E}_\lambda[f(\xi(\thinning{\eta}{V(G)}),V(G)))]$, where $V(G)$ is the vertex set of $G$. Therefore, we obtain from dominated convergence that
	\begin{align}\label{e3.44}
	\E_\lambda [f(\xi^v[\eta\setminus \eta_\infty],\eta_\infty)] 
	=\int\mathbb{E}_\lambda[f(\xi(\thinning{\eta}{A}),A)]\pla(\eta_\infty\in\dd A),
	\end{align}
first for special non-negative $f$, and then, by a monotone-class argument, for general $f$. This implies the assertion.
\end{proof}

Lemma~\ref{lem:stoppingset} is a quite general distributional identity. We will only require the following corollary:
\begin{corollary} \label{lem:thinning}
Let $v,u,x\in\Rd$ \col{be distinct}. Then, for $\pla(\C(v,\xi^v)\in \cdot)$-a.e.\ $A$,
		\[ \pla\big(\offconn{u}{x}{\xi^{u,x}}{\C(v, \xi^{v})} \mid \C(v,\xi^v) =A\big) = \pla (\conn{u}{x}{\xi^{u,x}[\thinning{\eta}{A}\cup\{u,x\}]}). \]
\end{corollary}
\begin{proof}
\cov{First note that
\begin{align*}
E:=\big\{\offconn{u}{x}{\xi^{u,x}}{\C(v, \xi^{v})}\big\}
  =\big\{\conn{u}{x}{\xi^{v,u,x}[\eta^{v,u,x}\setminus \C(v, \xi^{v})]}\big\}.
\end{align*}
Consider the graph
$\xi^{v}[\eta^{v}\setminus \C(v, \xi^{v})]$ and add two vertices $u,x$
together with independent connections between $u,x$ and
$\eta^{v}\setminus \C(v, \xi^{v})$. This gives a random graph
$\tilde\xi^{v,u,x}$. In fact, we can write in a generic way that
\begin{align*}
\tilde\xi^{v,u,x}=T(\eta^{v}\setminus \C(v, \xi^{v}),U,U'),
\end{align*}
where $U,U'$ are independent (double) sequences
of independent uniformly on $[0,1]$ distributed random variables,
independent of everything else, and $T$ is  a well-defined measurable
mapping. Since there are no connections between
$\C(v, \xi^{v})$ and $\eta^{v}\setminus \C(v, \xi^{v})$, 
we have that
\begin{align*}
\pla((\mathds 1_E,\C(v, \xi^{v}))\in\cdot)=\pla((\mathds 1_{\tilde E},\C(v, \xi^{v}))\in\cdot),
\end{align*}
where $\tilde{E}:=\big\{\conn{u}{x}{\tilde{\xi}^{v,u,x}}\big\}$.
Therefore, the left-hand side of the asserted identity equals
\begin{align*}
\iint \pla\big(\conn{u}{x}{T(\eta^{v}\setminus \C(v, \xi^{v}),\mathbf{u},\mathbf{u}')}
\mid \C(v,\xi^v)=A\big)
\,\pla(U\in \dd\mathbf{u})\,\,\pla(U'\in \dd\mathbf{u}').
\end{align*}
By Lemma \ref{lem:stoppingset}, this equals
$\pla(\conn{u}{x}{T(\eta_{\langle A\rangle},U,U')})$.
It remains to notice that $T(\eta_{\langle A\rangle},U,U')$ has the same
distribution as $\xi^{u,x}[\thinning{\eta}{A}\cup\{u,x\}]$.
}
\end{proof}

We next introduce the notion of \emph{pivotal points}. To this end,
let $\xi$ be an edge-marking of a PPP $\eta$ and let $v,u,x \in
\eta$. We say that $u\notin\{v,x\}$ is \emph{pivotal} for the
connection from $v$ to $x$ (and write $u \in \piv{v,x;\xi}$) if
\cov{$\conn{v}{x}{\xi}$ but
$\nconn{v}{x}{\xi[\eta\setminus\{u\}]}$.} Mind that, by definition,
$v$ and $x$ are never elements of $\piv{v,x;\xi}$. We list $\xi$ as an
argument after the semicolon to indicate decorations of $\xi$ with
extra points. This way, we can speak of the event
$\{ u \in \piv{v,x;\xi^{v,u,x}}\}$ for arbitrary $v,u,x\in\Rd$. We use
the same notation for events that are introduced later.

Note that $\piv{v,x;\xi} = \piv{x,v;\xi}$, but we use the notation to put emphasis on paths ``from $v$ to $x$''. A non-symmetric property of pivotal points that we use later is the fact that $\piv{v,x;\xi}$ can be ordered in the sense that there is a unique first (second, third, etc.) pivotal point that every path from $v$ to $x$ traverses first (second, third, etc.). Furthermore, for a locally finite set $A \subset\Rd$, and $v,u\in \Rd$, we define
		\eqq{E(v,u;A,\xi) := \{\xconn{v}{u}{\xi}{A} \} \cap \{\nexists w \in \piv{v,u;\xi}: \xconn{v}{w}{\xi}{A} \}. \label{eq:LE:def:E_event}}
Let us take the time to prove an elementary partitioning identity here, which will be useful at a later stage:

\begin{lemma}[Partition of connection events] \label{lem:partitioning_via_E_events}
Let $v,x\in\Rd$ and let $A \subset \Rd$ be a locally finite set. Then
	\[ \mathds 1_{\{ \xconn{v}{x}{\xi^{v,x}}{A} \}} = \mathds 1_{E(v,x;A, \xi^{v,x})} + \sum_{u \in \eta} \mathds 1_{E(v,u;A,\xi^{v,x})} \mathds 1_{\{ u \in \textsf{\textup{Piv}}(v,x;\xi^{v,x})\}}. \]
\end{lemma}
\begin{proof}
We prove ``$\geq$'' first. We first claim that the right-hand side is a sum of indicators of mutually disjoint events. Indeed, due to the ordering of pivotal points $y$ satisfying $\{ \xconn{v}{y}{\xi^{v}}{A} \}$, the choice of $u$ as the first such pivotal point is unique, making the union over first pivotal points $u$ a disjoint one. Moreover, $E(v,x;A, \xi^{v,x})$ is the event that the set of such pivotal points is empty.

Assume now that the right-hand side takes value $1$. On the one hand, if $E(v,x;A,\xi^{v,x})$ holds, then $\{\xconn{v}{x}{\xi^{v,x}}{A} \}$ holds as well by definition. On the other hand, assume that $\xi$ contains a point $u \in \eta=V(\xi)$ such that $\xi\in E(v,u;A,\xi^{v,x})$ \col{and $u \in \piv{v,x;\xi^{v,x}}$}. Due to the pivotality of $u$, any path $\gamma$ from $v$ to $x$ must be the concatenation of two disjoint paths $\gamma_1$ and $\gamma_2$ (i.e.,~$\gamma_1$ and $\gamma_2$ share no interior vertices), where $\gamma_1$ is a path from $v$ to $u$ and $\gamma_2$ is a path from $u$ to $x$. Since $\{\xconn{v}{u}{\xi^{v,x}}{A} \}$ holds, there must be a vertex $\pi_i(\eta) \in \gamma_1$ that is thinned out. By definition, there is some $\pi_j(A)$ such that $Y_{i,j} \leq \connf(\pi_j(A)-\pi_i(\eta))$. In other words, $\pi_i(\eta)$ is deleted in an $A$-thinning of $\eta$, and so $\{ \xconn{v}{x}{\xi^{v,x}}{A} \}$ holds. Thus, ``$\geq$'' holds.

To see ``$\leq$'', assume that $\{ \xconn{v}{x}{\xi^{v,x}}{A} \}$
holds. Then either $E(v,x;A, \xi^{v,x})$ holds, or there is at least
one pivotal point $y$ satisfying \cov{$\xconn{v}{y}{\xi^{v}}{A}$}. 
Since the pivotal points can be ordered, we can pick the first
such pivotal point and call it $u$. This point $u$ then satisfies
$E(v,u;A,\xi^{v,u})$.
\end{proof}

The following lemma has an analogue in discrete models, see e.g.~\cite[Lemma 2.1]{HarSla90}. In bond percolation, it is called the ``Cutting-bond lemma''. Since Lemma~\ref{lem:stoppingset} holds for arbitrary intensity, so does Lemma~\ref{lem:LE:cutting_point}.
\begin{lemma}[Cutting-point lemma] \label{lem:LE:cutting_point}
Let $\lambda \geq 0$ and let $v,u,x \in \Rd$ with $u \neq x$ and let $A\subset \Rd$ be locally finite. Then
		\[ \E_\lambda\big[\mathds 1_{E(v,u;A,\xi^{v,u,x})} \mathds 1_{\{ u \in \textsf{\textup{Piv}}(v,x;\xi^{v,u,x})\}} \big] 
					= \E_\lambda\left[\mathds 1_{E(v,u;A,\xi^{v,u})} \cdot \tlam^{\C(v,\xi^v) }(u,x) \right].\]
Moreover,
		\[ \pla\left(\dconn{\orig}{u}{\xi^{\orig,u,x}}, u \in \textsf{\textup{Piv}}(\orig,x; \xi^{\orig,u,x})\right) = \E_\lambda\left[\mathds 1_{\{\dconn{\orig}{u}{\xi^{\orig,u}}\}} \cdot \tlam^{\C(\orig,\xi^{\orig}) }(u,x) \right]. \]
\end{lemma}
Before proceeding with the proof, we want to stress the fact that $\tlam^{\C(v,\xi^v) }(u,x)$ is the random variable arising from $\tlam^A(u,x)$ by replacing the fixed set $A$ by the random set $\C(v,\xi^v)$.
\begin{proof}
First, note that
		\[E(v,u;A,\xi^{v,u,x}) \cap \{ u \in \piv{v,x; \xi^{v,u,x}}\} = E(v,u;A,\xi^{v,u}) \cap \{ u \in \piv{v,x; \xi^{v,u,x}}\}. \]
In words, we can take away vertex $x$ from $\xi^{v,u,x}$ in the event $E(v,u;A, \xi^{v,u,x})$, since if $x$ was necessary (or even relevant) for the connection from $v$ to $u$, then $u$ would not be pivotal. Furthermore, abbreviating $\C=\C(v, \xi^v)$ \chr{and denoting the complement of $x\sim y$ by $x\nsim y$},
		\[ \{ u \in \piv{v,x; \xi^{v,u,x}}\} = \{\conn{v}{u}{\xi^{v,u}} \} \cap \{\offconn{u}{x}{\xi^{u,x}}{\C}\} \cap \{ x \nsim y \text{ in } \xi^{v,x} \; \forall y \in \C \} \]
$\pla$-a.s.~by the following argument: If $u$ is pivotal, then $\C$ contains all vertices connected to $v$ by a path not using $u$, and in return any path from $x$ to $v$ visits $u$ before it hits $\C$. Both these statements use that $u \notin \C$ a.s. In particular, the first two connection events on the right-hand side hold and there cannot be a direct edge from $x$ to $\C$. This proves one inclusion.
Conversely, if $u$ and $x$ are connected off $\C$, then $x$ cannot lie in $\C$. Moreover, it cannot even lie in $\mathscr C(v, \xi^{v,x})$ as this would imply the existence of an edge from $x$ to $\C$. Consequently, every path from $v$ to $x$ must pass through $u$. As $u$ is connected to $v$, this makes $u$ a pivotal point in $\xi^{v,u,x}$, proving the second inclusion.

Since $E(v,u;A,\xi^{v,u})\subseteq \{\conn{v}{u}{\xi^{v,u}} \}$,
		\al{ E(v,u;A,\xi^{v,u}) \cap & \{ u \in \piv{v,x; \xi^{v,u,x}}\} \\
				&= E(v,u;A,\xi^{v,u}) \cap \{\offconn{u}{x}{\xi^{u,x}}{\C}\}\cap \{ x \nsim y \text{ in } \xi^{v,x} \; \forall y \in \C \}. }
Conditioning on $\xi'=\xi^{u,v}[\C(v,\xi^v) \cup \{u\}]$, we see that
		\[\E_\lambda\big[\mathds 1_{E(v,u;A,\xi^{v,u,x})} \mathds 1_{\{ u \in \textsf{\textup{Piv}}(v,x;\xi^{v,u,x})\}} \big] = \E_\lambda\big[ \mathds 1_{E(v,u;A,\xi^{v,u})} 
						\E_\lambda[ \mathds 1_{\{\offconn{u}{x}{\xi^{u,x}}{\C}\}} \mathds 1_{\{ x \nsim y \text{ in } \xi^{v,x} \; \forall y \in \C \}}\mid \xi'] \big] \]
by the fact that $E(v,u;A,\xi^{v,u})$ is measurable w.r.t.~$\sigma(\xi')$ (the $\sigma$-field generated by $\xi'$). Indeed, $\xi'$ is the graph induced by $u$ together with all points that can be reached from $v$ without traversing $u$. Now, conditionally on $\xi'$, the last two indicators are independent: $\{ x \nsim y \textrm{ in } \xi^{v,x} \, \forall \, y \in \C \}$ depends only on points in $\C \subseteq V(\xi')$ and edges between $\C$ and $x$. On the other hand, $\{\offconn{u}{x}{\xi^{u,x}}{\C}\}$ depends only on points in $\eta^{u,x}\setminus\C$ and on edges between those points.

Together with the identities $\pla(x \nsim y \textrm{ in } \xi^{v,x} \, \forall \, y \in \C \mid \xi') = \bar\connf(\C, x)$ (recall~\eqref{eq:LE:def:thinning_probability}) and
	\[ \bar\connf(B, x) \cdot \pla\big(\conn{u}{x}{\xi^{u,x}[\thinning{\eta}{B} \cup \{u,x\}]} \big) = \tlam^{B}(u,x)\]
for any locally finite set $B$ (recall the definition of $\tlam^B$ in~\eqref{eq:def:LE:offconn}), this leads to
		\al{\E_\lambda\big[\mathds 1_{E(v,u;A,\xi^{v,u,x})} \mathds 1_{\{ u \in \textsf{\textup{Piv}}(v,x;\xi^{v,x})\}} \big]
				& = \E_\lambda\big[ \mathds 1_{E(v,u;A,\xi^{v,u})} \cdot \bar\connf(\C, x) \cdot \E_\lambda[\mathds 1_{\{\offconn{u}{x}{\xi^{u,x}}{\C}\}} \mid \xi'] \big] \\
				&= \E_\lambda\big[ \mathds 1_{E(v,u;A,\xi^{v,u})} \cdot \bar\connf(\C, x) \cdot \pla\left(\offconn{u}{x}{\xi^{u,x}}{\C} \mid \C\right) \big] \\
				&= \E_\lambda\big[ \mathds 1_{E(v,u;A,\xi^{v,u})} \cdot \bar\connf(\C, x) \cdot \pla(\conn{u}{x}{\xi^{u,x}[\thinning{\eta}{\C} \cup \{u,x\}] }) \big] \\
				&= \E_\lambda\big[ \mathds 1_{E(v,u;A,\xi^{v,u})} \cdot \tlam^{\C}(u,x) \big].}
In the second line, we have used that $\sigma(\xi')$ and $\sigma(\C)$ (the $\sigma$-fields generated by $\xi'$ and $\C$ respectively) differ only in the information about the status of edges between points of $\C \cup \{u\}$. Since any connection event \emph{off} $\C$ is independent of such edges, we can replace $\xi'$ by $\C$ in the conditioning to use Corollary~\ref{lem:thinning} in the third line.

The second assertion of Lemma~\ref{lem:LE:cutting_point} follows upon applying the above arguments with $E(v,u;A,\xi^{v,u})$ replaced by $\{\dconn{\orig}{u}{\xi^{\orig,u}}\}$.
\end{proof}

\subsection{The derivation of the expansion} \label{sec:LE:derivation_of_LE}
For the following definition, we introduce a sequence of independent edge-markings $(\xi_i)_{i \in\N_0}$ of respective PPPs $(\eta_i)_{i\in\N_0}$. Recall that $\int\cdots\dd \vec u_{[0,n]}$ is a sequence of integrals (over $\R^d$) over $u_0,\dots,u_n$. 
\begin{definition}[Lace-expansion coefficients] \label{def:LE:lace_expansion_coefficients}
For $n\in\N$ and $x\in\Rd$, we define
	\algn{ \Pi_\lambda^{(0)}(x) &:= \pla (\dconn{\orig}{x}{\xi^{\orig,x}}) - \connf(x), \label{eq:LE:Pi0_def} \\
			\Pi_\lambda^{(n)}(x) &:= \lambda^n \int \pla \Big( \{\dconn{\orig}{u_0}{\xi^{\orig, u_0}_{0}}\}
					\cap \bigcap_{i=1}^{n} E(u_{i-1},u_i; \C_{i-1}, \xi^{u_{i-1}, u_i}_{i}) \Big) \dd \vec u_{[0,n-1]} , \label{eq:LE:Pin_def} }
where $u_n=x$ and $\C_{i} = \C(u_{i-1}, \xi^{u_{i-1}}_{i})$ is the cluster of $u_{i-1}$ in $\xi^{u_{i-1}}_i$. Further define
	\algn{ R_{\lambda, 0} (x) &:= - \lambda \int \pla \Big( \{\dconn{\orig}{u_0}{\xi^{\orig, u_0}_0}\} \cap \{\xconn{u_0}{x}{\xi^{u_0,x}_1}{\C_0}\}  \Big) \dd u_0, \label{eq:LE:R0_def}\\
		R_{\lambda, n}(x) &:= (-\lambda)^{n+1} \int \pla \Big( \{\dconn{\orig}{u_0}{\xi^{\orig, u_0}_0}\} \cap \bigcap_{i=1}^{n} E(u_{i-1},u_i; \C_{i-1}, \xi^{u_{i-1}, u_i}_{i}) \notag \\
				& \hspace{4cm} \cap \{ \xconn{u_n}{x}{\xi^{u_n,x}_{n+1}}{\C_n} \} \Big) \dd \vec u_{[0,n]}
							 \label{eq:LE:Rn_def}.}
Additionally, define $\Pi_{\lambda, n}$ as the alternating partial sum
	\eqq{ \Pi_{\lambda, n}(x) := \sum_{m=0}^{n} (-1)^m \Pi_\lambda^{(m)}(x). \label{eq:LE:PiN_def}} 
\end{definition}
We can relate $\Pi_\lambda^{(n)}$ and $R_{\lambda,n}$ in the following way. As $\pla(\xconn{u_n}{x}{\xi^{u_n,x}_{n+1}}{A} ) \leq \tlam(x-u_n)$ for an arbitrary locally finite set $A$, we can bound
	\eqq{ |R_{\lambda,n}(x)| \leq \lambda \int \Pi_\lambda^{(n)}(u_n) \tlam(x-u_n) \dd u_n \leq \lambda \ftlam(\orig) \big( \sup_{y\in\Rd} \Pi_\lambda^{(n)}(y) \big).
			\label{eq:LE:Rn_Pin_bound}  }
Our main result of this section is the following proposition:
\begin{prop}[Lace expansion] \label{thm:laceexpansionidentity}
Let $x\in\Rd$ and $\lambda \in [0,\lambda_c)$. Then, for $n \geq 0$,
	\eqq{ \tlam(x) = \connf(x) + \Pi_{\lambda, n}(x) + \lambda \big((\connf + \Pi_{\lambda, n}) \star \tlam\big)(x) + R_{\lambda, n}(x) .\label{eq:laceexpansionidentity}}
\end{prop}
\begin{proof}
The proof is by induction over $n$. After the base case (first step), we prove the case $n=1$ (second step). The case for general $n$ is analogous, but with heavier notation, and is only sketched (third step).

\underline{First step, $n=0$.} Using~\eqref{eq:LE:Pi0_def} in Definition~\ref{def:LE:lace_expansion_coefficients}, we observe that
		\eqq{\tlam(x) = \connf(x) + \Pi_\lambda^{(0)}(x) + \pla(\conn{\orig}{x}{\xi^{\orig,x}}, \ndconn{\orig}{x}{\xi^{\orig,x}}). \label{eq:lace_expansion_first_step_1}}
The event in the last term of the sum enforces the existence of a (first) pivotal point, and so, similarly to Lemma~\ref{lem:partitioning_via_E_events}, we can partition
	\eqq{ \mathds 1_{\{\conn{\orig}{x}{\xi^{\orig,x}}\} \cap \{\ndconn{\orig}{x}{\xi^{\orig,x}}\}} = \sum_{u \in \eta} \mathds 1_{\{\dconn{\orig}{u}{\xi^{\orig,x}} \} \cap 
									\{u \in \piv{\orig,x;\xi^{\orig,x}} \}}.	\label{eq:lace_expansion_pivotal_identity}}
We set $\C_0=\C(\orig,\xi^{\orig})$. Taking probabilities, we can use the Mecke formula~\eqref{eq:prelim:mecke_1} and then the Cutting-point Lemma~\ref{lem:LE:cutting_point} to rewrite
		\algn{ \pla(\conn{\orig}{x}{\xi^{\orig,x}}, \ndconn{\orig}{x}{\xi^{\orig,x}}) 
				&= \lambda \int \pla\left(\dconn{\orig}{u}{\xi^{\orig,u,x}}, u \in \textsf{\textup{Piv}}(\orig,x; \xi^{\orig,u,x})\right) \dd u \notag\\
				&= \lambda \int \E_\lambda\Big[ \mathds 1_{\{\dconn{\orig}{u}{\xi^{\orig,u}}\}} \tlam^{\C_0}(u,x) \Big] \dd u\label{eq:lace_expansion_first_key_step}.}
To deal with $\tlam^{\C_0}(u,x)$ in~\eqref{eq:lace_expansion_first_key_step}, note that taking probabilities in~\eqref{eq:LE:tau_thinning_split} gives
		\eqq{ \tlam^{A}(u,x) = \tlam(x-u) - \pla \big( \xconn{u}{x}{\xi^{u,x}_1}{A} \big) \label{eq:LE:incl_excl_probabilities}}
for a locally finite set $A$\ch{, where, by construction, $\xi^{u,x}_1$ is independent of $\xi^{\orig,x}=\xi_0^{\orig,x}$.} We can substitute~\eqref{eq:LE:incl_excl_probabilities} into~\eqref{eq:lace_expansion_first_key_step} with the fixed set $A=\C_0$. Inserting this back into~\eqref{eq:lace_expansion_first_step_1}, 
and using the independence of $\xi_0$ and $\xi_1$, 
we can express $\tlam$ as
		\algn{\tlam(x) &= \connf(x) + \Pi_\lambda^{(0)}(x) +\lambda \int \E_\lambda \left[ \mathds 1_{\{\dconn{\orig}{u}{\xi^{\orig,u}}\}}\right] \tlam(x-u) \dd u \notag\\
					& \quad - \lambda \int \E_\lambda\Big[ \mathds 1_{\{\dconn{\orig}{u}{\xi^{\orig,u}_0}\}} \mathds 1_{\{\xconn{u}{x}{\xi^{u,x}_1}{\C_0}\}} \Big] \dd u \label{eq:LE:first_step_error_term}\\
					&= \connf(x) + \Pi_\lambda^{(0)}(x) +\lambda \int \left(\connf(u) + \Pi_\lambda^{(0)}(u)\right) \tlam(x-u) \dd u + R_{\lambda, 0}(x), \label{eq:LE:first_key_step_B}}
using the definition of $R_{\lambda,0}$ in~\eqref{eq:LE:R0_def}. This proves~\eqref{eq:laceexpansionidentity} for $n=0$. Note that all appearing integrals in~\eqref{eq:LE:first_key_step_B} are finite (as the integrands are bounded by $\tlam$), and so the rewriting via~\eqref{eq:LE:incl_excl_probabilities} is justified. 

\underline{Second step, $n=1$.}
We consider the second indicator in~\eqref{eq:LE:first_step_error_term} and its probability, regarding $\C_0$ as a fixed set. Thanks to Lemma~\ref{lem:partitioning_via_E_events}, and recalling that $\C_1=\C(u, \xi_1^{u}) $, we have, \col{for any locally finite $B \subset \Rd$,}
	\algn{\pla \left( \xconn{u}{x}{\xi^{u, x}_1}{B} \right) &= \pla \big(E(u,x;B, \xi^{u,x}_1)\big)
				+ \E_\lambda\Big[\sum_{u_1 \in \eta_1} \mathds 1_{E(u,u_1;B, \xi^{u,x}_1)} \mathds 1_{\{u_1 \in \piv{u,x;\xi^{u,x}_1}\}} \Big] \notag\\
		&= \pla \left(E(u,x;B, \xi^{u,x}_1)\right) + \lambda \int \E_\lambda \big[ \mathds 1_{E(u,u_1;B, \xi^{u,u_1,x}_1)} \mathds 1_{\{u_1 \in \piv{u,x;\xi^{u,u_1,x}_1}\}} \big] \dd u_1 \notag\\
		&= \pla \left(E(u,x;B, \xi^{u,x}_1)\right) + \lambda \int \E_\lambda \left[ \mathds 1_{E(u,u_1;B, \xi^{u, u_1}_1) } \cdot \tlam^{\C_1}(u_1, x) \right] \dd u_1, \label{eq:LE_step2_keyidentities}}
where we have again employed Mecke's formula~\eqref{eq:prelim:mecke_1} and the Cutting-point Lemma~\ref{lem:LE:cutting_point}. Again, we apply~\eqref{eq:LE:incl_excl_probabilities} with $A=\C_1$ to~\eqref{eq:LE_step2_keyidentities}, which gives
	\algn{\E_\lambda \Big[ \mathds 1_{E(u,u_1;B, \xi^{u, u_1}_1) } \cdot \tlam^{\C_1}(u_1, x) \Big] &= \pla(E(u,u_1;B, \xi^{u, u_1}_1)) \tlam(x-u_1) \notag\\
		& \quad - \E_\lambda\left[ \mathds 1_{E(u,u_1;B, \xi^{u, u_1}_1)} \mathds 1_{\{\xconn{u_1}{x}{\xi^{u_1,x}_2}{\C_1}\}} \right] . \label{eq:throughconnidentity}}
We now insert~\eqref{eq:LE_step2_keyidentities} with $u=u_0$ as well as the set $B=\C_0$ into the expansion identity~\eqref{eq:LE:first_step_error_term}. Recalling the definition of $\Pi_\lambda^{(n)}$ in~\eqref{eq:LE:Pin_def}, we can extract $\Pi_\lambda^{(1)}$ and apply~\eqref{eq:throughconnidentity} to perform the next step of the expansion, yielding
	\al{\tlam(x) &= \connf(x) + \Pi_\lambda^{(0)}(x)-\Pi_\lambda^{(1)}(x) + \lambda \int \left(\connf(u) + \Pi_\lambda^{(0)}(u)\right) \tlam(x-u) \dd u\\
		& \quad - \lambda \int \tlam(x-u_1) \cdot \lambda\int \E_\lambda \left[ \mathds 1_{\{\dconn{\orig}{u_0}{\xi^{\orig,u_0}_0}\}} \mathds 1_{E(u_0,u_1;\C_0, \xi^{u_0, u_1}_1)} \right]\dd u_0\dd u_1 \\
		& \quad + \lambda^2 \int \E_\lambda \left[\mathds 1_{\{\dconn{\orig}{u_0}{\xi^{\orig,u_0}_0} \}} \int
				\left[ \mathds 1_{E(u_0,u_1;\C_0, \xi^{u_0, u_1}_1)} \mathds 1_{\{\xconn{u_1}{x}{\xi^{u_1,x}_2}{\C_1}\}} \right] \dd u_1 \right] \dd u_0 \\
		& = \connf(x) + \Pi_\lambda^{(0)}(x)-\Pi_\lambda^{(1)}(x) +\lambda \int \left(\connf(u) + \Pi_\lambda^{(0)}(u) - \Pi_\lambda^{(1)}(u)\right) \tlam(x-u) \dd u + R_{\lambda, 1}(x). }
This proves~\eqref{eq:laceexpansionidentity} for $n=1$. Again, we point out that the appearing integrals are finite since $\lambda<\lambda_c$. 

\underline{Third step, general $n$.} For general $n\geq 1$, we can repeat the arguments for $n=1$ and obtain
	\al{ \pla\big(\xconn{u_n& }{x}{\xi^{u_n,x}_{n+1}}{\C_n}\big) = \pla\big(E(u_n,x;\C_n, \xi^{u_n,x}_{n+1})\big) \\
		& \quad + \lambda \int \tlam(x-u_{n+1}) \pla\big(E(u_n,u_{n+1};\C_n, \xi^{u_n,u_{n+1}}_{n+1})\big) \dd u_{n+1} \\
		& \quad - \lambda\int \E_\lambda \Big[ \mathds 1_{E(u_n,u_{n+1};\C_n, \xi^{u_n,u_{n+1}}_{n+1})} \mathds 1_{\{\xconn{u_{n+1}}{x}{\xi^{u_{n+1},x}_{n+2}}{\C_{n+1}}\}} \Big] \dd u_{n+1}. }
Plugging this into $R_{\lambda, n}(x)$, the first term yields $\Pi_\lambda^{(n+1)}(x)$, the second one yields $\lambda (\Pi_\lambda^{(n+1)} \star \tlam)(x)$, and the last one yields $R_{\lambda, n+1}(x)$. By induction, this proves the claim.
\end{proof}

\section{Diagrammatic bounds} \label{sec:diagrammaticbounds}
\subsection{Warm-up: Motivation and bounds for $n=0$} \label{sec:DB:warmup}

The aim of this section is to bound the lace-expansion coefficients $\Pi_\lambda^{(n)}$, which we have identified in the previous section. The bounds will be formulated in terms of somewhat simpler quantities, so-called diagrams. To this end, we first interpret the integrand in $\Pi_\lambda^{(n)}$ as the probability of an event contained in some connection event, which we can illustrate pictorially. In the next step, these connection events are decomposed by heavy use of the BK inequality into diagrams, which will turn out to be easier to analyze (this analysis is performed in Section~\ref{sec:bootstrapanalysis}). A diagram is an integral over a product of two-point and connection functions. Its diagrammatic representation is illustrated in Figure~\ref{fig:Psidiagrams} and used heavily in the analysis in the later parts of this section.

To illustrate the idea of this lengthy procedure, we first \chr{demonstrate the proofs for} $n=0$. Since $\Pi_\lambda^{(0)}$ is fairly simple, this has the advantage of giving a rather compact overview of what we execute at length for general $n$ afterwards. 

The main results of this section are Propositions~\ref{prop-bounds-LA-coefficients-unweighted} and \ref{thm:PsiDiag_Bound_Derangement}. The former gives bounds on $\widehat\Pi_\lambda^{(n)}(k)$, the latter gives related bounds on $\widehat\Pi_\lambda^{(n)}(\orig) - \widehat\Pi_\lambda^{(n)}(k)$, which turn out to be important in Section~\ref{sec:bootstrapanalysis}. In preparation of the latter bounds, we state Lemma~\ref{lem:cosinesplitlemma}. Note that if $f(x) = f(-x)$, then
	\eqq{\widehat{f}(k) = \int f(x) \e^{\i k \cdot x} \dd x = \int f(x) \cos(k \cdot x) \dd x. \label{eq:fouriersymmetry}}
Consequently, $|\widehat \Pi_\lambda^{(n)}(\orig)-\widehat \Pi_\lambda^{(n)}(k)|=\int [1-\cos(k\cdot x)] \Pi_\lambda^{(n)}(x) \dd x$. The following lemma is well known in the lace-expansion literature and allows to decompose factors of the form $[1-\cos(k\cdot x)]$. 
\begin{lemma}[Split of cosines, \cite{FitHof16}, Lemma 2.13] \label{lem:cosinesplitlemma}
Let $t_i \in \R$ for $i=1, \ldots, m$ and $t = \sum_{i=1}^{m} t_i$. Then
	\[1-\cos (t) \leq m \sum_{i=1}^{m} [1- \cos(t_i)]. \]
\end{lemma}

The next definition and observation are to be seen as an intermezzo, as they are not necessary at this point. In fact, Definition~\ref{def:tillambubble} will not be of importance until Section~\ref{sec:bootstrapanalysis}. We state it here nonetheless to prove a basic relation to $\tlam$, which illustrates some key ideas recurring in many of the proofs to follow:

\begin{definition}[One-step connection probability] \label{def:tillambubble} For $x\in\Rd$, we define $\tillam(x) := \connf(x) + \lambda (\connf \star\tlam)(x)$.
\end{definition}

\begin{observation}[Relation between $\tlam$ and $\tillam$] \label{obs:tildebounds}
Let $x \in \Rd$. Then $\tlam(x) \leq \tillam(x)$.
\end{observation}
\begin{proof}
By combining Mecke's formula and the BK inequality, we obtain
	\al{ \tlam(x) & \leq  \connf(x) + \E_\lambda\Big[\sum_{y \in \eta} \mathds 1_{\{\orig \sim y \text{ in } \xi^{\orig} \} \cap \{\conn{y}{x}{\xi^{x}}\}} \Big] \\
		& = \connf(x) + \lambda \int \pla\big(\{\orig \sim y \text{ in } \xi^{\orig,y} \} \cap \{\conn{y}{x}{\xi^{x,y}}\}\big) \dd y \\
		& \leq \connf(x) +\lambda \int \connf(y) \tlam(x-y) \dd y = \tillam(x). \qedhere} 
\end{proof}
\col{In the last inequality, we have also used that the two intersected events are independent}. This is due to the fact that $\orig\notin \eta^{x,y}$ a.s. Whenever the first event is not a direct adjacency however (but instead also a connection event), \col{we need to use the BK inequality instead}.

We define two quantities that are of relevance for the following proposition, as well as later on in Section~\ref{sec:diagrammaticbounds_disp}:
\begin{definition}[Basic displacement functions]
The Fourier quantities $\connf_k$ and $\tklam$ are defined as 
	\eqn{
	\label{Delta-k-functions-def}
	\connf_k(x) = [1-\cos(k\cdot x)] \connf(x)
	\qquad
	\text{and} 
	\qquad
	\tklam(x) = [1-\cos(k\cdot x)]\tlam(x).
	}
\end{definition}
The following proposition deals with $\Pi_\lambda^{(0)}$ and its Fourier transform:
\begin{prop}[Bounds for $n=0$] \label{thm:DB:Pi0_bounds}
Let $k\in\Rd$ \col{and let $\lambda\in[0,\lambda_c)$}. Then
	\al{|\widehat\Pi_\lambda^{(0)}(k)| & \leq \lambda^2(\connf^{\star 2}\star \tlam^{\star 2})(\orig), \\
		|\widehat\Pi_\lambda^{(0)}(\orig)-\widehat\Pi_\lambda^{(0)}(k)| &\leq \lambda^2 \Big((\connf_k\star\connf \star\tlam^{\star 2})(\orig) + (\connf^{\star 2} \star \tklam\star\tlam)(\orig) \Big) .}
\end{prop}
\begin{proof}
We note that for the event $\{ \dconn{\orig}{x}{\xi^{\orig,x}} \}$ to hold, either there is a direct edge between $\orig$ and $x$, or there are vertices $y,z$ in $\eta$ that are direct neighbors of the origin and have respective disjoint paths to $x$ that both do not contain the origin. Hence, by the multivariate Mecke equation~\eqref{eq:prelim:mecke_n},
	\al{\pla( \dconn{\orig}{x}{\xi^{\orig,x}}) &\leq \connf(x) + \tfrac 12 \E_\lambda\Big[ \sum_{(y,z) \in \eta^{(2)}} \mathds 1_{(\{\orig\sim y \text{ in } \xi^{\orig}\}
			\cap \{\conn{y}{x}{\xi^x}\}) \circ (\{\orig \sim z \text{ in } \xi^{\orig}\} \cap \{ \conn{z}{x}{\xi^x} \})} \Big] \\
		& = \connf(x) + \tfrac 12 \lambda^2 \iint \pla\big( (\{\orig\sim y \text{ in } \xi^{\orig, y}\} \cap \{\conn{y}{x}{\xi^{x,y}}\}) \\
		& \hspace{3.5cm} \circ (\{\orig \sim z \text{ in } \xi^{\orig, z}\} \cap \{ \conn{z}{x}{\xi^{x,z}} \}) \big) \dd y \dd z. }
After applying the BK inequality to the above probability, the integral factors, and so
	\algn{\pla( \dconn{\orig}{x}{\xi^{\orig,x}}) & \leq \connf(x) + \tfrac 12 \lambda^2 \Big( \int \pla( \{\orig\sim y \text{ in } \xi^{\orig,y}\}
						 \cap \{\conn{y}{x}{\xi^{y,x}}\}) \dd y \Big)^2 \notag \\
		& = \connf(x) + \tfrac 12 \lambda^2 (\connf \star \tlam)(x)^2. \label{eq:DB:Pi0_dconn_bound}}
Thus, recalling that $\Pi_\lambda^{(0)}(x) = \pla(\dconn{\orig}{x}{\xi^{\orig,x}})-\connf(x) \geq 0$ and dropping the factor $\tfrac 12$,
	\algn{ |\widehat\Pi_\lambda^{(0)}(k)| &= \left\vert \int \cos(k \cdot x) \Pi_\lambda^{(0)}(x) \dd x \right\vert \leq \int \Pi_\lambda^{(0)}(x) \dd x
					 \leq \lambda^2 \int (\connf\star\tlam)(x)^2 \dd x \notag \\
			& = \lambda^2 \int (\connf\star\tlam)(x) (\connf\star\tlam)(-x) \dd x = \lambda^2 (\connf^{\star 2} \star\tlam^{\star 2})(\orig)\label{eq:DB:Pi0_diag_bd}}
using symmetry of $\connf$ and $\tlam$ as well as commutativity of the convolution. For the second bound of Proposition~\ref{thm:DB:Pi0_bounds}, we apply~\eqref{eq:DB:Pi0_dconn_bound} and obtain
	\eqq{\widehat\Pi_\lambda^{(0)}(\orig)-\widehat\Pi_\lambda^{(0)}(k) = \int [1-\cos(k\cdot x)] \Pi_\lambda^{(0)}(x) \dd x 
				\leq \frac{\lambda^2}{2} \int (\connf \star \tlam)(x) \int [1-\cos(k\cdot x)] \connf(y)\tlam(x-y) \dd y \dd x. 	\label{eq:DB:Pi0_disp}}
We call the factor $[1-\cos(k\cdot x)]$ a {\em displacement factor}. Writing $x=y + (x-y)$, the Cosine-split Lemma~\ref{lem:cosinesplitlemma} allows us to distribute it over the factors $\connf$ and $\tlam$ as
	\[ \int [1-\cos(k\cdot x)] \connf(y) \tlam(x-y) \dd y \leq 2\Big[ (\connf_k\star\tlam)(x) + (\connf\star\tklam)(x)\Big]. \]
Substituting this back into~\eqref{eq:DB:Pi0_disp} gives the desired result.
\end{proof}

\subsection{Statement of bounds on the lace-expansion coefficients}
\label{sec-statement-bounds-LA-coefficients}
\ch{In this section, we state bounds on the lace-expansion coefficients $\Pi_\lambda^{(n)}$ for $n\geq 1$. The main results in this section are Propositions \ref{prop-bounds-LA-coefficients-unweighted}, \ref{thm:PsiDiag_Bound_Derangement_N1} and \ref{thm:PsiDiag_Bound_Derangement} below. We defer the proofs of these propositions, which are technically challenging yet similar in spirit to the proofs for $n=0$ in Section \ref{sec:DB:warmup}, to Section \ref{sec:bounds-lace-expansion} below.}

\ch{Recall the two-point function triangles $\trilam(x)= \lambda^2 \tlam^{\star 3}(x)$ and $\trilam = \sup_x \trilam(x)$, as defined in \eqref{triangle-defs}.  We will rely on several related two-point functions and triangle diagrams:}

\ch{\begin{definition}[Modified two-point function and triangles] \label{def:bubble}
Define
	\eqn{
	\label{tau-dot-def}
	\tlamo(x) := \delta_{x,\orig} + \lambda\tlam(x),
	}
and 
	\eqn{
	\label{triangle-dots-def}
	\trilamo(x) = \lambda(\tlamo\star\tlam\star\tlam)(x)\qquad \text{and} \qquad \trilamoo(x) = (\tlamo\star\tlamo\star\tlam)(x).
	}
\chr{Further, define}
	\eqn{
	\trilamo = \sup_{x \in \Rd} \trilamo(x), \qquad \trilamoo = \sup_{x \in \Rd} \trilamoo(x), \qquad \trilame = \sup_{x \in\Rd: |x| \geq \varepsilon} \trilamo(x),
	}
and
	\eqn{
	\ballep = \Big( \lambda \int \mathds 1_{\{| z| < \varepsilon\}} \dd z \Big)^{1/2}.
	}
In terms of the above quantities, we set
	\eqn{
	U_\lambda = 3 \trilamoo\trilamo, \qquad \col{U_\lambda^{(\varepsilon)} = 5 \trilamoo \big(\trilam + \trilame + \ballep + (\ballep)^2 \big)} ,
				 \qquad \bar U_\lambda = 2(1+ U_\lambda)U_\lambda^{(\varepsilon)}.
	}
\end{definition}}
\ch{We remark that $\trilam \leq \trilamo \leq \trilamoo$. Moreover, as $\tlam(\orig) = 1$, we have $\trilamoo \geq 1$.}

\ch{In terms of the above quantities, we can bound $\widehat\Pi_\lambda^{(n)}(\orig)=\int \Pi_\lambda^{(n)}(x) \dd x$ as follows:}

\ch{\begin{prop}[Bound on lace-expansion coefficients for $n\geq 1$]
\label{prop-bounds-LA-coefficients-unweighted}
Let $n \geq 1$ and $\lambda\in [0,\lambda_c)$. Then
\eqn{
\int \Pi_\lambda^{(n)}(x) \dd x \leq \lambda^n 2(2\trilamo+\lambda+1) \big(U_\lambda \wedge \bar U_\lambda \big)^{n}.
}
\end{prop}}
\medskip

\ch{Proposition \ref{prop-bounds-LA-coefficients-unweighted} is the extension to $n\geq 1$ of the bound on $\int \Pi_\lambda^{(0)}(x) \dd x$ in Proposition \ref{thm:DB:Pi0_bounds}. We also need an extension to $n\geq 1$ of the bound on $|\widehat\Pi_\lambda^{(0)}(\orig)-\widehat\Pi_\lambda^{(0)}(k)|$ in Proposition \ref{thm:DB:Pi0_bounds}, which we refer to as {\em displacement bounds}. 
The main results on such displacement bounds are Propositions~\ref{thm:PsiDiag_Bound_Derangement_N1} and~\ref{thm:PsiDiag_Bound_Derangement} below. To state them, we need another definition:
\begin{definition}[Further displacement bound quantities] 
\label{def:displacement_quantities} 
Let $x,k\in\Rd$. Recall from \eqref{Delta-k-functions-def} that $\tklam(x) = [1-\cos(k\cdot x)] \tlam(x)$. We define the displacement bubble as
	\eqn{
	W_\lambda(x;k) = \lambda (\tklam \star\tlam)(x), \qquad W_\lambda(k) = \sup_{x\in\Rd} W_\lambda(x;k).
	}
Furthermore, let $H_\lambda(k) := \sup_{a, b} H_\lambda(a,b;k)$, where
	\eqn{
	H_\lambda(a,b;k) := \lambda^5 \int \tlam(z) \tklam(u-z)\tlam(t-u) \tlam(t-z) \tlam(w-t)\tlam(a-w) \tlam(x+b-w) \tlam(x-u) \dd (t,w,z,u,x).
	}
\end{definition}}

\ch{Propositions~\ref{thm:PsiDiag_Bound_Derangement_N1} and~\ref{thm:PsiDiag_Bound_Derangement} provide bounds on $ \lambda| \widehat\Pi_\lambda^{(n)}(\orig) - \widehat\Pi_\lambda^{(n)}(k)| =\lambda\int [1-\cos(k\cdot x)] \Pi_\lambda^{(n)}(x) \dd x$ for $n=1$ and $n\geq 2$, respectively:
\begin{prop}[Displacement bound for $n=1$] \label{thm:PsiDiag_Bound_Derangement_N1}
For $k\in\Rd$,
	\algn{ \lambda\int[1-\cos (k \cdot x)] \Pi_\lambda^{(1)}(x) \dd x &\leq 31 (1+\lambda) (U_\lambda \wedge \bar U_\lambda) W_\lambda(k) 
						+ 2\lambda^3 \big[(\tklam\star\tlam\star\connf^{\star 2})(\orig) + (\connf_k\star\tlam^{\star 2}\star\connf)(\orig) \big] \nonumber\\
		& \qquad + \min\Big\{\lambda [1-\fconnf(k)] \trilamo, \big(\lambda [1-\fconnf(k)] \trilame + 4 W_\lambda(k) \big(\ballep\big)^2 \big) \Big\}. }
\end{prop}
\begin{prop}[Displacement bounds on lace-expansion coefficient for $n\geq 2$] 
\label{thm:PsiDiag_Bound_Derangement}
For $n\geq 2$ and $k\in\Rd$,
	\algn{ \lambda\int [1- & \cos (k \cdot x)] \Pi_\lambda^{(n)}(x) \dd x \notag \\
		& \leq 60 (n+1) \big(\lambda+(\trilamoo)^2 \big)^2 \Big[ W_\lambda(k) \big(U_\lambda\big)^{1\vee (n-2)} \trilamoo(1+ U_\lambda) + \big(U_\lambda\big)^{n-2} H_\lambda(k) \Big].
						\label{eq:DB:disp:thm_U_lambda}}
Moreover,
	\algn{ \lambda \int [1-& \cos (k \cdot x)] \Pi_\lambda^{(n)}(x) \dd x \notag \\
		& \leq 60 (n+1) \big(\lambda+(\trilamoo)^2 \big)^2 \Big[ W_\lambda(k) \big( \bar U_\lambda \big)^{1 \vee (n-2)} \Big( \trilamoo + U_\lambda + \bar U_\lambda\Big)
						+ \big( \bar U_\lambda \big)^{n-2} H_\lambda(k) \Big] . \label{eq:DB:disp:thm_U_lambda_eps}}
\end{prop}}

\ch{Proposition~\ref{thm:PsiDiag_Bound_Derangement} is the main result on displacement bounds, and its result is also valid for $n=1$. Proposition~\ref{thm:PsiDiag_Bound_Derangement_N1}, however, gives an essential improvement to Proposition~\ref{thm:PsiDiag_Bound_Derangement} for $n=1$ that will later allow us to show that also for $n=1$, the displacement bounds are sufficiently small under the conditions stated in (H1), (H2), or (H3).
}

\chr{We close this section with bounds on $\Pi_\lambda^{(n)}(x)$, that are important to define $\Pi_{\lambda_c}$ in Section~\ref{sec:bootstrapanalysis}. 
\begin{corollary}\label{cor:DB:Pi_x_bounds}
Let $n\geq 1$. Then
	\[ \sup_{x\in\Rd} \Pi_\lambda^{(n)}(x) \leq 2(2\trilamo+\lambda+1)^2 \big(U_\lambda \wedge\bar U_\lambda \big)^{n-1}.\]
\end{corollary}
This corollary will follow straightforwardly from the proof of Proposition \ref{thm:Psi_Diag_bound}, we give its proof at the end of Section \ref{sec:diagrammaticbounds_no_disp}. 
}

\section{Bootstrap analysis}\label{sec:bootstrapanalysis}

So far we have worked with a general symmetric connection
function satisfying \eqref{eq:phisummability}. In this section (starting
with Section \ref{sec:consequences_of_BS_bounds}) and the next
we shall use the assumptions made on $\connf$ in (H1), (H2), or (H3).

\subsection{Re-scaling of the intensity measure} \label{sec:rescaling}
\col{To increase the readability and avoid cluttering the technical proofs, we assume throughout Section~\ref{sec:bootstrapanalysis} (as well as in the appendix) that}
	\eqq{ \phiint  = \int\connf(x) \dd x = 1 \label{eq:BA:phi_normalized}.}
A re-scaling argument similar to~\cite[Section 2.2]{MeeRoy96} justifies this without losing generality. Let us briefly elaborate on this: We can scale $\Rd$ by a factor of $\phiint^{-1/d}$ (that is, the new unit radius ball is the previous ball of radius $\phiint^{1/d}$) and what we obtain has the distribution of an RCM model with parameters
	\[ \lambda^* = \E_\lambda\big[\big\vert\eta \cap [0, \phiint^{1/d}]^d \big\vert\big] = \lambda \phiint, 	\quad \connf^*(x) = \connf\left(x \phiint^{1/d} \right), 
	\quad \fconnf^*(k) = \fconnf\left(k/\phiint^{1/d}\right).
	\]
For this new model,
	\[ \int \connf^*(x) \dd x = \phiint^{-1} \int \connf(y) \dd y = 1. \]
By the re-scaling, $\{\conn{\orig}{x}{\xi^{\orig,x}}\}$ becomes $\{ \conn{\orig}{x / \phiint^{1/d}}{\xi^{\orig,x/\phiint^{1/d}}} \}$, we also have
	\[\tlam(x) = \tau^*_{\lambda^*} \left(x / \phiint^{1/d}\right), \]
where $\tau^*_\mu$ is the two-point function in the RCM governed by the connection function $\connf^*$. Clearly, $\lambda_c \phiint=\lambda_c^*$. Another short computation shows that $\trilam(x) = (\lambda^*)^2 \big(\tau_{\lambda^*}^* \star \tau_{\lambda^*}^* \star \tau_{\lambda^*}^* \big)\big(x/\phiint^{1/d}\big)$, so the triangle condition holds in the original model precisely when it holds in the re-scaled model.

As an example, $\fgmu(k) = (1-\mu\fconnf(k))^{-1}$ under the normalization assumption in \eqref{eq:BA:phi_normalized}.

\subsection{Introduction of the bootstrap functions}\label{sec:bootstrap_functions}
The analysis of Section~\ref{sec:bootstrapanalysis} follows the arguments in the paper by Heydenreich et al.~\cite{HeyHofSak08}, adapting them to the continuum setting. Some parts follow the presentation given there almost verbatim. We define
	\[\mulam :=  \col{1 - \frac{1}{\ftlam(\orig)}}\]
for $\lambda \geq 0$. 
Note that $\ftlam(\orig)$ is increasing in $\lambda$ and $\widehat\tau_0(\orig)=1$. Furthermore, as $\lambda \nearrow \lambda_c$, we have $\ftlam(\orig) \nearrow \infty$ and so $\mulam \nearrow 1$. In summary, $\mulam\in[0,1]$. Setting
	\[\Delta_k a(l) := a(l-k) + a(l+k) - 2 a(l)\]
to be the discretized second derivative of a function $a\colon \Rd \to \mathbb C$, we are in a position to define $f := f_1 \vee f_2 \vee f_3$ with
	\eqq{ f_1(\lambda) := \lambda, \qquad f_2(\lambda) := \sup_{k \in \Rd} \frac{|\widehat\tau_\lambda(k)|}{\fgmu(k)}, 
				\qquad f_3(\lambda) := \sup_{k,l \in \Rd} \frac{|\Delta_k \ftlam(l)|}{\widehat U_{\mulam}(k,l)}, \label{defeq:fbootstrapfunction}}
where we recall~\eqref{eq:def:greens_fct} and~\eqref{eq:def:greens_fct_fourier} for the Green's function $\greensmu$. Moreover, $\widehat U_{\mulam}$ is defined as
	\[\widehat U_{\mulam}(k,l):= 84 [1-\fconnf(k)] \Big(\fgmu(l-k)\fgmu(l)+\fgmu(l)\fgmu(l+k)+\fgmu(l-k)\fgmu(l+k) \Big)\]
and will serve as an upper bound on $\Delta_k \fgmu(l)$ by Lemma~\ref{lem:discrete-derivative-bound} below. This function $\widehat U_{\mulam}$ has nothing to do with the functions $U_\lambda, \bar U_\lambda$ from Section~\ref{sec:diagrammaticbounds}.

Let us point out that we show $f(\lambda) < \infty$ for all $\lambda\in[0,\lambda_c)$ as part of the proof of Proposition~\ref{thm:bootstrapargument}. In particular, we show that $f(0)\leq 2$ and that for $\lambda<\lambda_c$, $f_1$, $f_2$, $f_3$ are differentiable on $[0,\lambda]$ with uniformly bounded derivative.

Let us now explain the introduction of $\Delta_k$. The crucial observation (see Lemma~\ref{lem:discrete-derivative-bound} (i)) is that
	\[\ftklam (l) = - \tfrac 12 \Delta_k \ftlam(l), \]
where we recall that $\tklam(x)= [1-\cos(k\cdot x)] \tlam(x)$ is defined in Definition~\ref{def:displacement_quantities} and appears in $W_\lambda(k)$. This is why we are interested in a bound on $f_3$. We next state a lemma that collects some simple facts about the discretized second derivative, relates it to quantities of interest and states an important bound. The proof of Lemma~\ref{lem:discrete-derivative-bound} can be found in the literature~\cite{BorChaHofSlaSpe05,Sla06} for the discrete setting and carries over verbatim. 
\begin{lemma}[Bounds on the discretized second derivative,~\protect{\cite[Lemma 5.7]{Sla06}}] \label{lem:discrete-derivative-bound}
Let $a \colon \Rd\to\R$ be a measurable, symmetric and integrable function. Let $a_k(x) := [1-\cos (k \cdot x)] a(x)$. Then for $k,l\in\Rd$, 
\begin{itemize}
\item[(i)] $\Delta_k \widehat a(l) = -2 \widehat a_k(l)$ and
\item[(ii)] $| \Delta_k \widehat a(l) | \leq 2 (\widehat{|a|}(\orig) - \widehat{|a|}(k))$ and in particular, using (i), $\fconnf_k(l) \leq 1-\fconnf(k)$. 
\item[(iii)] For $\widehat A(k) = (1-\widehat a(k))^{-1}$,
	\[ |\Delta_k \widehat A(l) | \leq \big[ \widehat{|a|}(\orig) - \widehat{|a|} (k) \big] \Big( \big(\widehat A(l-k) + \widehat A(l+k) \big) \widehat A(l) 
		 + 8 \widehat A(l-k)\widehat A(l)\widehat A(l+k) \big(\widehat{|a|}(\orig) - \widehat{|a|}(l) \big) \Big). \]
In particular, for $\mu\in[0,1]$, 
	\[|\Delta_k \fgreensmu(l) | \leq [1-\fconnf(k)] \Big( \fgreensmu(l) \fgreensmu(l+k) + \fgreensmu(l)\fgreensmu(l-k) + 8 \fgreensmu(l-k)\fgreensmu(l+k) \Big) \leq \widehat U_{\mu}(k,l).\]
\end{itemize}
\end{lemma}

The remainder of Section~\ref{sec:bootstrapanalysis} is organized as follows. In Section~\ref{sec:consequences_of_BS_bounds}, we prove among other things that the lace expansion converges for each fixed $\lambda\in[0,\lambda_c)$ provided that $\beta$ is sufficiently small (recall that this means $d$ large for (H1) and $L$ large for (H2), (H3)). Moreover, we prove that under the additional assumption $f\leq 3$ on $[0,\lambda_c)$, the smallness of $\beta$ required for the convergence of the lace expansion does not depend on $\lambda$. To do so, we derive several bounds on triangles and related quantities in terms of the function $f$. In Section~\ref{sec:forbidden_region_argument}, we prove that $f(0)\leq 2$ and that $f$ is continuous on $[0,\lambda_c)$. We then use the results obtained in Section~\ref{sec:consequences_of_BS_bounds} to show that in fact $f \leq 2$ on $[0,\lambda_c)$ whenever $\beta$ is sufficiently small. This in turn implies that the bounds obtained in Section~\ref{sec:consequences_of_BS_bounds} under the additional assumption $f \leq 3$ are true. In percolation theory, this (at first glance circular) argument is known as the bootstrap argument. From there, the main theorems follow with only little extra work.

\subsection{Consequences of the bootstrap bounds} \label{sec:consequences_of_BS_bounds}
We state Proposition~\ref{thm:convergenceoflaceexpansion}, which
proves bounds on the lace-expansion coefficients for fixed $\lambda$
and consequently shows that the lace-expansion
identity~\eqref{eq:laceexpansionidentity} becomes the Ornstein-Zernike
equation in the limit $n\to\infty$. 
Our proofs heavily rely on Propositions~\ref{thm:randomwalkestimates} 
and \ref{thm:randomwalkestimates2} in Appendix A, and therefore on the assumptions
made on $\connf$ in (H1), (H2), or (H3).

\begin{prop}[Convergence of the lace expansion and OZE] \label{thm:convergenceoflaceexpansion} 
\
Let $\lambda \in [0, \lambda_c)$ and $d>3(\alpha \wedge 2)$. 
Then there exists a constant $\const=c(f(\lambda))>0$ with $x\mapsto c(x)$ increasing such that
	\algn{ 
	\int \Pi_{\lambda}^{(n)}(x) \dd x \leq (\const \beta)^{n\vee 1}, &\qquad \int [1-\cos(k\cdot x)] \Pi_{\lambda}^{(n)}(x) \dd x \leq c [\fconnf(\orig)-\fconnf(k)] (\const \beta)^{n\vee 1}, 	
	\label{eq:BA:convLE_intPi_bds-prel}\\
	\sup_{x\in\Rd} & \Pi_{\lambda}^{(n)}(x) \leq (\const \beta)^{n\vee 1}. \label{eq:BA:convLE_Pi_bd}
	}
\end{prop}

We remark that $\fconnf(\orig)=1$ if $\phiint=1$. An immediate consequence of the proposition is the following:

\begin{corollary}[Uniform convergence of the lace expansion] \label{cor:convergenceoflaceexpansion_uniform}
\ch{Let $\lambda \in [0, \lambda_c)$ and assume that $f(\lambda)<3$.} Let \ch{$d>6$} be sufficiently large under (H1) and let $d>3(\alpha \wedge 2)$ and $L$ be sufficiently large under (H2) and (H3). Then there is $c$ (independent of $\lambda$ and, for (H1), also independent of $d$) such that
	\algn{ 
	\int \sum_{n \geq 0} \Pi_{\lambda}^{(n)}(x) \dd x \leq c \beta, &\qquad \int [1-\cos(k\cdot x)] \sum_{n \geq 0}\Pi_{\lambda}^{(n)}(x) \dd x \leq c [\fconnf(\orig)-\fconnf(k)] \beta, \label{eq:BA:convLE_intPi_bds}\\
			\sup_{x\in\Rd} & \sum_{n \geq 0} \Pi_{\lambda}^{(n)}(x) \leq c, \label{eq:BA:convLE_Pi_bd}}
and
	\eqq{ \sup_{x\in\Rd} |R_{\lambda, n}(x)|  \leq \lambda \ftlam(\orig) (c\beta)^n . \label{eq:BA:convLE_R_bd}}
Furthermore, the limit $\Pi_\lambda := \lim_{n\to\infty} \Pi_{\lambda,n}$ exists and is an integrable function with Fourier transform $\fpilam(k) = \lim_{n\to\infty} \widehat\Pi_{\lambda, n}(k)$ for $k\in\Rd$. Lastly, $\tlam$ satisfies the Ornstein-Zernike equation, taking the form
	\eqq{ \tlam = \connf + \Pi_\lambda + \lambda \big( (\connf+\Pi_\lambda) \star \tlam). \label{eq:LE_identity_OZE}}
\end{corollary}

We formulate our diagrammatic bounds in Lemmas \ref{lem:triangle-bubble-consequence-bound}-\ref{lem:H-G-consequence-bounds}, which rely on Lemma~\ref{lem:boundsonV}. All of these lemmas also deal with constants $\const$ and $c$. As in the statement of Proposition~\ref{thm:convergenceoflaceexpansion}, they are independent of $d$ for (H1). Proposition~\ref{thm:convergenceoflaceexpansion} will be a rather direct consequence of these lemmas.

We recall that $\tillam(x)=\connf(x) + \lambda(\connf\star\tlam)(x), \tklam(x) = [1-\cos(k \cdot x)] \tlam(x)$, $\connf_k(x)= [1-\cos(k\cdot x)] \connf(x)$, and we furthermore set $\tilklam(x) := [1-\cos(k\cdot x)] \tillam(x)$. With this, for $a,k\in\Rd, m,n\in\N_0$ and $\lambda \geq 0$, we define
	\al{ V_\lambda^{(m,n)}(a) := \lambda^{m+n-1} (\connf^{\star m} \star \tlam^{\star n})(a) & , \qquad W_\lambda^{(m,n)}(a;k) := \lambda^{m+n} (\tklam \star \connf^{\star m} \star \tlam^{\star n})(a), \\
						\widetilde W_\lambda^{(m,n)}(a;k) &:= \lambda^{m+n} (\connf_k \star \connf^{\star m} \star \tlam^{\star n})(a). }
Note that $W_\lambda^{(0,1)}(a;k) = W_\lambda(a;k)$, with $W_\lambda$ from Definition~\ref{def:displacement_quantities}.

\begin{lemma}[Bounds on $V_\lambda, W_\lambda, \widetilde W_\lambda$] \label{lem:boundsonV}
Let $\lambda \in [0, \lambda_c)$ and let $m,n\in\N_0$ be such that $n\leq 3$ and $m+n \geq 2$. Let $\ch{d>2n}$ under (H1) and $d>(\alpha \wedge 2)n$ under (H2) and (H3). Then there is $\const=\col{c(f(\lambda),m,n)}$, \ch{where $x\mapsto c(x,m,n)$ is increasing}, such that
	\al{ \sup_{a \in\Rd} V_\lambda^{(m,n)}(a) & \leq \begin{cases} \const \beta^{((m+n) \wedge 3) - 2} &\mbox{under (H1)}, \\ \const\beta &\mbox{under (H2), (H3)}, \end{cases} \\
			\sup_{a \in\Rd} \widetilde W_\lambda^{(m,n)}(a;k) & \leq \const [1-\fconnf(k)] \beta^{((m+n) \wedge 3) - 2} . }
If furthermore $n\leq 1$, then 
	\[ \sup_{a \in\Rd} W_\lambda^{(m,n)}(a;k) \leq \const [1-\fconnf(k)] \beta^{((m+n) \wedge 3) - 2} . \]
\end{lemma}
\ch{
\begin{remark}[The role of $c(f(\lambda),m,n)$]
{\rm We will see that $c(f(\lambda),m,n)$ takes the form of a constant depending on $m$ and $n$ times a power of $f(\lambda)$, where the power also depends on $m$ and $n$.\\
We remark that Lemma~\ref{lem:boundsonV} produces a bound in terms of $\beta$ as soon as $m+n \geq 3$.}\hfill$\diamondsuit$
\end{remark}
}
\begin{proof}
We start by observing that, for all $m\in\N_0$ and $n\in\N$,
	\eqq{ V_\lambda^{(m,n)} \leq V_\lambda^{(m+1,n-1)} + V_\lambda^{(m+1,n)} \leq 2 \max_{l \in\{n-1,n\}} V_\lambda^{(m+1,l)}, \label{eq:V112bound}}
which is a direct consequence of $\tlam \leq \connf+\lambda(\connf\star\tlam)$ \col{(i.e., Observation~\ref{obs:tildebounds})}. Analogous bounds hold for $W_\lambda^{(m,n)}$ and $\widetilde W_\lambda^{(m,n)}$. What~\eqref{eq:V112bound} means is that we can ``increase $m$ by possibly decreasing $n$''. We can therefore assume $m \geq 2$ without loss of generality (i.e.,~replace $m$ by $m+n$). The bound for $V_\lambda$ follows from the Fourier inverse formula, as
	\[V_\lambda^{(m,n)}(a) = \lambda^{m+n-1} \int \e^{-\i a \cdot l} \fconnf(l)^{m} \ftlam(l)^n \frac{\dd l}{(2\pi)^d}
					 \leq f(\lambda)^{m+n-1} \int |\fconnf(l)|^m |\ftlam(l)|^n \frac{\dd l}{(2\pi)^d}, \]
where we used that $\lambda\leq f(\lambda)$. We next apply the bound $|\ftlam(l)| \leq f_2(\lambda) \fgmu(l)$, yielding
	\eqq{ V_\lambda^{(m,n)}(a) \leq f(\lambda)^{m+2n-1} \int |\fconnf(l)|^m \fgmu(l)^n \frac{\dd l}{(2\pi)^d} = f(\lambda)^{m+2n-1} \int \frac{|\fconnf(l)|^{m}}{[1- \mulam\fconnf(l)]^n } \frac{\dd l}{(2\pi)^d}.
				\label{eq:Vmn_bound}}
Applying Proposition~\ref{thm:randomwalkestimates} gives the claim. 
Next, note that
	\[\widetilde W_\lambda^{(m,n)}(a;k) \leq \lambda^{m+n} \int |\fconnf_k(l)| |\fconnf(l)|^{m} |\ftlam(l)|^n \frac{\dd l}{(2\pi)^d}
				\leq f(\lambda)^{m+2n} [1- \fconnf(k)] \int |\fconnf(l)|^{m} \fgmu(l)^n \frac{\dd l}{(2\pi)^d} \]
due to Lemma~\ref{lem:discrete-derivative-bound}(ii). Since this is the same bound as in~\eqref{eq:Vmn_bound}, we can apply Proposition~\ref{thm:randomwalkestimates} again. Lastly,
	\al{W_\lambda^{(m,n)}(a;k) &\leq 84 f(\lambda)^{m+2n} [1-\fconnf(k)] \int |\fconnf(l)|^m \fgmu(l)^{n+1} \Big( \fgmu(l-k) + \fgmu(l+k) \Big) \frac{\dd l}{(2\pi)^d} \\
		& \quad + 84 f(\lambda)^{m+2n} [1-\fconnf(k)] \int |\fconnf(l)|^m \fgmu(l)^n \fgmu(l-k)\fgmu(l+k) \frac{\dd l}{(2\pi)^d}. }
Both of these terms can be bounded using Proposition~\ref{thm:randomwalkestimates2} \col{(if $n=0$ in the second line, we can multiply the integrand with $\fgmu(l)$ at the cost of a factor of $2$)}. Recall that we have replaced $m$ by $m+n$, which yields the exponents in the statement of Lemma~\ref{lem:boundsonV}.
\end{proof}

The following lemma applies Lemma~\ref{lem:boundsonV} to deduce bounds on several triangle quantities. It is crucial in the sense that it gives a bound on $\trilam$. We later prove that this bound is uniform in $\lambda$, implying the triangle condition.

\begin{lemma}[Bounds on triangles]\label{lem:triangle-bubble-consequence-bound}
Let $\lambda \in [0, \lambda_c)$ and let \ch{$d>6$} under (H1) and $d>3(\alpha \wedge 2)$ under (H2) and (H3). Let further $\varepsilon$ be given as in (H1.2). Then there is $\const=c(f(\lambda))$, \ch{where $x\mapsto c(x)$ is increasing}, such that
	\[\trilam \leq \const \beta, \qquad \trilame \leq \const \beta, \qquad \trilamoo \leq \const, \qquad \trilamo \leq \begin{cases}
			\const & \mbox{under (H1),} \\ \const \beta & \mbox{under (H2), (H3).}	\end{cases}\]
\end{lemma}
\begin{proof}
Observe that 
	\al{\trilam(x) & \leq \lambda^2 (\tlam\star\tillam\star\tillam)(x) = \lambda^2 \Big(\tlam\star \big(\connf+\lambda(\connf\star\tlam)\big)\star \big(\connf+\lambda(\connf\star\tlam)\big) \Big)(x)\\
				&= V_\lambda^{(2,1)}(x) + 2 V_\lambda^{(2,2)}(x) + V_\lambda^{(2,3)}(x) \leq \const \beta }
by Lemma~\ref{lem:boundsonV}. Similarly, we get $\trilamo(x) = \ch{\trilam(x)} + \lambda \tlam^{\star 2}(x) \leq \trilam + V_\lambda^{(2,0)}(x) + 2 V_\lambda^{(2,1)}(x) + V_\lambda^{(2,2)}(x)$. Applying Lemma~\ref{lem:boundsonV} again, this is bounded by $\const$ under (H1), and by $\const \beta$ under (H2), (H3). The bound $\trilamoo \leq 1+ \trilamo$ gives $\trilamoo \leq \const$. 
We use Observation~\ref{obs:tildebounds} twice and finally $\connf(x)\le\tlam(f)$ to get 
\begin{align*}
	\trilamo(x) &=\lambda(\tillam\star\tlam\star\tlam)(x)
	=\lambda (\tlam\star\tlam)(x)+\trilam(x)\\
	&\le\lambda (\tlam\star\connf)(x)+\lambda^2 (\connf\star\tlam\star\tlam)(x)+\trilam(x)\\
	&\le\lambda (\connf\star\connf)(x)+2\lambda^2 (\connf\star\tlam\star\tlam)(x)+\trilam(x)
	\le \lambda (\connf\star\connf)(x)+3\trilam(x),
\end{align*}
and thus
	\[ \trilame \leq \const\beta + \sup_{x:|x| \geq \varepsilon} \lambda (\connf\star\connf)(x) \leq \const' \beta\]
since $\connf\le1$ satisfies (H1.2) under (H1).
\end{proof}

\begin{lemma}[Bound on $W_\lambda$]\label{lem_W-consequence-bound}
Let $\lambda \in [0, \lambda_c)$ and let \ch{$d>6$} under (H1) and $d>3(\alpha \wedge 2)$ under (H2) and (H3). Then there is $\const=c(f(\lambda))$, \ch{where $x\mapsto c(x)$ is increasing}, such that
	\[W_\lambda(k) \leq \const [1-\fconnf(k)].\]
\end{lemma}

\begin{proof}
Recall the definition of $W_\lambda(k)$ in Definition~\ref{def:displacement_quantities}. By Observation~\ref{obs:tildebounds} and Lemma~\ref{lem:cosinesplitlemma}, we obtain
	\al{ W_\lambda(k) & \leq \sup_x \lambda \int [1-\cos(k\cdot y)] \big(\connf(y) + \lambda(\connf\star\tlam)(y)\big) \tillam(x-y) \dd y \\
		& \leq \sup_x \lambda \Big( \big[ \connf_k + 2\lambda \connf_k\star \tlam + 2 \lambda \connf \star\tklam \big] \star \big[\connf + \lambda (\connf\star\tlam)\big]\Big)(x) \\
		& = \sup_x \Big( \widetilde W_\lambda^{(1,0)}(x;k) + 3 \widetilde W_\lambda^{(1,1)}(x;k) + 2 \widetilde W_\lambda^{(1,2)}(x;k) + 2 W_\lambda^{(2,0)}(x;k) + 2 W_\lambda^{(2,1)}(x;k) \Big).}
We note that all summands except $\widetilde W_\lambda^{(1,0)}(x;k)$ are bounded by $\const[1-\fconnf(k)]$ by Lemma~\ref{lem:boundsonV}. The statement now follows from Lemma~\ref{lem:boundsonV} together with
	\[\widetilde W_\lambda^{(1,0)}(x;k) = \lambda \int \connf_k(y) \connf(x-y) \dd y \leq \lambda \int \connf_k(y) \dd y = f_1(\lambda) [1- \fconnf(k)].\qedhere\]
\end{proof}

The next lemma deals with the required extra treatment of the diagrams $\Pi_\lambda^{(0)}$ and $\Pi_\lambda^{(1)}$ with an added displacement:

\begin{lemma}[Displacement bounds on $\Pi_\lambda^{(0)}$ and $\Pi_\lambda^{(1)}$]\label{lem_Pi1-consequence-bound}
Let $\lambda \in [0, \lambda_c)$ and let \ch{$d>6$} under (H1) and $d>3(\alpha \wedge 2)$ under (H2) and (H3). Let further $i\in\{0,1\}$. Then there is $\const=c(f(\lambda))$, \ch{where $x\mapsto c(x)$ is increasing}, such that
	\[\lambda \int [1-\cos(k \cdot x)] \Pi_\lambda^{(i)}(x) \dd x \leq \const \beta [1-\fconnf(k)].\]
\end{lemma}
\begin{proof}
Let first $i=0$. With the bound \eqref{eq:DB:Pi0_diag_bd} and the Cosine-split Lemma~\ref{lem:cosinesplitlemma}, we get
	\al{\lambda \int [1-\cos(k \cdot x)] \Pi_\lambda^{(0)}(x) \dd x &\leq \lambda^3 \int [1-\cos(k\cdot x)] (\connf\star\tlam)(x)^2 \dd x \\
		& \leq 2 \widetilde W_\lambda^{(1,2)}(\orig;k) + 2 W_\lambda^{(2,1)}(\orig;k) \leq \const \beta [1-\fconnf(k)] }
by Lemma~\ref{lem:boundsonV}.
Let now $i=1$. Recalling the bound in Proposition~\ref{thm:PsiDiag_Bound_Derangement_N1}, the claimed result follows from noting that the two appearing convolutions are $W_\lambda^{(2,1)}$ and $\widetilde W_\lambda^{(1,2)}$.
\end{proof}

\begin{lemma}[Bound on $H_\lambda$] \label{lem:H-G-consequence-bounds}
Let $\lambda \in [0, \lambda_c)$ and let \ch{$d>6$} under (H1) and $d>3(\alpha \wedge 2)$ under (H2) and (H3). Then there is $\const=c(f(\lambda))$, \ch{where $x\mapsto c(x)$ is increasing}, such that
	\[H_\lambda(k) \leq \const \beta [1- \fconnf(k)].\]
\end{lemma}
\begin{proof}
\begin{figure}
	\centering
 \includegraphics{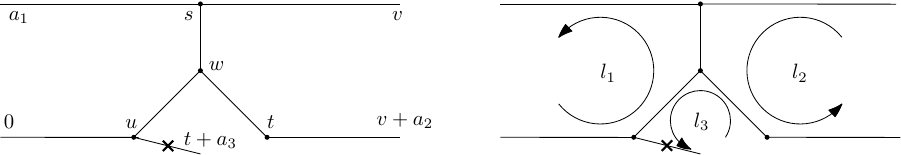}
	\caption{The diagram $H'_\lambda(a_1,a_2,a_3;k)$ and the schematic Fourier diagram $\widehat H'_\lambda(l_1,l_2,l_3;k)$.}
	\label{fig:Hdashdiagram}
\end{figure}
We define $\dtilklam (x) := \connf_k(x) + 2\lambda [ (\connf_k \star \tlam)(x) + (\connf \star \tilklam)(x)]$ and note that $\tilklam \leq \dtilklam$ by the Cosine-split Lemma~\ref{lem:cosinesplitlemma}. With this, by setting
	\al{ H'_\lambda(a_1,a_2,a_3;k) := \int & \tillam(s-a_1) \tillam(v-s) \tillam(s-w) \tillam(u) \tillam(w-u) \\
				& \quad \times \tillam(t-w) \tillam(v+a_2-t) \dtilklam(t+a_3-u) \dd(s,t,u,v,w), }
we have the bound $H_\lambda(k) \leq f(\lambda)^5\sup_{a_1, a_2} H'_\lambda(a_1, a_2, \orig; k)$. The Fourier inversion formula yields
	\eqq{ H'_\lambda(a_1, a_2, \orig;k) = \int \e^{-\i l_1 \cdot a_1}\e^{-\i l_2 \cdot a_2} \widehat H'_\lambda (l_1, l_2, l_3;k) \frac{\dd (l_1, l_2, l_3)}{(2 \pi)^{3d}}. \label{eq:Hdashidentity}}
Using that $b_1 = (b_1-s) + (s-w) + (w-u) + u, \ b_2 = (v+b_2-t) + (t-w) + (w-s) + (s-v)$ and $b_3 = (t+b_3-u) + (u-w) + (w-t)$, with appropriate substitution, this leads to
	\al{ \widehat H'_\lambda(l_1, l_2, l_3;k) &= \int \e^{\i l_1 \cdot b_1}\e^{\i l_2 \cdot b_2}\e^{\i l_3 \cdot b_3} H'_\lambda(b_1, b_2, b_3; k) \dd (b_1,b_2,b_3) \\
			&= \int \e^{\i l_1 \cdot u} \tillam(u) \e^{\i l_1 \cdot (b_1-s)} \tillam(b_1-s) \e^{\i l_2 \cdot (s-v)} \tillam(s-v) \e^{\i l_2 \cdot (v+b_2-t)} \tillam(v+b_2-t) \\
			& \qquad \times \e^{\i l_3 \cdot (t+b_3-u)} \dtilklam(t+b_3-u) \e^{\i (l_1-l_2) \cdot (s-w)} \tillam(s-w) \e^{\i (l_1-l_3) \cdot (w-u)} \tillam(w-u) \\
			& \qquad \times \e^{\i (l_2-l_3) \cdot (t-w)} \tillam(t-w) \dd (b_1,b_2,b_3,s,t,u,v,w) \\
			& = \ftillam(l_1)^2 \ftillam(l_2)^2 \fdtilklam(l_3) \ftillam(l_1-l_2) \ftillam(l_1-l_3) \ftillam(l_2-l_3), }
where we point to Figure~\ref{fig:Hdashdiagram} for an interpretation of the variables $l_i$ as cycles. Since $2 \fgreensmu \geq 1$,
	\[ |\ftillam(l)| \leq |\fconnf(l)| + \lambda |\fconnf(l)| |\ftlam(l)| \leq 3 f(\lambda)^2 |\fconnf(l)| \fgmu(l). \label{eq:ftillambound}\]
Similarly, with Lemma~\ref{lem:discrete-derivative-bound}(ii),
	\[|\fdtilklam(l_3)|\leq 4 f(\lambda)^2 [1-\fconnf(k)] \fgmu(l_3) + 2 f(\lambda)^2 | \fconnf(l_3)| \; \widehat U_{\mulam}(k,l_3).\] 
We can now go back and plug the above bounds into \eqref{eq:Hdashidentity} to obtain
	\algn{ H_\lambda(k) \leq\; &4 \times 3^7 f(\lambda)^{16} [1-\fconnf(k)] \int \fconnf(l_1)^2 \fgmu(l_1)^{2} \fconnf(l_2)^2 \fgmu(l_2)^{2} \fgmu(l_3) \notag\\
		& \qquad \times (|\fconnf| \cdot \fgmu)(l_1-l_2) (|\fconnf| \cdot \fgmu)(l_3-l_1) (|\fconnf| \cdot \fgmu)(l_3-l_2) \frac{\dd (l_1, l_2, l_3)}{(2\pi)^{3d}} \label{eq:Hdiagramfourierboundterm_a}\\
		& + 336 \times 3^7 f(\lambda)^{14} [1-\fconnf(k)] \int \fconnf(l_1)^2 \fgmu(l_1)^{2} \fconnf(l_2)^2 \fgmu(l_2)^{2}| \fconnf(l_3)| \notag\\
		& \qquad \times \Big[ \fgmu(l_3) \big( \fgmu(l_3+k) + \fgmu(l_3-k) \big) +\fgmu(l_3+k)\fgmu(l_3-k)\Big] \notag\\
		& \qquad \times (|\fconnf| \cdot \fgmu)(l_1-l_2) (|\fconnf| \cdot \fgmu)(l_1-l_3)
					(|\fconnf| \cdot \fgmu)(l_2-l_3) \frac{\dd (l_1, l_2, l_3)}{(2\pi)^{3d}}. 	\label{eq:Hdiagramfourierboundterm_b} }
Hence, $H_\lambda(k)$ is bounded by the sum of the terms~\eqref{eq:Hdiagramfourierboundterm_a} and~\eqref{eq:Hdiagramfourierboundterm_b}. The latter is itself a sum of two terms: By~\eqref{eq:Hdiagramfourierboundterm_b}(i) we refer to the term in~\eqref{eq:Hdiagramfourierboundterm_b} containing $\fgmu(l_3+k) + \fgmu(l_3-k)$, and by~\eqref{eq:Hdiagramfourierboundterm_b}(ii) we refer to the one containing $\fgmu(l_3+k)\fgmu(l_3-k)$. Hence, we have to bound the three terms~\eqref{eq:Hdiagramfourierboundterm_a},~\eqref{eq:Hdiagramfourierboundterm_b}(i),~\eqref{eq:Hdiagramfourierboundterm_b}(ii). We start with the term~\eqref{eq:Hdiagramfourierboundterm_a}, apply H\"older's inequality and bound the integral by
	\algn{&\left( \int (\fconnf(l_1)^2 \fgmu(l_1)^{3} (|\fconnf|\cdot \fgmu)(l_2)^{3} \fgmu(l_3) (|\fconnf|\cdot \fgmu)(l_3-l_1)^{1/2} (|\fconnf| \cdot \fgmu)(l_3-l_2)^{3/2} 
							 \; \frac{\dd (l_1, l_2, l_3)}{(2\pi)^{3d}} \right)^{2/3} \notag \\ 
		& \quad \times \left( \int \fconnf(l_1)^2 \fgmu(l_3) (|\fconnf| \cdot \fgmu)(l_1-l_2)^3  (\fconnf \cdot \fgmu)(l_3-l_1)^2 \; \frac{\dd (l_1, l_2, l_3)}{(2\pi)^{3d}} \right)^{1/3}.
						\label{eq:Hdiagramfourierboundterm_a_bound} }
The second integral in~\eqref{eq:Hdiagramfourierboundterm_a_bound} is simpler to deal with: We substitute $l_3'=l_3-l_1$ to then bound $\fgmu(l_3'+l_1) \leq \fgmu(l_3'+l_1) + \fgmu(l_3'-l_1)$ and use Proposition~\ref{thm:randomwalkestimates2} to resolve the integral over $l_3'$. We then resolve the integral over $l_2$ to obtain a factor $C \beta$ and note that the remaining integral over $l_1$ is bounded by 1.

To deal with the first integral in~\eqref{eq:Hdiagramfourierboundterm_a_bound}, we first consider the integral over $l_3$ and use H\"older's inequality to get
	\al{ \sup_{l_1, l_2} & \int (|\fconnf| \cdot \fgmu)(l_3-l_2)^{3/2} \fgmu(l_3) (|\fconnf|\cdot \fgmu)(l_1-l_3)^{1/2} \; \frac{\dd l_3}{(2\pi)^{d}} \\
		& \leq \sup_{l_1, l_2} \left( \int \fconnf(l_3')^2 \fgmu(l_3')^2 \fgmu(l_3'+l_2) \; \frac{\dd l_3'}{(2\pi)^{d}} \right)^{3/4}
					\left( \int \fconnf(l_3'')^2 \fgmu(l_3'')^2 \fgmu(l_3''+l_1) \; \frac{\dd l_3''}{(2\pi)^{d}} \right)^{1/4}, }
where we have substituted $l_3' = l_3-l_2$ and $l_3''=l_3-l_1$. Again, we use that $\fgmu$ is nonnegative and bound $\fgmu(l_3'+l_2) \leq \fgmu(l_3'+l_2) + \fgmu(l_3'-l_2)$ in the first integral and $\fgmu(l_3''+l_1) \leq \fgmu(l_3''+l_1) + \fgmu(l_3''-l_1)$ in the second. Proposition~\ref{thm:randomwalkestimates2} then completes the bounds. The remaining integral over $l_1$ and $l_2$ is handled straightforwardly, the latter yielding a factor of $C \beta$.

To bound the integral in~\eqref{eq:Hdiagramfourierboundterm_b}(i), let $\widehat D_{\mulam, k}(l) = \fgmu(l-k) + \fgmu(l+k)$. We apply the Cauchy-Schwarz inequality to bound the term from above by
	\al{&\left( \int \big[ \fconnf(l_1)^4 \fgmu(l_1)^{3} \big] \big[ (\fconnf\cdot \fgmu)(l_2-l_1)^{2} \fgmu(l_2) \big]\big[ \fconnf(l_3)^2 \fgmu(l_3)\widehat D_{\mulam,k}(l_3)^2 \big] 
						\; \frac{\dd (l_1, l_2, l_3)}{(2\pi)^{3d}} \right)^{1/2} \\ 
		& \quad \times \left( \int \big[(\fconnf \cdot \fgmu)(l_1-l_3)^2 \fgmu(l_1) \big] \big[ \fconnf(l_2)^4 \fgmu(l_2)^3\big] \big[(\fconnf\cdot\fgmu) (l_3-l_2)^2 \fgmu(l_3)  \big]
						\; \frac{\dd (l_1, l_2, l_3)}{(2\pi)^{3d}} \right)^{1/2}, }
which is easily decomposed as indicated by the square brackets. For the integral in~\eqref{eq:Hdiagramfourierboundterm_b}(ii), we use Cauchy-Schwarz to obtain a bound of the form
	\al{&\left( \int \big[ \fconnf(l_1)^4 \fgmu(l_1)^{3} \big] \big[ (\fconnf\cdot \fgmu)(l_2-l_3)^{2} \fgmu(l_2) \big]\big[ \fconnf(l_3)^2 \fgmu(l_3-k)^2 \fgmu(l_3+k) \big] 
						\; \frac{\dd (l_1, l_2, l_3)}{(2\pi)^{3d}} \right)^{1/2} \\ 
		& \quad \times \left( \int \big[(\fconnf \cdot \fgmu)(l_1-l_2)^2 \fgmu(l_1) \big] \big[ \fconnf(l_2)^4 \fgmu(l_2)^3\big] \big[(\fconnf \cdot\fgmu)(l_3-l_1)^2 \fgmu(l_3+k) \big]
						\; \frac{\dd (l_1, l_2, l_3)}{(2\pi)^{3d}} \right)^{1/2}. }
To resolve the integral over $l_3$ in the first factor, we use $\fgmu \geq 0$ to bound
	\eqq{ \int \fconnf(l_3)^2 \fgmu(l_3-k)^2 \fgmu(l_3+k) \frac{\dd l_3}{(2\pi)^d} \leq \int \fconnf(l_3)^2 \widehat D_{\mulam,k}(l_3)^3 \frac{\dd l_3}{(2\pi)^d} , \label{eq:BA:H_diagram_RW_estimates_extension}}
which is bounded by Proposition~\ref{thm:randomwalkestimates2}. The integral over $l_3$ in the second factor is handled similarly; the integrals over $l_1$ and $l_2$ can be handled in exactly the same way as in the bound on~\eqref{eq:Hdiagramfourierboundterm_b}(i).
\end{proof}

\begin{proof}[Proof of Proposition~\ref{thm:convergenceoflaceexpansion} \ch{and Corollary~\ref{cor:convergenceoflaceexpansion_uniform}}]
We first recall from Section~\ref{sec:diagrammaticbounds} and Propositions~\ch{\ref{prop-bounds-LA-coefficients-unweighted}} and~\ref{thm:PsiDiag_Bound_Derangement}, as well as Corollary~\ref{cor:DB:Pi_x_bounds}, that $\Pi_\lambda^{(n)}$, $\int \Pi_\lambda^{(n)}(x) \dd x$, and $\int [1-\cos(k \cdot x)] \Pi_\lambda^{(n)} (x) \dd x$ are bounded in terms of $\trilam, \trilamo, \trilamoo, W_\lambda(k)$, and $H_\lambda(k)$ for $n \geq 0$ (with an extra term for $\int[1-\cos(k\cdot x)]\Pi_\lambda^{(1)}(x) \dd x$ not of this form but addressed in Lemma~\ref{lem_Pi1-consequence-bound}). Recalling these bounds and combining them with the four lemmas just proved gives
	\al{\lambda\int \Pi_\lambda^{(n)}(x) \dd x \leq (\const' \beta)^{n \vee 1}, &\qquad \lambda\int [1-\cos (k\cdot x)] \Pi_\lambda^{(n)}(x) \dd x \leq [1-\fconnf(k)] (\const' \beta)^{(n-1) \vee 1}, \\
		 \Pi_\lambda^{(n)}(x) & \leq \const' (\const' \beta)^{(n-1)}, }
for some $\const' = c'(f(\lambda))$, \ch{where $x\mapsto c'(x)$ is increasing. This proves Proposition~\ref{thm:convergenceoflaceexpansion}.} 

\ch{For Corollary~\ref{cor:convergenceoflaceexpansion_uniform}, we assume $f(\lambda)<3$. Thus,} if $\const' \beta < 1$,
	\[\lambda \int \sum_{m=0}^{n}\Pi_{\lambda}^{(m)}(x)  \dd x = \sum_{m=0}^{n} \lambda \int \Pi_\lambda^{(m)}(x) \dd x \leq \sum_{m=0}^{\infty} (\const' \beta)^{m \vee 1} = \const' \beta (1 + (1-\const' \beta)^{-1}).\]
If $\const' \beta < 1/2$, then we can choose $\const = 4\const'$. The other two bounds follow similarly. 

Note that, by dominated convergence, this implies that the limit $\Pi_\lambda$ is well defined, and so is its Fourier transform. We are left to deal with $R_{\lambda, n}(x)$. Recalling the bound~\eqref{eq:LE:Rn_Pin_bound} and combining it with the bound on $\Pi_\lambda^{(n)}$ from Corollary~\ref{cor:DB:Pi_x_bounds} implies
	\[ \sup_{x\in\Rd} |R_{\lambda, n}(x)| \leq \lambda \ftlam(\orig) \sup_{x\in\Rd} |\Pi_\lambda^{(n)}(x)| \leq f(\lambda) \ftlam(\orig) (\const \beta)^n. \]
As $\ftlam(\orig)$ is finite for $\lambda<\lambda_c$, the right-hand side vanishes as $n\to\infty$ for sufficiently small $\beta$. As a consequence, $R_{\lambda,n} \to 0$ uniformly in $x$, which proves~\eqref{eq:LE_identity_OZE} for $\lambda<\lambda_c$.
\end{proof}

\subsection{The bootstrap argument}\label{sec:forbidden_region_argument}
The missing piece to prove our main theorems is Proposition~\ref{thm:bootstrapargument}, proving that the function $f=f_1 \vee f_2 \vee f_3$ defined in~\eqref{defeq:fbootstrapfunction} is continuous on $[0,\lambda_c)$, bounded by $2$ at $0$ and that $f \leq 3$ implies $f \leq 2$.

\begin{prop}[The forbidden-region argument] \label{thm:bootstrapargument}
The following three statements are true:
\begin{enumerate}
\item The function $f$ satisfies $f(0) \leq 2$.
\item The function $f$ is continuous on $[0, \lambda_c)$.
\item Moreover, $f(\lambda) \notin (2,3]$ for all $\lambda \in [0,\lambda_c)$ provided that $d> 3(\alpha \wedge 2)$ and $\beta \ll 1$ (i.e., $d$ sufficiently large for (H1) and $L$ sufficiently large for (H2), (H3)).
\end{enumerate}
Consequently, $f(\lambda) \leq 2$ holds uniformly in $\lambda<\lambda_c$ for $d>3(\alpha \wedge 2)$ and $\beta \ll 1$.
\end{prop}

\begin{proof}
We show that $(1.)$-$(3.)$ hold for the functions $f_1, f_2, f_3$ separately. The result then follows for $f$ itself.

\paragraph{(1.) Bound for $\lambda=0$.} Trivially, $f_1(0)=0$. Note that $\mu_0=0$, and so $\widehat G_{\mu_0} \equiv 1$. Also, $\tau_0 = \connf$ and so $\widehat\tau_0 = \fconnf$. From this we infer $f_2(\lambda) \leq 1$. Lastly, by Lemma~\ref{lem:discrete-derivative-bound}(ii),
	\[ f_3(0) = \sup_{k,l} \frac{|\Delta_k \fconnf(l)|}{252[1-\fconnf(k)]} \leq \frac{1}{126}.\]

\paragraph{(2.) Continuity of the bootstrap function.} The continuity of $f_1$ is obvious. The same idea used in the discrete $\mathbb Z^d$ case is used to handle the other two bootstrap functions. More precisely, we make use of a known result, formulated by Slade~\cite[Lemma 5.13]{Sla06}. The idea is summarized as follows:
\begin{compactitem}
\item We want to prove the continuity of the supremum of a family $(h_\alpha)_{\alpha \in B}$ of functions ($\alpha$ is either $k$ or the tuple $(k,l)$, and $B$ is either $\Rd$ or $(\Rd)^2$).
\item To this end, we fix an arbitrary $\rho>0$ and show that $(h_\alpha)$ is equicontinuous on $[0,\lambda_c-\rho]$, i.e. that for $\varepsilon>0$ there exists $\delta>0$ such that $|s-t|<\delta$ implies $|h_\alpha(s) - h_\alpha(t)| \leq \varepsilon$ uniformly in $\alpha\in B$.
\item We prove this equicontinuity by taking derivatives with respect to $\lambda$ and bounding this derivative uniformly in $\alpha$ on $[0,\lambda_c-\rho]$.
\item Furthermore, we prove that $(h_\alpha)_{\alpha\in B}$ is uniformly bounded on $[0,\lambda_c-\rho]$.
\item This implies that $t \mapsto \sup_{\alpha\in B} h_\alpha(t)$ is continuous on $[0, \lambda_c - \rho]$. As $\rho$ was arbitrary, we get the desired continuity on $[0,\lambda_c)$.
\end{compactitem}
A first important observation is that we actually have to deal with the supremum of a family $(|h_\alpha|)_{\alpha\in B}$ of functions, which might cause headaches when taking derivatives. However, as an immediate consequence of the reverse triangle inequality, the equicontinuity of $(h_\alpha)_{\alpha\in B}$ implies the equicontinuity of $(|h_\alpha|)_{\alpha\in B}$, as $\big||h_\alpha(x+t)| - |h_\alpha(x)|\big| \leq |h_\alpha(x+t)-h_\alpha(x)|$.

We start with $f_2$ and consider
	\[ \frac{\dd}{\dd\lambda} \frac{\ftlam(k)}{\fgmu (k)} = \frac{1}{\fgmu(k)^2} \left[ \fgmu(k) \frac{\dd \ftlam(k)}{\dd\lambda}
					- \ftlam(k) \left.\frac{\dd \fgreensmu(k)}{\dd\mu}\right\vert_{\mu=\mulam} \times \frac{\dd\mulam}{\dd\lambda} \right],	\]
which we treat by bounding every appearing term. With 2$\fgmu(k) \geq 1$, we start by noting
	\[ \frac 12 \leq \frac{1}{1- \mulam \fconnf(k)} = \fgmu(k) \leq \fgmu(\orig) = \ftlam(\orig) \leq \widehat \tau_{\lambda_c - \rho} (\orig) = \frac{\chi(\lambda_c- \rho) - 1}{\lambda_c - \rho},\]
where the last term is finite. The finiteness of $\ftlam(\orig)$ turns out to be helpful several times, as it also bounds $|\ftlam(k)| \leq \ftlam(\orig)$ uniformly in $k$. The derivative of the two-point function satisfies
	\[ \left\vert \frac{\dd}{\dd\lambda} \ftlam(k) \right\vert = \left\vert \int \e^{\i k\cdot x} \frac{\dd}{\dd\lambda} \tlam(x) \dd x \right\vert \leq \int \frac{\dd}{\dd\lambda} \tlam(x) \dd x
				= \frac{\dd}{\dd\lambda} \int \tlam(x) \dd x = \frac{\dd}{\dd\lambda} \ftlam(\orig) \leq \ftlam(\orig)^2.\]
The exchange of derivative and integral is justified as the integrand $\tlam$ is bounded uniformly in $\lambda$ by the integrable function $\tau_{\lambda_c-\rho}$. The last bound is Lemma~\ref{lem:differentialinequalityclustersize}, and so we make use of $\ftlam(\orig)$ being finite again.

By definition of $\fgreensmu$ (recall~\eqref{eq:def:greens_fct_fourier}), $|\tfrac{\dd}{\dd\mu} \fgreensmu(k) | \leq \fgreensmu(k)^2 \leq \fgreensmu(\orig)^2$ which, for $\mu=\mulam$, equals $\ftlam(\orig)^2$. Lastly, $\tfrac{\dd}{\dd\lambda} \mulam = \tfrac{\dd}{\dd\lambda}\ftlam(\orig)/ \ftlam(\orig)^2\leq 1$ by Lemma~\ref{lem:differentialinequalityclustersize}.

This proves the continuity of $f_2$. It is not hard to see that $f_3$ can be treated in a similar way:
	\[ \frac{\dd}{\dd\lambda} \frac{\Delta_k \ftlam(l)}{\widehat U_{\mulam}(k,l)} = \frac{1}{\widehat U_{\mulam}(k,l)^2} \left[\widehat U_{\mulam}(k,l) \frac{\dd \Delta_k\ftlam(l)}{\dd\lambda}
				- \Delta_k \ftlam(l) \left. \frac{\dd \widehat U_{\mu}(k,l)}{\dd\mu} \right\vert_{\mu=\mulam} \times \frac{\dd\mulam}{\dd\lambda} \right]. \]
Recalling the definitions of $\Delta_k \ftlam(l)$ and $\widehat U_{\mulam}(k,l)$, similar bounds as used for $f_2$ can be applied.

\paragraph{(3.) The forbidden region.}

To show the claim, we assume that \col{$f(\lambda) \leq 3$ for $\lambda \in [0,\lambda_c)$ and  show that this implies $f (\lambda)\leq 2$}. The assumption $f(\lambda) \leq 3$ allows us to apply \col{Corollary~\ref{cor:convergenceoflaceexpansion_uniform}}.

Throughout this whole part, we use $M$ and $\tilde M$ to denote constants whose exact value may change from line to line. We stress that they are independent of $d$ for (H1) and independent of $L$ for (H2), (H3). We start by setting $a := \connf + \Pi_\lambda$, and hence
	\[\widehat a(k) = \fconnf(k) + \fpilam(k).\]
By \ch{Corollary~\ref{cor:convergenceoflaceexpansion_uniform}}, $\ftlam$ takes the form $\ftlam(k) = \widehat a(k) / (1-\lambda \widehat a(k))$, and thus
	\eqq{ \mulam = \lambda + \frac{\fpilam(\orig)}{\widehat a(\orig)}. \label{eq:mulamBSidentity}}
Also, we will frequently use that $|\fpilam(k)| \leq M\beta$ uniformly in $k$.

Consider first $f_1$. Applying \ch{Corollary~\ref{cor:convergenceoflaceexpansion_uniform}} to~\eqref{eq:ftlam_chi_relation} gives $\chi(\lambda) = (1-\lambda \widehat a(\orig))^{-1}$, and so
	\eqq{ \lambda = \frac{1 - \chi(\lambda)^{-1}}{1+\fpilam(\orig)} \leq (1 - \chi(\lambda)^{-1}) (1+M\beta) \leq 1 + M\beta. \label{eq:lambda_c_asymptotics}}
Using $1 \geq \chi(\lambda)^{-1} \searrow 0$ for $\lambda \nearrow \lambda_c$ implies the required bound and, moreover, $\lambda = 1 + \mathcal O(\beta)$ for $\lambda\nearrow \lambda_c$. What we have used here, and will frequently use in this section, is that when we are confronted with an expression of the form $(1-\widehat g(k))^{-1}$, where $|\widehat g(k)| \leq M\beta$, we can choose $\beta$ small enough so that $M \beta <1$ and there exists a constant $\tilde M$ such that
	\[ 0 \leq \frac{1}{1-\widehat g(k)} \leq \frac{1}{1- M\beta} = \sum_{l \geq 0} (M\beta)^l \leq 1 + \tilde M \beta. \]
To deal with $f_2$, we use the ``split'' $f_2=f_4 \vee f_5$, where we introduce $A_\rho := \{k \in \Rd: |\fconnf(k)| \leq \rho \}$ and set
	\[f_4(\lambda) := \sup_{k \in A_\rho} \frac{|\ftlam(k)|}{\fgmu (k)}, \qquad f_5(\lambda) := \sup_{k \in A_\rho^c} \frac{|\ftlam(k)|}{\fgmu (k)}. \]
The precise value of $\rho$ does not matter, but for practical purposes, we set it to be $\rho=\ch{\tfrac{1}{4}}$. Now, for $k\in A_\rho$, we see that
	\[\frac{|\ftlam(k)|}{\fgmu(k)} = |\widehat a(k)| \cdot \frac{(1-\mulam \fconnf(k))}{1-\lambda \widehat a(k)} 
				\leq (\rho + M \beta) \cdot \frac{1+\rho}{1- (1+M\beta) (\rho + M\beta)} \leq 1+\tilde M \beta. \]
Hence, $f_4(\lambda) \leq 1+M\beta$. We turn to $f_5$ and consequently to $k\in A_\rho^c$. We define 
	\[\widehat N(k) = \widehat a(k) / \widehat a(\orig), \qquad \widehat F(k) = (1-\lambda \widehat a(k))/ \widehat a(\orig), \qquad \widehat Q(k) = (1+\fpilam(k)) / \widehat a(\orig), \]
so that $\ftlam(k) = \widehat N(k) / \widehat F(k)$. Rearranging gets us to
	\algn{\frac{\ftlam(k)}{\fconnf(k) \fgmu(k)} &= \widehat N(k) \frac{1-\mulam\fconnf(k)}{\fconnf(k) \widehat F(k)} = \widehat Q(k) + \widehat N(k) \frac{1-\mulam\fconnf(k)
					- \frac{\widehat Q(k)}{\widehat N(k)}\fconnf(k)\widehat F(k)}{\fconnf(k) \widehat F(k)} \notag \\ \label{eq:f5rearranging}
		&= \widehat Q(k) + \tfrac{\widehat N(k)}{\widehat F(k)} \fconnf(k)^{-1} \left[1-\mulam\fconnf(k) - \tfrac{\fconnf(k) \widehat Q(k)}{\widehat N(k)}\widehat F(k) \right].}
The extracted term $\widehat Q(k)$ satisfies $|\widehat Q(k)| \leq 1 + M\beta$. We further observe that
	\[\frac{\fconnf(k) \widehat Q(k)}{\widehat N(k)} = \frac{\fconnf(k) (1+ \fpilam(k))}{\fconnf(k) + \fpilam(k)} = 1 - \frac{[1-\fconnf(k)] \fpilam(k)}{\fconnf(k) + \fpilam(k)} =: 1- \widehat b(k). \]
Recalling identity~\eqref{eq:mulamBSidentity} for $\mulam$, we can rewrite the quantity $[1- \mulam \fconnf(k)- (1-\widehat b(k)) \widehat F(k)]$, appearing in \eqref{eq:f5rearranging}, as
	\al{ \frac{1+\fpilam(\orig) - \big[\lambda + \fpilam(\orig) + \lambda\fpilam(\orig)\big] \fconnf(k) - 1 + \lambda(\fconnf(k)+\fpilam(k)) + \widehat b(k)(1-\lambda\widehat a(k))}{1+\fpilam(\orig)} \\
		= \frac{ [1-\fconnf(k)]\left(\fpilam(\orig) + \lambda \fpilam(\orig)\right) + \lambda[\fpilam(k)-\fpilam(\orig)] + \widehat b(k)(1-\lambda\widehat a(k))}{1+\fpilam(\orig)}.}
Noting that $|\fpilam(\orig)-\fpilam(k)| \leq M [1-\fconnf(k)] \beta$ by \ch{Corollary~\ref{cor:convergenceoflaceexpansion_uniform}}, the first three terms are bounded by $M [1-\fconnf(k)] \beta$ for some constant $M$. Using~\eqref{eq:lambda_c_asymptotics}, the last term is
	\al{\left\vert\frac{\widehat b(k)(1-\lambda \widehat a(k))}{1+\fpilam(\orig)} \right\vert &= \left\vert \frac{1-\fconnf(k)}{1+\fpilam(\orig)} \cdot \frac{\fpilam(k)}{\fconnf(k) + \fpilam(k)}
				- \frac{\lambda[1-\fconnf(k)] \fpilam(k)}{1+\fpilam(\orig)} \right\vert \\
		& = [1-\fconnf(k)] \frac{|\fpilam(k)|}{|1+\fpilam(\orig)|}\left\vert \frac{1}{\fconnf(k) + \fpilam(k)} -\lambda \right\vert \\
		&\leq [1-\fconnf(k)] (1+M\beta) |\fpilam(k)| \Big( 2\rho^{-1} + \lambda (1+M\beta) \Big) \\& \leq \tilde M [1-\fconnf(k)] \beta.	}
Again, we require $\beta$ to be small; in this case we want that $|\fpilam(k)| \leq M\beta < \rho/2$. Putting these just acquired bounds back into \eqref{eq:f5rearranging}, we can find constants $M, \tilde M$ such that
	\algn{\left\vert\frac{\ftlam(k)}{\fconnf(k) \fgmu(k)} \right\vert &\leq | \widehat Q(k)| + M\beta \Big\vert\frac{\widehat N(k)}{\widehat F(k)}\Big\vert \left\vert \frac{1-\fconnf(k)}{\fconnf(k)} \right\vert \notag \\
		& \leq 1+ M\beta + 2 M\beta (1 + \tilde M \beta) \frac{1}{|\fconnf(k)|} \left\vert \frac{\ftlam(k)}{\fgmu(k)} \right\vert \leq 1+ M\beta + 2 \cdot 3 \cdot M \rho^{-1} \beta (1+\tilde M \beta)) \notag \\
		& \leq 1 + \bar M \beta, \label{eq:bootstrap:f_5_improvement} }
for some constant $\bar M$. Note that we have used the bound $[1-\fconnf(k)] \leq 2[1-\mulam \fconnf(k)]$ to get from $\widehat G_1(k)^{-1}$ to $\fgmu(k)^{-1}$. As $|\ftlam(k)|/\fgmu(k) \leq | \ftlam(k) / (\fconnf(k) \fgmu(k))|$, this concludes the improvement of $f_5$.

Before we treat $f_3$, we introduce $f_6$ given by
	\[f_6(\lambda) := \sup_{k\in\Rd} \frac{1-\mulam \fconnf(k)}{|1- \lambda (\fconnf(k)+\fpilam(k))|}. \]
We show that $f(\lambda) \leq 3$ implies $f_6(\lambda) \leq 2$. To do so, consider first $k\in A_\rho$ and choose $\beta$ small enough so that $1- \lambda(\rho + |\fpilam(k)|) >0$ (note that $\lambda \vee (1+|\fpilam(k)|) \leq 1+M\beta$). We then have
	\eqq{ \frac{1-\mulam \fconnf(k)}{1-\lambda \widehat a(k)} \leq \frac{1+\rho}{1- (1+M\beta) (\rho+M\beta)} \leq 2, \label{eq:bootstrap_f_6_first_bound}}
for $\rho=1/4$ and $\beta$ sufficiently small. Now, when $k \in A_\rho^c$, we have
	\eqq{\frac{1-\mulam \fconnf(k)}{1-\lambda \widehat a(k)} = \left\vert \frac{\ftlam(k)}{\fconnf(k) \fgmu(k)} \right\vert \cdot \left \vert \frac{\fconnf(k)}{\widehat a(k)} \right\vert
					\leq (1 + M \beta) \left\vert 1- \frac{\fpilam(k)}{\widehat a(k)} \right\vert \leq (1+M\beta) \left(1+\frac{M\beta}{\rho-M \beta} \right), \label{eq:bootstrap_f_6_second_bound} }
which is bounded by $1+\tilde M \beta$. Note that for the first bound in~\eqref{eq:bootstrap_f_6_second_bound}, we used the estimate established in~\eqref{eq:bootstrap:f_5_improvement}, which is stronger than a bound on $f_5$. In conclusion, \eqref{eq:bootstrap_f_6_second_bound} together with~\eqref{eq:bootstrap_f_6_first_bound} shows $f_6 (\lambda) \leq 2$. We are now equipped to improve the bound on $f_3$. As a first step, elementary calculations give the identity
	\[\Delta_k \ftlam(l) = \underbrace{\frac{\Delta_k\widehat a(l)}{1- \lambda \widehat a(l)}}_{\text{(I)}}
				+ \sum_{\sigma= \pm 1} \underbrace{\frac{\lambda \big(\widehat a(l + \sigma k) - \widehat a(l)\big)^2}{\big(1-\lambda\widehat a(l)\big)\big(1- \lambda \widehat a(l+\sigma k)\big)}}_{\text{(II)}}
				+ \underbrace{\widehat a(l) \Delta_k \left( \frac{1}{1-\lambda \widehat a(l)}\right)}_{\text{(III)}}. \]
We bound each of the three terms separately. We note that by Lemma~\ref{lem:discrete-derivative-bound} and \ch{Corollary~\ref{cor:convergenceoflaceexpansion_uniform}}, we have $|\Delta_k \fpilam(l)| \leq |\fpilam(\orig)-\fpilam(k)| \leq [1-\fconnf(k)] M\beta$. With this in mind,
	\al{ |\text{(I)}| &= \left\vert \Delta_k \widehat a(l)\right\vert \cdot \left\vert \frac{1-\mulam \fconnf(l)}{1-\lambda \widehat a(l)} \right\vert \cdot \fgmu(l) 
						\leq 2 \fgmu(l) \Big\vert \Delta_k \fconnf(l) + \Delta_k \fpilam(l) \Big\vert \\
		& \leq 2 \fgmu(l) \Big\vert 1-\fconnf(k) + [1-\fconnf(k)] M\beta \Big\vert = 2 (1+ M\beta) [1-\fconnf(k)] \fgmu(l)\\
		& \leq 4 (1+\tilde M \beta) [1-\fconnf(k)] \fgmu(l) \fgmu(l+k), }
where we have used the improved bound on $f_6$ and $2\fgmu(l+k) \geq 1$. This type of bound will be sufficient for our purposes, and we aim for similar bounds on (II) and (III). Recalling that $\partial_k^\pm g(l) = g(l\pm k)-g(l)$, we are interested in $\partial_k^\pm \fconnf(l)$ and $\partial_k^\pm \fpilam(l)$ to deal with (II). Note that for $g$ with $g(x)=g(-x)$, we have
	\algn{ |\partial_k^\pm \widehat g(l)| &= \left\vert \int \e^{\i l \cdot x} (\e^{\pm \i k \cdot x} - 1) g(x) \dd x \right\vert \leq \int \left\vert \e^{\pm \i k \cdot x} -1\right\vert \cdot |g(x)| \dd x \notag\\
		& \leq \int \Big( [1 -\cos(k\cdot x)] + |\sin (k \cdot x)| \Big) |g(x)| \dd x . \label{eq:discretepartialsbounds}}
Now, with the help of the Cauchy-Schwarz inequality and \eqref{eq:discretepartialsbounds},
	\al{|\partial_k^\pm \fconnf(l)| & \leq \left( \int \connf(x) \dd x \right)^{1/2} \left( \int \sin (k \cdot x)^2 \connf(x) \dd x \right)^{1/2} + \int [1-\cos(k \cdot x)] \connf(x) \dd x \\
		&= 1 \cdot \left( \int [1-\cos (k \cdot x)^2] \connf(x) \dd x \right)^{1/2} + [1-\fconnf(k)] \\
		&\leq 2 \left( \int [1-\cos (k \cdot x)] \connf(x) \dd x \right)^{1/2} + [1-\fconnf(k)] \\
		&= 2 [1-\fconnf(k)]^{1/2} + [1-\fconnf(k)] \leq 4 [1-\fconnf(k)]^{1/2}.}
Similarly,
	\al{|\partial_k^\pm \fpilam(l)| & \leq \left( \int |\Pi_\lambda(x)| \dd x \right)^{1/2} \left( 2\int [1-\cos(k \cdot x)] |\Pi_\lambda(x)| \dd x \right)^{1/2} + \int [1-\cos(k \cdot x)] |\Pi_\lambda(x)| \dd x \\
		& \leq (M \beta)^{1/2} (2M [1-\fconnf(k)] \beta)^{1/2} + M [1-\fconnf(k)] \beta \\
		& \leq \tilde M [1-\fconnf(k)]^{1/2} \beta. }
We deal with the denominator in (II) by noting that, for $\sigma\in\{-1,0,1\}$,
	\eqq{\frac{1}{1-\lambda \widehat a(l + \sigma k)} = \frac{1-\mulam \fconnf(l+\sigma k)}{1-\lambda \widehat a(l + \sigma k)} \fgmu(l+\sigma k) \leq 2 \fgmu(l+\sigma k), \label{eq:f6application} }
employing the improved bound on $f_6$ again. In summary,
	\al{\text{(II)} & \leq (1+M\beta) \Big((4 + M\beta) [1-\fconnf(k)]^{1/2} \Big)^2 4 \fgmu(l) \fgmu(l \pm k) \\
		& \leq 64 (1+\tilde M \beta) [1-\fconnf(k)] \fgmu(l) \fgmu(l \pm k). }
Turning to (III), we note that $|\widehat a(l)| \leq 1+M\beta$. We treat the second factor with Lemma~\ref{lem:discrete-derivative-bound}, and so we recall that we can recycle the bound observed in \eqref{eq:f6application} to get $(1-\lambda\widehat a(l))^{-1} \leq (1+M \beta) \fgmu (l)$. We furthermore obtain the bound
	\al{ \lambda \big( \widehat{|a|}(\orig) - \widehat{|a|}(k) \big) &= \lambda \int [1-\cos(k\cdot x)] |\connf(x) + \Pi_\lambda(x)| \dd x \\
		& \leq (1 + M \beta) [1-\fconnf(k)]. }
Substituting this into Lemma~\ref{lem:discrete-derivative-bound}, we obtain
	\al{\Delta_k \frac{1}{1-\lambda \widehat a(l)} &\leq (1 + M \beta)^3 \Big( \fgmu (l-k) + \fgmu (l+k) \Big) \fgmu(l) [1-\fconnf(k)] \\
		& \quad + 8 (1+ M\beta)^5 \fgmu(l-k) \fgmu(l) \fgmu(l+k) [1-\fconnf(l)] \cdot [1-\fconnf(k)] \\
		& \leq 16 (1+M\beta) [1 -\fconnf(k)] \times \Big( \fgmu(l-k) \fgmu(l) + \fgmu(l) \fgmu(l+k) + \fgmu(l-k)\fgmu(l+k) \Big). }
Putting everything together, we are done, as $|\Delta_k \ftlam(l)| \leq (1 + M \beta) \widehat U_{\mulam}(k,l)$. This finishes the proof of Proposition~\ref{thm:bootstrapargument}.
\end{proof}

\section{Proof of the main theorems} \label{sec:maintheorems}

Proposition \ref{thm:bootstrapargument} implies that $f(\lambda)\leq 2$ for all $\lambda\in[0,\lambda_c)$, and hence Lemma~\ref{lem:triangle-bubble-consequence-bound} gives that  $\trilam \leq C\beta$ uniformly in $\lambda$. These bounds are uniform in $\lambda$, so that
monotone convergence implies the triangle condition. 

For continuity of $\theta$, assume that both $\orig$ and $x$ are in the (a.s.~unique) infinite component. This implies $\orig \longleftrightarrow x$, and so $0 \leq \theta(\lambda_c)^2 \leq \tau_{\lambda_c}(x)$ for all $x\in\Rd$ via the FKG inequality~\eqref{eq:prelim:FKG}. But $\tau_{\lambda_c}(x) \to 0$ for almost all $|x| \to \infty$ due to the triangle condition (a little extra effort shows that this holds for \emph{all} $x$), which implies that $\theta(\lambda_c)=0$.

Proposition~\ref{thm:convergenceoflaceexpansion} \col{and Corollary~\ref{cor:convergenceoflaceexpansion_uniform}} allow us to extend the Ornstein-Zernike equation to $\lambda_c$:
\begin{corollary}[The OZE at the critical point] \label{lem:convergence_of_lace_expansion_corollary_lambda_c}
The Ornstein-Zernike equation~\eqref{eq:LE_identity_OZE} extends to $\lambda_c$. In particular, the limit $\Pi_{\lambda_c}^{(n)} = \lim_{\lambda\nearrow \lambda_c} \Pi_\lambda^{(n)}$ exists for every $n\in\N_0$, and so does $\Pi_{\lambda_c} = \sum_{n \geq 0} (-1)^n \Pi_{\lambda_c}^{(n)}$.
\end{corollary}
\chr{The proof uses diagrammatic bounds formulated in Proposition \ref{thm:DB:Pi_bound_Psi}, and is postponed to Section \ref{sec:DB:bounding_events}.}

\begin{proof}[Proof of Theorem~\ref{thm:maintheorem}]
We have established the triangle condition, it therefore remains to prove the infrared bound. 
Let $\lambda<\lambda_c$ first. We reuse the notation $a=\connf + \Pi_\lambda$. In Fourier space, Proposition~\ref{thm:convergenceoflaceexpansion} gives
	\eqq{ \ftlam(k) \big(1-\lambda (\fconnf(k) + \fpilam(k) )\big) = \fconnf(k) + \fpilam(k). \label{eq:MT:LE_fourier_identity}}
We claim that $1-\lambda(\fconnf(k) + \fpilam(k)) >0$ for all $k$. To this end, assume first that there exists $k\in\Rd$ with $\hat a(k)=\fconnf(k) + \fpilam(k) = \lambda^{-1}$. As
	\[ |\ftlam(k)| \leq \ftlam(\orig) < \infty, \]
the left-hand side of~\eqref{eq:MT:LE_fourier_identity} vanishes, and so also the right-hand side must satisfy $\widehat a(k)=0$, which directly contradicts the assumption $\widehat a(k) = \lambda^{-1}$.

However, for $k=\orig$, the right-hand side of~\eqref{eq:MT:LE_fourier_identity} is positive for sufficiently small $\beta$ (it is at least $\phiint-C \beta$); since also $\ftlam(\orig) > 0$, we must have $1-\lambda(\fconnf(\orig) + \fpilam(\orig)) >0$. From this, the continuity of the Fourier transform implies that there cannot be a $k$ with $1-\lambda(\fconnf(k) + \fpilam(k)) <0$, proving the claim.

We now divide by $\chr{1-\lambda \widehat a(k)}$ in~\eqref{eq:MT:LE_fourier_identity} to obtain \col{
	\algn{ \lambda|\ftlam(k)| &= \frac{\lambda |\widehat a(k)|}{1-\lambda \widehat a(k)} 
			= \frac{\lambda |\fconnf(k)+\fpilam(k)|}
						{\underbrace{1-\lambda(\fconnf(\orig)+\fpilam(\orig))}_{>0} + \lambda[\fconnf(\orig)-\fconnf(k)] + \lambda[\fpilam(\orig) - \fpilam(k)]} \notag\\
		& \leq \frac{|\fconnf(k)+\fpilam(k)|}{[\fconnf(\orig)-\fconnf(k)] - |\fpilam(\orig) - \fpilam(k)|} 
						\leq \frac{|\fconnf(k)|+\mathcal O(\beta)}{[\fconnf(\orig)-\fconnf(k)] (1+\mathcal O(\beta))}
						= \frac{|\fconnf(k)|+\mathcal O(\beta)}{[\fconnf(\orig)-\fconnf(k)]},  \label{eq:MT:infrared_bd} }}
\col{using the bound~\eqref{eq:BA:convLE_intPi_bds} for $\fpilam(\orig) - \fpilam(k)$}. This proves the infra-red bound for $\lambda<\lambda_c$.

Let now $\lambda=\lambda_c$ and $k \neq \orig$. Note that $\widehat a(\orig) = \lambda_c^{-1}$. For contradiction, assume that $\widehat a(k) = \lambda_c^{-1}$ as well. We can write
	\al{0=1-\lambda_c(\connf(k) + \widehat\Pi_{\lambda_c}(k)) &= \underbrace{1-\lambda_c \col{(\fconnf(\orig)+\widehat\Pi_{\lambda_c}(\orig))}}_{=0}
					 + \lambda_c [\fconnf(\orig)-\fconnf(k)] + \lambda_c [\widehat\Pi_{\lambda_c}(\orig) - \widehat\Pi_{\lambda_c}(k) ] \\
		&= \lambda_c [\fconnf(\orig)-\fconnf(k)] (1 + \mathcal O(\beta)),}
using the bound~\eqref{eq:BA:convLE_intPi_bds} for $\widehat\Pi_{\lambda_c}(\orig) - \widehat\Pi_{\lambda_c}(k)$. But as \col{$\fconnf(\orig)-\fconnf(k) >0$} for $k \neq \orig$, this yields a contradiction.

With Corollary~\ref{lem:convergence_of_lace_expansion_corollary_lambda_c}, we get identity~\eqref{eq:MT:LE_fourier_identity} for $\lambda_c$, and by the above argument, we can again divide by $(1-\lambda_c\widehat a(k))$ and obtain the same bound as in~\eqref{eq:MT:infrared_bd}.
\end{proof}

\begin{proof}[Proof of Theorem~\ref{thm:mainthmcorollaries}]
At the beginning of Section \ref{sec:maintheorems}, we have proven that $\theta(\lambda_c)=0$. 
Mind that the asymptotic behavior of $\lambda_c$ was noted in~\eqref{eq:lambda_c_asymptotics} and the line below. Identity~\eqref{eq:lambda_c_identity_fpilam} follows from the lace-expansion identity~\eqref{eq:MT:LE_fourier_identity} for $k=\orig$, keeping in mind that $\ftlam(\orig)$ diverges for $\lambda \nearrow \lambda_c$ (this was already used in the proof of Theorem~\ref{thm:maintheorem}).

To prove $\gamma \geq 1$, we rely on the work done in Lemmas~\ref{lem:tlam_differentiability} and~\ref{lem:differentialinequalityclustersize}. We start by proving $\widehat\tau_{\lambda_c}(\orig) = \infty$. Recall that $\tlam^n(x,y) = \pla(\conn{x}{y}{\xi^{x,y}_{\Lambda_n}})$ and define
	\[\chi^n(\lambda) := \sup_{x \in \Lambda_n} \int \tlam^n(x,y) \dd y \qquad \text{for } \lambda \geq 0.\]
Mind that this is \emph{not} the expected size of the largest cluster in $\Lambda_n$. We claim that $1/\chi^n(\lambda)$ is an equicontinuous sequence with $\chi^n(\lambda) \to \ftlam(\orig)$ for every $\lambda \geq 0$. From this, we get continuity of $1/\ftlam(\orig)$.

As $\tlam^n(x,y) \nearrow \tlam(x-y)$, we get $\chi^n(\lambda) \nearrow \ftlam(\orig)$ by monotone convergence and thus $1/\chi^n(\lambda) \to 1/\ftlam(\orig)$. For the equicontinuity, we show that $1/\chi^n(\lambda)$ has uniformly bounded derivative. First, note that the same arguments used in Lemma~\ref{lem:differentialinequalityclustersize} show that, uniformly in $x$,
	\eqq{ \frac{\dd}{\dd\lambda} \int \tlam^n(x,y)\dd y \leq \chi^n(\lambda)^2. \label{eq:mainthms:chi_n_differential_inequality}}
We want to relate~\eqref{eq:mainthms:chi_n_differential_inequality} to $\tfrac{\dd}{\dd\lambda} \chi^n(\lambda)$. Given $\lambda$ and $\epsilon$, let $v_{\lambda,\epsilon} \in \Lambda_n$ be a point such that
	\[\int \tlam^n(v_{\lambda,\epsilon}, y) \dd y \geq \chi^n(\lambda) - \epsilon^2.\]
The exact choice of $v_{\lambda,\epsilon}$ (as it is not unique) does not matter. This gives
	\algn{ \chi^n(\lambda+\epsilon) - \chi^n(\lambda) & \leq \int \big(\tau^n_{\lambda+\epsilon} (v_{\lambda+\epsilon,\epsilon}, y) - \tlam^n(v_{\lambda+\epsilon,\epsilon}, y)\big) \dd y
							 + \epsilon^2 \notag\\
		& \leq \sup_{v\in\Lambda_n} \int \big( \tau^n_{\lambda+\epsilon} (v, y) - \tlam^n(v, y) \big) \dd y + \epsilon^2. \label{eq:mainthms:chi_n_maximizer_bound}}
Similarly to Lemma~\ref{lem:tlam_differentiability}, we can show that $\lambda \mapsto\chi^n(\lambda)$ is continuous and almost everywhere differentiable. We divide~\eqref{eq:mainthms:chi_n_maximizer_bound} by $\epsilon$ and let $\epsilon\searrow 0$. We claim that
	\eqq{ \frac{\dd}{\dd\lambda} \chi^n(\lambda) \leq \sup_{v\in\Lambda_n} \int \frac{\dd}{\dd\lambda} \tlam^n(v,y)\dd y \leq \chi^n(\lambda)^2. \label{eq:mainthms:chi_n_derivative}}
The exchange of limit and supremum is justified since the integral in~\eqref{eq:mainthms:chi_n_maximizer_bound} converges to $\tfrac{\dd}{\dd\lambda} \int \tlam^n(v,y)\dd y$ uniformly in $v$, and then~\eqref{eq:mainthms:chi_n_derivative} follows by employing the Leibniz integral rule and~\eqref{eq:mainthms:chi_n_differential_inequality} for the second bound. Rearranging, this yields
	\[ \frac{\dd}{\dd\lambda} \frac{1}{\chi^n(\lambda)} \geq -1.\]
In summary, the functions $\lambda\mapsto 1/\chi^n(\lambda)$ form a non-increasing family with uniformly bounded derivative, which gives us equicontinuity. Since also $1/ \chi^n(\lambda) \to 1/\ftlam(\orig)$ pointwise, this limit is also continuous and, since it attains zero for every $\lambda>\lambda_c$, we have $1/\widehat\tau_{\lambda_c}(\orig)=0$.

We can now integrate the inequality $\tfrac{\dd}{\dd\lambda}\ftlam(\orig)^{-1} \geq -1$ between $\lambda$ and $\lambda_c$, so that
	\eqq{\frac{1}{\widehat\tau_{\lambda_c}(\orig)} - \frac{1}{\ftlam(\orig)} = -\frac{1}{\ftlam(\orig)} \geq -(\lambda_c-\lambda). \label{eq:gamma_lower_bound}}
Hence, $\ftlam(\orig) \geq (\lambda_c-\lambda)^{-1}$ and, with~\eqref{eq:ftlam_chi_relation}, we obtain $\chi(\lambda) \geq \lambda (\lambda_c-\lambda)^{-1}$. This shows $\gamma \geq 1$.

To prove $\gamma \leq 1$, let $\lambda\in(0,\lambda_c)$. We have to repeat some of the calculations from the diagrammatic bounds. Note that
	\[\pla(u \in \piv{\orig,x;\xi^{\orig,x}}) = \E_\lambda\big[\mathds 1_{\{\conn{\orig}{u}{\xi^{\orig,u}}\}} \tlam^{\C(\orig,\xi^{\orig})}(u,x) \big] \]
follows in the same manner from Lemma~\ref{lem:stoppingset} as Lemma~\ref{lem:LE:cutting_point}. Applying Lemma~\ref{lem:tlam_differentiability} and then~\eqref{eq:LE:tau_thinning_split} gives
	\algn{\frac{\dd}{\dd\lambda} \ftlam(\orig) &= \int \frac{\dd}{\dd\lambda} \tlam(x) \dd x = \iint \pla\big(u \in \piv{\orig,x;\xi^{\orig,u,x}}\big) \dd u \dd x \notag\\
		&= \iint \E_\lambda[ \mathds 1_{\{\conn{\orig}{u}{\xi^{\orig,u}}\}} \tlam(x-u)] \dd u \dd x \notag\\
		& \qquad - \iint \E_\lambda\left[ \mathds 1_{\{\conn{\orig}{u}{\xi_0^{\orig,u}}\}} \mathds 1_{\{\xconn{u}{x}{\xi_1^{u,x}}{\C(\orig,\xi_0^{\orig})}\}}\right] \dd u \dd x. 
					\label{eq:MT:ftlam_identity}}
The first integral on the right-hand side of~\eqref{eq:MT:ftlam_identity} is $\ftlam(\orig)^2$. 
\chr{
A bound on the second integral is formulated as a separate lemma, whose proof is deferred to Section \ref{sec:bounds-lace-expansion}: 
\begin{lemma}\label{lem:GammaIntBound}
For $u,x\in\Rd$ and $\lambda\ge0$,  
\[\iint \E_\lambda\left[ \mathds 1_{\{\conn{\orig}{u}{\xi_0^{\orig,u}}\}} \mathds 1_{\{\xconn{u}{x}{\xi_1^{u,x}}{\C(\orig,\xi_0^{\orig})}\}}\right] \dd u \dd x 
	\le \lambda^{-2} \trilam(\orig) \chi(\lambda)^2.
	\]
\end{lemma}
}
In summary, $\tfrac{\dd}{\dd\lambda} \ftlam(\orig) \geq \ftlam(\orig)^2 - \lambda^{-2} \trilam \chi(\lambda)^2$. Rearranging yields
	\[ \frac{\dd}{\dd\lambda} \frac{1}{\ftlam(\orig)} \leq -1 +  \trilam \frac{\lambda^{-2} \chi(\lambda)^2}{\ftlam(\orig)^2} \leq -1 + \trilam(1+\lambda^{-1} + \lambda^{-2}) \leq -1/2.  \]
In this, the last inequality holds when the triangle $\trilam$ is small enough (guaranteed by Theorem~\ref{thm:maintheorem}) and $\lambda>0$. We hence get a lower bound counterpart to~\eqref{eq:mainthms:chi_n_derivative} and can integrate as in~\eqref{eq:gamma_lower_bound} to get $\gamma\leq 1$.
\end{proof}

\section{Bounds on the lace expansion} 
\label{sec:bounds-lace-expansion}
\ch{In this section, we complete the proof of the bounds on the lace expansion, as stated in Propositions \ref{prop-bounds-LA-coefficients-unweighted}, \ref{thm:PsiDiag_Bound_Derangement_N1} and \ref{thm:PsiDiag_Bound_Derangement} in Section \ref{sec-statement-bounds-LA-coefficients}. This section is organised as follows. In Section \ref{sec:DB:bounding_events}, we formulate bounding events for the lace-expansion coefficients, and use these in combination with the BK-inequality to derive bounding diagrams on them. A bounding diagram is an appropriate integral over products of two-point functions and related quantities. \chr{At the end of this subsection, we give the proof of Lemma \ref{lem:convergence_of_lace_expansion_corollary_lambda_c}.} 
In Section \ref{sec:diagrammaticbounds_no_disp}, we considerably simplify these bounding diagrams by relating them to triangles as in Proposition \ref{prop-bounds-LA-coefficients-unweighted}, the proof of which then easily follows. \chr{There we also give the proof of Corollary~\ref{cor:DB:Pi_x_bounds}.} 
In Section \ref{sec:diagrammaticbounds_disp}, we extend the analysis in Section \ref{sec:diagrammaticbounds_no_disp} to deal with {\em weighted} diagrams and use it to prove the displacement bounds in Propositions \ref{thm:PsiDiag_Bound_Derangement_N1} and \ref{thm:PsiDiag_Bound_Derangement}.}

\subsection{Bounding events for the lace-expansion coefficients} \label{sec:DB:bounding_events}
The aim of this section is to take the first step into finding simple bounds on the lace-expansion coefficients $\Pi_\lambda^{(n)}$. We start by stating the central result of this section, Proposition~\ref{thm:DB:Pi_bound_Psi}, while the remainder of the section is concerned with its proof.

For the proof, we first introduce certain bounding events in Definition~\ref{def:Fdiagrams} that allow for a simple pictorial representation and that bound the $E$ events. As a second step, we bound these events by large products of two-point functions (through heavy use of the BK inequality), constituting the bound of Proposition~\ref{thm:DB:Pi_bound_Psi}. We continue to simplify this bound in Section~\ref{sec:diagrammaticbounds_no_disp}. 

Definition~\ref{def:DB:psi_functions} introduces the quantities in terms of which the bound of Proposition~\ref{thm:DB:Pi_bound_Psi} is formulated. It also introduces Dirac delta functions. We stress that in this paper, we use them primarily for convenient and more compact notation and to increase readability. In particular, they appear when applying the Mecke equation~\eqref{eq:prelim:mecke_1} to obtain
	\begin{equation}
	\E \big[ \sum_{y\in\eta^u} f(y, \xi^u) \big] = \int (\lambda+\delta_{y,u}) \E_\lambda [ f(y,\xi^{u,y})] \dd y, 
	\end{equation} 
and so the factor $\lambda+\delta_{y,u}$ encodes a case distinction of whether point $y$ coincides with $u$ or not. 

\begin{definition}[The $\psi$ functions] \label{def:DB:psi_functions}
Let $w,x,y,z\in\Rd$. Recall that $\tlamo(x) := \delta_{x,\orig} + \lambda\tlam(x)$ from \eqref{tau-dot-def}. Moreover, let
	\al{ \parl(x,y) := \tlamo(x) \tlam(y), & \quad \triangle(x,y,z) := \tlam(x-y) \tlam(y-z) \tlam(z-x), \\
		 \Box(w,x,y,z) &:=\tlam(w-x) \tlam(x-y) \tlam(y-z) \tlam(z-w).}
We define
	\al{&\psi_0^{(1)}(w,u) := \lambda\triangle (\orig,w, u), \quad \psi_0^{(2)}(w,u):=\lambda \delta_{w,\orig} \int \triangle (\orig,t,u) \dd t, \quad
						\psi_0^{(3)}(w,u):= \connf(u) \delta_{w,\orig}, \\
		&\psi_n^{(1)} (a,b,t,z,x) := \lambda \parl(t-b, z-a) \triangle(t,z,x), \quad \psi_n^{(2)}(a,b,t,z,x) := \delta_{t,z} \delta_{z,x}\tlam(t-b)\tlam(z-a),\\
		&\psi^{(1)}(a,b,t,w,z,u) := \lambda^2 \Box(t,w,u,z) \parl(t-b,z-a), \\
		&\psi^{(2)}(a,b,t,w,z,u) := \lambda \triangle(t,z,u) \tlamo(t-w) \parl(w-b,z-a), \\
		&\psi^{(3)}(a,b,t,w,z,u) := \delta_{z,u} \delta_{t,z} \tlam(t-w) \parl(w-b,z-a), }
and set $\psi_0 := \psi_0^{(1)}+\psi_0^{(2)}+\psi_0^{(3)}, \psi_n := \psi_n^{(1)} + \psi_n^{(2)}, \psi:= \psi^{(1)}+\psi^{(2)}+\psi^{(3)}$.
\end{definition}

See Figure~\ref{fig:Psidiagrams} for a pictorial representation of the functions $\psi^{(j)}$.

\begin{figure}[t]
 \includegraphics{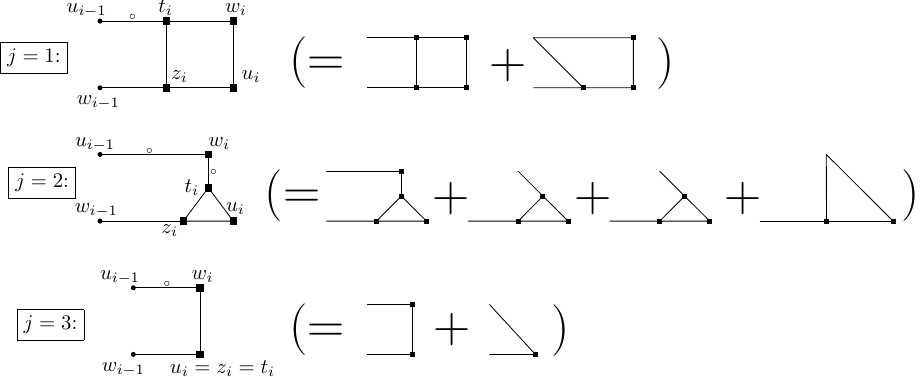}
	\caption{Diagrammatic representation of segment $i$, and hence of the functions $\psi^{(j)}$. Factors $\tlam$ are represented by lines, factors $\tlamo$ are represented by lines endowed with a `$\circ$'. The points $t_i, w_i, z_i, u_i$ (the ones labeled by index $i$) are depicted as squares---in the later decomposition of the full diagram into segments, these are the ones integrated over when bounding segment $i$. The small diagrams in brackets indicate the form that the diagrams take when expanding the two terms constituting $\tlamo$ (i.e.~when writing out all possible collapses). }
	\label{fig:Psidiagrams}
\end{figure}

\begin{prop}[Bound in terms of $\psi$ functions] \label{thm:DB:Pi_bound_Psi}
Let $n \geq 1, x\in\Rd$ and $\lambda\in [0,\lambda_c)$. Then
\[ \Pi_\lambda^{(n)}(x) \leq \lambda^n \int \psi_0(w_0,u_0) \Big(
  \prod_{i=1}^{n-1} \psi(\vec v_i) \Big) \psi_n(w_{n-1},u_{n-1}, t_n,
  z_n,x) \dd\big( (\vec w, \vec u)_{[0,n-1]}, (\vec t, \vec z)_{[1,n]}
  \big), \] where $\vec v_i = (w_{i-1},u_{i-1},t_i,w_i,z_i,u_i)$.
\end{prop}

The proof is given after Lemma \ref{lem:DB:bounds_E_by_F}.  Throughout
the paper, we use $\vec v_i$ as an abbreviation for various vectors of
elements in $\R^d$ and related expressions.

Recall that the edge-markings in~\eqref{eq:LE:Pin_def} are
independent, a fact that is heavily used in the
following. Unfortunately, the event taking place on graph $i$ is not
quite independent of the event taking place on graph $i-1$. However, a
little restructuring together with appropriate bounding events enables
us to guarantee such an independence. With the next steps, we achieve
two things: On the one hand, we bound the $E$ events by simpler ones
(see Definition~\ref{def:Fdiagrams} and
Lemma~\ref{lem:DB:bounds_E_by_F}), and on the other, we exploit the
independence structure.

We start by introducing a ``thinning connection'', defined for
edge-markings of sets of points (which may not be PPPs). To this end,
recall Definition \ref{def:LE:thinning_events} for the definition of a
thinning of a discrete set $A$.

\begin{definition}[Thinning connection]
  \ \label{def:DB:extended_thinning_events} Let $\xi_1, \xi_2$ be two
  independent edge-markings of two locally finite sets $A_1, A_2$
  (hence, $\xi_i=\xi_i(A_i)$). For $x,y\in\Rd$, we write
$\sqconn{x}{y}{(\xi_1, \xi_2)}$ if $x \in A_1$, $y\in A_2$
and $y\notin (A_2)_{\thinn{\C(x, \xi_1)}}$.
\end{definition}

Given $\C(x, \xi_1)$, which is determined by $\xi_1$,
$\{\sqconn{x}{y}{(\xi_1,\xi_2)} \}$ is just a thinning event in
$\xi_2$. On the other hand, given the thinning marks of $y$,
$\{\sqconn{x}{y}{(\xi_1,\xi_2)}\}$ is just a connection event in
$\xi_1$, as $x$ must be connected to some vertex $z$ in $\xi_1$ that
``thins out'' $y$.

We will apply Definition~\ref{def:DB:extended_thinning_events} only for pairs $(\xi_i, \xi_{i+1})$ from the sequence $(\xi_i)_{i \in \N_0}$ used in the definition of the lace-expansion coefficients.  

The next definition should be regarded as an extension of the disjoint
occurrence event to multiple connection events that may overlap in
their endpoints (similar to the application of the BK inequality
in~\eqref{eq:prelim:BK_application}), as well as to events involving
`$\sqarrow$' (living on two RCMs):

\begin{definition}[Multiple disjoint connection events] \label{def:DB:multi_circ} Let $m\in\N$ and
  $\vec x, \vec y \in (\Rd)^m$. \cov{We say that
  $\bigcirc_m^\leftrightarrow((x_j, y_j)_{1 \leq j \leq m}; \xi)$ occurs if there exist $m$ paths such
    $\conn{x_j}{y_j}{\xi}$ for every $j \in[m]$} 
  with the additional requirement that every point in $\eta$ is the
  interior vertex of at most one of the $m$ paths and none of the $m$
  paths contains an interior vertex in the set
  $\{x_j: j\in[m]\} \cup \{y_j: j\in [m]\}$.

  \ch{Moreover, let $\xi_1, \xi_2$ be two independent RCMs.}
\cov{We say
that    $\bigcirc_m^\sqarrow( (x_j,y_j)_{1 \leq j \leq m}; (\xi_1,\xi_2))$
occurs if}
\begin{itemize}
\item[(a)]
\cov{  $\bigcirc_{m-1}^\leftrightarrow((x_j,y_j)_{1 \leq j <m};\xi_1)$
  occurs and there exists paths connecting $x_j$ and $y_j$} for
  $j\in[m-1]$ that satisfy the requirements listed in the definition
  of $\bigcirc_{m-1}^\leftrightarrow((x_j,y_j)_{1 \leq j <m};\xi_1)$,
  and satisfy the additional requirement that they do not use $x_{m}$
  or $y_{m}$ as interior vertices;
\item[(b)]
\cov{  $\sqconn{x_{m}}{y_{m}}{(\xi_1[V(\xi_1)\setminus \{x_i, y_i\}_{1
      \leq i <m}],\xi_2)}\}$} occurs in such a way that at least one
  point $z$ in $\xi_1$ that is responsible for thinning out $y_m$ is
  connected to $x_m$ by a path $\gamma$ so that $z$, as well as all
  interior vertices of $\gamma$, are not contained in any path of the
  $\bigcirc_{m-1}^\leftrightarrow((x_j,y_j)_{1 \leq j <m};\xi_1)$
  event.
\end{itemize}
\ch{See Figure \ref{fig:DisjointConnection} for an example.}  
\begin{figure}[t]\centering
	\includegraphics[width=.5\textwidth]{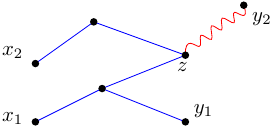}
	\caption{\ch{An example of the event
            $\bigcirc_2^\sqarrow( (x_1,y_1),(x_2,y_2);
            (\xi_1,\xi_2))$. The blue lines indicate edges in
            $\xi_1$. The wiggled line indicates that $z$ (a vertex in
            $\xi_1$) is in a $\xi_2$-thinning of $y_2$.}}
	\label{fig:DisjointConnection}
\end{figure}
\end{definition}

\cov{We shall work with the preceding definitions similarly as
explained after Definition \ref{def:LE:connection_events}.}

We only
require the events $\bigcirc_m^\leftrightarrow$ and
$\bigcirc_m^\sqarrow$ for distinct points $\vec x, \vec y$. We
moreover remark that
\cov{$\bigcirc_1^\leftrightarrow((x,y);\xi)$ means that $\conn{x}{y}{\xi}$ and
$\bigcirc_1^\sqarrow( (x,y); (\xi_1,\xi_2))$ means
that $\sqconn{x}{y}{(\xi_1,\xi_2)}$.} Furthermore, for distinct points
$u,v,x,y$, almost surely,
\[ \bigcirc_2^\leftrightarrow((u,v),(x,y);\xi^{u,v,x,y}) =
  \{\conn{u}{v}{\xi^{u,v}}\} \circ \{\conn{x}{y}{\xi^{x,y}}\}, \] and
so on. Crucially, $\bigcirc_m^\leftrightarrow$ is still amenable to
the use of the BK inequality (again, see the proof
of~\eqref{eq:prelim:BK_application}). In contrast, the thinning
connection as defined in
Definition~\ref{def:DB:extended_thinning_events} is not an increasing
event. \col{The reason for this is that adding points changes the
  ordering of points as given by~\eqref{eq:orderingofPPpoints} and
  thus the thinning variables}. As we would like to use the BK
inequality on $\bigcirc_m^\sqarrow$ later on, the following
observation gives an important identity for $\bigcirc_m^\sqarrow$:

\begin{lemma}[Relating $\bigcirc_m^\sqarrow$ and
  $\bigcirc_m^\leftrightarrow$] \label{lem:DB:squigarrow_tlam_equality}
  Let $m \in\N$ and $\vec x, \vec y \in (\Rd)^m$. Let $\xi_1, \xi_2$
  be two independent RCMs. Then
\[\pla\big( \bigcirc_m^\sqarrow\big((x_j,y_j)_{1 \leq j \leq m}; (\xi_1^{\vec x_{[1,m]}, \vec y_{[1,m-1]}}, \xi_2^{y_m}) 
\big) \big)
= \pla \big( \bigcirc_m^\leftrightarrow \big((x_j,y_j)_{1\leq j \leq m}; \xi_1^{(\vec x, \vec y)_{[1,m]}} \big)
          \big).\]
\end{lemma}
\begin{proof} 
\cov{Let us denote the event on the left-hand side by $E$ and
that on the right-hand side by $F$. 
The assertion follows once we have shown that
\begin{align}\label{e453}
\pla\big(E\mid \xi_1^{\vec x_{[1,m]}, \vec y_{[1,m-1]}} \big)
=\pla\big(F\mid \xi_1^{\vec x_{[1,m]}, \vec y_{[1,m-1]}} \big).
\end{align}
}

\cov{
In the remainder of the proof we fix $\xi_1^{\vec x_{[1,m]}, \vec y_{[1,m-1]}}$.
Let us call a point $x\in\eta_1\cup\{x_m\}$ {\em good} if
either $x=x_m$ or $F$ occurs with $x$ in place of $y_m$.
Otherwise we say that $x$ is {\em bad}.
If $x$ is good, then adding the edge 
$\{x,y_m\}$ to $\xi_1^{\vec x_{[1,m]}, \vec y_{[1,m-1]}}$
triggers the event $F$ to hold. Note that
adding one or more edges between $y_m$ and bad points cannot lead to the occurrence of $F$.
Let $A$ be the set of all good points. 
By the independence assumptions
of the RCM we then have that
\begin{align}\label{e758}
\pla\big(F\mid \xi_1^{\vec x_{[1,m]}, \vec y_{[1,m-1]}} \big)=1-\bar\connf(A, y_m).
\end{align}
On the other hand, if a point $x\in A$ thins out $y_m$,
then the event $E$ occurs. By definition of the thinning operation
we obtain \eqref{e758} with $E$ in place of $F$.
Hence \eqref{e453} follows and the proof is complete.}
\end{proof}

\chr{As an application of Lemma \ref{lem:DB:squigarrow_tlam_equality}, we now prove Lemma \ref{lem:GammaIntBound}.}
\begin{proof}[Proof of Lemma \ref{lem:GammaIntBound}]
We start by bounding the left-hand side of the asserted bound using as  
	\algn{ \iint& \E_\lambda\left[ \mathds 1_{\{\conn{\orig}{u}{\xi_0^{\orig,u}}\}} \mathds 1_{\{\xconn{u}{x}{\xi_1^{u,x}}{\C(\orig,\xi_0^{\orig})}\}}\right] \dd u \dd x\nonumber\\
	&=\E_\lambda \Big[ \mathds 1_{\{\conn{\orig}{u}{\xi_0^{\orig,u}}\}} \sum_{y \in \eta_1^x} \mathds 1_{\{\sqconn{\orig}{y}{(\xi_0^\orig, \xi_1^{x})} \}}
					\mathds 1_{\{\conn{u}{y}{\xi_1^u}\}\circ \{\conn{y}{x}{\xi_1^x}\}} \Big] \notag\\
		&= \E_\lambda \big[ \mathds 1_{\{\conn{\orig}{u}{\xi_0^{\orig,u}}\}} \mathds 1_{\{\sqconn{\orig}{x}{(\xi_0^\orig, \xi_1^{x})}\}} \big] \tlam(x-u) \notag\\
		& \quad + \lambda \int \E_\lambda \big[ \mathds 1_{\{\conn{\orig}{u}{\xi_0^{\orig,u}}\}} \mathds 1_{\{\sqconn{\orig}{y}{(\xi_0^\orig, \xi_1^{x,y})}\}} \big]
						 \pla \big( \{ \conn{u}{y}{\xi^{u,y}} \} \circ \{ \conn{y}{x}{\xi^{y,x}}\} \big) \dd y \notag\\
		& \leq \int \E_\lambda \big[ \mathds 1_{\{\conn{\orig}{u}{\xi_0^{\orig,u}}\}} \mathds 1_{\{\sqconn{\orig}{y}{(\xi_0^\orig, \xi_1^{y})}\}} \big]
						 \tlam(y-u) \tlamo(x-y) \dd y. \label{eq:MT:through_conn_bound}}
Note that
	\eqq{ \mathds 1_{\{\conn{\orig}{u}{\xi_0^{\orig,u}}\}} \mathds 1_{\{\sqconn{\orig}{y}{(\xi_0^\orig, \xi_1^{y})}\}} \leq \sum_{a \in \eta_0^\orig} 
					\mathds 1_{\bigcirc^\sqarrow_3 ((\orig, a), (a,u), (a,y); (\xi_0^{\orig,u}, \xi_1^y))},  \label{eq:MT:sconn_eta1_treegraph_bound}}
where we recall $\bigcirc_3^\sqarrow$ from Definition~\ref{def:DB:multi_circ}. 
sWe get that \eqref{eq:MT:through_conn_bound} is bounded above by
	\al{ \int \Big( \delta_{a,\orig} & \pla \big( \bigcirc^\sqarrow_2 ((\orig,u), (\orig,y); (\xi_0^{\orig,u}, \xi_1^y))\big)
				+ \lambda \pla \big( \bigcirc^\sqarrow_3 ((\orig, a), (a,u), (a,y); (\xi_0^{\orig,a,u}, \xi_1^y)) \big) \Big) \\
			& \qquad \times \tlam(y-u) \tlamo(x-y) \dd (a,u,x,y) \\
			\leq & \int \tlamo(a) \tlam(u-a) \tlam(y-a) \tlam(y-u) \tlamo(x-y) \dd(a,y,u,x) 
			=  \lambda^{-2} \trilam(\orig) \chi(\lambda)^2}
using Lemma~\ref{lem:DB:squigarrow_tlam_equality} and then the BK inequality. 
\end{proof}

We next define the events that will be used to bound the $E$ events: 
\begin{definition}[Bounding $F$ events] \label{def:Fdiagrams}
Let $\xi_1, \xi_2$ be two independent edge-markings, and let $n \geq 1$ and $a,b,t,w,z,u \in\Rd$. Define
	\al{F^{(1)}_0 (a, w, u, b; (\xi_1, \xi_2)) &:= \{a \nsim u \text{ in } \xi_1\} \cap \bigcirc_4^\sqarrow\big((a,u),(a,w),(u,w),(w,b);(\xi_1,\xi_2)\big), \\
		F^{(2)}_0 (a,w, u, b; (\xi_1,\xi_2)) &:= \{w=a\} \cap \{a\sim u \text{ in } \xi_1\} \cap \{ \sqconn{w}{b}{(\xi_1\setminus\{u\}, \xi_2)} \} , \\
		F_n (a,t,z,u; \xi) &:= \{|\{t, z , u\}| \neq 2 \} \cap \bigcirc_4^\leftrightarrow \big((a,t),(t,z),(t,u),(z,u); \xi \big), \\
		F^{(1)} (a,t,w,z,u,b; (\xi_1,\xi_2)) &:= \{|\{t,w,z,u\}|=4\} \cap \bigcirc_6^\sqarrow \big( (a,t),(t,z),(z,u),(t,w),(w,u),(w,b);(\xi_1,\xi_2) \big), \\
		F^{(2)} (a,t,w,z,u,b; (\xi_1,\xi_2)) &:= \{w \notin\{u,z\}, |\{t, z, u \}| \neq 2\} \\
			& \hspace{2.5cm} \cap \bigcirc_6^\sqarrow \big( (a,w),(w,t),(t,u),(t,z),(z,u),(w,b);(\xi_1,\xi_2) \big).}
In addition, let $F_0 = F_0^{(1)} \cup F_0^{(2)}$. 
\end{definition}
Figure~\ref{fig:Fdiagrams} illustrates the diagrammatic events $F_0, F^{(1)}, F^{(2)}$, and $F_n$. We say that a diagrammatic event \emph{collapses} when a subset of the arguments coincides. These collapses of points turn out to be a recurring source of trouble in this section. An example of a collapse is $z=t=u$ in the event $F^{(2)}$.
\begin{figure}
	\centering
 \includegraphics{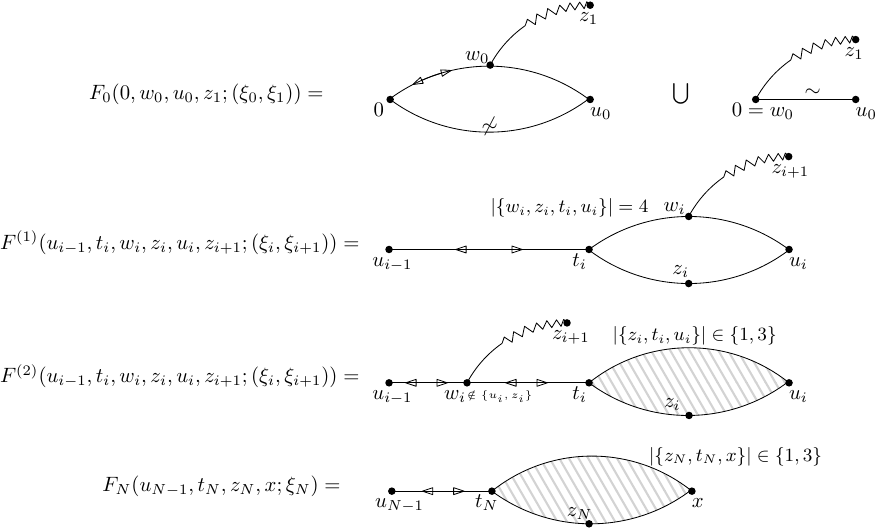}
	\caption{The full diagrammatic events. The line with a `$\sim$' symbol represents a direct edge. A line with a `$\nsim$' symbol indicates that this may \emph{not} be a direct edge. A partially squiggly lines represent the event $\{w_i \sqarrow z_{i+1} \text{ in }(\xi_i,\xi_{i+1})\}$, taking place on both $\xi_i$ and $\xi_{i+1}$. 
Arrows on a line indicate that the two endpoints of that line may coincide. The hatched area may collapse into a single point altogether.}
	\label{fig:Fdiagrams}
\end{figure}

The next lemma bounds the events $E$ by the simpler $F$ events. Recall the sequence $(\xi_i)_{i\in\N_0}$ from Section~\ref{sec:LE:derivation_of_LE} and recall that it denotes a sequence of independent RCMs. We denote the respective underlying PPPs by $\eta_i$ for $i\in\N_0$.
\begin{lemma}[Bounds in terms of $F$ events] \label{lem:DB:bounds_E_by_F}
Let $n\geq 1$ and let $u_0, \ldots, u_n=x \in \Rd$ be distinct points. Write $\C_i= \C(u_{i-1}, \xi_i^{u_{i-1}}), \xi'_i = \xi_i^{u_{i-1}, u_i}$, where $u_{-1}=\orig$. Then
	\al{ & \mathds 1_{\{ \dconn{\orig}{u_0}{ \xi'_0} \}} \prod_{i=1}^{n} \mathds 1_{E (u_{i-1}, u_i; \C_{i-1}, \xi'_i)}  \\
		& \qquad \leq \sum_{\vec z_{[1,n]}: z_i \in \eta_i^{u_i}} \Big( \sum_{w_0 \in \eta_0^{\orig}} \mathds 1_{F_0 (\orig,w_0, u_0, z_1; (\xi'_0,\xi'_1))} \Big)
						\Big( \sum_{t_n \in \eta_n^{u_{n-1}, x}} \mathds 1_{F_n(u_{n-1}, t_n, z_n, x;\xi'_n)} \Big) \\
		& \qquad \quad \times \prod_{i=1}^{n-1} \Big( \sum_{\substack{t_i \in \eta_i^{u_{i-1}}, \\ w_i \in \eta_i}} \mathds 1_{F^{(1)}(u_{i-1}, t_i,w_i, z_i, u_i, z_{i+1}; (\xi'_i, \xi'_{i+1}))}
		 				+ \sum_{\substack{w_i \in \eta_i^{u_{i-1}}, \\ t_i \in \eta_i^{w_i,u_i}}} \mathds 1_{F^{(2)}(u_{i-1}, t_i,w_i, z_i, u_i, z_{i+1};(\xi'_i, \xi'_{i+1}))} \Big). }
\end{lemma}

\begin{proof}
We first prove the following assertion:
	\algn{ \mathds 1_{E (u_{n-1}, x; \C_{n-1}, \xi'_n)} \leq \sum_{z_n \in \eta_n^{x}} \sum_{t_n \in \eta_n^{u_{n-1}, x}}  \mathds 1_{F_n(u_{n-1}, t_n, z_n, x; \xi'_n)} 
							\mathds 1_{\{\sqconn{u_{n-2}}{z_n}{(\xi_{n-1}^{u_{n-2}}, \xi'_n)}\}}. \label{eq:DB:bounds_E_by_F_N_step}}
Recall the definition of $E (u_{n-1}, x; \C_{n-1}, \xi'_n)$ in~\eqref{eq:LE:def:E_event} and, specifically, recall Definition~\ref{def:LE:connection_events}(2). The \col{event in the left-hand side of}~\eqref{eq:DB:bounds_E_by_F_N_step} is contained in the event that $u_{n-1}$ is connected to $x$, but this connection breaks down after a $\C_{n-1}$-thinning of $\eta_n^{x}$. We distinguish two cases under which this can happen:

\underline{Case (a):} In this case, $x$ is thinned out and $E (u_{n-1}, x; \C_{n-1}, \xi'_n)$ \col{is contained in}
	\[\{ \conn{u_{n-1}}{x}{\xi'_n} \} \cap \{\sqconn{u_{n-2}}{x}{(\xi_{n-1}^{u_{n-2}}, \xi'_n)}\}=F_n(u_{n-1}, x,x,x;\xi'_n) \cap \{\sqconn{u_{n-2}}{x}{(\xi_{n-1}^{u_{n-2}}, \xi'_n)}\} .\]

\underline{Case (b):} In this case, $x$ is not thinned out. Now, the occurrence of $E$ \col{implies that there is} at least one interior point on the path between $u_{n-1}$ and $x$ that is thinned out in a $\C_{n-1}$-thinning. We claim that we can pick one such point to be $z_n$ and satisfy the bound in~\eqref{eq:DB:bounds_E_by_F_N_step}.

Let $t_n$ be the last pivotal point in $\piv{u_{n-1}, x; \xi'_n}$ (again, we use that $\piv{u_{n-1}, x; \xi'_n}$ can be ordered in the direction from $u_{n-1}$ to $x$) and set $t_n=u_{n-1}$ when $\piv{u_{n-1}, x; \xi'_n}=\varnothing$. By definition of $t_n$ as last pivotal point, we have $\{ \dconn{t_n}{x}{\xi_n^{x}}\}$.

Moreover, the second part in the definition of the event $E (u_{n-1}, x; \C_{n-1}, \xi'_n)$ (recall~\eqref{eq:LE:def:E_event}) forces all of these paths from $t_n$ to $x$ to break down after a $\C_{n-1}$-thinning, while $t_n$ itself cannot be thinned out. Hence, there is a thinned-out point on each path between $t_n$ and $x$. We can pick any of them to be the point $z_n$ in the $F_n$ event in~\eqref{eq:DB:bounds_E_by_F_N_step}. This proves~\eqref{eq:DB:bounds_E_by_F_N_step}.

Abbreviating $\vec v_i=(u_{i-1}, t_i,w_i, z_i, u_i, z_{i+1})$ for $i \in [n-1]$, we further assert that, for $z_{i+1} \in \eta^{u_{i+1}}_{i+1}$,
	\algn{ & \mathds 1_{E (u_{i-1}, u_i; \C_{i-1}, \xi'_i)} \mathds 1_{\{\sqconn{u_{i-1}}{z_{i+1}}{(\xi_i^{u_{i-1}}, \xi'_{i+1})}\}} \notag\\
		  & \qquad \leq \sum_{z_i \in \eta_i^{u_i}} \mathds 1_{\{\sqconn{u_{i-2}}{z_i}{(\xi_{i-1}^{u_{i-2}}, \xi'_i)}\}}
					\Big( \sum_{\substack{t_i \in \eta_i^{u_{i-1}}, \\ w_i \in \eta_i}} \mathds 1_{F^{(1)}(\vec v_i; (\xi'_i, \xi'_{i+1}))}
					+ \sum_{\substack{w_i \in \eta_i^{u_{i-1}}, \\ t_i \in \eta_i^{w_i,u_i}}} \mathds 1_{F^{(2)}(\vec v_i; (\xi'_i, \xi'_{i+1}))} \Big). \label{eq:DB:bounds_E_by_F_inductive_step}}
Recall how the thinning events in Definition~\ref{def:LE:thinning_events} were introduced via the mappings $(\pi_j)$ from~\eqref{eq:orderingofPPpoints}. If $z_{i+1}=\pi_j(\eta_{i+1}^{u_{i+1}})$ and the second event on the left-hand side of~\eqref{eq:DB:bounds_E_by_F_inductive_step} occurs, then there must be a point $\pi_l(\eta_i^{u_{i-1}})\in\C(u_{i-1}, \xi_i^{u_{i-1}})$ such that $Y_{j,l} \leq \connf(\pi_l(\eta_i^{u_{i-1}})-\pi_j(\eta_{i+1}^{u_{i+1}}))$, where $(Y_{j,l})_{l\in\N}$ are the thinning variables associated to $z_{i+1}$. Informally speaking, $\pi_l(\eta_i^{u_{i-1}})$ is responsible for thinning out $z_{i+1}$. Let $\gamma$ denote a path in $\xi_i^{u_{i-1}}$ from $u_{i-1}$ to $\pi_l(\eta_i^{u_{i-1}})$.

Turning to the event $E$, again we start by considering the case $u_i \notin (\eta_i)_{\thinn{\C_{i-1}}}$, i.e.~the case where $u_i$ is thinned out. In this case, the event $E$ \col{implies}
	\[ \{\conn{u_{i-1}}{u_i}{\xi'_i} \} \cap \{\sqconn{u_{i-2}}{u_i}{(\xi_{i-1}^{u_{i-2}}, \xi'_i)}\}. \]
Letting $w_i$ be the last point $\gamma$ shares with the path from $u_{i-1}$ to $u_i$ (where $w_i=u_{i-1}$ is possible), we obtain $F^{(2)}(\vec v_i;(\xi'_i, \xi'_{i+1}))$ for $t_i=z_i=u_i$.

In the case where $u_i \in (\eta_i)_{\thinn{\C_{i-1}}}$, we set $t_i$ to be last pivotal point for the connection between $u_{i-1}$ and $u_i$ (if there is no pivotal point, set $t_i=u_{i-1}$). By definition of $E$, there is a path $\tilde \gamma$ between $u_{i-1}$ and $t_i$ (possibly of length $0$) and there must be two disjoint $t_i$-$u_i$-paths in $\xi^{u_i}$ (call them $\gamma'$ and $\gamma''$), both of length at least two and both containing an interior point that is thinned out.

Let $w_i$ be the last point $\gamma$ shares with $\tilde \gamma \cup \gamma' \cup \gamma''$. If $w_i$ lies on $\tilde \gamma$, we pick a thinned-out point on $\gamma'$ and call it $z_i$ to obtain the event $F^{(2)}$. If $w_i$ lies on $\gamma'$ or $\gamma''$, we pick a thinned-out point from the respective other path ($\gamma''$ or $\gamma'$), call it $z_i$, and obtain the event $F^{(1)}$. This proves~\eqref{eq:DB:bounds_E_by_F_inductive_step}.

We can now recursively bound the events in Lemma~\ref{lem:DB:bounds_E_by_F} (from $n$ to $1$), and, setting $\vec v_0=(\orig,w_0, u_0, z_1;\linebreak (\xi'_0, \xi_1'))$, it remains to prove that
	\eqq{ \mathds 1_{\{\dconn{\orig}{u_0}{\xi'_0}\}} \mathds 1_{\{\sqconn{\orig}{z_1}{(\xi_0^{\orig}, \xi'_1)} \}}
			 \leq \sum_{w_0 \in \eta^{\orig}_0} \mathds 1_{F_0 (\vec v_0)}. \label{eq:DB:bounds_E_by_F_base}}
Again, there must be a point $\pi_l(\eta_0^\orig) \in \C_0$ that is responsible for thinning $z_1$ out. Let $\gamma$ be a path in $\xi_0^\orig$ from $\orig$ to $\pi_l(\eta_0^\orig)$.

Moreover, we can partition the event $\{\dconn{\orig}{u_0}{\xi^{\orig,u_0}}\}_0$ as follows: When $\orig \nsim u_0$, there are two disjoint paths ($\gamma'$ and $\gamma''$ say) from $\orig$ to $u_0$, both of length at least $2$. On the other hand, when $\orig \sim u_0$, we consider $\gamma'$ and $\gamma''$ to be the degenerate paths containing only the origin $\orig$.

Let $w_0$ be the last vertex $\gamma$ shares with $\gamma'$ or $\gamma''$ (thus, $w_0=\orig$ is possible). Requiring the three paths $\gamma', \gamma'', \gamma$ to be present, and respecting the two cases of the double connection between $\orig$ and $u_0$, results precisely in $F_0(\vec v_0)$, proving~\eqref{eq:DB:bounds_E_by_F_base} and therefore Lemma~\ref{lem:DB:bounds_E_by_F}.\end{proof}

\begin{proof}[Proof of Proposition~\ref{thm:DB:Pi_bound_Psi}]
We use Lemma~\ref{lem:DB:bounds_E_by_F} to give a bound on $\Pi_\lambda^{(n)}(x)$. It involves sums over random points on each of the $n+1$ configurations $\xi_0, \ldots, \xi_n$. In the following, we intend to apply the Mecke formula to deal with these sums. In particular, we use the Mecke formula~\eqref{eq:prelim:mecke_1} on $\xi_0$ and the multivariate Mecke formula~\eqref{eq:prelim:mecke_n} on $\xi_i$ for $1 \leq i \leq n$. The $F$ events in the indicators imply that some of the ``extra point'' coincidences vanish (for example, under $z_n=x$, application of the Mecke formula for the sum over $t_n$ produces a term where $t_n\neq x = z_n$ a.s., but this term vanishes due to the restrictions in $F_n$). Taking this into consideration, recalling the definition of $\Pi_\lambda^{(n)}$ in~\eqref{eq:LE:PiN_def}, and abbreviating $\vec v_i=(u_{i-1}, t_i,w_i, z_i, u_i, z_{i+1})$, gives
	\algn{\Pi_\lambda^{(n)}(x) & \leq \lambda^n \int \E_\lambda \bigg[ \sum_{z_n\in \eta_n^x} \sum_{t_n \in \eta_n^{u_{n-1},x}} \mathds 1_{F_n(u_{n-1}, t_n, z_n, x;\xi'_n)} \notag\\
		& \qquad \times \prod_{i=1}^{n-1} \sum_{z_i \in \eta_i^{u_i}} \Big( \sum_{t_i \in \eta_i^{u_{i-1}}} \sum_{w_i \in\eta_i} \mathds 1_{F^{(1)}(\vec v_i;(\xi'_i, \xi'_{i+1}))} 
				 + \sum_{w_i \in \eta_i^{u_{i-1}}} \sum_{t_i \in \eta_i^{w_i,u_i}} \mathds 1_{F^{(2)}(\vec v_i;(\xi'_i, \xi'_{i+1}))} \Big) \notag\\
		& \qquad \times \sum_{w_0 \in \eta_0^\orig} \mathds 1_{F_0 (\orig,w_0, u_0, z_1;(\xi'_0;\xi'_1))} \bigg] \dd \vec u_{[0,n-1]} \notag\\
		&= \lambda^n \int \E_\lambda\bigg[ (\lambda + \delta_{w_0,\orig}) \mathds 1_{F_0 (\orig,w_0,u_0, z_1;(\xi''_0;\xi''_1))}
					\prod_{i=1}^{n-1} \Big( \lambda^2 (\lambda+\delta_{t_i, u_{i-1}})  \mathds 1_{F^{(1)}(\vec v_i;(\xi_i'', \xi''_{i+1}))}  \notag \\
		& \qquad \qquad + (\lambda(\lambda+\delta_{t_i,w_i}) + \delta_{z_i,u_i} \delta_{t_i, u_i}) (\lambda+\delta_{w_i,u_{i-1}}) \mathds 1_{F^{(2)}(\vec v_i;(\xi''_i, \xi''_{i+1}))} \Big) \notag \\
		& \qquad\quad \times (\lambda (\lambda+\delta_{t_n,u_{n-1}})+\delta_{z_n,x}\delta_{t_n,x}) \mathds 1_{ F_n(u_{n-1}, t_n, z_n, x;\xi''_n)} \bigg]
						\dd \big((\vec u, \vec w)_{[0,n-1]}, (\vec z, \vec t)_{[1,n]} \big), \label{eq:DB:F_bound_mecke}}
where we set $\xi''_0:= \xi_0^{\orig, w_0, u_0}, \xi''_i:= \xi_i^{u_{i-1}, t_i, w_i, z_i, u_i}$, and $\xi''_n := \xi_n^{u_{n-1}, t_n, z_n, x}$.

To simplify~\eqref{eq:DB:F_bound_mecke}, we exploit the independence of the $\xi_i$. Note that for every $i$, there are four events that depend on $\xi''_i$, namely $F^{(j)}(\vec v_i;(\xi_i'', \xi''_{i+1}))$ and $F^{(j)}(\vec v_{i-1};(\xi_{i-1}'', \xi''_{i}))$ (for $j=1,2$, respectively). The latter two are thinning events that depend on $\xi''_i$ only through $Y(z_i)$, the thinning mark associated to the deterministic point $z_i$ (see the remark after the definition of the thinning connection in Definition~\ref{def:DB:extended_thinning_events}). The former two are connection events in $\xi''_i$, and they are independent of $Y(z_i)$. As a consequence, the expectation on the right-hand side of~\eqref{eq:DB:F_bound_mecke} factorizes, and so
	\algn{\Pi_\lambda^{(n)}(x) & \leq \lambda^n \int \E_\lambda \Big[ (\lambda + \delta_{w_0,\orig}) \mathds 1_{F_0 (\orig,w_0,u_0, z_1;(\xi''_0;\xi''_1))} \Big]
					 \prod_{i=1}^{n-1} \E_\lambda \Big[ \lambda^2 (\lambda+\delta_{t_i, u_{i-1}})  \mathds 1_{F^{(1)}(\vec v_i;(\xi_i'', \xi''_{i+1}))} \notag \\
		& \qquad \qquad + \big(\lambda(\lambda+\delta_{w_i,t_i}) + \delta_{z_i,u_i} \delta_{t_i, u_i}\big)(\lambda+\delta_{w_i,u_{i-1}}) \mathds 1_{F^{(2)}(\vec v_i;(\xi''_i, \xi''_{i+1}))} \Big] \notag \\
		& \qquad \quad \times \E_\lambda \Big[ \big(\lambda(\lambda+\delta_{t_n,u_{n-1}})+\delta_{z_n,x}\delta_{t_n,x}\big) \mathds 1_{ F_n(u_{n-1}, t_n, z_n, x;\xi''_n)} \Big]
						\dd \big((\vec u, \vec w)_{[0,n-1]}, (\vec z, \vec t)_{[1,n]} \big) \notag\\
		& = \lambda^n \int (\lambda + \delta_{w_0,\orig}) \pla \big(  F_0 (\orig,w_0,u_0, z_1;(\xi''_0;\xi''_1)) \big)
					 \prod_{i=1}^{n-1} \Big[ \lambda^2 (\lambda+\delta_{t_i, u_{i-1}}) \pla \big( F^{(1)}(\vec v_i;(\xi_i'', \xi''_{i+1}))) \notag \\
		& \qquad \qquad + \big(\lambda(\lambda+\delta_{w_i,t_i}) + \delta_{z_i,u_i} \delta_{t_i, u_i}\big)(\lambda+\delta_{w_i,u_{i-1}}) \pla\big( F^{(2)}(\vec v_i;(\xi''_i, \xi''_{i+1}))\big) \Big] \notag \\
		& \qquad \quad \times \big(\lambda(\lambda+\delta_{t_n,u_{n-1}})+\delta_{z_n,x}\delta_{t_n,x}\big) \pla \big(F_n(u_{n-1}, t_n, z_n, x;\xi''_n) \big)
						\dd \big((\vec u, \vec w)_{[0,n-1]}, (\vec z, \vec t)_{[1,n]} \big). \label{eq:DB:F_bound_independence} } 
The goal is to bound the appearing events using the BK inequality, and thus bound $\Pi_\lambda^{(n)}(x)$ in terms of so-called diagrams (integrals of large products of two-point functions that are conveniently organized). Figure~\ref{fig:Fdiagrams} already suggests how to decompose the probability of the respective events via the BK inequality. Note that, abbreviating $\vec v_n=(u_{n-1}, t_n,z_n,x)$, we can directly use the BK inequality to obtain
	\algn{ & \int (\lambda(\lambda+\delta_{t_n,u_{n-1}})+\delta_{z_n,x}\delta_{t_n,x}) \pla(F_n(\vec v_n;\xi''_n)) \dd t_n \notag\\
		 \leq & \int \underbrace{\lambda \tlamo(t_n-u_{n-1}) \triangle (t_n, z_n,x)}_{=:\phi_n^{(1)}(\vec v_n)}
		 			 + \underbrace{\delta_{t_n, z_n}\delta_{z_n,x} \tlam(x-u_{n-1})}_{=:\phi_n^{(2)}(\vec v_n)} \dd t_n. \label{eq:DB:F_bound_N_BK}}
To decompose the other factors, we first make use of Lemma~\ref{lem:DB:squigarrow_tlam_equality} to deal with the $\bigcirc_m^\sqarrow$ event, and then proceed as in~\eqref{eq:DB:F_bound_N_BK} by applying the BK inequality. For $1 \leq i < n$, we recall that $\vec v_i = (u_{i-1}, t_i, w_i, z_i, u_i, z_{i+1})$ and bound, using Lemma~\ref{lem:DB:squigarrow_tlam_equality},
	\algn{ & \int \lambda^2 (\lambda+\delta_{t_i, u_{i-1}}) \pla\big(F^{(1)}(\vec v_i;(\xi_i'', \xi''_{i+1}))\big)  \dd (t_i, w_i)  \notag\\
		= & \int \lambda^2 (\lambda+\delta_{t_i, u_{i-1}}) \pla \big( \bigcirc_6^\leftrightarrow \big( (u_{i-1},t_i),(t_i,z_i),(z_i,u_i),(t_i,w_i),(w_i,u_i),(w_i,z_{i+1});\xi''_i \big) \big) \dd (t_i, w_i) \notag\\
		\leq & \int \underbrace{\lambda^2 \parl(t_i-u_{i-1}, z_{i+1}-w_i) \Box(t_i, w_i, u_i, z_i)}_{=:\phi^{(1)}(\vec v_i)} \dd (t_i, w_i), \label{eq:DB:F_bound_phi1}}
as well as
	\algn{ & \int (\lambda(\lambda+\delta_{w_i,t_i}) + \delta_{z_i,u_i} \delta_{t_i, u_i})(\lambda+\delta_{w_i,u_{i-1}}) \pla\big( F^{(2)}(\vec v_i;(\xi''_i, \xi''_{i+1}))\big)\big)  \dd (t_i,w_i)  \notag\\
		= & \int \Big[ \lambda(\lambda+\delta_{w_i,t_i})(\lambda+\delta_{w_i,u_{i-1}})
							\pla\big( \bigcirc_6^\leftrightarrow \big( (u_{i-1},w_i),(w_i,t_i),(t_i,u_i),(t_i,z_i),(z_i,u_i),(w_i,z_{i+1});\xi''_i \big) \notag\\
		& \qquad + \delta_{z_i,u_i} \delta_{t_i, u_i}(\lambda+\delta_{w_i,u_{i-1}}) \pla \big( \bigcirc_3^\leftrightarrow \big( (u_{i-1},w_i),(w_i,u_i),(w_i,z_{i+1});\xi''_i \big)\big) \Big] \dd (t_i, w_i) \notag\\
		 \leq & \int \Big[ \underbrace{\lambda \triangle(t_i,z_i,u_i) \tlamo(t_i-w_i) \parl(w_i-u_{i-1}, z_{i+1}-w_i) }_{=:\phi^{(2)}(\vec v_i)} \notag\\
		& \qquad + \underbrace{\delta_{z_i,u_i} \delta_{t_i,u_i} \tlam(t_i-w_i) \parl(w_i-u_{i-1}, z_{i+1}-w_i) }_{=:\phi^{(3)}(\vec v_i)} \Big] \dd(t_i,w_i). \label{eq:DB:F_bound_phi23}}
Analogously, with $\vec v_0= (\orig,w_0,u_0,z_1)$,
	\algn{ & \int (\lambda + \delta_{w_0,\orig}) \pla \big( F_0 (\orig,w_0,u_0, z_1;(\xi''_0,\xi''_1)) \big) \dd w_0 \notag\\
		=& \int (\lambda + \delta_{w_0,\orig}) \pla\big(\{\orig \nsim u_0 \text{ in } \xi_0''\}\cap \bigcirc_4^\leftrightarrow \big( (\orig, u_0), (\orig,w_0),(u_0,w_0),(w_0,z_{1});\xi''_0 \big)\big)
							+ \delta_{w_0,\orig} \connf(u_0) \tlam(z_1) \dd w_0  \notag\\
		\leq & \int \underbrace{\lambda \triangle(\orig,w_0,u_0) \tlam(z_1-w_0)}_{=: \phi_0^{(1)}(\vec v_0)}
							+ \delta_{w_0,\orig} (\tlam(u_0)-\connf(u_0)) \tlam(u_0) \tlam(z_1) + \underbrace{\delta_{w_0,\orig} \connf(u_0) \tlam(z_1)}_{=: \phi_0^{(3)}(\vec v_0)} \dd w_0 \notag\\
		\leq & \int \Big[ \phi_0^{(1)}(\vec v_0) + \underbrace{\delta_{w_0,\orig} \int \col{\triangle(\orig,t_0,u_0)} \dd t_0 \tlam(z_1)}_{=:\phi_0^{(2)}(\vec v_0)}
							+ \phi_0^{(3)}(\vec v_0) \Big] \dd w_0. \label{eq:DB_F_bound_phi0}}
Writing $\phi_0 := \sum_{j=1}^{3}\phi_0^{(j)}, \phi:= \sum_{j=1}^{3} \phi^{(j)}$, and $\phi_n := \sum_{j=1}^{2} \phi_n^{(j)}$, we can substitute these new bounds into~\eqref{eq:DB:F_bound_independence} and obtain
	\eqq{ \Pi_\lambda^{(n)}(x) \leq \lambda^n \int \phi_0(\vec v_0) \Big( \prod_{i=1}^{n-1} \phi(\vec v_i) \Big) \phi_n(\vec v_n) 
							\dd \big((\vec u, \vec w)_{[0,n-1]}, (\vec z, \vec t)_{[1,n]} \big). \label{eq:DB:F_bound_phi} }
The proof is completed with the observation that
	\[ \tlam(z_i-w_{i-1}) \phi(u_{i-1},t_i,w_i,z_i,u_i,z_{i+1}) = \tlam(z_{i+1}-w_i) \psi(w_{i-1},u_{i-1},t_i,w_i,z_i,u_i), \]
as well as $\phi_0(\orig,w_0,u_0,z_1)=\tlam(z_1-w_0) \psi_0(w_0,u_0)$ and $\tlam(z_n-w_{n-1}) \phi_n(u_{n-1},t_n,z_n,x) = \linebreak \psi_n(w_{n-1},u_{n-1},t_n,z_n,x)$ and a telescoping product identity.
\end{proof}

Proposition~\ref{thm:DB:Pi_bound_Psi} gives a bound on $\Pi_\lambda^{(n)}$ in terms of a diagram, which is, to be more accurate, itself a sum of $2 \cdot 3^{n}$ diagrams, i.e.,
	\eqq{ \Pi_\lambda^{(n)}(x) \leq \lambda^n \sum_{\vec j_{[0,n]}} \int \psi_0^{(j_0)} \left( \prod_{i=1}^{n-1} \psi^{(j_i)}\right) \psi_n^{(j_n)} 
								\dd\big( (\vec w, \vec u)_{[0,n-1]}, (\vec t, \vec z)_{[1,n]} \big), \label{eq:DB:piboundpsi_vectorj}}
where the sum is over all vectors $\vec j$ with $1\leq j_i \leq 3$ for $0 \leq i < n$ and $j_n \in\{1,2\}$, and where the arguments of the $\psi$ functions were omitted. If we were to expand every factor of $\tlamo$ (regarding it as a sum of two terms), then $\psi^{(1)}$ and $\psi^{(3)}$ turn into a sum of two terms each, whereas $\psi^{(2)}$ turns into a sum of four terms (similarly, $\psi_n^{(1)}$ turns into a sum of two terms). In that sense, there are eight types of interior segments. We point to Figure~\ref{fig:Psidiagrams} for an illustration of the $\psi$ functions and these eight types.

In Section~\ref{sec:diagrammaticbounds_no_disp}, we want to give an inductive bound on $\int \Pi_\lambda^{(n)}(x) \dd x$. To this end, define the function $\Psi^{(n)}$, which is almost identical to the bound obtained by Proposition~\ref{thm:DB:Pi_bound_Psi}, but better suited for the induction performed in Section~\ref{sec:diagrammaticbounds_no_disp}. Define
	\eqq{ \Psi^{(n)} (w_n, u_n) := \int \psi_0 (w_0,u_0) \prod_{i=1}^{n} \psi (w_{i-1}, u_{i-1}, t_i, w_i, z_i, u_i) \dd\big( (\vec w, \vec u)_{[0,n-1]}, (\vec t, \vec z)_{[1,n]} \big). \label{def:Psi_N}}
Note that $\Psi^{(n)}$ is similar to the bound in Proposition~\ref{thm:DB:Pi_bound_Psi}, but with $\psi_n$ replaced by $\psi$. Since
	\[ 	\psi_n(w_{n-1},u_{n-1}, t_n, z_n, x) \leq \sum_{j \in \{2,3\}} \int \psi^{(j)} (w_{n-1},u_{n-1}, t_n, w_n, z_n,x) \dd w_n , \]
we arrive at the new bound
	\eqq{ \int \Pi_\lambda^{(n)}(x) \dd x \leq \lambda^n \int \Psi^{(n-1)}(w,u) \psi_n(w,u,t,z,x) \dd(w,u,t,z,x) \leq \lambda^{n} \iint \Psi^{(n)}(w,u) \dd w \dd u. \label{eq:Pi_Psi_Diag_bound}}
In the following two sections, we heavily rely on the bound obtained in~\eqref{eq:Pi_Psi_Diag_bound}.
\chr{Before that, we show how a slight adaptation of the proof of Proposition~\ref{thm:DB:Pi_bound_Psi} proves Lemma~\ref{lem:convergence_of_lace_expansion_corollary_lambda_c}:}

\begin{proof}[Proof of Lemma \ref{lem:convergence_of_lace_expansion_corollary_lambda_c}]
We can define $\Pi_{\lambda_c}^{(n)}$ for $n\in\N_0$ by extending the definitions in~\eqref{eq:LE:Pi0_def} and~\eqref{eq:LE:Pin_def} to $\lambda_c$. We claim that $\Pi_\lambda^{(n)} \to \Pi_{\lambda_c}^{(n)}$ for every $n\in\N_0$ as $\lambda \nearrow \lambda_c$. 

In order to prove this, let $(\lambda_m)_{m\in\N}$ be an increasing sequence with $\lambda_m \nearrow \lambda_c$. We write~\eqref{eq:LE:Pi0_def} and~\eqref{eq:LE:Pin_def} as $\Pi_\lambda^{(n)}(x) = \lambda^n \int \pla(A^{(n)}) \dd \vec u$. Define
	\[ h_\lambda^{(n)}(x, \vec u) := \lambda^n \int \psi_0(w_0,u_0) \Big( \prod_{i=1}^{n-1} \psi(\vec v_i) \Big)
						 \psi_n(w_{n-1},u_{n-1}, t_n, z_n,x) \dd\big( \vec w_{[0,n-1]}, (\vec t, \vec z)_{[1,n]} \big) \]
with $\vec v_i = (w_{i-1}, u_{i-1}, t_i, w_i, z_i, u_i)$. By Proposition~\ref{thm:DB:Pi_bound_Psi}, we have $\Pi_\lambda^{(n)}(x) \leq \int h_\lambda^{(n)}(x, \vec u) \dd\vec u$. Moreover, Corollary~\ref{cor:DB:Pi_x_bounds} and \col{Corollary~\ref{thm:convergenceoflaceexpansion}} show that
	\eqq{ \col{\int h_\lambda^{(n)}(x,\vec u) \dd\vec u \leq C (C \beta)^{n-1} \label{eq:MT:dominating_h_function}}}
for all $\lambda<\lambda_c$, and $C$ is independent of $\lambda$. As $\lambda\mapsto h_\lambda^{(n)}$ is increasing, $h_{\lambda_c}^{(n)}$ is also integrable. A close inspection of the proof of Proposition~\ref{thm:DB:Pi_bound_Psi} shows that we may follow the steps in the proof with $x$ and $\vec u$ as fixed arguments (i.e.,~we do not integrate over the points $\vec u$) to get $\lambda^n\pla(A^{(n)}) \leq h_{\lambda_c}^{(n)}$ for all $\lambda \leq \lambda_c$. Hence, by dominated convergence, it suffices to show $|\p_{\lambda_c}(A^{(n)}) - \p_{\lambda_m}(A^{(n)})| \to 0$ as $m\to\infty$.

Recall that the event $A^{(n)}$ takes place on $n+1$ independent RCMs. For $0 \leq i \leq n$, let $(\eta_{i,m})_{m\in\N}$ and $(\tilde\eta_{i,m})_{m\in\N}$ be sequences of PPPs of intensities $\lambda_m, \tilde \lambda_m:= \lambda_c-\lambda_m$; and for fixed $m$, let those $2n+2$ PPPs be independent. Hence, the superposition of the two PPPs $\eta_{i,m}$ and $\tilde\eta_{i,m}$ forms a PPP of intensity $\lambda_c$. We can moreover couple the marks determining edge connections, so that $\tilde\xi_{i,m}' = \xi_i^{u_{i-1},u_i} (\eta_{i,m}')$ (where $\eta_{i,m}'= \eta_{i,m}^{u_{i-1}, u_i}$) forms an RCM of intensity $\lambda_m$ and $\xi_{i,m}' = \xi_i^{u_{i-1},u_i} (\eta_{i,m}' + \tilde\eta_{i,m})$ forms an RCM of intensity $\lambda_c$ with the further property that $\tilde\xi_{i,m}' = \xi_{i,m}'[\eta_{i,m}']$ (recall that $\xi(\eta)$ is the RCM with underlying PPP $\eta$). We now write
	\algn{\big|\p_{\lambda_c}(A^{(n)}) - \p_{\lambda_m}(A^{(n)})\big| &=  \Big| \E \Big[ \mathds 1_{\{\dconn{\orig}{u_0}{\xi_{0,m}'}\}} \prod_{i=1}^{n} \mathds 1_{E(u_{i-1},u_i;\xi_{i,m}')} \notag \\
			& \hspace{2cm}- \mathds 1_{\{\dconn{\orig}{u_0}{\tilde\xi_{0,m}'}\}} \prod_{i=1}^{n} \mathds 1_{E(u_{i-1},u_i;\tilde\xi_{i,m}')} \Big] \Big| , 
				 \label{eq:MT:Pilam_crit_conv_first_bound}}
where $\E$ denotes the joint probability measure. In order to bound~\eqref{eq:MT:Pilam_crit_conv_first_bound}, we define $\C_{i,m}=\C(u_{i-1}, \linebreak \xi_i^{u_{i-1}}(\eta_{i,m}^{u_{i-1}} + \tilde\eta_{i,m}))$ as well as $\tilde\C_{i,m} = \C(u_{i-1}, \xi_i^{u_{i-1}}(\eta_{i,m}^{u_{i-1}}))$ for $0 \leq i \leq n$ and claim that
	\eqq{ \p (\C_{i,m} \neq \tilde\C_{i,m}) \xrightarrow{m\to\infty} 0.  \label{eq:MT:cluster_crit_conv}}
Note that $\theta(\lambda_c)=0$ has been proven at the beginning of Section \ref{sec:maintheorems}, so we may use that $\C_{i,m}$ is finite almost surely. Next, let $(\tilde \Lambda_m)_{m\in\N}$ be an increasing sequence of subsets of $\Rd$ exhausting $\Rd$ and satisfying $\tilde \lambda_m \text{Leb}(\tilde\Lambda_m)  \to 0$ as $m\to\infty$. Then 
	\al{ \p (\C_{i,m} \neq \tilde\C_{i,m}) &\leq \p\big( \tilde\eta_{i,m} \cap \tilde\Lambda_m \neq \varnothing \big) + \p_{\lambda_c} (\conn{\orig}{\tilde\Lambda_m^c}{\xi^\orig_{i,m}}) \\
		&= 1-\e^{-\tilde \lambda_m \text{Leb}(\tilde\Lambda_m) } + \p_{\lambda_c} (\conn{\orig}{\tilde\Lambda_m^c}{\xi^\orig_{i,m}}) \xrightarrow{m\to\infty} 0, }
proving~\eqref{eq:MT:cluster_crit_conv}. We now proceed to bound~\eqref{eq:MT:Pilam_crit_conv_first_bound} by observing that conditionally on $\C_{i,m}=\tilde\C_{i,m}$ for all $0 \leq i \leq n$, the event $A^{(n)}$ occurs in $(\xi'_{i,m})_{i=0}^n$ if and only if it occurs in $(\tilde\xi'_{i,m})_{i=0}^n$. Consequently,
	\[ \big|\p_{\lambda_c}(A^{(n)}) - \p_{\lambda_m}(A^{(n)})\big| \leq \p\big(\exists \; i\in \{0,\ldots, n\}: \C_{i,m} \neq \tilde\C_{i,m} \big) \leq (n+1) \p(\C_{0,m} \neq \tilde\C_{0,m}) \xrightarrow{m\to\infty} 0. \]
This proves that $|\p_{\lambda_c}(A^{(n)}) - \p_{\lambda_m}(A^{(n)})| \to 0$ as $m \to\infty$.

\col{The functions $\int h_\lambda^{(n)}(\cdot ,\vec u) \dd \vec u$ serve as dominating functions for $\Pi_\lambda^{(n)}(\cdot)$, and they are summable in $n$  by~\eqref{eq:MT:dominating_h_function} for sufficiently small $C\beta$. Hence, $\sum_{n \geq 0} (-1)^n \Pi_{\lambda_c}^{(n)} =: \Pi_{\lambda_c}$ converges absolutely and satisfies $\Pi_{\lambda_c} = \lim_{\lambda\nearrow \lambda_c} \Pi_{\lambda}$ by dominated convergence.}

Moreover, $\Pi_{\lambda_c}$ is integrable by the uniform bounds in~\eqref{eq:BA:convLE_intPi_bds} \col{of Corollary~\ref{cor:convergenceoflaceexpansion_uniform}}. Hence, we can take the limit $\lambda \nearrow \lambda_c$ in~\eqref{eq:LE_identity_OZE}, extending it to $\lambda_c$ with $\Pi_{\lambda_c}$ as defined above.
\end{proof}

\subsection{Diagrammatic bounds on the lace-expansion coefficients} 
\label{sec:diagrammaticbounds_no_disp}

\ch{Having obtained the bound~\eqref{eq:Pi_Psi_Diag_bound} is a good start, but this bound is still a highly involved integral. The aim of this section is to decompose $\Psi^{(n)}$ into much simpler objects, namely triangles $\trilam$ and the related quantities as introduced in Definition~\ref{def:bubble} in Section \ref{sec-statement-bounds-LA-coefficients}. The central result of this section is the following proposition:}
\begin{prop}[Bounds on $\Psi^{(n)}$ for general $n$] \label{thm:Psi_Diag_bound}
Let $n \geq 0$. Then
	\[ \lambda^{n+1} \iint \Psi^{(n)}(w,u) \dd w \dd u \leq 2(2\trilamo+\lambda+1) \big(U_\lambda \wedge \bar U_\lambda \big)^{n}.\]
\end{prop}

\noindent
\ch{{\it Proof of Proposition \ref{prop-bounds-LA-coefficients-unweighted}.} Combined with~\eqref{eq:Pi_Psi_Diag_bound}, Proposition \ref{thm:Psi_Diag_bound} immediately proves Proposition \ref{prop-bounds-LA-coefficients-unweighted}.}
\qed
\medskip

Before working towards the proof of Proposition~\ref{thm:Psi_Diag_bound}, let us motivate it. Similarly to discrete percolation, we want to bound $\int \Pi_\lambda^{(n)}(x) \dd x$ in terms of $\trilam$ and $\trilamo$. In turn, we hope to prove that the latter two quantities become small as $\beta$ becomes small. To see our motivation for introducing the $\varepsilon$-triangle, consider for a moment the Poisson blob model, $\connf= \mathds 1_{\unitball}$, as a representative of the finite-variance model (H1). We have no hope here of $\trilamo$ becoming small, as
	\[ \lambda^{-1}\trilamo(\orig) \geq (\tlam\star\tlam)(\orig) \geq (\connf\star\connf)(\orig) = 1.\]
This issue arises \emph{only} for the finite-variance model, and most prominently for the Poisson blob model---under (H2) and (H3), we later prove that $\trilamo$ is small whenever $\beta$ is small. However, as it turns out, we are able to prove that $\trilame$ becomes small (as $\beta$ becomes small) for some $\varepsilon$ (namely, the one assumed to exist under assumption (H1.2)). On the other hand, it is clear that for any $\varepsilon$ smaller than the radius of the unit volume ball (i.e., $\varepsilon < r_d = \pi^{-1/2} \Gamma(\tfrac d2 + 1)^{1/d}$), we have $\ballep \xrightarrow{d\to\infty} 0$.

Proposition~\ref{thm:Psi_Diag_bound} implies a bound which avoids $\varepsilon$ completely \ch{(as $U_\lambda \wedge \bar U_\lambda\le U_\lambda$)}, and which is substantially easier to prove. This bound suffices for the connection functions of (H2) and (H3). Additionally, we have a bound containing $\trilame$ and $\ballep$, which is necessary for (H1). We prove Proposition~\ref{thm:Psi_Diag_bound} without specifying $\varepsilon$ (we do this later). However, $\varepsilon$ should be thought of as an arbitrary, but small enough, value (smaller than $r_d$ suffices). 

As a first step, we introduce some related quantities, which will be of help not only in the proof of Proposition~\ref{thm:Psi_Diag_bound}, but also in Section~\ref{sec:diagrammaticbounds_disp} below. We define
	\al{\breve\Psi^{(n)}(w_0,u_0, w_n,u_n) &= \sum_{\vec j_{[1,n]} \in [3]^n} \int \prod_{i=1}^{n} \psi^{(j_i)}(w_{i-1}, u_{i-1}, t_i, w_i, z_i, u_i)
					\dd\big( (\vec w, \vec u)_{[1,n-1]}, (\vec t, \vec z)_{[1,n]} \big), \\
		\breve\Psi^{(n, \geq \varepsilon)}(w_0,u_0, w_n,u_n) &=	\mathds 1_{\{|w_0-u_0| \geq \varepsilon\}} \breve\Psi^{(n)}(w_0,u_0, w_n, u_n), \\
		\breve\Psi^{(n, < \varepsilon)}(w_0,u_0, w_n,u_n) &= \mathds 1_{\{|w_n-u_n| < \varepsilon\}} \breve\Psi^{(n)}(w_0,u_0, w_n, u_n).}
The following lemma, providing some bounds on the quantities just introduced, will be at the heart of the proof of Proposition~\ref{thm:Psi_Diag_bound}:
\begin{lemma}[Bound on $\breve\Psi^{(n)}$ diagrams] \label{lem:DB:breve_psi_bounds}
Let $n\geq 1$ and $\varepsilon>0$. Then
	\al{ \sup_{a,b\in\Rd} \lambda^n \iint \breve\Psi^{(n)}(a,b,x,y) \dd x \dd y & \leq \min\Big\{ \big(U_\lambda\big)^n, U_\lambda \big(\bar U_\lambda\big)^{n-1} \Big\}, \\
		\max_{\bullet \in \{<\varepsilon, \geq \varepsilon\} }\sup_{a,b\in\Rd} \lambda^n \iint \breve\Psi^{(n, \bullet )} (a,b,x,y) \dd x \dd y & \leq \big(\bar U_\lambda\big)^{n-1} U_\lambda^{(\varepsilon)}, \\
		\sup_{a,b\in\Rd} \lambda^n \iint \mathds 1_{\{|a-b| \geq \varepsilon\}}\breve\Psi^{(n,<\varepsilon)}(a,b,x,y) \dd x \dd y & \leq \big(\bar U_\lambda\big)^{n-1} \big( U_\lambda^{(\varepsilon)} \big)^{2}. }
\end{lemma}
\begin{proof}
The proof is by induction on $n$. The induction hypothesis is that the three inequalities in Lemma~\ref{lem:DB:breve_psi_bounds} hold for $n-1$.

\underline{Base case, bound on $\breve\Psi^{(1)}$.} Let $n=1$. By translation invariance,
	\eqq{ \sup_{a,b} \lambda \int \psi^{(j)}(a,b,t,w,z,u) \dd(t,w,z,u) = \sup_{a} \lambda \int \psi^{(j)}(\orig,a,t,w,z,u) \dd(t,w,z,u) \label{eq:DB:breve_Psi_induction_base} }
for $j=1,2,3$. Starting with $j=1$, the integral on the right-hand side of~\eqref{eq:DB:breve_Psi_induction_base} is equal to
	\al{ \lambda^3 \iint \tlam(z) & \tlam(t-z) \tlamo(a-t) \Big( \iint \tlam(u-z) \tlam(w-u) \tlam(t-w) \dd u \dd w \Big) \dd z \dd t \\
			& = \lambda \iint \tlam(z) \tlam(t-z) \tlamo(a-t) \trilam(t-z) \dd z \dd t \\
			& \leq \trilam \trilamo(a) \leq \trilamoo \trilamo,}
as $\trilam \leq \trilamoo$. For $j=2$, we substitute $y'=y-u$ for $y \in\{t,w,z\}$, and we can bound~\eqref{eq:DB:breve_Psi_induction_base} by
	\al{ \lambda^2 & \int \tlam(z) \tlam(u-z) \tlam(t-u) \tlam(t-z) \tlamo(w-t) \tlamo(a-w) \dd (t,w,z,u) \\
			& = \lambda^2 \iint \tlam(z') \tlam(t') \tlam(t'-z') \Big( \iint \tlam(z'+u) \tlamo(a-w'-u) \tlamo(w'-t') \dd w' \dd u \Big) \dd z' \dd t' \\
			& = \lambda^2 \iint \tlam(z') \tlam(t') \tlam(t'-z') \trilamoo(a+z'-t') \dd z' \dd t' \\
			& \leq \trilamoo \trilam(\orig) \leq \trilamoo \trilamo.}
For $j=3$, the integral in~\eqref{eq:DB:breve_Psi_induction_base} is $\trilamo(a) \leq \trilamo \leq \trilamoo\trilamo$. In total, this gives the claimed bound for $n=1$.

We show how we represent bounds of the above form in pictorial format by repeating the above bounds for $j=1,2$. This is redundant at this point, but as the pictorial bounds are more accessible as well as more efficient, this will make later bounds easier to read. For $j=1$, letting $\vec v=(t,w,z,u)$, the above bound is executed pictorially as
	\[ \lambda \sup_{a} \int\psi^{(1)}(\orig,a,\vec v) \dd\vec v = \lambda^3 \sup_{{\color{darkorange}{\bullet}}} \int \mathrel{\raisebox{-0.25 cm}{\includegraphics{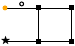}}}
			 \ \leq \lambda^3 \sup_{\textcolor{darkorange}{\bullet}} \Big( \int \mathrel{\raisebox{-0.25 cm}{\includegraphics{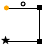}}}
			 	\Big( \sup_{\textcolor{lblue}{\bullet}, \textcolor{green}{\bullet}} \int \mathrel{\raisebox{-0.25 cm}{\includegraphics{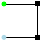}}} \Big)\Big) \leq \trilamo \trilam. \]
Let us explain the above line in more detail. As was the case in Figure~\ref{fig:Psidiagrams}, factors of $\tlam$ become lines, factors of $\tlamo$ become lines with a `$\circ$', points integrated over become black squares, and points over which we take the supremum become colored disks. The color indicates the location of the point in the diagram (however, using different colors is not essential). We note that the factor $\tlam(z)$ is interpreted as a line between $z$ and the origin $\orig$. We denote the origin either by `$\bigstar$' or by putting no symbol at all. To avoid cluttering the diagrams, we only use the `$\bigstar$' symbol for the origin to highlight changes in position due to substitutions.

To give the pictorial bound for $j=2$, we have to represent the performed substitution. Note that after the substitution, the variable $u$ appears in the two factors $\tlamo(w'+u-a)$ and $\tlam(z'+u)$. We interpret this as a line between $-u$ and $z'$ as well as a line between $a-u$ and $w'$. In this sense, the two lines do not meet in $u$, but they have endpoints that are a constant vector $a$ apart. We represent this as
	\[ \int \tlamo(\textcolor{blue}{w'}+u-\textcolor{altviolet}{a}) \tlam(\textcolor{darkgreen}{z'}+u) \dd u = \int \parl(\textcolor{blue}{w'}+u-\textcolor{altviolet}{a},\textcolor{darkgreen}{z'}+u) \dd u
					 = \int \mathrel{\raisebox{-0.25 cm}{\includegraphics{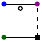}}}.\]
In other words, we represent the pair of points $u$ and $u-a$ with a dashed line. One endpoint of this dashed line will always be a square (representing $u$ in our case), the other a colored disk. With this notation, the pictorial bound for $j=2$ is
	\al{ \lambda \sup_{a} \int\psi^{(2)}(\orig,a,\vec v) \dd\vec v & = \lambda^3 \sup_{\textcolor{darkorange}{\bullet}} \int \mathrel{\raisebox{-0.25 cm}{\includegraphics{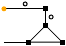}}}
			 \ = \lambda^3 \sup_{\textcolor{darkorange}{\bullet}} \int \mathrel{\raisebox{-0.25 cm}{\includegraphics{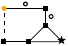}}} \\
			 & \leq \lambda^3 \int \Big( \Big( \sup_{\textcolor{lblue}{\bullet},\textcolor{green}{\bullet},\textcolor{darkorange}{\bullet}} 
			 	\int \mathrel{\raisebox{-0.25 cm}{\includegraphics{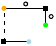}}} \Big)
			 	\mathrel{\raisebox{-0.25 cm}{\includegraphics{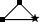}}} \Big) \leq \trilamoo \trilam. }
Note that the origin moved (from lower left to lower right) after the substitution.

\underline{Base case, bound on $\breve\Psi^{(1,\geq\varepsilon)}$.} To deal with $\breve\Psi^{(1,\geq \varepsilon)}$, we again have to bound the three types. For $j=1,2$, we can drop the indicator and recycle the bounds obtained on~\eqref{eq:DB:breve_Psi_induction_base}. For $j=3$, we observe that
	\[ \lambda \int \mathds 1_{\{|a| \geq \varepsilon\}} \psi^{(3)}(\orig,a,t,w,z,u) \dd(t,w,z,u) = \mathds 1_{\{|a| \geq \varepsilon\}}\trilamo(a) \leq \trilame,\]
which gives the desired bound.

\underline{Base case, bound on $\breve\Psi^{(1,<\varepsilon)}$.} \col{We show that we can bound $\breve\Psi^{(1,<\varepsilon)}$ by either $U_\lambda^{(\varepsilon)}$ or $(U_\lambda^{(\varepsilon)})^2$ respectively, giving} the desired second and third inequality of Lemma~\ref{lem:DB:breve_psi_bounds} and thus concluding the base case. The first bound on $\breve\Psi^{(1,<\varepsilon)}$ (the one for $j=1$) is given pictorially as
	\algn{ \lambda \sup_{a} \int & \mathds 1_{\{|w-u| < \varepsilon\}}\psi^{(1)}(\orig,a,t,w,z,u) \dd(\vec v) = \lambda^3 \sup_{\textcolor{darkorange}{\bullet}} 
			\int \mathrel{\raisebox{-0.25 cm}{\includegraphics{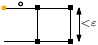}}} \notag \\
		& \leq \lambda^4 \sup_{\textcolor{darkorange}{\bullet}} \int \mathrel{\raisebox{-0.25 cm}{\includegraphics{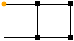}}}
			 + \lambda^3 \sup_{\textcolor{darkorange}{\bullet}} \int \mathrel{\raisebox{-0.25 cm}{\includegraphics{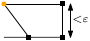}}}
			\ \leq \trilam^2 + \lambda^3 \sup_{\textcolor{darkorange}{\bullet}} \int \mathrel{\raisebox{-0.25 cm}{\includegraphics{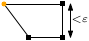}}}. \label{eq:DB:breve_psi_1_2_eps}}
An arrow with a `$< \varepsilon$' denotes the indicator of the two endpoints being less than $\varepsilon$ apart. The disappearing line in the last bound means that we have applied the (rough) bound $\tlam\leq 1$. We investigate the second term in the bound of~\eqref{eq:DB:breve_psi_1_2_eps} and write
	\[\tlame(x) := \mathds 1_{\{|x| < \varepsilon\}} \tlam(x).\]
In the subsequent lines, the substitutions $w'=w-u$ and $z'=z-u$ give
	\al{\lambda^3 \sup_{\textcolor{darkorange}{\bullet}} \int \mathrel{\raisebox{-0.25 cm}{\includegraphics{Psi_5__first_bound_eps.pdf}}}
			& = \lambda^3 \sup_a \int \tlam(a-z)\tlam(z-u) \tlame(w-u)\tlam(w-a) \dd(w,z,u) \\
		& = \lambda^3 \sup_a \int \tlame(w') \Big( \iint \tlam(u+w'-a) \tlam(a-u-z') \tlam(z') \dd u \dd z' \Big) \dd w' \\
		& = \lambda \int \tlame(w') \trilam(w') \dd w' \leq \big(\ballep\big)^2 \trilam.}
It is here that we see why $\ballep$ was defined with a square root: It allows us to extract two factors of $\ballep$ in the above. For $j=2$, the bounds are
	\al{ \lambda \sup_{a} \int & \mathds 1_{\{|w-u| < \varepsilon\}}\psi^{(2)}(\orig,a,t,w,z,u) \dd(\vec v) = \lambda^2 \sup_{\textcolor{darkorange}{\bullet}} 
			\int \mathrel{\raisebox{-0.25 cm}{\includegraphics{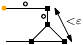}}} \\
		& = \lambda^3 \sup_{\textcolor{darkorange}{\bullet}} \int \mathrel{\raisebox{-0.25 cm}{\includegraphics{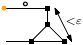}}}
			+ \lambda^2 \sup_{\textcolor{darkorange}{\bullet}} \int \mathrel{\raisebox{-0.25 cm}{\includegraphics{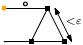}}} \\
		& \leq \lambda^4 \sup_{\textcolor{darkorange}{\bullet}} \int \mathrel{\raisebox{-0.25 cm}{\includegraphics{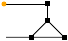}}}
			+ \lambda^3 \sup_{\textcolor{darkorange}{\bullet}} \int \mathrel{\raisebox{-0.25 cm}{\includegraphics{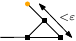}}}
			+ \lambda \sup_{\textcolor{darkorange}{\bullet}} \Big( \int \mathrel{\raisebox{-0.25 cm}{\includegraphics{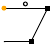}}}
				\Big( \lambda \sup_{\textcolor{lblue}{\bullet}} \int \mathrel{\raisebox{-0.25 cm}{\includegraphics{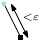}}} \Big)\Big) \\
		& \leq \trilam^2 + \lambda^3 \sup_{\textcolor{darkorange}{\bullet}} \int \mathrel{\raisebox{-0.25 cm}{\includegraphics{Psi_2__second_collapse_eps.pdf}}}
			\ + \trilamo \big(\ballep \big)^2 . }
To deal with the middle term, we see that, uniformly in $a\in\Rd$,
	\[\lambda^3 \int \mathrel{\raisebox{-0.25 cm}{\includegraphics{Psi_2__second_collapse_eps.pdf}}} \ \leq
		\lambda^3 \int \mathrel{\raisebox{-0.25 cm}{\includegraphics{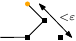}}} \ =
		\lambda^3 \int \tlam(z) \tlam(t-z) \tlam(a-t) \Big( \int \mathds 1_{\{|a-u| < \varepsilon\}} \dd u\Big) \dd z \dd t \leq \trilam \big( \ballep \big)^2. \]
Finally, the contribution of $j=3$ is bounded by
	\al{ \lambda \sup_{a} \int & \mathds 1_{\{|w-u| < \varepsilon\}}\psi^{(3)}(\orig,a,t,w,z,u) \dd\vec v = \lambda \sup_{\textcolor{darkorange}{\bullet}} 
			\int \mathrel{\raisebox{-0.25 cm}{\includegraphics{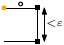}}} \\
		& \leq \lambda^2 \sup_{\textcolor{darkorange}{\bullet}} \int \mathrel{\raisebox{-0.25 cm}{\includegraphics{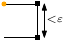}}}
			+ \lambda \sup_{\textcolor{darkorange}{\bullet}} \int \mathrel{\raisebox{-0.25 cm}{\includegraphics{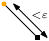}}} 
			\ \leq \lambda^2 \sup_{\textcolor{darkorange}{\bullet}} \int \mathrel{\raisebox{-0.25 cm}{\includegraphics{Psi_7__no_collapse_eps.pdf}}} + \big( \ballep \big)^2. }
We next bound the first term in the above. With the change of variables $z'=z-w$,
	\algn{\lambda^2 \int \mathrel{\raisebox{-0.25 cm}{\includegraphics{Psi_7__no_collapse_eps.pdf}}} 
			\ &= \lambda^2 \iint \tlam(z) \tlame(w-z) \tlam(a-w) \dd z \dd w \notag \\
		& = \lambda \int \tlame(z') \Big(\lambda\int \tlam(z'+w)\tlam(a-w) \dd w \Big) \dd z' \leq \trilamo \big( \ballep \big)^2, \label{eq:DB:triangle_with_epsilon_ballep_bound}}
as $\lambda(\tlam\star\tlam)(x) \leq \trilamo(x)$. Summing these contributions, $\lambda \iint \breve\Psi^{(1, <\varepsilon)} (a,b,x,y) \dd x \dd y$ is bounded by
	\[ 2 \trilam^2 + 2 \trilam \big(\ballep\big)^2 + 2 \trilamo \big(\ballep\big)^2 + \big(\ballep\big)^2 \leq 2 \trilam^2 + (4 \trilamo+1) \big(\ballep\big)^2
				 \leq \col{ U_\lambda^{(\varepsilon)} \wedge \big( U_\lambda^{(\varepsilon)} \big)^2},\]
as required. This concludes the base case.

\underline{Inductive step, $n>1$.} Let now $n>1$ and assume that the lemma is true for $n-1$. Then
	\[ \lambda^n \iint \breve\Psi^{(n)}(a,b,x,y) \dd x \dd y = \lambda \iint \breve\Psi^{(1)}(a,b,s,t) \Big( \lambda^{n-1} \iint \breve\Psi^{(n-1)}(s,t,x,y) \dd x \dd y \Big) \dd s \dd t \leq \big( U_\lambda\big)^n. \] 
For the second bound, a case distinction between $|s-t| \geq \varepsilon$ and $|s-t| <\varepsilon$ gives
	\al{\lambda^n \iint \breve\Psi^{(n)}(a,b,x,y) & \dd x \dd y = \lambda \iint \breve\Psi^{(1)}(a,b,s,t) \Big( \lambda^{n-1} \iint \breve\Psi^{(n-1, \geq\varepsilon)}(s,t,x,y) \dd x \dd y \Big) \dd s \dd t \\
			& \quad + \lambda \iint \breve\Psi^{(1,<\varepsilon)}(a,b,s,t) \Big( \lambda^{n-1} \iint \breve\Psi^{(n-1)}(s,t,x,y) \dd x \dd y \Big) \dd s \dd t \\
			& \leq U_\lambda (2(1+U_\lambda))^{n-2} \big(U_\lambda^{(\varepsilon)}\big)^{n-1} + U_\lambda^{(\varepsilon)} U_\lambda \big((2(1+U_\lambda)) U_\lambda^{(\varepsilon)} \big)^{n-2}\\
			& = 2 U_\lambda (2(1+U_\lambda))^{n-2} \big(U_\lambda^{(\varepsilon)}\big)^{n-1}.}
The same case distinction for $\breve\Psi^{(n,\geq\varepsilon)}$ yields
	 \al{\lambda^n \iint \breve\Psi^{(n,\geq\varepsilon)}(a,b,x,y) & \dd x \dd y = \lambda \iint \breve\Psi^{(1,\geq\varepsilon)}(a,b,s,t) 
	 					\Big( \lambda^{n-1} \iint \breve\Psi^{(n-1,\geq\varepsilon)}(s,t,x,y) \dd x \dd y \Big) \dd s \dd t \\
	 		& \quad + \lambda \iint \mathds 1_{\{|a-b| \geq \varepsilon\}}\breve\Psi^{(1,<\varepsilon)}(a,b,s,t) \Big( \lambda^{n-1} \iint \breve\Psi^{(n-1)}(s,t,x,y) \dd x \dd y \Big) \dd s \dd t \\
	 		& \leq U_\lambda^{(\varepsilon)} (2(1+U_\lambda))^{n-2} \big(U_\lambda^{(\varepsilon)}\big)^{n-1} + \big( U_\lambda^{(\varepsilon)}\big)^2 U_\lambda \big((2(1+U_\lambda)) U_\lambda^{(\varepsilon)} \big)^{n-2},}
which is at most $(2(1+U_\lambda))^{n-1} \big(U_\lambda^{(\varepsilon)}\big)^{n}$. The bounds for $\breve\Psi^{(n,<\varepsilon)}$ follow similarly. Having initiated and advanced the induction hypothesis, the claim follows by induction.
\end{proof}
\begin{proof}[Proof of Proposition~\ref{thm:Psi_Diag_bound}]
For $n=0$,
	\eqq{ \lambda\iint \Psi^{(0)}(w,u) \dd w \dd u = \sum_{j=1}^{3} \lambda \iint \psi_0^{(j)}(\orig,w,u) \dd w \dd u = 2\trilam(\orig) + \lambda \int \connf(u) \dd u \leq 2\trilam + \lambda, \label{eq:DB:prop_N0_base}}
which is certainly bounded by $2\trilamo+\lambda+1$. Let now $n \geq 1$ and note that
	\[\lambda^{n+1} \iint \Psi^{(n)}(x,y) \dd x \dd y = \lambda^{n+1} \iint \Psi^{(0)}(w,u) \Big( \iint \breve\Psi^{(n)}(w,u,x,y) \dd x \dd y \Big) \dd w \dd u. \]
At this stage, we can employ the bound on $\breve\Psi^{(n)}$ from Lemma~\ref{lem:DB:breve_psi_bounds} to obtain the bound in terms of $U_\lambda$ (without $U_\lambda^{(\varepsilon)}$). To obtain the second bound, we continue and observe that
	\al{ \lambda^{n+1} \iint \Psi^{(n)}(x,y) \dd x \dd y & \leq \Big( \lambda \iint \Psi^{(0)}(w,u) \dd w \dd u \Big) \Big( \sup_{w,u} \lambda^n \iint \breve\Psi^{(n,\geq\varepsilon)}(w,u,x,y) \dd x \dd y \Big) \\
			& \quad + \Big( \lambda \iint \mathds 1_{\{|w-u| < \varepsilon\}}\Psi^{(0)}(w,u) \dd w \dd u \Big) \Big( \sup_{w,u} \lambda^n \iint \breve\Psi^{(n)}(w,u,x,y) \dd x \dd y \Big) \\
			& \leq (2\trilam+\lambda)\big(\bar U_\lambda\big)^n \\
			& \quad + \Big( \lambda \iint \mathds 1_{\{|w-u| < \varepsilon\}}\Psi^{(0)}(w,u) \dd w \dd u \Big) U_\lambda \big(\bar U_\lambda\big)^{n-1}. }
To finish the proof, we need a bound similar to~\eqref{eq:DB:prop_N0_base} with the extra indicator $\mathds 1_{\{|w-u| < \varepsilon\}}$, i.e.~we still are confronted with a sum of three terms. The one for $j=3$ is directly bounded by $(\ballep)^2$. For the two terms $j=1,2$, we proceed similarly to~\eqref{eq:DB:triangle_with_epsilon_ballep_bound} (setting $a=\orig$). This results in the bound
	\[ \lambda^2 \iint \tlam^{(\varepsilon)}(z) \tlam(z-y) \tlam(y) \dd z \dd y \leq \trilamo \big(\ballep\big)^2.\]
Thus, \col{using that $\ballep \leq 1$,}
	\al{ \lambda^{n+1} \iint \Psi^{(n)}(x,y) \dd x \dd y &\leq \big(\bar U_\lambda\big)^{n-1} 
					\Big( (2\trilam+\lambda) \bar U_\lambda +  U_\lambda (2\trilamo + 1) \big(\ballep\big)^2 \Big) \\
		& \leq 2 (2\trilamo + \lambda + 1) \big(\bar U_\lambda\big)^{n}. \qedhere}
\end{proof}

We finally present the proof of Corollary~\ref{cor:DB:Pi_x_bounds}, which is a consequence of  Proposition~\ref{thm:Psi_Diag_bound}. 
\begin{proof}[Proof of Corollary \ref{cor:DB:Pi_x_bounds}]
Let $x\in\Rd$. Note that writing the statement of Proposition~\ref{thm:DB:Pi_bound_Psi} in terms of $\Psi^{(n)}$ (defined in~\eqref{def:Psi_N}) gives
	\al{ \Pi_\lambda^{(n)}(x) &\leq \lambda^n \iint \Psi^{(n-1)}(w,u)\Big( \int \psi_n(w,u,t,z,x) \dd(t,z) \Big) \dd (w,u) \\
		& \leq \big( \trilamo + 1 \big) \lambda^n \iint \Psi^{(n-1)}(w,u) \dd (w,u). }
Applying Proposition~\ref{thm:Psi_Diag_bound} implies the statement.
\end{proof}

\subsection{Diagrammatic bounds with displacement} 
\label{sec:diagrammaticbounds_disp}
In this section, we prove Propositions~\ref{thm:PsiDiag_Bound_Derangement_N1} and \ref{thm:PsiDiag_Bound_Derangement}. Recall the definitions of $W_\lambda(a;k)$ and $H_\lambda(a,b;k)$ in Definition \ref{def:displacement_quantities}. In terms of pictorial diagrams as introduced in Section \ref{sec:diagrammaticbounds_no_disp}, we can represent $W_\lambda(a;k)$ and $H_\lambda(a,b;k)$ as
	\[ W_\lambda(\textcolor{altviolet}{a};k) = \lambda \int \mathrel{\raisebox{-0.25 cm}{\includegraphics{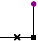}}} \qquad
		\text{and } H_\lambda(\textcolor{altviolet}{a},\textcolor{blue}{b};k) = 
						\lambda^5 \int \mathrel{\raisebox{-0.25 cm}{\includegraphics{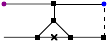}}}. \]
Since $W_\lambda(a;k) = \lambda\int \tklam(y)\tlam(a-y)\dd y$, the line carrying the factor $[1-\cos(k\cdot y)]$, which is the one representing $\tklam(y)$, is marked with a `$\times$'. This once more illustrates the power of the pictorial diagrams.
\medskip

We first prove Proposition~\ref{thm:PsiDiag_Bound_Derangement}, after which we sharpen the bound obtained for $n=1$ as stated in Proposition~\ref{thm:PsiDiag_Bound_Derangement_N1}.
The proof of Proposition~\ref{thm:PsiDiag_Bound_Derangement} needs another preparatory lemma (and definition), which we now state. We introduce $\bar\Psi^{(0)}$ and $\bar\Psi^{(n)}$ for $n \geq 1$, which are similar to $\Psi^{(n)}$, as
	\al{ \col{\bar \Psi^{(0)} (w_0,z_0)} &\col{:= \lambda \triangle(w_0,z_0,\orig) +  \delta_{\orig, w_0} \delta_{\orig, z_0},} \\
		\bar\Psi^{(n)}(w_n, z_n) &:= \int \phi_n(u_1,w_0,z_0,\orig) \\
			& \times \prod_{i=1}^{n-1} \Big( \phi^{(1)}(u_{i+1},w_i,t_i,z_i,u_i,z_{i-1}) + \sum_{j=2}^{3} \phi^{(j)}(u_{i+1},t_i,w_i,z_i,u_i,z_{i-1}) \Big) \\
			& \times \bigg( \lambda^2 \Box(w_n,t_n,u_n,z_n) \tlam(z_{n-1}-t_n) + \lambda \triangle(t_n,z_n,u_n) \parl(t_n-w_n, z_{n-1}-w_n) \\
			& \quad + \delta_{z_n,u_n}\delta_{t_n,z_n} \tlam(t_n-w_n) \tlam(z_{n-1}-w_n) \bigg) \dd \big((\vec w, \vec z)_{[0,n-1]}, (\vec t, \vec u)_{[1,n]} \big). }
Much like $\Psi^{(n)}$, $\bar\Psi^{(n)}$ is a product over ``segments'', and these segments (mostly) are a sum of three terms each. The three terms of the $\Psi^{(n)}$ segments and the $\bar\Psi^{(n)}$ segments are quite similar in nature---see the proof sketch of Lemma~\ref{lem:DB:bar_Psi_bounds} below. We stress the fact that in $\phi^{(1)}$, the labels of the points $w_i$ and $t_i$ are swapped. We also define $\bar\Psi^{(n,<\varepsilon)}(w,z) := \bar\Psi^{(n)}(w,z) \mathds 1_{\{|w-z| < \varepsilon\}}$. The following lemma is very much in the spirit of Lemma~\ref{lem:DB:breve_psi_bounds} and Proposition~\ref{thm:Psi_Diag_bound}:
\begin{lemma} \label{lem:DB:bar_Psi_bounds}
For $n \geq 0$,
	\al{ \lambda^{n+1} \iint \bar\Psi^{(n)}(w,z) \dd w \dd z & \leq (\trilam+\lambda) \big(U_\lambda \wedge \bar U_\lambda \big)^{n}, \\
			\lambda^{n+1} \iint \Big(\bar\Psi^{(n, <\varepsilon)}(w,z) - \mathds 1_{\{n=0\}}\,\delta_{0,z}\delta_{0,w}\Big) \dd w \dd z & \leq 
			\red{\bar U_\lambda^{n+1}. }}
\end{lemma}

The details of the proof of Lemma \ref{lem:DB:bar_Psi_bounds} are omitted, as they are analogous to those of Proposition~\ref{thm:Psi_Diag_bound}. \red{However, we sketch how to obtain the bounds. Note that for the first bound,} the factor $\trilam + \lambda$ stems from the base case \red{$\bar\Psi_0$} and is not identical to the one in Proposition~\ref{thm:Psi_Diag_bound} \red{(since $\Psi_0$ is different)}. We give a pictorial sketch of why we obtain the same bounds in the inductive step. We note that $\Psi^{(n)}$ consists of $\Psi^{(0)}$ times $n$ factors of $\psi$, which is a sum of three terms represented pictorially as
	\[ \int \Big( \mathrel{\raisebox{-0.25 cm}{\includegraphics{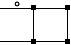}}}
		+\mathrel{\raisebox{-0.25 cm}{\includegraphics{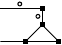}}} 
		+ \mathrel{\raisebox{-0.25 cm}{\includegraphics{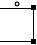}}} \Big). \]
On the other hand, $\bar\Psi^{(n)}$ consists of $\bar\Psi^{(0)}$ times $n$ factors of the form
	\[ \int \Big( \mathrel{\raisebox{-0.25 cm}{\includegraphics{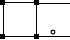}}}
		+\mathrel{\raisebox{-0.25 cm}{\includegraphics{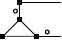}}} 
		+ \mathrel{\raisebox{-0.25 cm}{\includegraphics{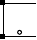}}} \Big). \]
The pictures are identical for $j=1,3$ and almost identical for $j=2$, and, most importantly, they can be bounded in the exact same way. \red{Regarding the second bound in Lemma \ref{lem:DB:bar_Psi_bounds}, note that for $n=0$, the left-hand side is equal to $\trilam$ ($\le \bar U_\lambda$) and immediately gives the base case. The inductive step then follows as in the first bound.}

\begin{proof}[Proof of Proposition~\ref{thm:PsiDiag_Bound_Derangement}]
What makes this proof more cumbersome is the displacement factor $[1-\cos(k \cdot x)]$. In making use of Lemma~\ref{lem:cosinesplitlemma}, we would like to ``split it up'' and distribute it over the single segments of the diagram. The $(n+1)$st segment diagram $\phi_0(\vec v_0) \big(\prod_{i=1}^{n-1} \phi(\vec v_i) \big) \phi_n(\vec v_n)$ that we have derived in~~\eqref{eq:DB:F_bound_phi} contains a product of factors of $\tlam$, which sit along the ``top'' of the diagram and whose arguments sum up to $x$. Hence, we can rewrite $x = \sum_{i=0}^{n} d_i$, where the term $d_i$ is the displacement which falls on the $i$-th segment along the top. We note that we can replace ``top'' by ``bottom'' in the previous sentences, and that in our later depictions of diagrams, we often use the bottom to carry the displacement.

Since the even-indexed segments appear in the diagram in the way displayed in Figure~\ref{fig:Psidiagrams} and the odd-indexed ones appear upside down (of course, this is just a matter of perspective), the displacement $d_i$ depends on the parity of $i$ and might be of the form $d_i =w_i-u_{i-1}$ or $d_i= u_i-w_{i-1}$ for $i\in \{1,\ldots, n\}$ (where $u_n=w_n=x$), whereas $d_0=w_0$, since we can fix the orientation of the first segment. Depending on the particular type of segment $i$, other forms are possible (for degenerate segments, $d_i$ might collapse to $\orig$ altogether). A key step is the inequality
	\[ 1-\cos(k \cdot x) \leq (n+1) \sum_{i=0}^{n} [1-\cos(k \cdot d_i)] \]
due to the Cosine-split Lemma~\ref{lem:cosinesplitlemma}, which allows us to split the diagram into a sum of $(n+1)$ diagrams, each of which only contains a local displacement $d_i$. We can thus hope to use the bounds on $\Psi$ and $\bar\Psi$ provided by Proposition~\ref{thm:Psi_Diag_bound} and Lemma~\ref{lem:DB:bar_Psi_bounds} for all but one segment.

Before we get to actual bounds, we may have to split $d_i$ once more into a sum of two terms for $i>0$, so that each of these terms appears as an argument of some factor $\tlam$. This strengthens the hope of obtaining a bound on $\widehat\Pi_\lambda^{(n)}(\orig)-\widehat\Pi_\lambda^{(n)}(k)$ in terms of a sum of $(n+1)$ terms, each looking rather similar to the bound we have obtained for $\widehat\Pi_\lambda^{(n)}(\orig)$---that is, $n$ out of the $(n+1)$ segments are bounded by known quantities and one designated factor contains $W_\lambda(k)$ (or the related $H_\lambda(k)$ diagram). For $\vec v_i=(w_{i-1}, u_{i-1}, t_i, w_i, z_i, u_i)$ to be specified below, we define
	\al{ \psi^{(4)}(\vec v_i) &= \tlam(w_i-u_{i-1})\tlam(u_i-w_i)\tlam(w_{i-1}-u_i) \delta_{z_i,u_i} \delta_{t_i,u_i}, \\
			\psi^{(5)}(\vec v_i) &= \tlam(u_i-u_{i-1})\tlam(w_{i-1}-u_i) \delta_{w_i,u_{i-1}}\delta_{z_i,u_i} \delta_{t_i,u_i}, }
so that $\psi^{(3)} = \psi^{(4)} + \psi^{(5)}$. The reason to split $\psi^{(3)}$ up further is to single out $\psi^{(5)}$, which will need some special treatment at a later stage of the proof. For $j \in \{1,2,4,5\}$ and $i\notin \{0,n\}$, we aim to bound
	\eqq{ \lambda^{n+1} \int \psi_0(\vec v_0) \Big( \prod_{l=1}^{i-1} \psi(\vec v_l)\Big) [1-\cos(k\cdot d_i)] \psi^{(j)}(\vec v_i)
				 \prod_{l=i+1}^{n} \psi(\vec v_l) \dd \big((\vec t, \vec z)_{[1,n]},(\vec w, \vec u)_{[0,n-1]}, x\big)
							 \label{eq:DB:disp:Proof_Goal},}
where we write $\vec v_0=(w_0,u_0)$ and $\vec v_l=(w_{l-1}, u_{l-1}, t_l, w_l, z_l, u_l)$ (with $u_n=x$) for $l\in\{1,\ldots, n\}$. 
For $i=0$ and $i=n$, the quantities analogous to~\eqref{eq:DB:disp:Proof_Goal} that we need to bound are
	\algn{ \lambda^{n+1} \int [1-\cos(k\cdot d_0)] \psi_0(\vec v_0) \prod_{l=1}^{n} \psi(\vec v_l) \dd \big((\vec t, \vec z)_{[1,n]},(\vec w, \vec u)_{[0,n-1]}, x\big)
							 \label{eq:DB:disp:Proof_Goal_i0}, \\
			 \lambda^{n+1} \int \psi_0(\vec v_0) \Big( \prod_{l=1}^{n-1} \psi(\vec v_l)\Big) [1-\cos(k\cdot d_n)] \psi_n(\vec v_n) \dd \big((\vec t, \vec z)_{[1,n]},(\vec w, \vec u)_{[0,n-1]}, x\big)
							 \label{eq:DB:disp:Proof_Goal_in}. }
Our proof proceeds as follows. We devise a strategy of proof which gives good enough bounds for all $n\geq 2$, all $j$ and all displacements $d_i$---except when $n=2$ and $j=5$. In a second step, we consider this special scenario separately. We divide the general proof into three cases: 
\begin{compactitem}
\item[(a)] The displacement is on the first segment, i.e. $i = 0$.
\item[(b)] The displacement is on the last segment, i.e. $i = n$.
\item[(c)] The displacement is on an interior segment, i.e.~$0 < i< n$.
\end{compactitem}

\underline{Case (a):} The only option for $d_0$ is $d_0=w_0$, and so $j=1$ is the only contributing case (otherwise $d_0=\orig$ and thus $[1-\cos(k\cdot d_0)] =0$). Next, note that we can rewrite~\eqref{eq:DB:disp:Proof_Goal_i0} as
	\begin{align} & \lambda^{n+1} \int [1-\cos(k\cdot w)] \psi_0^{(1)}(\orig,w,u) \parl(t+x-u,z+x-w) \bar\Psi^{(n-1)}(t,z) \dd(w,u,t,z,x) \notag\\
			\leq & (\lambda + \trilam) \big( U_\lambda\wedge\bar U_\lambda \big)^{n-1}
					\times \sup_{a\in\Rd} \int \lambda [1-\cos(k\cdot (w-x))] \psi_0^{(1)}(x,w,u) \parl(u,w-a) \dd (w,u,x), \label{eq:DB:disp:i=0}\end{align}
where the bound is by virtue of Lemma~\ref{lem:DB:bar_Psi_bounds}. Here we also see the need to introduce $\bar\Psi^{(n)}$. The integral in the right-hand side of~\eqref{eq:DB:disp:i=0} is
	\eqq{ \lambda^2 \int \Big( \int \tklam(w-x) \tlam(u-x) \dd x \Big) \tlam(w-u) \tlamo(u) \tlam(a-w) \dd(w,u) \leq W_\lambda(k) \trilamo(a). \label{eq:DB:disp:i=0_bound}}
As previously, we show how we represent this bound pictorially. We can bound~\eqref{eq:DB:disp:i=0_bound} as
	\[ \lambda^2 \int \mathrel{\raisebox{-0.25 cm}{\includegraphics{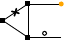}}} \ \leq 
			\lambda \int \Big( \Big( \sup_{\textcolor{green}{\bullet}, \textcolor{lblue}{\bullet}} \lambda \int \mathrel{\raisebox{-0.25 cm}{\includegraphics{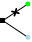}}} \Big)
					\mathrel{\raisebox{-0.25 cm}{\includegraphics{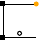}}} \Big) \leq W_\lambda(k) \trilamo. \]

\underline{Case (b):} We turn to $i=n$ and note that, depending on the parity of $n$, either $d_n=x-w_{n-1}$ or $d_n= x-u_{n-1}$. Suppose first that $d_n=x-w_{n-1}$. We can write~\eqref{eq:DB:disp:Proof_Goal_in} as
	\al{\lambda^{n+1} \int & [1-\cos(k\cdot (x-w))] \Psi^{(n-1)}(w,u) \psi_n(w,u,t,z,x) \dd(w,u,t,z,x) \\
			& \leq 2(2\trilamo + \lambda+1) \big(U_\lambda\wedge \bar U_\lambda\big)^{n-1} \times \lambda \sup_{a\in\Rd} \int [1-\cos(k\cdot x)] \big( \psi_n^{(1)} + \psi_n^{(2)} \big) (\orig,a,t,z,x) \dd (t,z,x),}
where the bound is by Proposition~\ref{thm:Psi_Diag_bound}. Tending to $\psi_n^{(1)}$ first, we use the Cosine-split Lemma~\ref{lem:cosinesplitlemma} to write $x=(x-z)+z$, which yields
	\eqan{\lambda \int [1-\cos(k & \cdot x)] \psi_n^{(1)}(\orig,a,t,z,x) \dd (t,z,x) = \lambda^2\int \mathrel{\raisebox{-0.25 cm}{\includegraphics{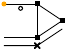}}} 
	\label{bound-with-explicit-triangle}\\
			& \leq 2 \lambda^2 \Big[\int \mathrel{\raisebox{-0.25 cm}{\includegraphics{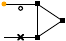}}}
						\ + \int \mathrel{\raisebox{-0.25 cm}{\includegraphics{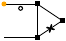}}} \Big] \nn\\
			& \leq 2 \lambda^2 \bigg[ \int \mathrel{\raisebox{-0.25 cm}{\includegraphics{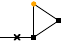}}}
						\ + \lambda \int \mathrel{\raisebox{-0.25 cm}{\includegraphics{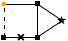}}}
						\ + \int \Big( \mathrel{\raisebox{-0.25 cm}{\includegraphics{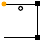}}}
						\Big( \sup_{\textcolor{green}{\bullet}, \textcolor{lblue}{\bullet}} \int \mathrel{\raisebox{-0.25 cm}{\includegraphics{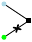}}} \Big)\Big) \bigg] \nn\\
			& \leq 2\lambda^2 \int \Big( \mathrel{\raisebox{-0.25 cm}{\includegraphics{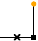}}}
						\Big( \sup_{\textcolor{green}{\bullet}, \textcolor{lblue}{\bullet}} \int \mathrel{\raisebox{-0.25 cm}{\includegraphics{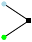}}} \Big)\Big)
						 + 2\lambda^3 \int \Big(\Big( \sup_{\textcolor{lblue}{\bullet}, \textcolor{green}{\bullet}}
								 \int \mathrel{\raisebox{-0.25 cm}{\includegraphics{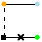}}} \Big) 
						 \mathrel{\raisebox{-0.25 cm}{\includegraphics{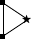}}}\Big) + 2 \trilamo W_\lambda(k) \nn\\
			& \leq 2 \Big( \trilamo W_\lambda(k) + \trilam W_\lambda(k) + \trilamo W_\lambda(k) \Big) \leq 6 \trilamo W_\lambda(k), \nn}
where we recall the usage of the `$\star$' symbol for the origin after substitution. In the above, note that
	\[ \int \tklam(\textcolor{darkgreen}{b_1}+x) \tlam(\textcolor{blue}{b_2}+x-\textcolor{altviolet}{a}) \dd x 
				= \int \mathrel{\raisebox{-0.25 cm}{\includegraphics{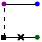}}} \ = W_\lambda(b_2-b_1-a;k).\]
For $\psi_n^{(2)}$,
	\[ \lambda \int [1-\cos(k \cdot x)] \psi_n^{(2)}(\orig,a,t,z,x) \dd (t,z,x) = \lambda\int \mathrel{\raisebox{-0.25 cm}{\includegraphics{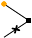}}}
				\ = W_\lambda(a;k) \leq W_\lambda(k).\]
The bounds for $d_n = x-u_{n-1}$ are the same due to symmetry. In total, noting that $\trilam \leq 2 \trilamo + 1 \leq (\trilamoo)^2$, cases (a) and (b) contribute at most
	\[ 2\big(\lambda+(\trilamoo)^2 \big) \big(U_\lambda\wedge \bar U_\lambda\big)^{n-1} W_\lambda(k) \big( \trilamo + 6 \trilamo + 1\big)
				\leq 16 \big(\lambda+(\trilamoo)^2 \big) \big(U_\lambda\wedge \bar U_\lambda\big)^{n-1} W_\lambda(k) \trilamoo. \]

\underline{Case (c):} Let $i \in \{1, \ldots, n-1\}$. We rewrite~\eqref{eq:DB:disp:Proof_Goal} as
	\algn{\lambda^{n+1} \int \Psi^{(i-1)}(w_{i-1},u_{i-1}) [1-\cos(k\cdot d_i)] \psi^{(j)}(\vec v_i) & \parl(b_1+x-u_i,b_2+x-w_i) \notag\\
				& \quad \times \bar\Psi^{(n-i-1)}(b_1,b_2) \dd (\vec v_i,\vec b_{[1,2]},x), \label{eq:DB:disp:case_b} }
where $\vec v_i = (w_{i-1}, u_{i-1}, t_i, w_i, z_i, u_i)$ and $d_i \in \{w_i-u_{i-1}, u_i-w_{i-1}\}$. In the next lines, we drop the subscript $i$, set $\vec v=(\orig,a,t,w,z,u,b_1,b_2,x;d,k),$ $\vec y=(t,w,z,u,x)$, $d\in\{w-a,u\}$, and write
	\eqn{
	\label{tilde-psi-j-bound}
	\tilde\psi_k^{(j)}(\vec v) = [1-\cos(k\cdot d)] \psi^{(j)}(\orig,a,t,w,z,u) \parl(b_1+x-u,b_2+x-w).
	}
Employing our bounds on $\Psi^{(i-1)}$ in Proposition~\ref{thm:Psi_Diag_bound} and on $\bar\Psi^{(n-i-1)}$  in  Lemma~\ref{lem:DB:bar_Psi_bounds}, and noting that $(1+\trilam)\wedge(1+2\trilamo)\le (\trilamoo)^2$, we see that~\eqref{eq:DB:disp:case_b} is bounded by
\begin{align} \label{newid-a} 
\lambda^{n+1} \int \Psi^{(i-1)} & (a_1,a_2) \bigg[ \sup_a \int \tilde\psi_k^{(j)}(\vec v)\, \bar\Psi^{(n-i-1)}(b_1,b_2) \dd (\vec y, \vec b_{[1,2]}) \bigg] \dd \vec a_{[1,2]}  \\
		\leq 2 \big(\lambda+& (\trilamoo)^2\big)^2 \big(U_\lambda\wedge\bar U_\lambda\big)^{n-2} \sup_{a,b_1,b_2} \int \lambda \tilde\psi_k^{(j)}(\vec v) \dd \vec y,  \nonumber 
\end{align}
and this is the key to the proof of \eqref{eq:DB:disp:thm_U_lambda}. 
Similarly, when we combine the bounds on $\Psi^{(i-1)}$ in Proposition~\ref{thm:Psi_Diag_bound} and 
both bounds in Lemma~\ref{lem:DB:bar_Psi_bounds}, we see that~\eqref{eq:DB:disp:case_b} is also bounded by 
\begin{align} 
\label{psi-j-integral-three-terms} 
\lambda^{n+1} \int \Psi^{(i-1)} & (a_1,a_2) \bigg[ \sup_a \int \tilde\psi_k^{(j)}(\vec v)\Big(\mathds 1_{\{|b_1-b_2|\geq\varepsilon\}} \bar\Psi^{(n-i-1)}(b_1,b_2) \\
		& \qquad\qquad\qquad\qquad\qquad\qquad + \bar\Psi^{(n-i-1,<\varepsilon)}(b_1,b_2) \Big) \dd (\vec y, \vec b_{[1,2]}) \bigg] \dd \vec a_{[1,2]} \nonumber \\
		\leq 2 \big(\lambda+& (\trilamoo)^2\big)^2 \bar U_\lambda^{n-2} \sup_{a,b_1,b_2} \int \lambda \tilde\psi_k^{(j)}(\vec v) 
		\big( \mathds 1_{\{|b_1-b_2| \geq \varepsilon \}} + \bar U_\lambda + \mathds 1_{\{i=n-1\}}\,\delta_{0,b_1}\delta_{0,b_2}\big) \dd \vec y,   \nonumber 
\end{align} 
which, together with \eqref{newid-a} is the key to \eqref{eq:DB:disp:thm_U_lambda_eps}. 

In the sequel we prove the upper bounds
\eqn{\label{eq:jIntBd}
	\sup_{a,b_1,b_2} \int \lambda \tilde\psi_k^{(j)}(\vec v) \dd \vec y
	\le
	\begin{cases}
		2 W_\lambda(k) \big(\trilam\trilamoo+\trilam\trilamo+\trilamo\trilamo\big)&\text{for~} j=1;\\
		W_\lambda(k)\big(4\trilamo+8\trilam\trilamoo+2\trilam\trilamo\big)+2H_\lambda(k) \quad&\text{for~} j=2;\\
		W_\lambda(k)\trilamo&\text{for~} j=4;\\
		W_\lambda(k)\trilamoo&\text{for~} j=5,
	\end{cases}	
}
as well as 
\eqn{\label{eq:jIntIndicBd}
	\sup_{a,b_1,b_2} \int \lambda \tilde\psi_k^{(j)}(\vec v) 
		\mathds 1_{\{|b_1-b_2| \geq \varepsilon\}} \dd \vec y
	\le
	\begin{cases}
		2 W_\lambda(k) \big(\trilam\trilamoo+\trilam\trilamo+\trilamo\trilame\big)&\text{for~} j=1;\\
		W_\lambda(k)\big(4\trilame+8\trilam\trilamoo+2\trilam\trilamo\big)+2H_\lambda(k) &\text{for~} j=2;\\
		W_\lambda(k)\trilamo&\text{for~} j=4;\\
		W_\lambda(k)\trilamoo&\text{for~} j=5.
	\end{cases}	
}
The main difference between \eqref{eq:jIntBd} and \eqref{eq:jIntIndicBd} lies in the operator $\trilame$ appearing for $j=1$ and $j=2$ in \eqref{eq:jIntIndicBd}, which is changed to $\trilamo$ in \eqref{eq:jIntBd}. 

We next explain how \eqref{eq:jIntBd} and \eqref{eq:jIntIndicBd} can be used to derive \eqref{eq:DB:disp:thm_U_lambda} and \eqref{eq:DB:disp:thm_U_lambda_eps}. 
For this, we recall from Section \ref{sec-statement-bounds-LA-coefficients}  the bounds $\trilam\le\trilamo\le\trilamoo$, $1\le\trilamoo$ and $\trilame\le\trilamo\le U_\lambda/3$ as well as $\trilame\le\bar U_\lambda/3$ and finally $\trilam\trilame\le\trilam\trilamo\le\trilam\trilamoo\le(U_\lambda\wedge\bar U_\lambda)/3$. 
It might further be seen that 
	$$\trilamo\le (\trilamoo(1+U_\lambda))\wedge(\trilamoo + U_\lambda + \bar U_\lambda).$$   
These bounds enable us to deduce from \eqref{eq:jIntBd} that 
\eqn{\label{eq:IntPsiJBd1}
	\sum_{j=1,2,4,5}\sup_{a,b_1,b_2} \int \lambda \tilde\psi_k^{(j)}(\vec v) \dd \vec y
	\le 30\big(W_\lambda(k)\trilamoo(1+U_\lambda)+H_\lambda(k)\big).
}
We now combine \eqref{eq:IntPsiJBd1} with \eqref{newid-a} to establish \eqref{eq:DB:disp:thm_U_lambda} for $n\ge3$ (recall that the factor $n+1$ arises from the Cosine-split Lemma \ref{lem:cosinesplitlemma}).  For $n=2$ we need to extract an additional factor $U_\lambda$ on the right-hand side of \eqref{eq:IntPsiJBd1}. The factor is extractable from the given bounds in \eqref{eq:jIntBd} for $j=1,2,4$ since $\trilam\le\trilamo\le U_\lambda$, and we show separately (at the end of the proof) via a direct argument for $n=2$ and $j=5$ that 
\eqn{ \label{eq:j5n2bd}\lambda^{3} \int \psi_0(\vec v_0) [1-\cos(k\cdot d_1)] \psi^{(5)}(\vec v_1)
				 \psi(\vec v_2) \dd \big((\vec t, \vec z)_{[1,2]},(\vec w, \vec u)_{[0,1]}, x\big)
	\leq 3 W_\lambda(k) (U_\lambda \wedge \bar U_\lambda).
}
This establishes \eqref{eq:DB:disp:thm_U_lambda} subject to \eqref{eq:jIntBd} and \eqref{eq:j5n2bd}. 

In order to show \eqref{eq:DB:disp:thm_U_lambda_eps}, it is sufficient to establish that 
\begin{align}
	&	\sum_{j=1,2,4,5}\sup_{a,b_1,b_2} \int \lambda \tilde\psi_k^{(j)}(\vec v) 
		\big( \mathds 1_{\{|b_1-b_2| \geq \varepsilon \}} + \bar U_\lambda + \mathds 1_{\{i=n-1\}}\,\delta_{0,b_1}\delta_{0,b_2}\big) \dd \vec y \nonumber \\
	&\hskip10em \le 30\big(W_\lambda(k)\trilamoo(\trilamoo+U_\lambda+\bar U_\lambda)\bar U_\lambda^{{\mathds 1} \{n=2\}}+H_\lambda(k)\big).
\end{align}
We consider the three summands appearing in the integral on the right-hand side separately. 
The term involving the indicator $\mathds 1_{\{|b_1-b_2| \geq \varepsilon \}}$ we get as before, but now with \eqref{eq:jIntIndicBd} rather than \eqref{eq:jIntBd} (again together with \eqref{eq:j5n2bd} for the case $n=2$ and $j=5$). 
For the remaining two terms, we apply \eqref{eq:jIntBd} directly (and ignore the indicators in the last summand for $n\ge3$). 
However, for $n=2$ and the term involving the indicators $\mathds 1_{\{i=n-1\}}\,\delta_{0,b_1}\delta_{0,b_2}$, we need a special argument to extract the additional factor $\bar U_\lambda$, and we give this argument at the end of the proof. This establishes \eqref{eq:DB:disp:thm_U_lambda_eps} subject to \eqref{eq:jIntBd}, \eqref {eq:jIntIndicBd}, \eqref{eq:j5n2bd}, and the special cases that will be handled below.

\paragraph{\bf Proof of \eqref{eq:jIntBd} and \eqref{eq:jIntIndicBd}.}
We will mostly focus on the proof of \eqref{eq:jIntIndicBd}, as \eqref{eq:jIntBd} will then follow rather straightforwardly. 
First, we note that
	\[ 
	\lambda^2 \sup_{a,b_1,b_2} \int \tilde\psi_k^{(j)} (\vec v) \dd \vec y = \lambda^2 \sup_{a,b} \int [1-\cos(k\cdot d)] \psi^{(j)}(\orig,a,t,w,z,u) \parl(x-u,b+x-w) \dd \vec y.
	\]
We now turn to the particular values for $j$ and $d$. Starting with $j=1$, recall that $d\in\{u,w-a\}$. Again, we turn to pictorial bounds. In the following lines, we use an arrow together with `$ \geq \varepsilon$' to represent indicators of the form $\mathds 1_{\{|\cdot| \geq \varepsilon\}}$.

Setting $\vec v =(\orig,a,t,w,z,u,b,\orig,x;u,k)$ and $\vec y=(t,w,z,u,x)$, the bound for $d=u$ can be obtained as
	\begin{align} \lambda \int \tilde\psi_k^{(1)} (\vec v) & \mathds 1_{\{|b| \geq \varepsilon\}} \dd \vec y = 
				\lambda^3 \int [1-\cos(k\cdot u)] \tlam(z) \tlamo(a-t) \Box(z,t,w,u) \notag \\ & \qquad \times \tlamo(x-u) \tlam(b+x-w) \mathds 1_{\{|b| \geq \varepsilon\}} \dd\vec y \notag \\
		& = \lambda^3 \int \mathrel{\raisebox{-0.25 cm}{\includegraphics{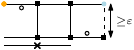}}} 
				\ \leq 2\lambda^3 \int \mathrel{\raisebox{-0.25 cm}{\includegraphics{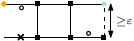}}} 
					+ 2\lambda^3 \int \mathrel{\raisebox{-0.25 cm}{\includegraphics{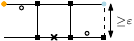}}} \notag \\
		& \leq 2\lambda^3 \int \mathrel{\raisebox{-0.25 cm}{\includegraphics{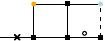}}}
				+ 2 \lambda^4 \int \mathrel{\raisebox{-0.25 cm}{\includegraphics{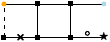}}}
				+ 2 \lambda^3 \int \Big( \mathrel{\raisebox{-0.25 cm}{\includegraphics{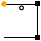}}}
					\Big( \sup_{\textcolor{green}{\bullet}} \int \mathrel{\raisebox{-0.25 cm}{\includegraphics{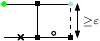}}} \Big) \Big), \label{eq:DP_bounds_psi_1}
	\end{align}
where we have used the Cosine-split Lemma~\ref{lem:cosinesplitlemma}. The first summand in the right-hand side of~\eqref{eq:DP_bounds_psi_1} is bounded via
	\[ \lambda^3 \int \mathrel{\raisebox{-0.25 cm}{\includegraphics{Disp_i_1__split_left_collapse.pdf}}} 
			\ \leq \lambda \int \Big( \mathrel{\raisebox{-0.25 cm}{\includegraphics{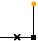}}}
				\Big( \sup_{\textcolor{green}{\bullet}, \textcolor{violet}{\bullet}} \int \mathrel{\raisebox{-0.25 cm}{\includegraphics{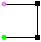}}}
				\Big( \sup_{\textcolor{green}{\bullet}, \textcolor{violet}{\bullet}} \int \mathrel{\raisebox{-0.25 cm}{\includegraphics{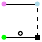}}} \Big)\Big)\Big)
			\leq W_\lambda(k) \trilam \trilamoo. \]
The second summand in the r.h.s.~of~\eqref{eq:DP_bounds_psi_1} is bounded by
	\[ \lambda^4 \int \mathrel{\raisebox{-0.25 cm}{\includegraphics{Disp_i_1__split_left_subst.pdf}}} 
			\ \leq \lambda^4 \int \Big( \Big( \sup_{\textcolor{green}{\bullet}, \textcolor{violet}{\bullet}} 
						\int \mathrel{\raisebox{-0.25 cm}{\includegraphics{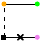}}} \Big)
						\mathrel{\raisebox{-0.25 cm}{\includegraphics{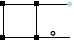}}} \Big) \leq W_\lambda(k) \trilam\trilamo. \]
In the third summand in the r.h.s.~of~\eqref{eq:DP_bounds_psi_1}, we obtain the bound
	\eqan{\lambda^3 \int \Big( \mathrel{\raisebox{-0.25 cm}{\includegraphics{Disp_i_1__split_right_bound1.pdf}}}
					\Big( \sup_{\textcolor{green}{\bullet}} \int \mathrel{\raisebox{-0.25 cm}{\includegraphics{Disp_i_1__split_right_bound2.pdf}}} \Big) \Big)
			& \leq \trilamo \sup_{\textcolor{green}{\bullet}} \lambda^2 \int \mathrel{\raisebox{-0.25 cm}{\includegraphics{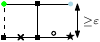}}}\label{bound-j=1-vep}\\
			& \leq \trilamo \sup_{\textcolor{green}{\bullet}} \lambda^2 \int \Big( \sup_{\textcolor{darkorange}{\bullet}, \textcolor{violet}{\bullet}}
							\int \mathrel{\raisebox{-0.25 cm}{\includegraphics{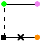}}} \Big)
				\mathrel{\raisebox{-0.25 cm}{\includegraphics{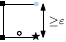}}} \Big) \leq \trilamo W_\lambda(k) \trilame .
				\nn}
Note that this third summand is the only one where we obtain a bound in terms of $\trilame$. 
\red{For the corresponding expression in \eqref{eq:jIntBd}, where we do not have the indicator $\mathds 1_{\{|b| \geq \varepsilon\}}$ available, we simply replace $\trilame$ in \eqref{bound-j=1-vep} by $\trilamo$ (and the same adaptation shall be used numerous times below).} 
Again, the displacement for $d=w-a$ yields the same upper bound due to symmetry. The contribution of~\eqref{eq:DP_bounds_psi_1} is therefore bounded by $6 W_\lambda(k) (U_\lambda\wedge\bar U_\lambda)$.

We turn to $j=2$ and see that, similarly to~\eqref{eq:DP_bounds_psi_1},
	\begin{align} \lambda \int \tilde\psi_k^{(2)} (\vec v) \mathds 1_{\{|b| \geq \varepsilon\}} \dd \vec y &= 
				\lambda^2 \int [1-\cos(k\cdot u)] \tlam(z) \tlamo(a-w) \tlamo(t-w) \triangle(z,t,u) \notag \\ &\qquad \times \tlamo(x-u) \tlam(b+x-w) \mathds 1_{\{|b| \geq \varepsilon\}} \dd\vec y \notag \\
			& = \lambda^2 \int \mathrel{\raisebox{-0.25 cm}{\includegraphics{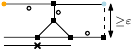}}}
					\ = \lambda^2 \bigg[ \int \mathrel{\raisebox{-0.25 cm}{\includegraphics{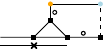}}}
					 \ + \lambda \int \mathrel{\raisebox{-0.25 cm}{\includegraphics{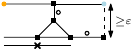}}} \bigg] \notag \\
			& \leq \lambda^2 \bigg[ \int \mathrel{\raisebox{-0.25 cm}{\includegraphics{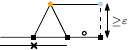}}}
					+ \lambda \int \mathrel{\raisebox{-0.25 cm}{\includegraphics{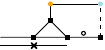}}} \bigg] \label{eq:DB:disp:psi_2_coll} \\
			& \qquad + 2\lambda^3 \bigg[ \int \mathrel{\raisebox{-0.25 cm}{\includegraphics{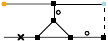}}}
					+ \int \mathrel{\raisebox{-0.25 cm}{\includegraphics{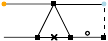}}}
					+ \lambda \int \mathrel{\raisebox{-0.25 cm}{\includegraphics{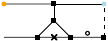}}} \bigg]. \label{eq:DB:disp:psi_2_no_coll}
	\end{align}
We investigate the five bounding diagrams separately. The first summand in~\eqref{eq:DB:disp:psi_2_coll} is
	\eqn{
	\label{bound-j=2-vep}
	\lambda^2 \int \mathrel{\raisebox{-0.25 cm}{\includegraphics{Disp_i_2__collapse_1_coll_2.pdf}}}
		\ \leq \lambda^2 \int \mathrel{\raisebox{-0.25 cm}{\includegraphics{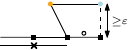}}}
		\ \leq \lambda^2 \int \Big( \Big(\sup_{\textcolor{green}{\bullet}} \int \mathrel{\raisebox{-0.25 cm}{\includegraphics{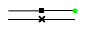}}} \Big)
			\mathrel{\raisebox{-0.25 cm}{\includegraphics{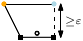}}} \Big)
		\leq 4 W_\lambda(k) \trilame.
	}
\red{Again for the corresponding expression in \eqref{eq:jIntBd}, we simply replace $\trilame$ in \eqref{bound-j=1-vep} by $\trilamo$.}

The second summand in~\eqref{eq:DB:disp:psi_2_coll} is bounded by
	\al{ \lambda^3 \int & \mathrel{\raisebox{-0.25 cm}{\includegraphics{Disp_i_2__collapse_1_no_coll_2.pdf}}} 
			\ \leq \lambda^3 \int \Big( \mathrel{\raisebox{-0.25 cm}{\includegraphics{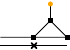}}}
				\Big( \sup_{\textcolor{green}{\bullet}, \textcolor{violet}{\bullet}} \int \mathrel{\raisebox{-0.25 cm}{\includegraphics{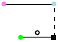}}} \Big)\Big)
		 	\leq 2 \trilamoo \lambda^3 \bigg[ \int \mathrel{\raisebox{-0.25 cm}{\includegraphics{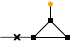}}}
		 		+ \int \mathrel{\raisebox{-0.25 cm}{\includegraphics{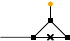}}} \bigg] \\
		 & \leq 2 \trilamoo \lambda^3 \bigg[ \int \mathrel{\raisebox{-0.25 cm}{\includegraphics{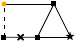}}}
		 		 \; + \int \Big( \mathrel{\raisebox{-0.25 cm}{\includegraphics{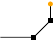}}}
		 		 \Big( \sup_{\textcolor{green}{\bullet}, \textcolor{lblue}{\bullet}} 
		 		 \int \mathrel{\raisebox{-0.25 cm}{\includegraphics{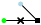}}} \Big)\Big) \bigg]
		 	\leq 4 \trilamoo \trilam W_\lambda(k). }
The first summand in~\eqref{eq:DB:disp:psi_2_no_coll} is
	\[ \lambda^3 \int \mathrel{\raisebox{-0.25 cm}{\includegraphics{Disp_i_2__no_collapse_1_split_left.pdf}}} 
			\ = \lambda^3 \int \mathrel{\raisebox{-0.25 cm}{\includegraphics{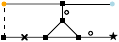}}}
			\ \leq \lambda^3 \int \Big( \Big( \sup_{\textcolor{green}{\bullet}, \textcolor{violet}{\bullet}} 
					\int \mathrel{\raisebox{-0.25 cm}{\includegraphics{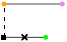}}} \Big)
					\mathrel{\raisebox{-0.25 cm}{\includegraphics{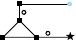}}} \Big) 
			\leq \trilamoo\trilam W_\lambda(k), \]
the second is
	\al{ \lambda^3 \int \mathrel{\raisebox{-0.25 cm}{\includegraphics{Disp_i_2__no_collapse_1_split_right_collapse_2.pdf}}} 
			\ &\leq \lambda^3 \int \Big( \mathrel{\raisebox{-0.25 cm}{\includegraphics{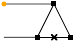}}}
				\Big( \sup_{\textcolor{green}{\bullet}, \textcolor{violet}{\bullet}} 
				\int \mathrel{\raisebox{-0.25 cm}{\includegraphics{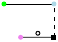}}} \Big) \Big) 
			\leq \lambda^3 \trilamoo \int \Big( \mathrel{\raisebox{-0.25 cm}{\includegraphics{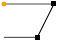}}}
					\Big( \sup_{\textcolor{green}{\bullet}, \textcolor{lblue}{\bullet}} 
					\mathrel{\raisebox{-0.25 cm}{\includegraphics{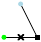}}} \Big)\Big) \\
		& \leq \trilamoo \trilam W_\lambda(k),}
and the third is
	\al{ \lambda^4 \int \mathrel{\raisebox{-0.25 cm}{\includegraphics{Disp_i_2__no_collapse_1_split_right_no_collapse_2.pdf}}} 
			\ &= \lambda^4 \int \mathrel{\raisebox{-0.25 cm}{\includegraphics{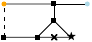}}}
			 + \lambda^5 \int \mathrel{\raisebox{-0.25 cm}{\includegraphics{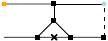}}} \\
			& \leq \lambda^4 \int \Big(\Big( \sup_{\textcolor{lblue}{\bullet}} \int \Big( \sup_{\textcolor{green}{\bullet}, \textcolor{violet}{\bullet}} 
					 \int \mathrel{\raisebox{-0.25 cm}{\includegraphics{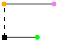}}} \Big)
				 \mathrel{\raisebox{-0.25 cm}{\includegraphics{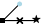}}} \Big)
				 \mathrel{\raisebox{-0.25 cm}{\includegraphics{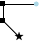}}} \Big) + H_\lambda(k)
				\leq \trilamo \trilam W_\lambda(k) + H_\lambda(k). }
The displacement $d=w-a$ is handled like the first term of~\eqref{eq:DB:disp:psi_2_no_coll} by symmetry, and thus the bound for $d=u$ suffices. We conclude that the joint contribution of~\eqref{eq:DB:disp:psi_2_coll} and~\eqref{eq:DB:disp:psi_2_no_coll} is bounded by $14 W_\lambda(k)(U_\lambda\wedge\bar U_\lambda(k)) + 2 H_\lambda(k)$. 

We turn to $j=4$, for which we get
	\eqan{
	\label{bound-j=4-vep}
	\lambda \int \tilde\psi_k^{(4)} (\vec v) \mathds 1_{\{|b| \geq \varepsilon\}} \dd \vec y &= 
				\lambda^2 \int [1-\cos(k\cdot u)] \tlam(u) \tlamo(a-w) \tlam(u-w) \\
			& \qquad \times \tlamo(x-u) \tlam(b+x-w) \mathds 1_{\{|b| \geq \varepsilon\}} \dd(w,u,x) \nn\\
			& = \lambda^2 \int \mathrel{\raisebox{-0.25 cm}{\includegraphics{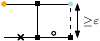}}}
				\ = \lambda^2 \int \mathrel{\raisebox{-0.25 cm}{\includegraphics{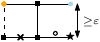}}} 
				\ \leq \lambda \int \Big( \Big( \sup_{\textcolor{green}{\bullet},\textcolor{violet}{\bullet}}
					 \int \mathrel{\raisebox{-0.25 cm}{\includegraphics{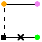}}} \Big)
				\mathrel{\raisebox{-0.25 cm}{\includegraphics{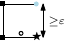}}} \Big) \nn\\
			& \leq \trilame W_\lambda(k). \nn}

Finally, $j=5$ (where $d=u$ is the only option) yields
	\eqan{\lambda \int \tilde\psi_k^{(5)} (\vec v) \mathds 1_{\{|b| \geq \varepsilon\}} \dd \vec y 
	&= 
				\lambda \int [1-\cos(k\cdot u)] \tlam(u) \tlam(z-a) \tlamo(x-u) \tlam(b+x-a) \mathds 1_{\{|b| \geq \varepsilon\}} \dd(u,x) \nonumber \\
			& = \lambda \int \mathrel{\raisebox{-0.25 cm}{\includegraphics{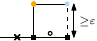}}}
				\ \leq \lambda \int \Big( \mathrel{\raisebox{-0.25 cm}{\includegraphics{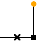}}} \Big( \sup_{\textcolor{green}{\bullet},\textcolor{violet}{\bullet}}
				 \int \mathrel{\raisebox{-0.25 cm}{\includegraphics{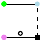}}} \Big)\Big)
				\leq \trilamoo W_\lambda(k). 
	}			
\red{This concludes the proof of \eqref{eq:jIntBd} and \eqref{eq:jIntIndicBd}, and to prove \eqref{eq:DB:disp:thm_U_lambda} and \eqref{eq:DB:disp:thm_U_lambda_eps} we are thus left with dealing with the special cases arising for $n=2$.}

\paragraph{\bf Remaining contributions for $n=2$.} 
We first prove \eqref{eq:j5n2bd}, which deals with the integral over $\psi^{(5)}$ in the setting when $n=2$. The left-hand side of \eqref{eq:j5n2bd} can be written as a sum of  
	\eqq{
	\lambda^3 \int \psi_0^{(s)}(\orig,w,u) \tklam(z-w) \tlam(z-u) \psi_2(u,z,y_1,y_2,x) \dd(w,u,z,y_1,y_2,x) 
	\label{eq:DB:disp:N2_j5_goal}
	}
\red{over $s=1,2,3$. For $s=1,2$,} we can bound~\eqref{eq:DB:disp:N2_j5_goal} by
	\al{ \lambda^3 & \int \psi_0^{(s)}(\orig,w,u) \Big( \sup_{w,u} \int \tklam(z-w) \tlam(z-u) \Big( \sup_{z,u} \int \psi_2(u,z,y_1,y_2,x) \dd(y_1,y_2,x) \Big) \dd u \Big) \dd(w,u) \\
		 &= \lambda^3 \int \psi_0^{(s)}(\orig,w,u) \Big( \sup_{\textcolor{darkorange}{\bullet}, \textcolor{lblue}{\bullet}} 
		 		\int \mathrel{\raisebox{-0.25 cm}{\includegraphics{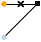}}} \Big( \sup_{\textcolor{green}{\bullet}, \textcolor{violet}{\bullet}}
		 		\lambda \int \mathrel{\raisebox{-0.25 cm}{\includegraphics{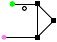}}} 
		 		\ + \sup_{\textcolor{green}{\bullet}, \textcolor{violet}{\bullet}}
		 		\int \mathrel{\raisebox{-0.25 cm}{\includegraphics{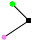}}}\Big) \Big) \dd(w,u) \\
		 & \leq \lambda W_\lambda(k) \trilamo(1 + \trilamo) \int \psi_0^{(j)}(\orig,w,u) \dd (w,u) \leq W_\lambda(k) \trilam\trilamo (1+\trilamo) \leq 2 W_\lambda(k) (U_\lambda \wedge \bar U_\lambda),	 }
as is easily checked for both $s=1,2$. For $s=3$, we shift~\eqref{eq:DB:disp:N2_j5_goal} by $-x$, 
whereupon~\eqref{eq:DB:disp:N2_j5_goal} is equal to
	\al{\lambda^3 & \int \psi_0^{(s)}(-x,w,u) \tklam(z-w) \tlam(z-u) \psi_2(u,z,y_1,y_2,\orig) \dd(w,u,z,y_1,y_2,x) \\
		& \leq \lambda^4 \int \mathrel{\raisebox{-0.25 cm}{\includegraphics{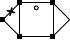}}}
				+ \lambda^3 \int \mathrel{\raisebox{-0.25 cm}{\includegraphics{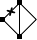}}}
			\ \leq \lambda^4 \int \Big(\Big( \sup_{\textcolor{darkorange}{\bullet}, \textcolor{lblue}{\bullet}} \int \mathrel{\raisebox{-0.25 cm}{\includegraphics{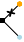}}} \Big)
				\mathrel{\raisebox{-0.25 cm}{\includegraphics{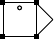}}} \Big)
				+ \lambda^3 \int \Big(\Big( \sup_{\textcolor{darkorange}{\bullet}, \textcolor{lblue}{\bullet}} \int \mathrel{\raisebox{-0.25 cm}{\includegraphics{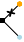}}} \Big)
				\mathrel{\raisebox{-0.25 cm}{\includegraphics{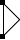}}} \Big) \\
		& \leq W_\lambda(k) \trilam(\trilamo+1) \leq W_\lambda(k) (U_\lambda \wedge \bar U_\lambda).}
\red{All these contributions satisfy the desired bound, as required.}

\red{For the contribution to $n=2$ involving the term $\mathds 1_{\{i=1\}}\,\delta_{0,b_1}\delta_{0,b_2}$, we need to extract an extra factor of $\bar U_\lambda$. Problematic is the occurrence of terms $\trilamo$, because we can not bound it from above by $\bar U_\lambda$. 
Again we develop a direct argument and start from a representation as in \eqref{eq:DB:disp:N2_j5_goal}, where $\psi^{(5)}$ is replaced by the difficult contributions due to $\psi^{(1)}, \psi^{(2)}, \psi^{(4)}$ and $\psi^{(5)}$. The above bound for $\psi^{(5)}$ applies verbatim. We next discuss the problematic contributions due to $\psi^{(1)}$, $\psi^{(2)}$ and $\psi^{(4)}$ in \eqref{bound-j=1-vep}, \eqref{bound-j=2-vep} and \eqref{bound-j=4-vep}, respectively. We are dealing with $\psi_2(u,z,y_1,y_2,x)$ (recall \eqref{eq:DB:disp:N2_j5_goal}), for which $\psi_2 = \psi_2^{(1)} + \psi_2^{(2)}$, and we refer to Definition \ref{def:DB:psi_functions}. The contribution due to $\psi_2^{(1)}$ contains an explicit triangle, and therefore can be dealt with as in, for example, \eqref{bound-with-explicit-triangle}. For the contribution involving $\psi_2^{(2)}$, we obtain a bound that is very alike the ones that we have dealt with before. The one major difference is that the term $\parl(b_1+x-u,b_2+x-w)=\parl(x-u,x-w)$, appearing for $b_1=b_2=0$ in \eqref{tilde-psi-j-bound}, is replaced by $ \tlam(x-u) \tlam(x-w)$, i.e., $\tlamo(x-u)$ is replaced by $\tlam(x-u)$. After this, the problematic terms in \eqref{bound-j=1-vep}, \eqref{bound-j=2-vep} and \eqref{bound-j=4-vep} obtain a factor $\trilam$ instead of $\trilamo$, so that these bound are as desired as well.
}

Carefully putting together all these bounds finishes the proof.
\end{proof}

\begin{proof}[Proof of Proposition~\ref{thm:PsiDiag_Bound_Derangement_N1}]
Let now $n=1$. By \eqref{eq:DB:F_bound_phi}, we get a bound on $\lambda\int [1-\cos(k\cdot x)] \Pi_\lambda^{(1)}(x) \dd x $ of the form
	\eqq{ \lambda^2 \int [1-\cos(k\cdot x)] \sum_{j_0=1}^{3}\sum_{j_1=1}^{2}\psi_0^{(j_0)}(w,u) \psi_1^{(j_1)} (w,u,t,z,x) \dd (w,u,t,z,x). \label{eq:Pi1bound}}
This results in a sum of six diagrams, which we bound one by one. Again, we want to use the Cosine-split Lemma~\ref{lem:cosinesplitlemma} in order to break up the displacement factor $[1-\cos(k\cdot x)]$ and distribute it over edges of the diagrams.
Most terms follow analogously to $n\geq 2$, and so we only do the pictorial representations thereof. For $(j_0,j_1)=(1,1)$, we get the bound
	\eqq{ \lambda^4 \int \mathrel{\raisebox{-0.25 cm}{\includegraphics{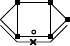}}}
			\ \leq 3\lambda^4 \bigg[ \int \mathrel{\raisebox{-0.25 cm}{\includegraphics{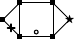}}} 
				+ \lambda \int \mathrel{\raisebox{-0.25 cm}{\includegraphics{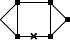}}}
				+ \int \mathrel{\raisebox{-0.25 cm}{\includegraphics{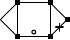}}} \bigg]. \label{eq:DB:Pi1bound_stepA}}
The first and third diagram on the r.h.s.~of~\eqref{eq:DB:Pi1bound_stepA} are the same by symmetry, and so
	\al{ \lambda^4 \int \mathrel{\raisebox{-0.25 cm}{\includegraphics{N1__1,1__general.pdf}}} &
		 \leq 3\lambda^5 \int \Big( \mathrel{\raisebox{-0.25 cm}{\includegraphics{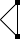}}}
					\Big( \sup_{\textcolor{darkorange}{\bullet}} \int \mathrel{\raisebox{-0.25 cm}{\includegraphics{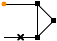}}} \Big)\Big)
				+ 6\lambda^4 \int \mathrel{\raisebox{-0.25 cm}{\includegraphics{N1__1,1__split_right.pdf}}} \\
		& \leq 3 \trilam \sup_{\textcolor{darkorange}{\bullet}} \lambda^3 \int \mathrel{\raisebox{-0.25 cm}{\includegraphics{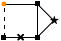}}}
				+ 6\lambda^4 \int \Big( \mathrel{\raisebox{-0.25 cm}{\includegraphics{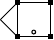}}}
				\Big( \sup_{\textcolor{darkorange}{\bullet}, \textcolor{lblue}{\bullet}}
					\int \mathrel{\raisebox{-0.25 cm}{\includegraphics{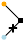}}} \Big) \Big) \\
		&\leq 3 \trilam \sup_{\textcolor{darkorange}{\bullet}} \lambda^3 \int \Big( \Big(\sup_{\textcolor{lblue}{\bullet}, \textcolor{green}{\bullet}} \int
					\mathrel{\raisebox{-0.25 cm}{\includegraphics{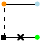}}} \Big)
					\mathrel{\raisebox{-0.25 cm}{\includegraphics{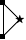}}} \Big) + 6 W_\lambda(k) \trilam\trilamo \\
		& \leq 3W_\lambda(k) \trilam^2 + 6 W_\lambda(k) \trilam\trilamo \leq 9 W_\lambda(k) \trilam\trilamo. }
Similarly, symmetry for $(j_0,j_1)=(1,2)$ gives
	\[ \lambda^3 \int \mathrel{\raisebox{-0.25 cm}{\includegraphics{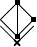}}} 
		\ \leq 4\lambda^3 \int \mathrel{\raisebox{-0.25 cm}{\includegraphics{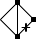}}}
		\ \leq 4\lambda^3 \int \Big( \mathrel{\raisebox{-0.25 cm}{\includegraphics{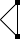}}}
			\Big( \sup_{\textcolor{darkorange}{\bullet}, \textcolor{lblue}{\bullet}} \int \mathrel{\raisebox{-0.25 cm}{\includegraphics{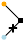}}} \Big)\Big)
		\leq 4 W_\lambda(k) \trilam. \]
By substitution, we can reduce the diagrams $(j_0,j_1)=(2,1)$ to the one from $(j_0,j_1)=(1,1)$, as
	\[ \lambda^4 \int \mathrel{\raisebox{-0.25 cm}{\includegraphics{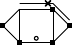}}}
		\ = \lambda^4 \int \mathrel{\raisebox{-0.25 cm}{\includegraphics{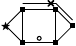}}} 
		\ = \lambda^4 \int \mathrel{\raisebox{-0.25 cm}{\includegraphics{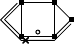}}} \ \leq 16 W_\lambda(k) \trilam\trilamo. \]
The diagram $(j_0,j_1)=(2,2)$, which is $\lambda^2\int \triangle(\orig,w,u) W_\lambda(u;k) \dd u \dd w$, can be bounded directly by $W_\lambda(k) \trilam$. When $(j_0,j_1)=(3,1)$, we see a direct edge, which is pictorially represented with an extra `$\sim$'. The diagram therefore is
	\al{ \lambda^3 \int \mathrel{\raisebox{-0.25 cm}{\includegraphics{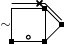}}} &
		\ \leq 2 \lambda^3 \bigg[ \int \mathrel{\raisebox{-0.25 cm}{\includegraphics{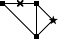}}}
			+ \lambda \int \mathrel{\raisebox{-0.25 cm}{\includegraphics{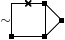}}}
			+ \int \mathrel{\raisebox{-0.25 cm}{\includegraphics{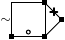}}} \bigg] \\
		& \leq 2W_\lambda(k) \trilam + 2 \lambda^4 \int \Big( \mathrel{\raisebox{-0.25 cm}{\includegraphics{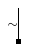}}}
					\Big( \sup_{\textcolor{darkorange}{\bullet}} \int \mathrel{\raisebox{-0.25 cm}{\includegraphics{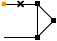}}} \Big)\Big)
					 + 2\lambda W_\lambda(k) (\connf\star\trilamo)(\orig) \\
		& \leq 2 W_\lambda(k) \big( 2\trilam + \lambda(\connf\star\trilam)(\orig)\big) +2 \lambda^4 \Big(\int\connf(x) \dd x \Big) \sup_{\textcolor{darkorange}{\bullet}}
					 \int \mathrel{\raisebox{-0.25 cm}{\includegraphics{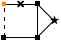}}}  \\
		& \leq 2W_\lambda(k) \Big(2 \trilam + \lambda\trilam\int\connf(u)\dd u\Big) +2 \lambda^4 \sup_{\textcolor{darkorange}{\bullet}}
					 \int \Big( \Big( \sup_{\textcolor{lblue}{\bullet}, \textcolor{green}{\bullet}} \int
					 	\mathrel{\raisebox{-0.25 cm}{\includegraphics{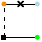}}} \Big)
					  	\mathrel{\raisebox{-0.25 cm}{\includegraphics{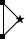}}} \Big) \\
		& \leq 4W_\lambda(k) \trilam (1+\lambda) . }
When $(j_0, j_1)= (3,2)$, we apply Observation~\ref{obs:tildebounds} to get a bound of the form
	\al{\lambda^2 (\tklam\star\tlam\star\connf)(\orig) & = \lambda^2 (\connf_k\star\tlam\star\connf)(\orig) + \lambda^3 \int [1-\cos(k\cdot x)](\connf\star\tlam)^2(x) \dd x \\
		& \leq \lambda^2 (\connf_k\star\tlam\star\connf)(\orig) + 2 \lambda^3 \Big( (\tklam\star\tlam\star\connf^{\star 2})(\orig) + (\connf_k\star\tlam^{\star 2}\star\connf)(\orig) \Big), }
where the Cosine-split Lemma~\ref{lem:cosinesplitlemma} was used in the second line to distribute the factor $[1-\cos(k\cdot x)]$ over $(\connf\star\tlam)$. We still have to handle the first summand. To this end, we use $1 = \mathds 1_{\{|x| \geq \varepsilon\}} + \mathds 1_{\{|x| < \varepsilon\}}$ and obtain
	\al{\lambda^2(\connf_k\star\tlam\star\connf)(\orig) & \leq \lambda^2 \int \connf_k(x) \Big( \mathds 1_{\{|x| \geq \varepsilon\}} \trilamo(x) \Big) \dd x
											+ \lambda^2 \int \tlam^{(\varepsilon)}(x) [1-\cos(k\cdot x)](\tlam\star\tlam)(x) \dd x \\
		& \leq \lambda \trilame \int \connf_k(x) \dd x + 4 \lambda^2 \int \tlam^{(\varepsilon)}(x) (\tklam\star\tlam)(x) \dd x \\
		& \leq \lambda \trilame \int \connf_k(x) \dd x + 4 \big(\ballep\big)^2 W_\lambda(k),}
where, again, we have used the Cosine-split Lemma~\ref{lem:cosinesplitlemma} in the second line. By just the bound from the first line, we get $\lambda^2(\connf_k\star\tlam\star\connf)(\orig) \leq \lambda \trilamo \int \connf_k(x) \dd x$. To finish the proof, note that $\int \connf_k(x) \dd x = 1 - \fconnf(k)$. Carefully putting together all six bounds gives the statement.
\end{proof}

\begin{appendix}
\section{Random walk properties} \label{sec:appendix}

We first prove Proposition~\ref{thm:prelim:conn_fct_examples}. For simplicity of presentation, we 
assume throughout the appendix that $\phiint =1$ (see Section~\ref{sec:rescaling} for an explanation why we may do so).

\begin{proof}[Proof of Proposition~\ref{thm:prelim:conn_fct_examples} (b)]
\col{By definition, (H2.1) and (H2.2) hold}. It remains to prove (H2.3). Note that the constants $b,c_1,c_2$ do not depend on $L$, but on the function $h$ and hence implicitly also on the dimension $d$. (The function $h$ is fixed throughout the proof.)

\col{ Note first that $\fconnf_L(k)=\widehat h(Lk)$. Hence, without loss of generality, we prove (H2.3) for $\connf=h$ and $L=1$. }

For each $k\in\Rd$, the first and second partial derivatives of
$\widehat h$ are given by
\begin{align*}
\frac{\partial}{\partial k_j}\widehat h(k)=\i\int x_j h(x) \e^{\i k\cdot x}\dd x, 	
\quad \frac{\partial^2}{\partial k_j \partial k_l} \widehat h(k)=-\int x_jx_l h(x) \e^{\i k\cdot x}\dd x. 
\end{align*}
Thus, we obtain from the multivariate Taylor theorem (applied to the
origin) with the Peano form of the remainder term that
\begin{align*} 
\widehat{h}(k)=1-\frac{1}{2}\sum^d_{j,l=1}k_jk_l\int x_jx_l h(x)\dd x
-\frac{1}{2}\sum^d_{j,l=1}k_jk_l\int x_jx_lh(x)\big(\e^{\i s k\cdot
  x}-1\big)\dd x, 
\end{align*}
where $s=s(k)\in(0,1)$, and where we have used that
$\int x_jh(x)\dd x=0$ for all $j\in\{1,\ldots,d\}$. 
The symmetric matrix $M$ with entries $\int x_jx_lh(x)\dd x$ is positive semidefinite.
Since we have assumed this matrix to be regular we have that
\begin{align*}
\sum^d_{j,l=1}k_jk_l\int x_jx_l h(x)\dd x\ge c|k|^2,
\end{align*}
where $c>0$ is the smallest eigenvalue of $M$.
By dominated convergence,
$\int x_jx_lh(x)(\e^{\i k'\cdot x}-1)\dd x\to 0$ as $|k'|\to 0$ for
all $j,l\in\{1,\ldots,d\}$, so that the first inequality in (H2.3)
holds for a suitable $c_1$ and all sufficiently small $b>0$. The
second inequality in (H2.3) then follows from the fact that
$\widehat{h}$ is bounded away from $1$ outside any compact
neighborhood of the origin. (Otherwise $h(x)=0$ for almost every
$x\in\Rd$, a contradiction.)
\end{proof}

The next proofs will use the following classical fact on the Fourier transform of the indicator function of a ball; see e.g.~\cite[Section B.5]{Gra14}. Let $r>0$ and $g_r$ be the indicator function of the ball in $\Rd$ with radius $r$ centered at the origin. Then
	\eqq{ \widehat{g}_r(k)=\Big(\frac{2\pi r}{|k|} \Big)^{d/2} J_{d/2}(|k|r), \quad k\in\R^d, \label{eq:app:FT_indicator_ball_bessel} }
where, for $a>-1/2$, the {\em Bessel function}
$J_a\colon\R_{\ge 0}\to\R$ is given by 
	\[ J_a(x):=\sum^\infty_{m=0}\frac{(-1)^m}{m!\,\Gamma(m+a+1)}\Big(\frac{x}{2}\Big)^{2m+a}, \quad x\ge 0. \]
It is helpful to note here that $b_d:=\pi^{d/2} \Gamma(d/2+1)^{-1}$ is the volume of a ball in $\Rd$ with radius 1 \col{and $r_d := \pi^{-1/2} \Gamma(d/2+1)^{1/d}$ is the radius of the unit volume ball (we denote the latter by $\unitball$).}

\begin{proof}[Proof of Proposition~\ref{thm:prelim:conn_fct_examples}(c)] \col{Since we assume that $\phiintL =1$, we assume $h$ and $\connf_L$ to be given as}
	\[ \col{ h(x) = (|x| \vee r_d)^{-d-\alpha} \qquad \text{and} \qquad \connf_L(x) = \frac{h(x/L)}{\int h(x/L)\dd x}. }\]
Note that we only need to consider $\alpha \leq 2$, as the other case is covered in the spread-out (finite-variance) model. We only consider the first bound in (H3.3), as the other properties follow similarly to (H2).

Given $L \geq 1$ and $k\in\R^d$, we again use that $\fconnf_L(k)=\fconnf_1(Lk)=c_h\widehat{h}(Lk)$, where
	\[ c_h^{-1}=\int h(x)\dd x=\int_{|x|\le r_d} (r_d)^{-d-\alpha} \dd x+\int_{|x| >r_d} |x|^{-d-\alpha} \dd x =(r_d)^{-d-\alpha}  + \frac{d b_d}{\alpha (r_d)^\alpha}, \]
and where we have used polar coordinates. Therefore it is no loss of generality to assume $L=1$. \col{Our goal is thus to control $1-\widehat h(k)$. We do this by splitting area of integration into three regions.} By \eqref{eq:app:FT_indicator_ball_bessel},
	\[ \int_{|x|\le r_d}\big(1-\e^{\i k\cdot x}\big)\dd x = \bar b_d-\sum^\infty_{m=0}\frac{(-1)^m \pi^{d/2} (r_d)^d}{m!\Gamma(m+\frac d2 +1)} \Big(\frac{|k| r_d}{2} \Big)^{2m} 
			=\sum^\infty_{m=1}\frac{(-1)^{m+1} \pi^{d/2} (r_d)^d}{m!\Gamma(m+\frac d2 +1)} \Big(\frac{|k| r_d}{2} \Big)^{2m}, \]
which is \col{$\mathcal O(|k|^2) = o(|k|^\alpha)$.}

\col{By choosing $b$ in (H3.3) sufficiently small, we may assume that $r_d |k| < 1$. We use that $1-\cos(k\cdot x) \geq c |k|^2|x|^2$ for $|x| < 1 / |k|$ and obtain
	\[ \int_{r_d < |x| < 1 / |k|} [1-\cos(k\cdot x)] h(x) \dd x \geq c |k|^2 \int_{r_d < |x| < 1 / |k|} |x|^{2-d-\alpha} \dd x.  \]
for some constant $c$. Moreover,
	\[ |k|^2 \int_{r_d < |x| < 1 / |k|} |x|^{2-d-\alpha} \dd x = \begin{cases} c' |k|^\alpha & \mbox{for } \alpha< 2, \\
				c' |k|^2 \log \frac{1}{|k|} & \mbox{for } \alpha =2, \end{cases}  \]
for some constant $c'$. }

\col{Lastly, using $1-\cos(k\cdot x) \leq 2$,
	\[ \int_{|x| > 1 / |k|} [1-\cos(k\cdot x)] h(x) \leq 2 \int_{|x| >1 / |k|} |x|^{-d-\alpha} \dd x = c'' |k|^\alpha, \]
for some constant $c''$. This completes the proof} \end{proof}

Before giving the proof of Proposition~\ref{thm:prelim:conn_fct_examples}(a), we make the following observation about $\connf=\mathds 1_{\unitball}$:
\begin{observation}\label{obs:appendix:unitball_convolutions}
Let $\varepsilon>0$, $m \geq 3$ and $\connf = \mathds 1_{\unitball}$. Then there exists $\rho \in (0,1)$ and $C>0$ such that
	\[ \text{(i) } \sup_{x\in\Rd} \connf^{\star m} (x) \leq C \rho^d, \qquad \text{(ii) } \sup_{x\in\Rd: |x| \geq \varepsilon} \connf^{\star 2} (x) \leq C \rho^d.\]
\end{observation}
\begin{proof}
Throughout this proof, $|\cdot |$ refers to the Lebesgue measure for sets as well as to the Euclidean norm for vectors. To prove $(i)$, we note that $\sup_x \connf^{\star (m+1)}(x) \leq \sup_x \connf^{\star m}(x)$, so we only consider $m=3$. Next, note that the supremum is in fact a maximum, attaining its maximal value at $x=\orig$. This follows, for example, from the logconcavity of $\connf$ (see, e.g.,~\cite[Theorem 2.18]{DhaJoa88})

Recall that $r_d = \pi^{-1/2} \Gamma(\tfrac d2 + 1)^{1/d}$. We use $B(x,r)$ to denote a ball around $x$ of radius $r$, so that $\unitball(x) = B(x,r_d)$. Fix $\delta \in (0,1)$. We see that
	\al{ \connf^{\star 3} (\orig) &= \int \mathds 1_{\unitball}(y) \int \mathds 1_{\unitball}(y-z) \mathds 1_{\unitball} (-z) \dd z \dd y = \int_{\unitball} \int_{\unitball} \mathds 1_{\unitball}(y-z) \dd z \dd y \\
		&= \int_{\unitball} | \unitball(\orig)\cap \unitball(y) | \dd y \leq \int_{\delta \unitball} 1 \dd z 
							+ |\unitball \setminus \delta \unitball| \sup_{y \in \unitball \setminus \delta\unitball} | \unitball(\orig)\cap \unitball(y) | . }
The first term is simply $\delta^d$. We estimate $| \unitball \setminus \delta\unitball| \leq 1$ and are left to treat $|\unitball(\orig) \cap \unitball(y)| =: f(|y|)$. Note that $f(t)$ is non-increasing in $t$, and so the supremum is attained for $t=\delta r_d$. Let $y$ be a point with $|y|=\delta r_d$, for example $y=(\delta r_d, 0, \ldots, 0)$. We claim that
	\eqq{ \unitball(\orig) \cap \unitball(y) \subseteq B\Big(y/2, \sqrt{r_d^2 - |y|^2/4} \Big) = B\Big(y/2, r_d \sqrt{1- \delta^2/4} \Big) = y/2 + \sqrt{1-\delta^2/4} \; \unitball. \label{eq:appendix:intersection_2_balls}}
Assuming this claim, $(i)$ follows directly from $|\sqrt{1-\delta^2/4} \; \unitball| = (1-\delta^2/4)^{d/2}$. It remains to prove~\eqref{eq:appendix:intersection_2_balls}.

Let $x \in \unitball(\orig) \cap \unitball(y)$. Due to symmetry, we \ch{may assume that $x_1 \geq t/2$  (where $t=|y|=\delta r_d$)}. Now, recalling $y=(t, 0 \ldots, 0)$,
	\[ | x- y/2|^2 = (x_1 - t/2)^2 + \sum_{i=2}^{d} x_i^2 = | x |^2 - x_1 t + \tfrac 14 t^2 \leq r_d^2 - t^2/4,\]
which proves $x \in B(y/2, \sqrt{r_d^2 - t^2/4})$.

To prove $(ii)$, note that, by~\eqref{eq:appendix:intersection_2_balls}, we have $\connf^{\star 2}(x) = |\unitball(\orig) \cap \unitball(x)| = f( | x|) \leq (1-\varepsilon^2/4)^{d/2}$ for $|x| \geq \varepsilon$.
\end{proof}

\begin{proof}[Proof of Proposition~\ref{thm:prelim:conn_fct_examples}(a)]
It is clear that both densities have all moments. The Gaussian case, when $\connf(x)=(2\pi)^{-d/2} \exp(-|x|^2/2)$, is easy to handle. In this case, $\fconnf (k) = \exp(-\tfrac 12 |k|^2)$, and so $1-\fconnf (k) = \tfrac 12 |k|^2 + o(|k|^2)$ \ch{and $\fconnf$ is bounded away from 1 uniformly outside a fixed neighborhood}. Moreover, $\connf^{\star m}(\orig) \leq (\connf \star \connf) (\orig) = (2 \sqrt{\pi})^{-d}$.

The convolution statements for $\connf=\mathds 1_{\unitball}$ are shown in Observation~\ref{obs:appendix:unitball_convolutions}. It remains to prove (H1.3). Taking $k\in\Rd$ and choosing $r=r_d=\pi^{-1/2}\Gamma(d/2+1)^{1/d}$ in \eqref{eq:app:FT_indicator_ball_bessel} gives
	\eqq{ 1-\fconnf(k)=(2\pi)^{d/2}\sum^\infty_{m=1}\frac{(-1)^{m+1}}{m!\Gamma(m+d/2+1)} \frac{|k|^{2m}r^{2m+d}_d}{2^{2m+d/2}}=a_d|k|^{2}+|k|^{4}R_d(k), \label{eq:app:blob_fconnf_bessel} }
where $a_d:=(2\pi)^{d/2}r^{2+d}_d\Gamma(d/2+2)^{-1} 2^{-2-d/2}$ and
	\[ R_d(k):=\col{(2\pi)^{d/2}}\sum^\infty_{m=2}\frac{(-1)^{m+1}}{m!\Gamma(m+d/2+1)} \frac{|k|^{2m-4}r^{2m+d}_d}{2^{2m+d/2}}. \]
Since $\Gamma(d/2+2)=(d/2+1)\Gamma(d/2+1)$,
	\[ a_d=\frac{r^{2}_d}{4(d/2+1)}=\frac{\Gamma(d/2+1)^{2/d}}{4\pi (d/2+1)}. \]
At this stage, we recall the well-known bounds $(2\pi x)^{1/2}(x/ \e)^x<\Gamma(x+1)<(2\pi x)^{1/2}(x/\e)^x \e$, valid for each $x>0$. 
Using the first inequality with $x=d/2$ gives
	\eqq{ a_d>\frac{(\pi d)^{1/d} d}{8\pi \e (d/2+1)}\ge c>0, \label{eq:app:a_d_Gamma_bounds}}
where $c>0$ does not depend on $d$. Further, for $|k|\le 1$,
	\[ |R_d(k)|\le \sum^\infty_{m=2}\frac{r^{2m+d}_d}{m!\Gamma(m+d/2+1)2^{2m+d/2}} =\sum^\infty_{m=2}\frac{\Gamma(d/2+1)^{(2m+d)/d}}{m!\Gamma(m+d/2+1)2^{2m+d/2}\pi^{(2m+d)/2}}. \]
Using the above bounds for the Gamma function, we obtain that
	\[ |R_d(k)| \le \sum^\infty_{m=2}\frac{(\pi d)^{(2m+d)/(2d)}(d/(2 \e))^{(2m+d)/2} \e^{(2m+d)/d}} {m!(2\pi(m+d/2))^{1/2}((m+d/2)/ \e)^{(m+d/2)}2^{2m+d/2}\pi^{(2m+d)/2}}. \]
Since, trivially, $d/(2 \e)\le (m+d/2)/ \e$, we see that
	\al{ |R_d(k)| & \le \sum^\infty_{m=2}\frac{(\pi d)^{(2m+d)/(2d)} \e^{(2m+d)/d}} {m! (2\pi(m+d/2))^{1/2} 2^{2m+d/2}\pi^{(2m+d)/2}}
					\le \sum^\infty_{m=2}\frac{(\pi d)^{m/d} (\pi d)^{1/2} \e^{2m/d} \e} {m!(2\pi)^{1/2}(d/2)^{1/2} 2^{2m}\pi^{m}} \\
			&=\sum^\infty_{m=2}\frac{(\pi d)^{m/d} \e^{2m/d} \e} {m! 2^{2m}\pi^{m}}. }
The latter series is bounded uniformly above in $d$ 
\ch{and $k$, say by some constant $c'$. Together with \eqref{eq:app:blob_fconnf_bessel} and \eqref{eq:app:a_d_Gamma_bounds}, we get that $1-\fconnf(k)\geq c|k|^{2}-c'|k|^{4}>(c/2)|k|^2$ when $|k|^2 c'\leq c/2$, thus proving \eqref{eq:Boolean_D_infraredbound} \chr{with $c_1=c/2$ and $b=c/(2c')$}.}

\ch{For \eqref{eq:Boolean_D_infraredbound2}, we  argue as in the proof of Proposition~\ref{thm:prelim:conn_fct_examples} (b) that $\hat\varphi$ is bounded away from 1 outside a neighborhood of the origin, as otherwise we would get $\varphi\equiv0$ almost everywhere. We thus need to establish uniformity in the dimension. We start from \eqref{eq:app:FT_indicator_ball_bessel}, and sketch the argument here. 
Bessel functions of the first kind satisfy $|J_{d/2}|\le Md^{-1/3}$ (for a $d$-independent constant $M$), and $|J_{d/2}|$ achieves its global maximum at 
$\bar x_{d/2}=d/2+\gamma_1 (d/2)^{1/3}+O(d^{-1/3})$ (for a specific constant $\gamma_1\approx0.81\dots$), \cite{AbrSte70}. Therefore, \chr{using $r_d\sim \sqrt d / \sqrt {2\pi {\rm e}}$,} 
\[ 	|\fconnf(k)|
	=\Big(\frac{2\pi r_d}{|k|} \Big)^{d/2} |J_{d/2}(|k|r_d)|
	\le\Big(\frac{2\pi r_d}{|k|} \Big)^{d/2} Md^{-1/3}
	\to 0\qquad\text{as $d\to\infty$} \]
whenever $|k|r_d\ge \bar x_{d/2}$. 
On the other hand, the relationship between Bessel functions (see \cite[Section B.2]{Gra14}) yields 
\[ \frac{\dd}{\dd |k|}\fconnf(k)=-r_d\left(\frac{2\pi r_d}{|k|}\right)^{d/2} J_{d/2+1}(r|k|),\]
and thus $|k|\mapsto \fconnf(k)$ is decreasing until the first positive zero of $J_{d/2+1}$ located at 
$\bar y_{d/2+1}=(d/2+1)+\gamma_2(d/2+1)^{1/3}+O(d^{-1/3})$ 
with $\gamma_2\approx1.86\dots$, \cite{AbrSte70}.  
Since $\gamma_1<\gamma_2$, and thus $\bar x_{d/2}<\bar y_{d/2+1}$ for $d$ large enough, this proves the desired uniformity. Further details are provided in \cite[Appendix B]{DicHey22}.} 
\end{proof}

The following two propositions are used in a crucial way in Section~\ref{sec:bootstrapanalysis}:

\begin{prop}[Random walk $s$-condition] \label{thm:randomwalkestimates} \
\begin{compactitem}
\item[1.] Let $\connf$ satisfy (H1). Let $s \in \N_0$, $2 \leq m \in\N$, and $\ch{d>2s}$. Then there is a constant $c_s^{\sss\rm{RW}}$ (independent of $d$) such that, for all $\mu\in[0,1]$ and with $\beta =g(d)^{1/4}$ as in~\eqref{eq:def:beta},
	\[ \int \frac{|\fconnf(k)|^m}{[1-\mu\fconnf(k)]^s} \frac{\dd k}{(2\pi)^d} \leq c_s^{\sss\rm{RW}} \beta^{2((m \wedge 3)-2)}. \]
\item[2.] Let $\connf$ satisfy (H2) or (H3). Let $s \in \N_0$, $2 \leq m \in\N$, and $d>(\alpha \wedge 2)s$ \col{(if $\alpha=2$, either $d > 2s$ or $d=2s \geq 4$)}. Then there is a constant $c_s^{\sss\rm{RW}}$ (independent of $L$) such that, for all $\mu\in[0,1]$ and with $\beta=L^{-d}$ as in~\eqref{eq:def:beta},
	\[ \int \frac{|\fconnf(k)|^m}{[1-\mu\fconnf(k)]^s} \frac{\dd k}{(2\pi)^d} \leq c_s^{\sss\rm{RW}} \beta. \]
\end{compactitem}
\end{prop}

\begin{prop}[Related Fourier integrals] \label{thm:randomwalkestimates2}
\col{Let \ch{$d>6$} for (H1) and $d> 3(\alpha \wedge 2)$ for (H2) and (H3).} Then for $n\in \{1,2,3\}$ and $m\in\{2,3\}$, uniformly in $\mu \leq 1$ and $k \in \Rd$, and with $\beta$ as in~\eqref{eq:def:beta},
	\algn{ \int |\fconnf(l)|^m \fgreensmu(l)^{3-n} [\fgreensmu (l+k) + \fgreensmu (l-k)]^n \frac{\dd l}{(2 \pi)^d} &\leq 2^n c_3^{\sss\rm{RW}} \beta^{m-2}, \label{eq:RWrelatedtriangles1} \\
		\int |\fconnf(l)|^m \fgreensmu(l) \fgreensmu (l+k) \fgreensmu (l-k) \frac{\dd l}{(2 \pi)^d} &\leq c_3^{\sss\rm{RW}} \beta^{m-2}, \label{eq:RWrelatedtriangles2} }
where $c_3^{\sss\rm{RW}}$ is as in Proposition~\ref{thm:randomwalkestimates}.
\end{prop}

\begin{proof}[Proof of Proposition~\ref{thm:randomwalkestimates} for the finite-variance model (H1).]	
We first note that w.l.o.g. we can restrict to considering $m\in\{2,3\}$. 

Note that as soon as $\mu<1$, the denominator is bounded away from zero and the boundedness follows from the integrability of the enumerator. To get a bound that is uniform in $\mu$, we set $\mu=1$.

\ch{We split the area of integration and first consider $\{k\colon |k| \leq b\}$, where we choose $b>0$ as in \eqref{eq:Boolean_D_infraredbound}. 
Applying~\eqref{eq:Boolean_D_infraredbound},
	\[ 
	\int_{|k| \leq b} \frac{|\fconnf(k)|^m}{[1-\fconnf(k)]^{s}} \frac{\dd k}{(2\pi)^d} \leq c_1^{-s} \int_{|k|\leq b} |k|^{-2s} \frac{\dd k}{(2\pi)^d}. 
	\]
\chr{The RHS is $c_1^{-s}b^{d-2s}|S_{d-1}|$ (where $|S_{d-1}|=d\pi^{d/2}/\Gamma(d/2+1)$ is the surface of the $d$-dimensional sphere), and thus finite if $d>2s$ and smaller than a constant times $\beta^2=g(d)^{1/2}$ by our assumption that $g(d)$ decays at most exponentially.} 
For $|k|>b$, we have $1-\fconnf(k) >c_2>0$ by \eqref{eq:Boolean_D_infraredbound2}. 
Consequently,
	\[ \int_{|k| > b} \frac{|\fconnf(k)|^m}{[1-\fconnf(k)]^{s}} \frac{\dd k}{(2\pi)^d} 
						\leq c_2^{-s} \int |\fconnf(k)|^m \frac{\dd k}{(2\pi)^d}.
	\]
This integral is just $(\connf\star\connf)(\orig)\leq 1$ for $m=2$. For $m=3$, by Cauchy-Schwarz, \chr{with $c=c_1\wedge c_2$,}
 	\[
	c^{-s} \int |\fconnf(k)|^3 \frac{\dd k}{(2\pi)^d}\leq c^{-s} \left(\int \fconnf(k)^{4} \frac{\dd k}{(2\pi)^d}\right)^{1/2}\left(\int \fconnf(k)^{2} \frac{\dd k}{(2\pi)^d}\right)^{1/2}.
	\]
The integral in the first factor is $\connf^{\star 4}(\orig) = \mathcal O(\beta^4)$, the integral in the second factor is $(\connf\star\connf)(\orig)\leq 1$.
}						
\end{proof}

\begin{proof}[Proof of Proposition~\ref{thm:randomwalkestimates} for the spread-out models (H2), (H3).]
Again, we set $\mu=1$, since otherwise the statement is clear. We consider the regions $|k| \leq b/L$ and $|k| > b/L$. Applying (H2.3) (resp., (H3.3)),
	\eqq{ \int_{|k| \leq b L^{-1}} \frac{|\fconnf_L(k)|^m}{[1-\fconnf_L(k)]^s} \frac{\dd k}{(2\pi)^d} \leq \frac{1}{c_1^s L^{(\alpha \wedge 2) s}} \int_{|k| \leq b L^{-1}} 
				\frac{1}{|k|^{(\alpha \wedge 2) s}} \frac{\dd k}{(2\pi)^d} \leq C L^{-d} \label{eq:RW_Fourier_bound_spreadout}}
for $d > (\alpha \wedge 2)s$. Note that $C$ depends on $b,c_1, \alpha, d$, but not on $L$. \col{If $\alpha=2$ and $d=2s$, then we can replace the bound on the right-hand side of~\eqref{eq:RW_Fourier_bound_spreadout} by
	\[\frac{1}{c_1^s L^{d}} \int_{|k| \leq b L^{-1}} \frac{1}{|k|^d  \log(\frac{\pi}{2L|k|})^{d/2} }\frac{\dd k}{(2\pi)^d} \leq C L^{-d}, \]
so long as $d \geq 4$.} For the second region,
	\al{ \int_{|k| > b L^{-1}} \frac{|\fconnf_L(k)|^m}{[1-\fconnf_L(k)]^s} \frac{\dd k}{(2\pi)^d} &\leq c_2^{-s} \int |\fconnf_L(k)|^m\frac{\dd k}{(2\pi)^d} \leq c_2^{-s} (\connf_L\star\connf_L)(\orig) \\
		&\leq c_2^{-s} \| \connf_L \|_\infty \int \connf_L(x) \dd x \leq C L^{-d} ,}
using (H2.2) in the last bound.
\end{proof}

\begin{proof}[Proof of Proposition~\ref{thm:randomwalkestimates2}]
We show the proof for (H1); the spread-out models work similarly. To prove \eqref{eq:RWrelatedtriangles2}, we point out that
	\eqq{ \fgreensmu(l \pm k) = \int [\cos(l \cdot x) \cos (k \cdot x) \mp \sin(l \cdot x) \sin(k \cdot x)] \greensmu(x) \dd x. \label{eq:Greenfct-trigonometric-sum-identity}}
Setting $G_{\mu,k}(x) = \cos(k \cdot x) \greensmu (x)$, we thus have
	\eqq{ 0 \leq \fgreensmu(l-k) \fgreensmu(l+k) = \widehat G_{\mu,k} (l) ^2 - \left( \int \sin(l \cdot x) \sin(k \cdot x) \greensmu(x) \dd x \right)^2 \leq \widehat G_{\mu,k} (l)^2, \label{eq:RWestimates2_C_mu,k}}
and so, using the Cauchy-Schwarz inequality, \eqref{eq:RWrelatedtriangles2} is bounded by
	\algn{ & \left( \int \fconnf(l)^{2m-2} \fgreensmu(l) \widehat G_{\mu,k}(l)^2 \frac{\dd l}{(2\pi)^d} \right)^{1/2} \left(\int \fconnf(l)^2 \fgreensmu(l) \widehat G_{\mu,k}(l)^2 \frac{\dd l}{(2\pi)^d} \right)^{1/2} 
						\notag\\
			 & \qquad = \Big(\big( \connf^{\star 2m-2} \star \greensmu \star G_{\mu,k}^{\star 2} \big)(\orig)
			 				 \big( \connf^{\star 2} \star \greensmu \star G_{\mu,k}^{\star 2} \big) (\orig) \Big)^{1/2}. \label{eq:RWestimates2_C_mu,k_bound} }
Using that $\big( \connf^{\star 2m-2} \star \greensmu \star G_{\mu,k}^{\star 2} \big) (\orig) \leq \big( \connf^{\star 2m-2} \star \greensmu^{\star 3} \big) (\orig)$, we continue as in the proof of Proposition~\ref{thm:randomwalkestimates}. By \eqref{eq:Greenfct-trigonometric-sum-identity}, $\fgmu(l+k)+\fgmu(l-k) = 2 \widehat G_{\mu, k}(l)$, and so we can write \eqref{eq:RWrelatedtriangles1} as
	\[\int |\fconnf(l)|^m \fgreensmu(l)^{3-n} 2^n \widehat G_{\mu,k}(l)^n \frac{\dd l}{(2\pi)^d},\]
which is bounded analogously to above.
\end{proof}

\end{appendix}

\paragraph{Acknowledgments.} We are most grateful to Matthew Dickson and Mathew Penrose for comments on an earlier version of the manuscript. The work of RvdH is supported by the Netherlands Organisation for Scientific Research (NWO) through the Gravitation {\sc Networks} grant 024.002.003. MH and GL acknowledge the support of the German Science Foundation (DFG) through the research group ``Geometry and Physics of Spatial Random Systems'' (FOR 1548) and priority program ``Random Geometric Systems'' (SPP 2265).

\providecommand{\bysame}{\leavevmode\hbox to3em{\hrulefill}\thinspace}
\providecommand{\MR}{\relax\ifhmode\unskip\space\fi MR }
\providecommand{\MRhref}[2]{%
  \href{http://www.ams.org/mathscinet-getitem?mr=#1}{#2}
}
\providecommand{\href}[2]{#2}

\end{document}